\documentclass[11pt]{amsart}

\def\flnm{caf-}
\def\version{12}
\def\datum{Preprint, April 27, 2014}

\let\iflabels\iffalse
\let\ifcitenumber\iftrue
\let\iffilename\iffalse
\let\iftxfnts\iftrue
\let\ifscrltx\iffalse
\let\ifindexindication\iffalse

\usepackage{epsf}

\iflabels
\usepackage[notcite,notref]{showkeys}
\fi

\ifscrltx
\usepackage{srcltx}
\fi

\iftxfnts
\usepackage[varg]{txfonts}
\fi

\usepackage[all]{xy}

\usepackage{multicol}

\iffilename
\newcommand\subsnb{\ifnum\arabic{section}>0\quad\S\arabic{section}%
\ifnum\arabic{subsection}=0
\else.\arabic{subsection}\fi\fi}
\makeatletter\def\@oddfoot{\rm{\footnotesize{File
{\tt\flnm\version.tex}\subsnb}\qquad\today\hfil \thepage}}
\let\@evenfoot\@oddfoot 
\makeatother
\fi

\ifindexindication
 
 \newcommand\il[2]{\label{inl-#1}\indexindication{#2}{l}}
 \newcommand\ir[2]{\indexindication{$#2$}{r}}
\else
 \newcommand\il[2]{\label{inl-#1}}
 \newcommand\ir[2]{}
\fi

\usepackage{color}

\definecolor{blue}{rgb}{0,0,1}
\definecolor{red}{rgb}{1,0,0}
\definecolor{green}{rgb}{0,.6,.2}
\definecolor{purple}{rgb}{1,0,1}

\long\def\red#1\endred{\textcolor{red}{#1}}
\long\def\blue#1\endblue{\textcolor{blue}{#1}}
\long\def\purple#1\endpurple{\textcolor{purple}{ #1}}
\long\def\green#1\endgreen{\textcolor{green}{#1}}

\newtheorem{thm}{Theorem}[section]
\newtheorem{prop}[thm]{Proposition}
\newtheorem{cor}[thm]{Corollary}
\newtheorem{lem}[thm]{Lemma}
\newtheorem{mainthm}{Theorem}

\newcounter{rst}

\theoremstyle{definition}
\newtheorem{defn}[thm]{Definition}
\newtheorem{rmk}[thm]{Remark}

\newcommand\rmrk[1]{\medskip\par\noindent\emph{#1. }}
\newcommand\rmrkn[1]{\medskip\par\noindent\emph{#1 }}
\newcommand\rmrke{\medskip\par\noindent\emph{Remark. }}
\newcommand\rmrks{\medskip\par\noindent\emph{Remarks. }}
\newcounter{rmrkcnt}
\renewcommand\thermrkcnt{(\alph{rmrkcnt})}
\newcommand\itmi{\setcounter{rmrkcnt}{0}%
\refstepcounter{rmrkcnt}\thermrkcnt\ \ }
\newcommand\itm{\smallskip\par\noindent%
\refstepcounter{rmrkcnt}\thermrkcnt\ \ }

\newcommand\al{\alpha}
\newcommand\bt{\beta}
\newcommand\Gm{\Gamma}
\newcommand\gm{\gamma}
\newcommand\Dt{\Delta}
\newcommand\dt{\delta}
\newcommand\e{\varepsilon}
\newcommand\eps{\epsilon}
\newcommand\zt{\zeta}
\newcommand\Eta{{\mathsf H}}
\renewcommand\th{\vartheta}
\renewcommand\k{\kappa}
\newcommand\Ld{\Lambda}
\newcommand\ld{\lambda}
\newcommand\s{\sigma}
\newcommand\Ph{\Phi}
\newcommand\ph{\varphi}
\newcommand\ch{\chi}

\newcommand\ps{\psi}
\newcommand\Om{\Omega}
\newcommand\om{\omega}

\newcommand\EE{{\mathbf{E}}}
\newcommand\HH{{\mathbf{H}}}
\newcommand\VV{{\mathbf{V}}}
\newcommand\WW{{\mathbf{W}}}
\newcommand\XX{{\mathbf{X}}}
\newcommand\YY{{\mathbf{Y}}}

\newcommand\CC{\mathbb{C}}

\newcommand\QQ{\mathbb{Q}}
\newcommand\RR{\mathbb{R}}
\newcommand\ZZ{\mathbb{Z}}

\renewcommand\setminus{\smallsetminus}
\let\ovln\=
\renewcommand\={\;=\;}
\newcommand\SL{{\mathrm{SL}}}
\newcommand\PSL{{\mathrm{PSL}}}

\newcommand\SO{{\mathrm{SO}}}
\newcommand\oh{{\mathrm{O}}}
\newcommand\re{\mathrm{Re}\,}
\newcommand\im{\mathrm{Im}\,}
\newcommand\proj[1]{{\mathbb P}^1_{#1}}
\newcommand\hp{{\mathfrak{H}}}
\newcommand\uhp{\hp}
\newcommand\lhp{\hp^-}
\newcommand\coh[2]{\mathbf{r}_{#1}^{#2}}

\newcommand\bcoh[2]{\mathbf{q}_{#1}^{#2}}
\newcommand\parb{{\scriptscriptstyle\mathrm{pb}}}
\newcommand\hpar{H_\parb}
\newcommand\zpar{Z_\parb}
\newcommand\fs{{\om^\ast}}
\newcommand\fsn{{\om^0}}
\newcommand\wdg{{\mathrm{exc}}}
\let\exc\wdg

\newcommand\cusp[1]{S_{\!#1}}
\newcommand\harm{{\mathrm{Harm}}}
\newcommand\sharm{{\mathcal{H}}}
\newcommand\sharmb[1]{{\mathcal{H}^\mathrm{b}_{#1}}}
\newcommand\sharmh[1]{{\mathcal{H}^\mathrm{h}_{#1}}}
\newcommand\shad{{\xi}}
\newcommand\Ci{\iota}
\newcommand\Gmod{{\Gamma(1)}}
\newcommand\tGmod{{\widetilde{\Gmod}}}
\newcommand\bsing{\mathsf{BdSing}\,}
\newcommand\singr{\mathsf{Sing}_r\,}
\newcommand\tG{{\tilde G}}

\newcommand\tZ{{\tilde Z}}
\newcommand\tP{{\tilde P}}
\newcommand\tGm{{\tilde\Gm}}

\newcommand\hol{{\mathcal{O}}}

\newcommand\prj{{\scriptscriptstyle\mathrm{prj}}}
\newcommand\Prj[1]{\mathsf{prj}_{#1}}
\newcommand\indlim{{\displaystyle\lim_{\longrightarrow}\,}}

\newcommand\Gf{\Gamma} 

\newcommand\smp{{\mathrm{smp}}}

\newcommand\hypg[2]{{}_{#1}\!F_{\!#2}}

\newcommand\rs{\rho}
\newcommand\rsp{\rs^\prj}
\newcommand\Gr[2]{\mathcal{G}_{\!#1}^{#2}}
\newcommand\N[2]{\mathcal{N}_{\!#1}^{#2}}
\newcommand\I[2]{\mathcal{I}_{\!#1}(#2)}
\newcommand\tess{{\mathcal T}}
\newcommand\fd{{\mathfrak F}}
\newcommand\fdcu{\fd^{\mathrm{cu}}}

\newcommand\pnt{{\hbox{.}}} 

\newcommand\cu{{\mathcal{C}}}
\newcommand\Q{{\mathcal{Q}}}
\newcommand\R{{\mathcal{R}}}
\newcommand\qA{{}^q\!A}
\newcommand\qC{{}^q\!C}
\newcommand\pr{{\mathrm{pr}}}
\newcommand\sign{{\mathrm{Sign}\,}}
\newcommand\pol{{\mathrm{pol}}}
\newcommand\Pcal{{\mathcal P}}
\newcommand\av[1]{\mathrm{A\!v}_{\!#1}}
\newcommand\avGm[1]{\mathrm{A\!v}_{\Gm,#1}}
\newcommand\cond{{\mathrm{cond}}}
\newcommand\Id{{\mathrm{I}}}
\newcommand\E{{\mathcal E}}
\newcommand\F[1]{F_{\!#1}}
\newcommand\M[1]{\mathrm{M}_{#1}}

\renewcommand\H[1]{\mathrm{H}_{\!#1}}
\renewcommand\P[1]{\mathrm{P}_{\!#1}}
\newcommand\ca{\mathfrak a}
\newcommand\cb{\mathfrak b}

\newcommand\glie{{\mathfrak g}}
\newcommand\hk[1]{[\![#1]\!]}

\newcommand\V[2]{{\mathcal V}_{\!#1}^{#2}}
\newcommand\dsv[2]{\mathcal{D}_{\!#1}^{#2}}

\newcommand\bg{\mathcal{B}}
\newcommand\W[2]{{\mathcal W}_{\!#1}^{#2}}
\newcommand\Whol[2]{{}^{\mathrm{h}}{\mathcal W}_{\!#1}^{#2}}
\newcommand\esv[2]{\mathcal{E}_{#1}^{#2}}
\newcommand\Sg[2]{{\mathcal S}_{#1}^{#2}}

\newcommand\be{\begin{equation}}
\newcommand\ee{\end{equation}}
\newcommand\bad{\be\begin{aligned}}
\newcommand\ead{\end{aligned}\ee}
\newcommand\badl[1]{\be\label{#1}\begin{aligned}}
\newcommand\eadl{\end{aligned}\ee}

\makeatletter
\newcommand\hmatc[4]{\left[ {#1\@@atop #3}{#2\@@atop #4}\right]}
\newcommand\hmatr[4]{\left[ {\hfill #1\@@atop\hfill #3}{\hfill
#2\@@atop\hfill #4}\right]}
\newcommand\matc[4]{\left( {#1\@@atop #3}{#2\@@atop #4}\right)}
\newcommand\matr[4]{\left( {\hfill #1\@@atop\hfill #3}{\hfill
#2\@@atop\hfill #4}\right)} \makeatother

\newcommand\wdth{.475\textwidth}

\newcommand\twocolwithpictl[2]{\par\noindent
\refstepcounter{figure}
\begin{minipage}{\wdth}\hspace*{\fill}$#1$\hspace*{\fill}\bigskip\\
\hspace*{\fill}{\sc Figure \arabic{figure}}\hspace*{\fill}
\end{minipage}
\hspace{\fill}
\begin{minipage}{\wdth}#2\end{minipage}}

\newcommand\twocolwithpictr[2]{\par\noindent
\refstepcounter{figure}\hspace*{-1mm}
\begin{minipage}{\wdth}#1\end{minipage}
\hspace*{\fill}
\begin{minipage}{\wdth}\hspace*{\fill}$#2$\hspace*{\fill}\bigskip\\
\hspace*{\fill}{\sc Figure \arabic{figure}}\hspace*{\fill}
\end{minipage}}

\begin{document}

\title[Automorphic forms and cohomology]{Holomorphic
automorphic forms and cohomology}
\author{Roelof Bruggeman}
\address{Mathematisch Instituut Universiteit Utrecht, Postbus 80010,
3508 TA \ Utrecht, Nederland}
\email{r.w.bruggeman@uu.nl}
\author{YoungJu Choie}
\address{Dept.~of Mathematics and PMI, Postech, Pohang, Korea
1790--784}
\email{yjchoie@gmail.com}
\author{Nikolaos Diamantis}
\address{School of Mathematical Sciences, University of Nottingham,
Nottingham NG7 2RD, United Kingdom}
\email{nikolaos.diamantis@nottingham.ac.uk}
\date{Version \version, \datum}

\begin{abstract}We investigate the correspondence between holomorphic
automorphic forms on the upper half-plane with complex weight and
parabolic cocycles. For integral weights at least $2$ this
correspondence is given by the Eichler integral. We use Knopp's
generalization of this integral to real weights, and apply it to
complex weights that are not an integer at least~$2$. We show that
for these weights the generalized Eichler integral gives an injection
into the first cohomology group with values in a module of
holomorphic functions, and characterize the image. We impose no
condition on the growth of the automorphic forms at the cusps. Our
result concerns arbitrary cofinite discrete groups with cusps, and
covers exponentially growing automorphic forms, like those studied by
Borcherds, and like those in the theory of mock automorphic forms.

For real weights that are not an integer at least~$2$ we similarly
characterize the space of cusp forms and  the space of entire
automorphic forms. We give a relation between the cohomology classes
attached to holomorphic automorphic forms of real weight and the
existence of harmonic lifts.

A tool in establishing these results is the relation to cohomology
groups with values in modules of ``analytic boundary germs'', which
are represented by harmonic functions on subsets of the upper
half-plane. It turns out that for integral weights at least~$2$ the
map from general holomorphic automorphic forms to cohomology with
values in analytic boundary germs is injective. So cohomology with
these coefficients can distinguish all holomorphic automorphic forms,
unlike the classical Eichler theory.
\end{abstract}



\maketitle

\setcounter{tocdepth}{1} \tableofcontents

\section*{Introduction}

Classically, the interpretation of holomorphic modular forms of
integral weight on the complex upper half-plane $\uhp$ in terms of
group cohomology has provided a tool that has had many important
applications to the geometry of modular forms, the study of their
periods, the arithmetic of special values of their $L$-functions, for
instance in \cite{Sh59, Kn74, Ma73, KZ84}.

A similar interpretation for Maass forms had to wait until the
introduction of periods of Maass forms given by Lewis and Zagier
\cite{Le97, LZ1}. The analogue of Eichler cohomology and the
Eichler-Shimura isomorphism for Maass forms of weight zero was
established in \cite{BLZm}.

We recall that Eichler \cite{Ei57} attached a cocycle $\ps$ to
meromorphic automorphic forms $F$ of weight~$k\in 2\ZZ_{\geq 1}$ by
\be\label{cocdefE} \ps_{F,\gm}(t) \= \int_{\gm^{-1}z_0}^{z_0} F(z)\,
(z-t)^{k-2}\, dt\,.
\ee
This cocycle has values in the space of polynomial functions of degree
at most $k-2$, with the action of weight $2-k$. The action of a
Fuchsian group $\Gm$ is induced by the action $|_{2-k}$ on functions
$f: \uhp \to \CC$, and is given by
$$(f|_{2-k} \gm)(z):=(cz+d)^{k-2} f(\gm z).$$
The class of the cocycle does not depend on the base point $z_0$
in~$\uhp$. To get independence of the integrals on the path of
integration $F$ is supposed to have zero residue at all its
singularities. This is the case for cusp forms~$F$. In that case we
can put the base point $z_0$ at a cusp, and arrive at so-called
parabolic cocycles.

For cusp forms for the modular group $\Gamma(1)
= \SL_2(\ZZ)$ one takes $z_0$ at~$\infty$. Then the cocycle is
determined by its value on $S=\matr0{-1}10$. One calls $\ps_{F,S}$ a
\emph{period polynomial} of~$F$, whose coefficients are values of the
$L$-function of~$F$ at integral points in the critical strip.

Knopp \cite{Kn74} generalized this approach to automorphic forms with
arbitrary real weight. Then a multiplier system is needed in the
transformation behavior of holomorphic automorphic forms. The factor
$(z-\nobreak t)^{k-2}$ becomes ambiguous if one replaces the positive
even weight $k$ by a real weight~$r$. Knopp solves this problem by
replacing $t$ by $\bar t$ for points $t\in\uhp$, and restoring
holomorphy by complex conjugation of the whole integral. The values
of the resulting cocycle are holomorphic functions on the upper
half-plane. Knopp \cite{Kn74} shows that for cusp forms $F$ they have
at most polynomial growth as $t$ approaches the boundary. In this way
he obtains an antilinear map between the space of cusp forms and the
first cohomology group with values in a module of holomorphic
functions with polynomial growth. He showed, \cite{Kn74}, that for
many real weights, this map is a bijection, and conjectured this for
all $r\in \RR$. Together with Mawi \cite{KM10} he proved it for the
remaining real weights.

For positive even weights this seems to contradict the classical
results of Eichler \cite{Ei57} and Shimura \cite{Sh59}, which imply
that the parabolic cohomology with values in the polynomials of
degree at most $k-2$ is isomorphic to the direct sum of the space of
cusp forms of weight~$k$ and its complex conjugate. The apparent
contradiction is explained by the fact that Knopp uses a larger
module for the cohomology. Half of the cohomology classes for the
classical situation do not survive the extension of the module.
\smallskip

In the modular case the period function of a modular cusp form of
positive even weight satisfies functional equations
({\it Shimura-Eichler relations}). Zagier noticed that a functional
equation with a similar structure occurs in Lewis's discussion in
\cite{Le97} of holomorphic functions attached to even Maass cusp
forms. Together \cite{LZ1} they showed that there is a cohomological
interpretation. In \cite{BLZm} this relation is extended to arbitrary
cofinite discrete groups of motions in the upper half-plane and Maass
forms of weight zero with spectral parameters in the vertical strip
$0<\re s < 1$. It gives an isomorphism between spaces of Maass cusp
forms of weight $0$ and a number of parabolic cohomology groups, and
for the spaces of all invariant eigenfunctions to larger cohomology
groups.
\smallskip

In this paper we study relations between the space of automorphic
forms without growth condition at the cusps and various parabolic
cohomology groups. We use the approach of \cite{BLZm} in the context
of holomorphic automorphic forms for cofinite discrete groups of
motions in the upper half-plane that have cusps. Like in \cite{BLZm}
we do not need to impose growth conditions at the cusps, and speak of
\emph{unrestricted} holomorphic automorphic forms. We take the module
of holomorphic functions in which the cocycles take their values as
small as possible. That means the classical space of polynomials of
degree at most $k-2$ for weights $k\in \ZZ_{\geq 2}$. Also for other
weights we use a smaller module than Knopp \cite{Kn74}. To avoid the
complex conjugation we use modules of holomorphic functions on the
lower half-plane~$\lhp$.  The flexibility gained by considering
functions on the lower half plane was recognized by the last two
authors \cite{CD7} in the context of Eichler cohomology associated to
non-critical values of $L$-functions.  It turns out that, for
the main results, working with arbitrary weights in
$\CC\setminus \ZZ_{\geq 2}$ takes no more effort than working with
real weights; so that is the generality that we choose where
possible. We shall show that the definition in \eqref{cocdefE},
suitably interpreted, gives a bijection between the spaces of
unrestricted holomorphic automorphic forms and several isomorphic
parabolic cohomology groups.

There are several motivations and potential applications for this.
Knopp's approach could ``see'' only cusp forms, we work with smaller
modules of analytic vectors in a highest weight subspace of a
principal series representation, and obtain a cohomological
description of all automorphic forms. In particular, this covers the
case of automorphic forms with exponential growth at the cusps. This
case is important especially because of its prominent role in
Borcherds's theory \cite{Bo} and in the theory of mock modular forms.

In the same way that representation theory has provided an important
unified setting for holomorphic and Maass forms, our construction
reflects a common framework for the cohomology of holomorphic and
Maass forms.

There are a lot of important relations between the theory of
cohomology of modular forms and various problems in number theory.
For instance, Zagier \cite{Za93} gives a new elementary proof of the
Eichler-Selberg trace formula using the explicit description of the
action of Hecke operators on the space of cohomology groups. In the
same paper Zagier connects cocycles with double zeta values, in which
many interesting further results are developed recently (\cite{IKZ6},
\cite{Za13}). Another application is the possibility of an
interpretation of the higher Kronecker limit formula in terms of
cohomology group~\cite{VZ13}. 

Finally, we note that one of striking applications of Eichler
cohomology concerns algebraicity results for critical values of
$L$-functions of classical (integral weight) cusp forms, eg, Manin's
periods theorem \cite{Ma73}, or \cite{Ma72}. The results obtained
were later extended, at least conjecturally, to other values and to
values of derivatives in a manner eventually formalized in the
conjectures of Deligne, Beilinson, Bloch-Kato and others.
See~\cite{KnZ1}.

In the case of values of derivatives, the main pathway to such results
did not involve directly Eichler cohomology. However, for $f$ of
weight $2$, in \cite{Go95}, (resp. \cite{Di99}), $L_f'(1)$ (resp.
$L_f^{(n)}(1)$) is expressed in terms of a ``period'' integral
similar to an Eichler cocycle, when $L_f(1)=0$. Despite the
similarity, this ``period'' integral does not seem at first to have a
direct cohomological interpretation. Nevertheless, in \S\ref{sect-Lf}
we are able to show that, $L_f'(1)$ can be expressed as a derivative
with respect to a the parameter of a family of  parabolic 
cocycles $r\mapsto \ps^\infty_{f_r,\cdot}$
associated to a family $r\mapsto f_r$ of automorphic forms. With
\cite{Di99}, similar expressions can be proved for higher
derivatives. We hope that better insight into the cohomology whose
foundations we establish here should yield information about the
algebraic structure of derivatives of $L$-functions along the lines
of the algebraicity results for critical values derived with the help
of classical Eichler cohomology.
\medskip

We now proceed with a discussion of the results of this paper. We
avoid many technicalities, and state the main theorems giving only
rough descriptions of the cohomology groups and coefficient modules
involved. In the next sections we define precisely all objects
occurring in the statements.

Let $\Gm$ be a cofinite discrete subgroup of $\SL_2(\RR)$ with cusps.
We take a \emph{complex} weight $r\in \CC$ and an associated
multiplier system $v:\Gm\rightarrow\CC^\ast$. We denote by
$A_r(\Gm,v)$ the space of all holomorphic functions
$F:\uhp\rightarrow\CC$ such that
\[ F(\gm z) \= v(\gm)\, (cz+d)^r\, F( z)\qquad \text{for all }\gm =
\matc\ast\ast cd\in \Gm,\; z\in \uhp\,.\]

For a fixed $z_0\in \uhp$ and an $F\in A_r(\Gm,v)$ consider the map
$\ps_{F}^{z_0}:\gm \mapsto \ps_{F,\gm}^{z_0}$ on $\Gm$, where
$\ps_{F,\gm}^{z_0}$ is the function of $t\in \lhp$ given by
$t$\ir{psiz0def}{\ps_{F,\ast}^{z_0}}
\be\label{psiz0def} \ps_{F,\gm}^{z_0}(t) \;:=\;
\int_{\gm^{-1}z_0}^{z_0} (z-t)^{r-2}\, F(z)\, dz\,. \ee
 We take the branch of $(z-\nobreak t)^{r-2}$ with
$-\frac\pi2<\arg(z-\nobreak t)
<\frac{3\pi}2$.

Our first main theorem is:
\begin{mainthm}\label{THMac}Let $\Gm$ be a cofinite discrete subgroup
of~$\SL_2(\RR)$ with cusps. Let $r\in \CC\setminus \ZZ_{\geq 2}$, and
let $v$ be an associated multiplier system.
\begin{enumerate}
\item[i)] The assignment $\ps_F^{z_0}:\gm\mapsto \ps^{z_0}_{F,\gm}$ is
a cocycle, and $F\mapsto \ps_F^{z_0}$ induces an injective linear map
\be \coh r \om:A_r(\Gm,v) \rightarrow H^1(\Gm;\dsv{v,2-r}\om)\,.\ee
Here $\dsv{v,2-r}\om$ denotes a space of holomorphic functions on the
lower half-plane $\lhp$ that are holomorphically continuable to a
neighbourhood of $\lhp\cup\RR$, together with an action depending
on~$v$.
\item[ii)] The image \il{cohrom-thm}{$\coh r \om$}$\coh r \om
A_r(\Gm,v)$ is equal to the mixed parabolic cohomology group
$\hpar^1\bigl(\Gm;\dsv{v,2-r}\om,\dsv{v,2-r}{\fsn,\wdg} \bigr)$,
which consists of elements of $H^1(\Gm;\dsv{v,2-r}\om)$ represented
by cocycles whose values on parabolic elements of $\Gm$ satisfy
certain additional conditions at the cusps.
\end{enumerate}
\end{mainthm}
This result is comparable to Theorem~C in Bruggeman, Lewis, Zagier
\cite{BLZm} where a linear injection of Maass forms of weight $0$
into a cohomology group is established.

The proof of Theorem~\ref{THMac} will require many steps, and will be
summarized in Subsection~\ref{sect-recap-THMac}.\smallskip

We characterize the images under $\coh r \om$ of the spaces $\cusp
r(\Gm,v)$ of cusps forms and $M_r(\Gm,v)$ of entire automorphic
forms:
\begin{mainthm}\label{THMcc}Let $\Gm$ be a cofinite discrete subgroup
of~$\SL_2(\RR)$ with cusps. Let $r\in \RR$ and let $v$ be a
\emph{unitary} multiplier system on $\Gm$ for the weight~$r$.
\begin{enumerate}
\item[i)] If $r\in\RR\setminus\ZZ_{\geq 2}$
\[ \coh r \om \, \cusp r(\Gm,v) \= \hpar^1(\Gm; \dsv{v,2-r}\om,
\dsv{v,2-r}{\fsn,\infty,\wdg})\,, \]
where $\dsv{v,2-r}{\fsn,\infty,\wdg}$ is a subspace of
$\dsv{v,2-r}{\fsn,\wdg}$ defined by a smoothness condition.
\item[ii)] If $r\in \RR\setminus\ZZ_{\geq 1}$
\[ \coh r \om \, M_ r(\Gm,v) \= \hpar^1(\Gm; \dsv{v,2-r}\om,
\dsv{v,2-r}{\fsn,\smp,\wdg})\,, \]
with the $\Gm$-module $\dsv{v,2-r}{\fsn,\smp,\wdg}\supset
\dsv{v,2-r}{\fsn,\infty,\wdg}$ also contained in
$\dsv{v,2-r}{\fsn,\wdg}$.
\end{enumerate}
\end{mainthm}
Here we give only a result for real weights. It seems that for
non-real weights the cusp forms do not form a very special subspace
of the space of all automorphic forms. There is, as far as we know,
no nice bound for the Fourier coefficients and it seems hard to
define $L$-functions in a sensible way.

In Theorems \ref{THMac} and~\ref{THMcc} automorphic forms of weight
$r$ are related to cohomology with values in a module with the ``dual
weight'' $2-\nobreak r$.\il{dw0}{dual weight}
\smallskip

The characterization in Theorems \ref{THMac} and~\ref{THMcc} of the
images of spaces of automorphic forms is one of several possibilities
given in Theorem~\ref{THMiso}, which we state in
Subsection~\ref{sect-isocg}, after some $\Gm$-modules
containing~$\dsv{v,2-r}\om$ have been defined. There we see that the
map $\coh r \om$ in Theorem~\ref{THMac} is far from surjective. In
Section~\ref{sect-qaf} we discuss a space of \emph{quantum
automorphic forms}, for which there is, if $r\not\in \ZZ_{\geq 1}$, a
surjection to the space $H^1(\Gm;\dsv{v,2-r}\om)$.
\smallskip

In \S\ref{sect-KM} we will compare Part~i) of Theorem~\ref{THMcc} to
the main theorem of Knopp and Mawi \cite{KM10}, which gives an
isomorphism $\cusp r
(\Gm,v) \rightarrow H^1(\Gm;\dsv{v,2-r}{-\infty})$ for some
larger $\Gm$-module $\dsv{v,2-r}{-\infty}\supset \dsv{v,2-r}\om$. The
combination of the theorem of Knopp and Mawi with Theorem~\ref{THMac}
shows that there are many automorphic forms $F\in A_r(\Gm,v)$ for
which $\coh r \om F $ is sent to zero by the natural map
$H^1(\Gm;\dsv{v,2-r}\om)\rightarrow H^1(\Gm;\dsv{v,2-r}{-\infty})$.
This means that the cocycle $\gm\mapsto\ps_{F,\gm}^{z_0}$ becomes a
coboundary when viewed over the module $\dsv{v,2-r}{-\infty}$, i.e.,
that there is $\Ph\in \dsv{v,2-r}{-\infty}$ such that
$\ps_{F,\gm}^{z_0}=\Ph|_{v,2-r}(\gm-\nobreak 1)$ for all $\gm\in
\Gm$.
\smallskip

The following result relates the vanishing of the cohomology class of
$\gm\mapsto\ps_{F,\gm}^{z_0}$ over a still larger module $
\dsv{v,2-r}{-\om}\supset\dsv{v,2-r}{-\infty}$ to the existence of
harmonic lifts, a concept that we will discuss in Subsections
\ref{sect-hlaf} and~\ref{sect-prfhl}.
\begin{mainthm}\label{THMhl}Let $\Gm$ be a cofinite discrete subgroup
of $\SL_2(\RR)$ with cusps. Let $r\in \CC$ and let $v$ be a
multiplier system for the weight~$r$. The following statements are
equivalent for $F\in A_r(\Gm,v)$:
\begin{enumerate}
\item[a)] The image of $\coh r \om F$ under the natural map
$H^1(\Gm;\dsv{v,2-r}\om)\rightarrow H^1(\Gm;\dsv{v,2-r}{-\om})$
vanishes.
\item[b)] The automorphic form $F$ has a harmonic lift in
$\harm_{2-\bar r}(\Gm,\bar v)$; ie, $F$ is in the image of the
antilinear map $ \harm_{2-\bar r}(\Gm,\bar v) \rightarrow A_r(\Gm,v)$
given by $H\mapsto 2i y^{2-r} \,\overline{\partial_{\bar
z} H}$.
\end{enumerate}
\end{mainthm}

We prove this theorem in Subsection~\ref{sect-prfhl}. Combining the
theorem of Knopp and Mawi \cite{KM10} with Theorem~\ref{THMhl} we
obtain the existence of harmonic lifts in many cases. See
Theorem~\ref{thm-c-hl}  and Corollary~\ref{cor-Kra}.
\medskip

\rmrk{Boundary germs}An essential aspect of the approach in
\cite{BLZm} is the use of ``analytic boundary germs''. These germs
form $\Gm$-modules isomorphic to the modules in \cite{BLZm}
corresponding to $\dsv {v,2-r}\om$ and $\dsv{v,2-r}{\fsn,\wdg}$ in
our case. In \cite{BLZm} the boundary germs are indispensable for the
proof of the surjectivity of the map from Maass forms of weight zero
to cohomology. The same holds for this paper.

In Sections \ref{sect-bg}--\ref{sect-hws} we study the spaces of
boundary germs that are relevant for our present purpose. In
particular we define spaces $\esv{v,r}\om$ and $\esv{v,r}{\fsn,\wdg}$
that are for weights in $\CC\setminus \ZZ_{\geq 2}$ isomorphic to
$\dsv{v,2-r}\om$ and $\dsv{v,2-r}{\fsn,\wdg}$, respectively. In
Theorem~\ref{thmbg} we obtain, for all complex weights~$r$, an
injective map
\be \bcoh r \om : A_r(\Gm,v) \rightarrow H^1(\Gm;\esv{v,r}\om)\ee
and study the image.

For weights $ r\in\CC\setminus\ZZ_{\geq 2}$ we use Theorem~\ref{thmbg}
in the proof of Theorem~\ref{THMac}. Theorem~\ref{thmbg} is also
valid for weights in $\ZZ_{\geq 2}$. For these weights it leads to
the following result:
\begin{mainthm}\label{THMaci}
Let $r\in \ZZ_{\geq 2}$, let $\Gm$ be a cofinite discrete subgroup of
$\SL_2(\RR)$ with cusps, and let $v$ be a multiplier system on $\Gm$
with weight~$r$.
\begin{enumerate}
\item[i)] Put \il{crt}{$c_r$}$c_r = \frac i{2\,(r-1)!}$, let $\rs_r$
denote the natural morphism $\esv{v,r}\om \rightarrow
\dsv{v,2-r}\om$, and let $\dsv{v,2-r}\pol$ denote the submodule of
$\dsv{v,2-r}\om$ consisting of polynomial functions of degree at most
$r-2$. The following diagram commutes:
\badl{Ediag} \xymatrix@C=.5cm{ H^1(\dsv{v,r}\om) \ar[r]
& H^1(\esv {v,r} \om) \ar[rr]^{\rs_r}
&& H^1(\dsv{v,2-r}\pol) \ar[r]
&0\\
\hpar^1(\dsv{v,r}\om,\dsv{v,r}{\fsn,\wdg})
\ar[r] \ar@{^{(}->}[u]
& \hpar^1(\esv{v,r}\om,\esv {v,r}{\fsn,\wdg}) \ar[rr]^{\rs_r}
 \ar@{^{(}->}[u]
&& \hpar^1(\dsv{v,2-r}\pol) \ar@{^{(}->}[u]
\\
A_{2-r}(v) \ar[r]^{c_r\,\partial_\tau^{r-1}}
\ar[u]_\cong^{\coh{2-r}\om}
& A_r^0(v) \ar[u]_\cong^{\bcoh r \om} \ar@{^{(}->}[rr]
&& A_r(v) \ar@/_2.5pc/[uu]_(.7){\coh r \om}
\ar@/_.8pc/@{_{(}->}[lluu]_(.6){\bcoh r \om} } \eadl
{\rm (To save space the group $\Gm$ is suppressed in the notation.)}
\item[ii)] The top row and the middle row are exact.
\item[iii)] The maps $H^1(\Gm;\dsv{v,r}\om)\rightarrow
H^1(\Gm;\esv{v,r}\om)$ in the top row and the map and
$\hpar^1(\Gm;\dsv{v,r}\om,\dsv{v,r}{\fsn,\wdg})
\rightarrow \hpar^1(\Gm;\dsv{v,r}\om,\dsv{v,r}{\fsn,\wdg})
$ in the middle row are injective, unless $r=2$ and $v$ is the trivial
multiplier system. In that exceptional case both maps have a kernel
isomorphic to the trivial $\Gm$-module~$\CC$.
\end{enumerate}
\end{mainthm}
\rmrks
\itmi For $r\in \ZZ_{\geq 1}$ the $\Gm$-module $\dsv{v,r}\om$ can be
considered as a submodule of $\esv{v,r}\om$. The space $A^0_r(\Gm,v)$
is the space of unrestricted holomorphic automorphic form for which
the Fourier terms of order zero at all cusps vanish.

\itm We note that automorphic forms both of weight $r$ and of the dual
weight $2-r$ occur in the diagram. The theorem shows that boundary
germ cohomology in some sense interpolates between the cohomology
classes attached to automorphic forms of weight $2-r$ and of weight
$r$, with $r\in \ZZ_{\geq 2}$.

\itm The second line in diagram~\eqref{Ediag} has no closing
$\rightarrow 0$. In \S\ref{sect-cocl} we will discuss how this
surjectivity can be derived by classical methods, provided that we
assume that the multiplier system $v$ is unitary.
\medskip

\rmrk{Comparison with \cite{BLZm}} This paper has much in common with
the notes \cite{BLZm}. Both give isomorphisms between spaces of
functions with automorphic transformation behavior and mixed
parabolic cohomology groups. The main difference is in the modules in
which the cohomology groups have their values. The $\Gm$-modules in
\cite{BLZm} are spherical principal series representations. The
linear map in \cite{BLZm} analogous to our map $\coh r \om$ sends
Maass forms of weight zero to cohomology classes in $H^1(\Gm;\V s
\om)$, where $\V s \om$ is the space of analytic vectors in the
principal series representation of $\PSL_2(\RR)$ with spectral
parameter~$s$. The assumption $0<\re s < 1$ ensures that the
representation $\V s \om$ is irreducible. Holomorphic automorphic
forms of weight~$r\in \CC$ correspond to a spectral parameter $\frac
r2$, for which the corresponding space of analytic vectors is
reducible. Hence here we work with the highest weight subspace. It is
irreducible precisely if $r\not\in \ZZ_{\geq 2}$, which explains that
in this paper weights in $\ZZ_{\geq 2}$ require a special treatment.

Another complication arises as soon as the weight is not an integer.
This means that we deal with highest weight subspace of principal
series representations of the universal covering group of
$\SL_2(\RR)$. In the main text of these notes we have avoided use of
the covering group. We discuss it in the Appendix.

Although the main approach of this paper relies heavily on
methods from \cite{BLZm}, and also on ideas in~\cite{LZ1}, it was far
from trivial to handle the complications not present in~\cite{BLZm}.

\rmrk{Overview of the paper} In Sections
\ref{sect-defnot}--\ref{sect-osa} we discuss results that can be
formulated with the modules $\dsv{v,2-r}\ast$. Here the proof of
Theorem~\ref{THMcc} is reduced to that of Theorem~\ref{THMac}.
Sections \ref{sect-hf}--\ref{sect-pol} give results for harmonic
functions and boundary germs. In section~\ref{sect-hf} one finds the
proof of Theorem~\ref{THMhl}. We use the boundary germs in
Sections~\ref{sect-hws}--\ref{sect-abg-iw} to determine the image of
automorphic forms in cohomology, and prove Theorems \ref{THMac}
and~\ref{THMaci}. Sections \ref{sect-isos-pc} and~\ref{sect-coc-sing}
give the proof of Theorem~\ref{THMiso} (which itself is stated on
page~\pageref{THMiso}). The map $\coh r \om$ in Theorem~\ref{THMac}
is not surjective. In Section~\ref{sect-qaf} we discuss how quantum
automorphic forms are mapped, under some conditions, onto
$H^1(\Gm;\dsv{v,2-r}\om)$.

At the end of most sections we mention directly related literature. In
Section~\ref{sect-lit} we give a broader discussion of literature
related to the relation between automorphic forms and cohomology. In
the Appendix we give a short discussion of the universal covering
group and its principal series representations.

\rmrk{Acknowledgements} The preparation of these notes has taken
several years. Essential has been the opportunities to work on it
together. Bruggeman was enabled by the Pohang Mathematical Institute,
Korea, to visit Choie for a month in 2013. During part of the work
on this project, Diamantis was a guest at the Max-Planck Institut f\"ur
Mathematik, Bonn, Germany, whose support he gratefully acknowledges.
The three authors participated in 2014 in the Research in
Pairs program of the Mathe\-matisches Forschungsinstitut Oberwolfach,
Germany. We thank the PMI and the MFO for this support and for the
excellent working condition provided.

Choie has been partially supported by NRF 2013053914,
NRF-2011-0008928 and NRF-2013R1A2A2A01068676.

We thank Seokho Jin for his remarks on an earlier version of these
notes.
\medskip

\numberwithin{equation}{section}


\part{Cohomology with values in holomorphic functions}


\section{Definitions and notations} \label{sect-defnot}

We work with the \il{uhpl}{upper half-plane}\emph{upper half-plane}
\il{ulhp}{$\uhp, \lhp$}$\uhp= \bigl\{ z\in \CC\;:\; \im z>0\}$ and the
\il{lhpl}{lower half-plane}\emph{lower half-plane} $\lhp$ defined by
$\im z<0$. For $z\in \uhp\cup\lhp$ we will often use without further
explanation \il{yulhp}{$y=\im z$}$y=\im z$,
\il{xulhp}{ $x=\re z$}$x=\re z$. Both half-planes are disjoint open
sets in the \il{cpl}{complex projective line}complex projective line
\il{proj}{$\proj\CC,\; \proj \R$}$\proj\CC = \CC \cup\{\infty\}$, with
the \il{rpl}{real projective line}real projective line $\proj\RR =
\RR\cup\{\infty\}$ as their common boundary.

\subsection{Operators on functions on the upper and lower
half-plane}\label{sect-actions} Let $r\in \CC$. For functions $f$ on
the upper or lower half-plane\ir{SL2op}{|_rg}
\be\label{SL2op} f|_r g\,(z) \;:=\; (c z+d)^{-r}\, f\Bigl( \frac{a
z+b}{c z+d} \Bigr)\qquad \text{for }g=\matc abcd\in \SL_2(\RR)\,, \ee
where we compute $(cz+\nobreak d)^{-r}$ according to the
\ir{ac}{\text{argument convention for }cz+d}\emph{argument
convention} to take
\be\label{ac} \arg(cz+d)\in (-\pi,\pi]\text{ if }z\in \uhp\,,
\qquad\arg(c z+\nobreak d)\in [-\pi,\pi)\text{ if }z\in \lhp\,.
\ee
These operators $|_r g$ do not define a representation of
$\SL_2(\RR)$.
(One may relate it to a representation of the universal covering group
of $\SL_2(\RR)$. See the Appendix, \S\ref{app-wf}.)
There are however two useful identities. Set \ir{G0}{G_0}
\be\label{G0} G_0\;:=\; \biggl\{ \matc abcd \in \SL_2(\RR)\;:\; -\pi <
\arg(ci+d)<\pi\biggr\}\,. \ee
Then, for all $g \in G_0$ and $p=\matc y x 0 {y^{-1}}$ with $x\in \RR$
and $y>0$:
\begin{align}
(f|_r g^{-1})|_rg &\=(f|_r g)|_rg^{-1}\=f \,,\\
\label{ncj}
f|_r g p g^{-1} &\= \bigl((f|_r g)|_r p\bigr) |_r g^{-1}\,.
\end{align}

To interchange functions on the upper and the lower half-plane we use
the antilinear involution $\Ci$ given by\ir{Ci}{\Ci}
\be\label{Ci} \bigl(\Ci f\bigr)(z) \;:=\; \overline{f(\bar z)}\,. \ee
It maps holomorphic functions to holomorphic functions, and satisfies
\be\label{Ci-act} \Ci\,\bigl( f|_r g\bigr) \= (\Ci f)|_{\bar r }
g\qquad\bigl(g\in \SL_2(\RR)\bigr)\,. \ee

\subsection{Discrete group} Everywhere in this paper we denote by
\il{Gm}{$\Gm$}$\Gm$ a cofinite
\il{ds}{discrete subgroup}dis\-crete subgroup of $\SL_2(\RR)$ with
cusps, containing $\matr{-1}00{-1}$.
\il{cft}{cofinite}\emph{Cofinite} means that the quotient
$\Gm\backslash\uhp$ has finite volume with respect to the
\il{hpm}{hyperbolic measure}hyperbolic measure $\frac{dx\,dy}{y^2}$.
The presence of cusps implies that the quotient is not compact. The
standard example is the
\il{mgr}{modular group}\emph{modular group}
\il{Gmod}{$\Gmod$}$\Gmod=\SL_2(\ZZ)$.

\rmrk{Multiplier system}A \il{ms}{multiplier system}\emph{multiplier
system} on $\Gm$ \il{mssuit}{multiplier system for a given weight}for
the weight $r\in \CC$ is a map \il{vnc}{$v$ multiplier
system}$v:\Gm\rightarrow \CC^\ast$ such that the function on
$\Gm\times\uhp$ given by
\be\label{jvr}
j_{v,r}\Bigl( \matc abcd,z\Bigr) \= v\matc abcd\;
(cz+d)^r\ee
satisfies the following conditions:
\badl{mscond} j_{v,r}(\gm\dt,z) &\= j_{v,r}(\gm,\dt z)\,
j_{v,r}(\dt,z)&\qquad&\text{for }\gm,\dt\in \Gm\,,\\
j_{v,r}\biggl( \matr{-a}{-b}{-c}{-d},z\biggr)
&\= j_{v,r}\biggl(\matc abcd,z\biggr)
&&\text{for }\matc abcd\in \Gm \,. \eadl
We call a multiplier system \il{ums}{unitary multiplier
system}\il{msu}{multiplier system, unitary}\emph{unitary} if
$\bigl|v(\gm)\bigr|=1$ for all $\gm\in \Gm$.

\rmrk{Action of the discrete group}Let $v$ be a multiplier system on
$\Gm$ for the weight~$r$. For functions on $\uhp$ and $p\equiv r
\bmod 2$ we put for $\gm=\matc abcd \in \Gm$:\il{vq-act}{$|_{v,p}$}
\be\label{vq-act+}
f|_{v,p}\gm \,(z) \;:=\; v(\gm)^{-1}\, (cz+d)^{-p}\, f(\gm z) \=
j_{v,p}(\gm,z)^{-1}\, f(\gm z)\,,\ee
and for functions on $\lhp$ and $p\equiv
-r\bmod 2$
\be\label{vq-act-}
f|_{v,p}\gm \,(z) \;:=\; v(\gm)^{-1}\, (cz+d)^{-p}\, f(\gm z)\,.\ee
The operator $|_{v,p}$ defines a holomorphy-preserving action of $\Gm$
on the spaces of functions on $\uhp$ and on $\lhp$, {\sl ie.},
$(f|_{v,p}\gm)|_{v,p}\dt = f|_{v,p}\gm\dt$ for all $\gm,\dt\in \Gm$.
Furthermore, $f|_{v,p}\matr{-1}00{-1}=f$, hence we have an action of
\il{bGm}{$\bar \Gm$}$\bar \Gm := \Gm/\{1,-1\}\subset\PSL_2(\RR)$.
Finally,
\be \Ci \bigl( f|_{v,r}\gm\bigr) \= (\Ci f)|_{\bar v,\bar r}\gm\qquad
\text{for }\gm\in \Gm\,.
\ee

\subsection{Automorphic forms}\label{sect-af}
We consider automorphic forms without any growth condition.
\begin{defn}\label{uhaf-def}A \il{uhaf}{unrestricted holomorphic
automorphic form}\il{hafu}{holomorphic automorphic form,
unrestricted}\emph{unrestricted holomorphic automorphic form} on
$\Gm$ with weight \il{rwt}{$r$ weight}$r\in \CC$ and multiplier
system $v$ on~$\Gm$ for the
\il{wt}{weight}weight~$r$ is a holomorphic function
$F:\uhp\rightarrow\CC$ such that
\be\label{autcond} F|_{v,r}\gm \= F \qquad\text{for all }\gm\in \Gm\,.
\ee
We use \il{Ar}{$A_r(\Gm,v)$}$A_r(\Gm,v)$ to denote the space of all
such unrestricted holomorphic automorphic forms. We often abbreviate
\emph{unrestricted holomorphic automorphic form} to
\il{haf}{holomorphic automorphic form}\il{afuh}{ automorphic form,
unrestricted holomorphic}\il{afh}{automorphic form,
holomorphic}\emph{holomorphic automorphic form} or to
\il{af}{automorphic form}\emph{automorphic form}.
\end{defn}

\rmrk{Cusps} A \il{cu}{cusp}cusp of $\Gm$ is a point $\ca\in
\proj\RR=\RR\cup\{\infty\}$ such that the
\il{stab}{stabilizer}\emph{stabilizer} 
\il{stabxi}{$\Gm_{\ca}$}$\Gm_{\ca}:= \bigl\{\gm\in \Gm\;:\;
\gm\,\ca=\ca \bigr\}$ is infinite and has no other fixed points in
$\proj\CC$. This group is of the form $\Gm_\ca =\bigl\{\pm
\pi_\ca^n\;:\;n\in \ZZ\}$ for an element $\pi_\ca\in \Gm$
that is conjugate to \il{Tmat}{$T=\matc 1101$}$T=\matc 1101$
in~$\SL_2(\RR)$. The elements $\pi_\ca^n$ have trace~$2$, and are,
for $n\neq0$, called \il{parbm}{parabolic matrix}\emph{parabolic}.
The elements $\pi_\ca$ and $\pi_\ca^{-1}$ are \il{prpar}{primitive
parabolic}\emph{primitive parabolic} since they are not of the form
$\gm^n$ with $\gm\in \Gm$ and $n\geq 2$.

For each cusp $\ca$ there are (non-unique) \il{gxi}{$\sigma_{\ca
}$}$\s_\ca\in G_0$ such that \il{pica}{$\pi_\ca$}$\pi_\ca=\s_\ca T
\s_\ca^{-1}$. We arrange the choice such that for all $\gm\in \Gm$ we
have $\s_{\gm\ca}=\pm \gm \s_\ca T^n$ for some $n\in \ZZ$.

The set of cusps of a given discrete group $\Gm$ is a finite union of
$\Gm$-orbits. Each of these orbits is an infinite subset of
$\proj\RR$.

\rmrk{Fourier expansion}Each $F\in A_r(\Gm,v)$ has at each cusp $\ca$
of $\Gm$ a \il{Fe}{Fourier expansion}\emph{Fourier expansion}
\be \label{FexpF-xi}F|_r \sigma_{\ca}\;(z)
\= \sum_{n\equiv \al_{\ca}\bmod 1} a_n(\ca,F)\, e^{2\pi i n z}\,, \ee
with \il{alca}{$\al_\ca$}$\al_{\ca}$ such that $v(\pi_{\ca})
= e^{2\pi i \al_{\ca}}$. The
\il{Fc}{Fourier coefficient}Fourier coefficients
\il{anxi}{$a_n(\ca,F)$}$a_n(\ca,F)$ depend (by a non-zero factor) on
the choice of $\sigma_{\ca}$ in $\SL_2(\RR)$. In general,
$\sigma_{\ca}\not\in \Gm$, so we have to use the operator $|_r
\sigma_{\ca}$, and not the action $|_{v,r}$ of~$\Gm$.

If the multiplier system is not unitary, it may happen that
$\bigl|v(\pi_{\ca})\bigr|\neq 1$ for some cusps~$\ca$. Then
$\al_{\ca}\in \CC\setminus \RR$, and the Fourier term orders $n$
in~\eqref{FexpF-xi} are not real.

\begin{defn}\label{def-cf}We define the following subspaces of
$A_r(\Gm,v)$:
\begin{enumerate}
\item[i)] The space of \il{cf}{cusp form}\emph{cusp forms}
is\il{Scf}{$\cusp r(\Gm,v)$}
\[\cusp r(\Gm,v)\;:=\; \Bigl\{ F\in A_r(\Gm;v)\;:\; \forall_{\ca\text{
cusp }} \forall_{n\equiv \al_{\ca}(1)} \; \re n\leq 0 \Rightarrow
a_n(\ca,F)=0\Bigr\}\,.\]
\item[ii)] The space of \il{eaf}{entire automorphic form}\emph{entire
automorphic forms} is\il{Mdef}{$M_r(\Gm,v)$}
\[ M_r(\Gm,v) \;:=\; \Bigl\{ F\in A_r(\Gm,v)\;:\; \forall_{\ca \text{
cusp }} \forall _{n\equiv\al_{\ca}(1)}\; \re n<0 \Rightarrow
a_n(\ca,F)=0\Bigr\}\,.\]
\end{enumerate}
\end{defn}
If $v(\pi_\ca)\neq 1$ the name ``entire'' is not very appropriate,
since then the Fourier expansion at $\ca$ in \eqref{FexpF-xi} needs
 non-integral powers of $q=e^{2\pi i z}$.

This implies that $F\in \cusp r(\Gm,v)$ has \il{expdecc}{exponential
decay at cusps}\emph{exponential decay at all cusps}:
\bad \forall_{\ca\text{ cusp of
$\Gm$}}&\,\forall{X>0}\,\exists_{\e>0}\, \forall_{x\in [-X,X]}\,
 \\
& F\bigl(\s_\ca (x+iy)\bigr)\= \oh(e^{-\e y})\text{ as
}y\rightarrow\infty\,. \ead
If $v$ is not unitary we need to restrict $x$ to compact sets.
Similarly, functions $F\in M_r(\Gm,v)$ have at most
\il{polgrc}{polynomial growth at the cusps}\emph{polynomial growth} at
the cusps:
\badl{polgrc} \forall_{\ca\text{ cusp of
$\Gm$}}&\,\forall{X>0}\,\exists_{a>0}\, \forall_{x\in [-X,X]}\,
 \\
& F\bigl(\s_\ca (x+iy)\bigr)\= \oh(y^a)\text{ as
}y\rightarrow\infty\,. \eadl

\subsection{Cohomology and mixed parabolic
cohomology}\label{sect-coh-mpc}
We recall the basic definitions of group cohomology. Let $V$ be a
right $\CC[\Gm]$-module. Then the \il{cg}{cohomology group}first
cohomology group $H^1(\Gm;V)$ is\ir{1coh}{H^1(\cdot;\cdot)}
\be\label{1coh} H^1(\Gm;V) \= Z^1(\Gm;V) \bmod B^1(\Gm;V)\,,\ee
where \il{Z1}{$Z^1(\cdot;\cdot)$}$Z^1(\Gm;V)$ is the space of
\il{coc}{cocycle}\emph{$1$-cocycles} and
\il{B1}{$B^1(\cdot;\cdot)$}$B^1(\Gm;V)\subset
Z^1(\Gm;V)$ the space of
\il{cob}{coboundary}\emph{$1$-coboundaries}. A $1$-cocycle is a map
$\ps:\Gm\rightarrow V:\gm\mapsto \ps_\gm$ such that $\ps_{\gm\dt} =
\ps_\gm|\dt+\ps_\dt$ for all $\gm,\dt\in \Gm$ and a $1$-coboundary is
a map $\ps:\Gm\rightarrow V$ of the form $\ps_\gm = a|\gm-a$ for some
$a\in V$ not depending on~$\gm$.

\begin{defn}\label{mpcg}Let $V\subset W$ be right $\Gm$-modules. The
\il{mpcg}{mixed parabolic cohomology group}\emph{mixed parabolic
cohomology group}
\il{parbc}{$\hpar^1(\Gm;V,W)$}$\hpar^1(\Gm;V,W)\subset H^1(\Gm;V)$ is
the quotient $\zpar^1(\Gm;V,W)/B^1(\Gm;V)$,
where\ir{zparb}{\zpar^1(\Gm;V,W)}
\be\label{zparb} \zpar^1(\Gm;V,W) \= \bigl\{ \ps\in Z^1(\Gm;V)\;:\;
\ps_\pi\in W|(\pi-1)\text{ for all parabolic }\pi\in \Gm\bigr\}\,.
\ee
The space \il{hpar}{$\hpar^1(\Gm;V)$}$\hpar^1(\Gm;V)
:= \hpar^1(\Gm;V,V)$ is the usual \il{pchgp}{parabolic cohomology
group}\emph{parabolic cohomology group}.

We call cocycles in $\zpar^1(\Gm;V,W)$ \il{mpcoc}{mixed parabolic
cocycle}\emph{mixed parabolic cocycles}, and \il{pbcoc}{parabolic
cocycle}\emph{parabolic cocycles} if $V=W$.
\end{defn}

In~\eqref{zparb} it suffices to check the condition $\ps_\pi \in
W|(\pi-\nobreak 1)$ for $\pi=\pi_\ca$ with $\ca$ in a (finite) set of
representatives of the $\Gm$-orbits of cusps.

\subsection{Modules}\label{sect-modules}
The coefficient modules that we will use for cohomology are based on
the following spaces:
\begin{defn}\label{Dsomdef}Let $r\in \CC$. For functions $\ph$ define
the function $\Prj{2-r} \ph$ by \ir{Prj}{\Prj{2-r}}
\be\label{Prj} (\Prj{2-r} \ph)(t) \;:=\; (i-t)^{2-r}\, \ph(t)\,,\ee
where $(i-\nobreak t)^{2-r} $ denotes the branch with $\arg(i-\nobreak
t)\in \bigl(-\frac\pi2,\frac{3\pi}2\bigr)$.\il{dsv*}{$\dsv{2-r}\ast$}
\begin{enumerate}
\item[i)] $\dsv{2-r}{-\om}:=\Bigl\{\ph :\lhp\rightarrow\CC\;:\;\ph
\text{ is holomorphic} \Bigr\}$.
\item[ii)] $\dsv {2-r}{-\infty} := \Bigl\{ \ph \in \dsv
{2-r}{-\om}\;:\; \exists_{B>0} \; \ph(t) = \oh\bigl( |\im
t|^{-B}\bigr) + \oh\bigl( |t|^{B } \bigr)\text{ on }\lhp
\Bigr\}$, the space of functions with at most \il{polgr}{polynomial
growth}\emph{polynomial growth}.
\item[iii)] $\displaystyle \dsv{2-r}\infty \= \Bigl\{\ph\in \dsv{2-r}{
-\om }\;:\; \Prj{2-r}\ph \in C^\infty( \lhp\cup\nobreak\proj\RR)
\Bigr\}$
\item[iv)] $\displaystyle \dsv{2-r}\om \=\Prj{2-r}^{-1}\indlim\hol(U)$
where $U$ runs over the open neighbourhoods of $\lhp\cup\proj\RR$
in~$\proj\CC$, and
\il{holdef}{$\hol$}$\hol$ denotes the sheaf of holomorphic functions
on~$\proj\CC$.
\item[v)] For $r\in \ZZ_{\geq 2}$ we put
\il{dsvpol}{$\dsv{2-r}{\pol}$}$\dsv{2-r}\pol:=\Bigl\{ \ph \in
\dsv{2-r}\om\;:\; \ph $ is given by a polynomial function on $\CC$ of
degree at most $r-2\Bigr\}$.
\end{enumerate}
\end{defn}

\rmrk{Discussion}
\itmi The largest of these space, $\dsv{2-r}{-\om}$, consists of all
holomorphic functions on the lower half-plane. The subspace
$\dsv{2-r}{-\infty}$ is determined by behavior of $\ph(t)$ as $t$
approaches the boundary $\proj\RR$ of $\lhp$. The real-analytic
function $Q(t) = \frac {|\im t|}{|t-i|^2}$ on $\proj\CC\setminus
\{i\}$ satisfies $0< Q(t) \leq 1$ on the lower
half-plane and zero on its boundary. A more uniform definition of
polynomial growth requires that functions $f$ satisfy $f(t)
\ll Q(t)^{-B}$ for some $B>0$. In Part~ii) we use Knopp's formulation
 in \cite{Kn74}, transformed to the lower half-plane. Both are
equivalent. To see this, we use in one direction that (for $t\in
\lhp$)
\[ \frac { |\im t|^{-B} + |t|^{B}} {Q(t)^{-B}} \;\leq\;
\frac{1 + |t|^{2B}}{|t-i|^{2B}} \;\leq\; 1+1\,. \]
In the other direction we carry out separate estimates for the
following three cases
(1)~$|t|\leq 1$, with $Q(t)^{-B}\leq |\im t|^{-2B}$; (2)~$|t|\geq 1$,
$|\im t|\geq \frac12$ with $Q(t)^{-B} \ll |t|^{2B} + 1$;
(3)~$|t|\geq 1$, $|\im t|\leq \frac12$, with $ Q(t)^{-B} \leq
\frac{|t|^{2B}}{|\im t|^B} + |\im t|^{-B}\,\leq\, |t|^{4B}
+|\im t|^{-2B}+ |\im t|^{-B} $.

\itm With $t\in \lhp$, the factor $(i-\nobreak t)^{2-r} $ in
\eqref{Prj} is $\oh(1)$ if $\re r\geq 2$ and $\oh(|t|^{2-\re r})$ if
$\re r\leq 2$, and its inverse $(i-\nobreak t)^{r-2}$ satisfies
similar estimates. So the function $\ph$ on $\lhp$ has at most
polynomial growth if and only $\Prj{2-r}\ph$ has polynomial growth.
So we could formulate the definition of $\dsv{v,2-r}{-\infty}$ with
$\Prj{2-r}\ph$ instead of~$\ph$.

\itm The polynomial growth in Part~ii) concerns the behavior of
$\ph(t)$ as $t$ approaches the boundary $\proj\RR$ of $\lhp$ at any
point. The polynomial growth at the cusps in~\eqref{polgrc} concerns
the approach of $F(z)$ as $z$ approaches cusps in the boundary
$\proj\RR$ of~$\uhp$.

\itm For some holomorphic $\ph$ on $\lhp$ it may happen that
$\Prj{2-r}\ph$ extends from $\lhp$ to yield a function that is smooth
on $\lhp\cup\proj\RR$. Then $\Prj{2-r}\ph$ satisfies near $\xi\in
\RR$ a Taylor approximation of any order $N$
\[ \Prj{2-r}\ph(t) \= \sum_{n=0}^{N-1} a_n \, (t-\xi)^n + \oh\bigl(
(t-\xi)^N \bigr)
\]
as $t$ approaches $\xi$ through $\lhp\cup\RR$. Near $\infty$ we have a
Taylor approximation in $-1/t$. This defines the space in Part~iii)
as a subspace of $\dsv{2-r}{-\om}$.

These Taylor expansions imply that $\Prj{2-r}\ph$ has at most
polynomial growth at the boundary. So $\dsv{2-r}\infty$ is in fact a
subspace of $\dsv{2-r}{-\infty}$.

\itm Instead of Taylor expansions of any order, we may require that
$\Prj{2-r}\ph$ is near each $\xi\in \proj\RR$ given by a convergent
power series expansion. Then it extends as a holomorphic function to
a neighbourhood of $\lhp\cup\proj\RR$ in~$\proj\CC$. That defines the
space $\dsv{2-r}\om$ in Part~iv).

The formulation with an inductive limit implies that we consider two
extensions to be equal if they have the same restriction to~$\lhp$.

\itm If $r\in \ZZ_{\geq 2}$, and $\ph$ is a polynomial function of
degree at most $r-2$ the function $\Prj{2-r}\ph(t)$ extends
holomorphically to $\proj\CC \setminus \{i\}$.

\itm We have defined a decreasing sequence of spaces of holomorphic
functions on the lower half-plane: $\dsv{2-r}{ -\om} \supset
\dsv{2-r}{ - \infty} \supset\dsv{2-r}{\infty}\supset
\dsv{2-r}{\om}\supset\dsv{2-r}\pol$ (the last one only if $r\in
\ZZ_{\geq 2}$).

One can show that the spaces $\dsv{2-r}\ast$ arise as highest weight
subspaces occurring in principal series representations of the
universal covering group of $\SL_2(\RR)$. Then $\dsv{2-r}\om$
corresponds to a space of \emph{analytic vectors}, $\dsv{2-r}\infty$
to a space of $C^\infty$-vectors, $\dsv{2-r}{-\infty}$ to a space of
\emph{distribution vectors}, and $\dsv{2-r}{-\om}$ to a space of
\emph{hyperfunction vectors}. This motivates the choice of the
superscripts $\om$, $\infty$, $-\infty$ and $-\om$. See
\S\ref{app-ps} in the Appendix.

\itm The vector spaces $\dsv{2-r}\om$ and $\dsv{2-r}\infty$ depend
on~$r$, the spaces $\dsv{2-r}{-\infty}$ and $\dsv{2-r}{-\om}$ do not.

\rmrk{Projective model} We have characterized the spaces
$\dsv{2-r}\ast$ in iii) and~iv) in Definition~\ref{Dsomdef} by
properties of $\Prj{2-r}\ph$, not of $\ph$ itself, and could also
equally well use $\Prj{2-r}\ph$ in i) and~ii).

We call $\Prj{2-r}\dsv{2-r}\ast$ the \il{prjm}{projective
model}\emph{projective model} of $\dsv{2-r}\ast$. Advantages of the
projective model are the simpler definitions and the fact that none
of the spaces $\Prj{2-r}\dsv{2-r}\ast$ depends on~$r$. Moreover, the
projective model focuses our attention to the behavior of the
functions near the boundary $\proj\RR$ of the lower half-plane.

A big advantage of the spaces $\dsv{2-r}\ast$ themselves is the simple
form of the operators $|_{2-r}g$ with $g\in \SL_2(\RR)$. We will
mostly work with these spaces, and use the projective model only
where it makes concepts or proofs easier.

The formula in~\eqref{SL2op} for the operators $|_r g$ is the usual
formula when one works with holomorphic automorphic forms. Of course
these operators can be formulated in the projective model, as is done
in Proposition~\ref{prop-opprj} below. At first sight that
description looks rather complicated. However, even this formula has
its advantage, as will become clear in the proof of
Proposition~\ref{prop-inv}.

\begin{prop}\label{prop-opprj}Let $r\in \CC$. Under the linear map
$\Prj{2-r}$ the operators $|_{ 2-r } g$ with $g=\matc abcd\in
\SL_2(\RR)$ correspond to operators $|^\prj_{2-r}g$ given on $h$ in
the projective model by\ir{prjact}{|^\prj_{2-r} g}
\be\label{prjact}
h |^\prj_{2-r} g (t) \= (a-ic)^{r-2}\, \Bigl( \frac{t-i}{t-g^{-1} i}
\Bigr)^{2-r}\, h(gt)\,, \ee
for $t\in \lhp$ and the choice $\arg(a-\nobreak ic)\in [-\pi,\pi)$.
\end{prop}
\begin{proof}
We want to determine the operator $|_{2-r}g$ for $g=\matc abcd \in
\SL_2(\RR)$ such that the following diagram commutes:
\[ \xymatrix{ \Prj{2-r}\dsv{2-r}{-\om} \ar[r]^{|^\prj_{2-r}g}
& \Prj{2-r}\dsv{2-r}{-\om} \\
\dsv{2-r}{-\om} \ar[r]^{|_{2-r}g} \ar[u]_{\Prj{2-r}}
& \dsv{2-r}{-\om} \ar[u]^{\Prj{2-r}} }\]
For $\ph\in \dsv{2-r}{-\om}$ put $h=\Prj{2-r}\ph$. So $\ph(t)
=
(i-\nobreak t)^{r-2}\,h(t)$. Then $h|_{2-r}^\prj g\,(t)$ should be
given by
\[ \bigl( \Prj{2-r} (\ph|_{2-r}g) \bigr)(t)
\= (i-t)^{2-r} \,(ct+d)^{r-2}\,(i-gt)^{r-2}\, h(gt)
\]
So we need to check that
\[ (i-t)^{2-r}\, (ct+d)^{r-2}\, (i-g t)^{r-2}\= (a-ic)^{r-2}\, \Bigl(
\frac{t-i}{t-g^{-1}i}\Bigr)^{2-r}\,.\]
For $g$ near to the identity in $\SL_2(\RR)$ and $t$ near $-i$ this
can be done by a direct computation. The equality extends by
analyticity of both sides to $(t,g)\in \lhp\times G_0$. (See
\eqref{G0} for $G_0$.)

All factors are real-analytic in $(t,g)$ on $\lhp \times \SL_2(\RR)$,
except $(ct+\nobreak d)^{r-2}$ and $(a-ic)^{r-2}$. So we have to
check that the arguments of these two factors tend to the same limit
as $g=\matc abcd \rightarrow\matr{-p}{q}{0}{-p^{-1}}\in
\SL_2(\RR)\setminus G_0$, with $p>0$ and $q\in \RR$. We have indeed
$\arg(ct+\nobreak d)\rightarrow-\pi$, and
$\arg(a-ic)\rightarrow-\pi$.
\end{proof}

\begin{prop}\label{prop-inv} Each of the spaces $\dsv{2-r}\pol$,
$\dsv{2-r}\om$, $\dsv{2-r}\infty$, $\dsv{2-r}{-\infty}$ and
$\dsv{2-r}{-\om}$ is invariant under the operators $|_{2-r}g$ with
$g\in \SL_2(\RR)$.
\end{prop}
\begin{proof}
 We work with the projective model. The factor $\bigl(
\frac{t-i}{t-g^{-1}i}\bigr)^{2-r}$ and its inverse are holomorphic on
$\proj\CC\setminus p$, were $p$ is a path in $\uhp$ from $i$ to
$g^{-1}i$ in~$\uhp$. Multiplication by this factor preserves the
projective models of each of the last four spaces. The invariance of
$\dsv{2-r}\pol$ is easily checked without use of the projective
model.
\end{proof}

\begin{defn}Let $\Gm$ be a cofinite subgroup of $\SL_2(\RR)$ and let
$v$ be a multiplier system for the weight $r\in \CC$. For each choice
of $\ast\in \bigl\{-\om,-\infty,\infty, \om,\pol\bigr\}$, we define
\il{dsv-v2-r}{$\dsv{v,2-r}\ast$}$\dsv{v,2-r}\ast$ as the space
$\dsv{2-r}\ast$ with the action $|_{v,2-r}$ of~$\Gm$, defined
in~\eqref{vq-act-}.
\end{defn}

\rmrke The finite-dimensional module $\dsv{v,2-r}\pol$ is the
coefficient module used by Eichler \cite{Ei57}. Knopp \cite{Kn74}
used an infinite-dimensional module isomorphic (under $\Ci$ in
\eqref{Ci}) to $\dsv{v,2-r}{-\infty}$ for the cocycles attached to
cusp forms of real weight. In our approach $\dsv {v,2-r}\om$ will be
the basic $\Gm$-module.

\subsection{Semi-analytic vectors}\label{sect-sav}For a precise
description of the image of the map $\coh r \om$ from automorphic
forms to cohomology with values in $\dsv{v,2-r}\om$, we need more
complicated modules, in spaces where we relax the conditions in
Part~iv) of Definition~\ref{Dsomdef} in a finite number of points
of~$\proj\RR$.

\begin{defn}\label{def-sav}\textrm{ Semi-analytic
vectors.}\il{sav}{semi-analytic vector}
\begin{enumerate}
\item[i)] Let $ \xi_1,\ldots, \xi_n\in \proj\RR$.\ir{dsvomxi}{
\dsv{2-r}\om[\xi_1,\ldots,\xi_n]}
\be \label{dsvomxi}\dsv{2-r}\om[\xi_1,\ldots,\xi_n]\;:=\;
\Prj{2-r}^{-1}\,\indlim\hol(U)\,, \ee
where $U$ runs over the open sets in $\proj\CC$ that contain $\lhp $
and $\proj\RR\setminus\{\xi_1,\ldots,\xi_n\}$.
\item[ii)]$\dsv{2-r}\fs :=
\indlim \dsv{2-r}\om[\xi_1,\ldots,\xi_n]$, where
$\{\xi_1,\ldots,\xi_n\}$ runs over the finite subsets of $\proj\RR$.
\item[iii)]$\dsv{2-r}\fsn := \indlim
\dsv{2-r}\om[\ca_1,\ldots,\ca_n]$, where $\{\ca_1,\ldots,\ca_n\}$
runs over the finite sets of cusps of $\Gm$.
\item[iv)]For $\ph \in \dsv{2-r}\fs$ we define the set of
\il{singul}{boundary singularity}\emph{boundary singularities}
\il{bsing}{$\bsing \ph$}$\bsing \ph$ as the minimal set
$\{\xi_1,\ldots,\xi_n\}$ such that $\ph \in
\dsv{2-r}\om[\xi_1,\ldots,\xi_n]$.
\end{enumerate}
\end{defn}

\rmrk{Conditions on the singularities} The elements of the spaces in
Definition~\ref{def-sav} can be viewed as real-analytic functions
on~$\RR\setminus E$ for some finite set~$E$, without conditions on
the nature of the singularities at the exceptional points in~$E$. We
will define subspaces by putting restrictions on the singularities
that we allow.

If $\ph \in \dsv{2-r}\om$ then $h=\Prj{2-r}\ph$ is holomorphic at each
point $\xi\in \proj\RR$, hence its power series at $\xi$ represents
$h$ on a neighbourhood of $\xi$ in~$\proj\CC$:
\be \label{powser}
 h(t)\= \sum_{n\geq 0} a_n \, (t-\xi)^n\quad(\xi\in \RR)\,,\qquad
h(t)\= \sum_{n\geq 0}a_n\, t^{-n}\quad(\xi=\infty)\,.\ee
If $\ph$ is in the larger space $\dsv{2-r}\infty$, then there need not
be a power series that converges to the function $h=\Prj{2-r}\ph$,
but only an \il{asmpts}{asymptotic series}\emph{asymptotic
series}\ir{asser}{\sim}
\be\label{asser}
 h(t)\;\sim\; \sum_{n\geq 0} a_n \, (t-\xi)^n\quad(\xi\in
 \RR)\,,\qquad h(t)\;\sim\; \sum_{n\geq 0}a_n\,
 t^{-n}\quad(\xi=\infty)\,,
\ee
valid as $t$ approaches $\xi$ through $\lhp \cup \proj\RR$. By this
formula we mean that for any order $N\geq 1$ we have
\[ h(t) \= \sum_{n=0}^{N-1} a_n\, (t-\xi)^n + \oh \bigl(
(t-\xi)^N\bigr)\]
as $t\rightarrow\xi$ through $\lhp\cup\proj\RR$, and analogously for
$\xi=\infty$.

\rmrk{Smooth semi-analytic vectors}The first condition on the
singularities that we define is rather strict:
\begin{defn}\il{dsvominf}{$\dsv{2-r}{\om,\infty}[\xi_1,\ldots,\xi_n]$}
$\dsv{2-r}{\om,\infty}[\xi_1,\ldots,\xi_n]:=
\dsv{2-r}{\om}[\xi_1,\ldots,\xi_n]\cap \dsv{2-r}\infty$. We call it a
space of \il{savs}{semi-analytic vector, smooth}\il{ssav}{smooth
semi-analytic vector}\emph{smooth semi-analytic vectors}.
\end{defn}

\rmrk{Semi-analytic vectors with simple singularities} We may also
allow the asymptotic expansions in~\eqref{asser} to run over $n\geq
-1$. This gives the following space of semi-analytic vectors with
\emph{simple singularities}:
\begin{defn}\label{sav-ss}We define spaces of
\il{sav-ss}{semi-analytic vectors with simple
singularities}\emph{semi-analytic vectors with simple singularities}
by\ir{smpdef}{\dsv{2-r}{\om,\smp}[\xi_1,\ldots \xi_n]}
\badl{smpdef} \dsv{2-r}{\om,\smp}&[\xi_1,\ldots \xi_n]\;:=\; \Bigl\{
\ph \in \dsv{2-r}\om[\xi_1,\ldots,\xi_n]\;:\; \\
&\quad t\mapsto(t-\xi_j)\, (\Prj{2-r}\ph)(t) \text{ is in
}C^\infty(\lhp\cup\RR)\text{ if }\xi_j\in \RR\,,\\
&\quad t\mapsto t^{-1}\, (\Prj{2-r}\ph)(t)\text{ is in }C^\infty(\lhp
\cup\proj\RR\setminus\{0\})\text{ if }\xi_j=\infty\Bigr\}\,. \eadl
\end{defn}

\rmrk{Example}Elements of $\dsv{2-r}{\om,\smp}[\cdots]$ turn up
naturally. Often we will be interested in equations like the
following one:
\[ h(t+1)- h(t) \= \ph(t)\,,\]
where $\ph$ is given. In the case $\ph \in \dsv{2-r}\pol$ with $r\in
\ZZ_{\geq 2}$, we cannot find a solution $h$ in $\dsv{2-r}\om$ if
$\ph$ is a (nonzero) polynomial with degree equal to $r-2$. If there
is a solution $h$ of the equation given by a polynomial, then
$\mathrm{deg}\, h = r-1$, and $h$ cannot be in~$\dsv {2-r}\pol$.
 Further note that such a solution $h$ is even not in $\dsv{2-r}\om$,
since $(\Prj{2-r}h)(t)
=
(i-\nobreak t)^{2-r}\, h(t)$ is not holomorphic at~$\infty$. However,
$t\mapsto t^{-1}\,(\Prj{2-r}h)(t)$ is holomorphic at~$\infty$, hence
$h\in \dsv{2-r}{\om,\smp}[\infty]$.

\rmrk{Semi-analytic vectors supported on an excised neighbourhood} Much
more freedom leaves the last condition that we define. It does not
work with asymptotic expansions, but with the nature of the domain to
which the function can be holomorphically extended.
\begin{defn}\label{dec-en}A set $\Om \subset \proj\CC$ is an
\il{en}{excised neighbourhood}\emph{excised neighbourhood} of $\lhp \cup
\proj\RR$, if it contains a set of the form
\[ U \setminus \bigcup_{\xi\in E} W_{\xi}\,, \]
where $U$ is a standard neighbourhood of $\lhp\cup\proj\RR$
in~$\proj\CC$, where $E$ is a finite subset of $\proj\RR$, called the
\il{es}{excised set}\emph{excised set}, and where $W_{\xi}$ has the
form
\[ W_{\xi} \= \bigl\{ h_{\xi} z \in \uhp\;:\; |\re z|\leq a\text{ and
}\im z>\e \bigr\}\,,\]
with $h_{\xi}\in \SL_2(\RR)$ such that $h_{\xi}\infty=\xi$, and
$a,\e>0$\,.

Instead of ``excised neighbourhood of $\lhp\cup\proj\RR$ with excised
set $E$'' we shall often write \il{Een}{$E$-excised
neighbourhood}\emph{$E$-excised neighbourhood}.
\end{defn}

A typical excised neighbourhood $\Om$ of $\lhp\cup\proj\RR$ with
excised set $ E=\{\infty,\xi_1,\xi_2\}$ looks as indicated in
Figure~\ref{fig-excnbh}.
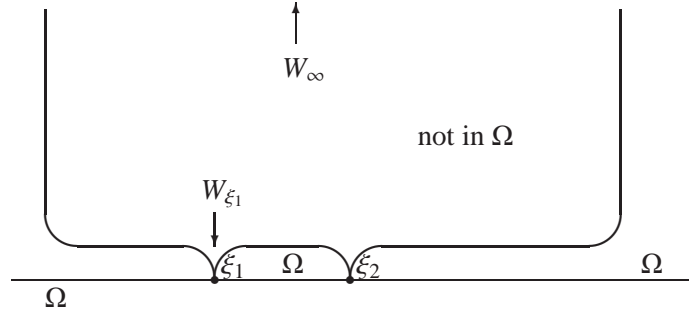
\begin{figure}[ht]
\[\setlength\unitlength{.9cm}
\begin{picture}(10,5)(0,-1)
\put(0,0){\line(1,0){10}}
\put(3.1,.1){$\xi_1$}
\put(5.1,.1){$\xi_2$}
\put(4,.1){$\Om$}
\put(.5,-.4){$\Om$}
\put(9.3,.1){$\Om$}
\put(3,0){\circle*{.1}}
\put(5,0){\circle*{.1}}
\put(3,1){\vector(0,-1){.5}}
\put(2.8,1.2){$W_{\xi_1}$}
\put(4,3){$W_\infty$}
\put(4.2,3.5){\vector(0,1){.6}}
\put(6,2){not in $\Om$}
\thicklines
\put(2.5,0){\oval(1,1)[rt]}
\put(3.5,0){\oval(1,1)[lt]}
\put(4.5,0){\oval(1,1)[rt]}
\put(5.5,0){\oval(1,1)[lt]}
\put(3.5,.5){\line(1,0){1}}
\put(1,1){\oval(1,1)[lb]}
\put(1,.5){\line(1,0){1.5}}
\put(.5,1){\line(0,1){3}}
\put(8.5,1){\oval(1,1)[rb]}
\put(9,1){\line(0,1){3}}
\put(5.5,.5){\line(1,0){3}}
\end{picture}
\]
\caption{An $\{\infty,\xi_1,\xi_2\}$-excised
neighbourhood.}\label{fig-excnbh}
\end{figure}

\begin{defn}For $\xi_1,\ldots,\xi_n\in \proj\RR$ we define spaces of
\il{esav}{excised semi-analytic vector}\il{save}{semi-analytic vector,
excised}\emph{excised semi-analytic
vector}\ir{dsvwdg}{\dsv{2-r}{\om,\wdg}[\xi_1,\ldots,\xi_n], \;
\dsv{2-r}{\fs,\wdg}}
\be \label{dsvwdg}
 \dsv{2-r}{\om,\wdg}[\xi_1,\ldots,\xi_n]\;:=\; \Prj{2-r}^{-1}\,\indlim
 \hol(\Om)\,,\ee
where $\Om$ runs over the $\{\xi_1,\ldots,\xi_n\}$-excised
neighbourhoods.
\end{defn}

\begin{defn}
For $\cond\in \{ \infty,\smp,\wdg\}$ we
define\ir{fs-fsn-cond}{\dsv{2-r}{\fs,\cond},\; \dsv{2-r}{\fsn,\cond}}
\badl{fs-fsn-cond} \dsv{2-r}{\fs,\cond} &\=
\indlim\dsv{2-r}{\om,\cond}[\xi_1,\ldots,\xi_n]\,,\\
\dsv{2-r}{\fsn,\cond} &\= \indlim
\dsv{2-r}{\om,\cond}[\ca_1,\ldots,\ca_n]\,, \eadl
where $\{\xi_1,\ldots \xi_n\}$ runs over the finite subsets of
$\proj\RR$, and $\{\ca_1,\ldots \ca_n\}$ over the finite sets of
cusps of~$\Gm$.
\end{defn}

\rmrk{Notation}The conditions $\infty$, and `$\smp$' can be combined
with `$\wdg$'.
\il{combcond}{$\dsv{v,2-r}{\fs,\cond_1,\cond_2},\;
\dsv{v,2-r}{\fsn,\cond_1,\cond_2}$}For instance, by
$\dsv{2-r}{\fs,\infty,\wdg}$ we mean $\dsv{2-r}{\fs,\infty}\cap
\dsv{2-r}{\fs,\wdg}$.

\begin{prop}\label{prop-sasi}
\begin{enumerate}
\item[i)] $\dsv{2-r}\om [\xi_1,\ldots,\xi_n] |_{2-r}g =
\dsv{2-r}\om[g^{-1}\xi_1,\ldots,g^{-1}\xi_n]$ for each $g\in
\SL_2(\RR)$. Hence
\begin{enumerate}
\item[a)] The space $\dsv{2-r}\fs$ is invariant under the operators
$|_{2-r}g$ with $g\in \SL_2(\RR)$.
\item[b)] The space $\dsv{2-r}\fsn$ is invariant under the operators
$|_{2-r}\gm$ for $\gm \in \Gm$.
\end{enumerate}
\item[ii)]The same holds for the corresponding spaces with condition
$\infty$, $\smp$ or $\wdg$ at the singularities.
\item[iii)] $\bsing (\ph|_{2-r}g) = g^{-1}\, \bsing\ph$ for $\ph\in
\dsv{2-r}\fs$ and $g\in \SL_2(\RR)$.
\end{enumerate}
\end{prop}
\begin{proof}Most is clear. For Part~ii) we check that the conditions
are stable under the operators $|_{2-r} g$.
\end{proof}

\rmrk{Notation}We denote for each of these spaces $\dsv{2-r}\fs$,
$\dsv{2-r}{\fsn}$, $\dsv{2-r}{\fs,\cond}$, $\dsv{2-r}{\fsn,\cond}$,
by \il{gmactfsfsn}{$\dsv{v,2-r}\ast$}$\dsv{v,2-r}\ast$ that space
provided with the action $|_{v,2-r}$ of~$\Gm$.

\subsection{Isomorphic cohomology groups}\label{sect-isocg}

Theorems \ref{THMac} and~\ref{THMcc} give one characterization of the
images of $A_r(\Gm,v)$, $\cusp r(\Gm,v)$ and $M_r(\Gm,v)$ under the
map $\coh r \om$ in Theorem~\ref{THMac} to the \il{analcoh}{analytic
cohomology}analytic cohomology group $H^1(\Gm;\dsv{v,2-r}\om)$. At
this point we have available all $\Gm$-modules to give several more
characterizations of these images, thus extending Theorems
\ref{THMac} and~\ref{THMcc}.

\begin{mainthm}\label{THMiso}Let $\Gm$ be a cofinite discrete subgroup
of $\SL_2(\RR)$ with cusps, and let $v$ be a multiplier system for
the weight $r\in \CC$.
\begin{enumerate}
\item[i)] Suppose that $r\not\in \ZZ_{\geq 2}$.
\begin{enumerate}
\item[a)] The image $\coh r \om
A_r(\Gm,v)=\hpar^1(\Gm;\dsv{v,2-r}\om,\dsv{v,2-r}{\fsn,\wdg})$ is
equal to
\[\hpar^1(\Gm;\dsv{v,2-r}\om,\dsv{v,2-r}{\fs,\wdg})\,,\]
and canonically isomorphic to
\[\hpar^1(\Gm;\dsv{v,2-r}{\fsn,\wdg})\,.\]
\item[b)] The codimension of
$\hpar^1(\Gm;\dsv{v,2-r}\om,\dsv{v,2-r}{\fsn,\wdg})$ in
$H^1(\Gm;\dsv{v,2-r}\om)$ is infinite.
\item[c)] The natural map $\hpar^1(\Gm;\dsv{v,2-r}{\fsn,\wdg})
\rightarrow \hpar^1(\Gm;\dsv{v,2-r}{\fs,\wdg})$ is injective, and its
image has infinite codimension in
$\hpar^1(\Gm;\dsv{v,2-r}{\fs,\wdg})$.
\end{enumerate}
\item[ii)] Suppose that $r\in \RR\setminus \ZZ_{\geq 2}$.

The image $\coh r \om \cusp
r(\Gm,v)=\hpar^1(\Gm;\dsv{v,2-r}\om,\dsv{v,2-r}{\fsn,\infty,\wdg})$
is equal to
\[ \hpar^1(\Gm;\dsv{v,2-r}\om,\dsv{v,2-r}{\fsn,\infty})\,,\quad
\hpar^1(\Gm;\dsv{v,2-r}\om,\dsv{v,2-r}{\fs,\infty})\,,\]
and canonically isomorphic to
\[ \hpar^1(\Gm;\dsv{v,2-r}{\fsn,\infty})\,,\quad
\hpar^1(\Gm;\dsv{v,2-r}{\fs,\infty}). \]
\item[iii)] Suppose that $r\in \RR\setminus \ZZ_{\geq 1}$.
\begin{enumerate}
\item[a)] The image $\coh r \om  M_r n
(\Gm,v)=\hpar^1(\Gm;\dsv{v,2-r}\om,\dsv{v,2-r}{\fsn,\smp,\wdg})$ is
 equal to
\[\hpar^1(\Gm;\dsv{v,2-r}\om,\dsv{v,2-r}{\fsn,\smp})\,,\quad
\hpar^1(\Gm;\dsv{v,2-r}\om,\dsv{v,2-r}{\fs,\smp})\,\]
and canonically isomorphic to $\hpar^1(\Gm;\dsv{v,2-r}{\fsn,\smp})$.
\item[b)] The space
$\hpar^1(\Gm;\dsv{v,2-r}\om,\dsv{v,2-r}{\fsn,\smp,\wdg})$ is
canonically isomorphic to the space
$\hpar^1(\Gm;\dsv{v,2-r}{\fs,\smp})$ if $v(\gm)\neq
e^{-r\ell(\gm)/2}$ for all primitive hyperbolic elements $\gm\in
\Gm$, where $\ell(\gm)$ is the hyperbolic length of the closed
geodesic associated to~$\gm$
\end{enumerate}
\end{enumerate}
\end{mainthm}

\rmrks
\itmi In the statement of the theorem we speak of equality of mixed
parabolic cohomology groups, all contained in
$H^1(\Gm,\dsv{v,2-r}\om)$, and of canonical isomorphisms, given by
natural maps in cohomology corresponding to inclusions of
$\Gm$-modules.

\itm Some of the isomorphisms underlying this theorem are valid for
a wider class of weights. See the results in Sections
\ref{sect-isos-pc} and~\ref{sect-coc-sing}.

\itm Proposition~\ref{prop-icd-hyp} gives some additional information
concerning $\hpar^1(\Gm;\dsv{v,2-r}{\fs,\smp})$ if
$v(\gm)=e^{-r\ell(\gm)/2}$ for some primitive hyperbolic $\gm\in
\Gm$.

\itm We will obtain Theorem~\ref{THMiso} in many steps. We recapitulate the
proof in Subsection~\ref{sect-recap-iso}

\subsection{Harmonic lifts of holomorphic automorphic
forms}\label{sect-hlaf}
The spaces of holomorphic automorphic forms are contained in larger
spaces of harmonic automorphic forms.

\begin{defn}\label{haf-def}Let $r\in \CC$.
\begin{enumerate}
\item[i)] If $U\subset \uhp$ is open and the function $F$ on $U$ is
twice continuously differentiable, then we call $F$ an
\il{r-harm}{$r$-harmonic}$r$-harmonic function on $U$ if $\Dt_r F=0$
for the differential operator \ir{Dtr}{\Dt_r}
\be\label{Dtr} \Dt_r \= - 4 y^2 \frac{\partial^2}{\partial
z\,\partial\bar z} + 2ir y\frac{\partial}{\partial\bar z}\,. \ee

\item[ii)]An \il{harmaf}{harmonic automorphic form}\emph{$r$-harmonic
automorphic form} with the multiplier system $v$ is a function
$F:\uhp\rightarrow \CC$ that satisfies
\begin{enumerate}
\item[a)] $F|_{v,r}\gm = F$ for all $\gm\in \Gm$.
\item[b)] $F$ is $r$-harmonic.
\end{enumerate}
We denote the linear space of such forms by
\il{hn}{$\harm_r(\Gm,v)$}$\harm_r(\Gm,v)$.
\end{enumerate}
\end{defn}

\begin{defn}Let $r\in \CC$. We call the following map $\shad_r$ the
\il{so}{shadow operator}\emph{shadow operator}:\ir{shad}{\shad_r f}
\be\label{shad} (\shad_r F)(z) = 2i\, y^{\bar r}
\,\overline{\frac{\partial
}{\partial \bar z} F(z)}\,. \ee
\end{defn}

A useful property of the shadow operator, which allows us to detect
$r$-har\-monic\-ity, is the following equivalence:
\be\label{harm-detect} F\in C^2(U)\text{ is
$r$-harmonic}\;\Leftrightarrow\; \shad_r F\text{ is holomorphic}\,.
\ee
This is based on the relation $\frac{\partial
}{\partial \bar z} \left ( \shad_r F \right )=-\frac{i y^{\bar
r-2}}{2} \overline{\Delta_r F} $.

The shadow operator induces an antilinear map
$$\shad_r : \harm_r(\Gm,v)
\rightarrow A_{2-\bar r}(\Gm,\bar v)$$
because $\shad_r$ sends elements in the kernel of $\Dt_r$ to
holomorphic functions, and
\be\label{xi-g} \shad_r \bigl( F|_r g) \= (\shad_r F)|_{2-\bar r} g
\quad\text{for each }g\in \SL_2(\RR)
\,. \ee

We have an exact sequence of $\RR$-linear maps
\be 0 \rightarrow A_r(\Gm,v) \rightarrow \harm_r(\Gm,v)
\stackrel{\shad_r}\rightarrow A_{2-\bar r}(\Gm,\bar v)\,. \ee

\begin{defn}\label{hldef}Let $F \in A_{2-\bar r}(\Gm,\bar v)$. We call
$H$ a
\il{hl}{harmonic lift}\emph{harmonic lift} of $F$ if
\[ H\in \harm_r(\Gm,v)\quad\text{ and }\quad \shad_r H \= F\,.\]
\end{defn}

In \S\ref{app-do} in the Appendix we discuss $r$-harmonic automorphic
forms on the universal covering group.

\begin{rmk}
The action $|_{v,r}$ of $\Gamma$ in the functions on $\uhp$ gives rise
to various spaces of invariants, for instance:
\begin{align*}
C_{v,r}^\infty(\Gamma\backslash\uhp)&\= \Bigl\{ f\in
C^\infty(\uhp)\;:\; f|_{v.r}\gm=f\text{ for all }\gm\in \Gm\Bigr\}\;;
\displaybreak[0]\\
\ker \Bigl( \Dt_r-\ld&: C_{v,r}^\infty(\Gm\backslash\uhp)
\longrightarrow C_{v,r}^\infty(\Gm\backslash\uhp)\Bigr)
\quad\text{ with }\ld\in \CC\,,\\
&\qquad \text{real-analytic automorphic forms}\;;
\displaybreak[0]\\
\harm_r(\Gm,v)&=\ker\Bigl( \Dt_r: C_{v,r}^\infty(\Gm\backslash\uhp)
\longrightarrow C_{v,r}^\infty(\Gm\backslash\uhp)\Bigr)\,,\\
 &\qquad\text{harmonic automorphic forms}\;;
 \displaybreak[0]\\
A_r(\Gm,v)&= \harm_r(\Gm,v) \cap\hol(\uhp)\,,\\
&\qquad\text{holomorphic automorphic forms}\,.
\end{align*}
For each of these spaces growth conditions at the cusp give rise to
subspaces.

\rmrk{Real-analytic}A function on an open set $U\subset \RR$ is
real-analytic if on an open neighbourhood of each $x_0\in U$ it is
given by a convergent power series of the form $\sum_{n\geq 0}
c_n\,(x-\nobreak x_0)^n$. This gives a holomorphic
extension of the function to a neighbourhood of $U$ in~$\CC$.

A function on an open set $U\subset\CC$) is real-analytic if an
open neighbourhood of each point $z_0=x_0+iy_0\in U$ it is given by an
absolutely convergent power series $\sum_{n,m\geq 0} c_{n,m}
(x-\nobreak x_0)^n\,(y-\nobreak y_0)^m$, or
equivalently by a convergent power series $\sum_{n,m\geq
0}d_{n,m}\bigl(z-\nobreak
z_0\bigr)^n\,\bigl(\overline{z-z_0}\bigr)^m$.
The latter representations give a holomorphic extension of the
function to some neighbourhood in $\CC^2$ of the image of the domain
of the function under the map $z\mapsto (z,\bar z)$.

On $\proj\CC$ one proceeds similarly, using power series in $1/z$
and $1/\bar z$ on a neighbourhood of $\infty$.
\end{rmk}


\section{Modules and cocycles}\label{sect-mod-coc}
In Section~\ref{sect-defnot} we fixed the notations and defined most
of the modules occurring in the main theorems in the Introduction.
Now we turn to the map from automorphic forms to cohomology induced
by~\eqref{psiz0def}. We also discuss the relation with the theorem of
Knopp and Mawi \cite{KM10}.

\subsection{The map from automorphic forms to
cohomology}\label{sect-af-coh}

\begin{defn}Let $F$ be any holomorphic function on $\uhp$.
\ir{omr}{\om_r(F;t,z)}
\be\label{omr}
\om_r(F;t,z) \;:=\; \bigl( z-t\bigr)^{r-2}\, F(z)\, dz\, \ee
for $z\in \uhp$ and $t\in \lhp$; we take $-\frac\pi2 < \arg(z-\nobreak
t ) < \frac{3\pi}2$.
\end{defn}

This defines $\om_r(F;t,z)$ as a holomorphic $1$-form in the variable
$z$. The presence of the second variable enables us to view it as a
differential form with values in the functions on~$\lhp$.

\begin{lem}\label{lem-om-sp}
\begin{enumerate}
\item[i)]The differential form $\om_r(F;\cdot,z)$ has values in
$\dsv{2-r}\om$.
\item[ii)] If $r\in \ZZ_{\geq 2}$ it has values in the subspace
$\dsv{2-r}\pol$.
\end{enumerate}
\end{lem}
\begin{proof}In the projective model the differential form looks as
follows:\ir{omr-prj}{\om_r^\prj(F;t,z)}
\be\label{omr-prj}
\om_r^\prj(F;t,z) \;:=\; \bigl(\Prj{2-r} \om_r(F; \cdot,z)\bigr)
\,(t) \= \Bigl( \frac{z-t}{i-t}\Bigr)^{r-2}\, F(z)\, dz\,,
\ee
where for $t\in \lhp$ and $z\in \uhp$ we have $\arg\frac{z-t}{i-t} \in
(-\pi,\pi)$. The factor $\Bigl( \frac{z-t}{i-t}\Bigr)^{r-2}$ is
holomorphic for $t\in \proj\CC\setminus p$, where $p$ is a path in
$\uhp$ from $i$ to $z$, which implies Part~i). Part~ii) is clear
from~\eqref{omr}.
\end{proof}

\begin{lem}\label{lem-omprop}Let $F$ be holomorphic on~$\uhp$.
\begin{enumerate}
\item[i)] $\om_r(F;\cdot,g z) |_{2-r} g = \om_r\bigl( F|_r
g;\cdot,z)$ for each $g\in \SL_2(\RR)$.
\item[ii)] $\om_r(F;\cdot,\gm z) |_{v,2-r} \gm = \om_r\bigl( F|_{v,r}
\gm;\cdot,z\bigr)$ for each $\gm\in \Gm$.
\item[iii)] $\int_{\gm z_1}^{\gm z_2} \om_r(F;\cdot,z)|_{v,2-r}\gm =
\int_{z_1}^{z_2}\om_r\bigl(F|_{v,r}\gm;\cdot,z)$
for $\gm\in \Gm$ and $z_1,z_2\in\uhp$. The integral is independent of
the choice of the path.
\end{enumerate}
\end{lem}
\begin{proof}
\begin{enumerate}
\item[i)] The relation amounts for $g=\matc abcd\in \SL_2(\RR)$ to
\[
(ct+d)^{r-2}\, \Bigl(\frac{z-t}{(ct+d)\,(cz+d)} \Bigr)^{r-2}\, F(gz)\,
\frac{dz}{(cz+d)^2}\= (z-t)^{r-2} \, (cz+d)^{-r}\, F(z)\, dz\,. \]
With the argument conventions in~\eqref{ac} for $\arg(cz+\nobreak d)$
with $z\in \uhp$ and $z\in \lhp$ this equality turns out to hold for
$t=-i$ and $z=i$. It extends holomorphically for $t\in \lhp$ and
$z\in \uhp$.
\item[ii)] With $g=\gm\in \Gm$ we multiply the relation in Part~i) by
$v(\gm)^{-1}$.
\item[iii)] We note that
\[\int_{\gm z_1}^{\gm z_2} \om_r(F;\cdot,z)|_{v,2-r}\gm \;(t)
\= \int_{z_1}^{z_2}\om_r(F; \cdot;\gm z)|_{v,2-r}\gm \;(t)
\= \int_{z_1}^{z_2} \om_r\bigl( F|_{v,r}\gm;t, z)\,. \]
The differential form is holomorphic, hence closed, and the integral
does not depend on the path of integration, only on the end-points.
\end{enumerate}
\end{proof}

\begin{prop}\label{prop-cohrom}Let $F\in A_r(\Gm,v)$.
\begin{enumerate}
\item[i)] The map $\ps^{z_0}_F:\gm \mapsto \ps^{z_0}_{F,\gm}$ defined
in~\eqref{psiz0def} in the introduction is an element of
$Z^1(\Gm;\dsv{v,2-r}\om)$.\il{psiz0prop}{$\ps_F^{z_0}$}
\item[ii)] The linear map \il{cohom}{$\coh r \om$}$\coh r \om :
A_r(\Gm,v)
\rightarrow H^1(\Gm;\dsv{v,2-r}\om)$ associating to $F$ the cohomology
class of $\ps^{z_0}_{F}$ is well defined.
\item[iii)] If $r\in \ZZ_{\geq 2}$ then $\coh r \om\,
A_r(\Gm,v)\subset H^1(\Gm;\dsv{v,2-r}\pol)$.
\end{enumerate}
\end{prop}
\begin{proof}
\begin{enumerate}
\item[i)]Since we integrate over a compact set in $\uhp$ the values
$\ps^{z_0}_{F,\gm}$ are in $\dsv{2-r}\om$. For the cocycle relation
we compute for $\gm,\dt\in \Gm$:
\begin{align*} \ps^{z_0}_{F,\gm\dt}-&\ps^{z_0}_{F,\dt} \= \biggl(
\int_{\dt^{-1}\gm^{-1}z_0}^{z_0} - \int_{\dt^{-1}z_0}^{z_0} \biggr)\,
\om_r(F;\cdot,z)
\= \int_{\dt^{-1}\gm^{-1}z_0}^{\dt^{-1}z_0} \om_r(F;\cdot z)\\
&\stackrel{\text{Part iii) in Lemma \ref{lem-omprop}}}=
\int_{\gm^{-1}z_0}^{z_0}\, \om\bigl(
F|_{v,r}\dt^{-1};\cdot,z)|_{v,2-r}\dt
\= \ps^{z_0}_{F,\gm}|_{v,2-r}\dt\,.
\end{align*}
\item[ii)] \ To see that the cohomology class of $\ps_F^{z_0}$ does
not depend on the choice of the base point $z_0$ we check that with
two base points $z_0$ and $z_1$ the difference is a coboundary:
\begin{align*}
\ps^{z_0}_{F,\gm} &- \ps^{z_1}_{F,\gm} \= \biggl(
\int_{\gm^{-1}z_0}^{z_0} - \int_{\gm^{-1}z_1}^{z_1}\biggr)\,
\om_r(F;\cdot,z)\\
&\= \bigl( \int_{\gm^{-1}z_0}^{\gm^{-1}{z_1}} - \int_{z_0}^{z_1}
\biggr)\, \om_r(F;\cdot,z)
\stackrel{\text{Part iii) in Lemma \ref{lem-omprop}}}= b|_{v,2-r}\gm -
b\,,
\end{align*}
with $b = \int_{z_0}^{z_1}\om_r(F;\cdot;z)$ in $\dsv{v,2-r}\om$. Hence
$\coh r \om$ is well defined.
\item[iii)] See Part~ii) of Lemma~\ref{lem-om-sp}.\qedhere
\end{enumerate}
\end{proof}

\subsection{Cusp forms} A cusp form $F\in \cusp r(\Gm,v)$ decays
exponentially at each cusp $\ca$ of~$\Gm$, and we can define for the
cusp~$\ca$\ir{psparb}{\ps^\ca_F}
\be\label{psparb} \ps^\ca_F: \gm\mapsto\ps^\ca_{F,\gm} (t)
\;:=\; \int_{\gm^{-1}\ca}^\ca \om_r(F;t,z)\,. \ee
We use $\s_\ca$ such that $\ca=\s_\ca\infty$ and $\pi_\ca= \s_\ca
T\s_\ca^{-1}$ as in \S\ref{sect-af}. If $|v(\pi_{\ca})|\neq 1$, then
$F\bigl( \sigma_{\ca}(x+\nobreak iy)\bigr)$ may be unbounded as a
function of $x\in \RR$. Then it is important to approach the cusps
$\ca$ and $\gm^{-1}\ca$ along a geodesic half-line.

\rmrks
\itmi If $|v(\pi_\ca)|\neq 1$ some care is needed in the choice of the
path of integration in its approach of~$\ca$. Now $F(\s_\ca z)$ may
have exponential growth in $x=\re z$, although for a given $x$ it has
exponential decay as $\im(z)=y\uparrow \infty$. The integral
converges uniformly if we restrict $x$ to a suitable compact set, for
instance by requiring that the path approaches $\ca$ along a geodesic
half-line.

\itm Proposition~\ref{prop-cohrom} extends easily to the situation
with $\ca$ as the base point, and we see that $\ps^\ca_F$ is a
cocycle, and that a change in the choice of the cusp $\ca$ adds a
coboundary. The following lemma prepares the identification of
$\dsv{v,2-r}{\fsn,\infty,\wdg}$ as a $\Gm$-module in which
$\psi_F^{\ca}$ takes its values.

\begin{lem}\label{lem-inf}Let $\ca=g \infty$ with $g\in \SL_2(\RR)$.
Suppose that $F$ is a holomorphic function on~$\uhp$ and that there
is $a>0$ such that $F(g z )
\= \oh(e^{-a y})$ as $\im(z)=y \rightarrow\infty$ for each value of
$x=\re z$. For $z_0\in \uhp$ and $t\in \lhp$ we define $h$ by
\be\label{h-int-p} h(t) \= \int_{z_0}^\ca \om_r(F;t,z)\,, \ee
Then $h$ extends holomorphically across $\proj\RR\setminus\{\ca\}$ and
defines an element of the space $\dsv{2-r}{\om,\infty,\wdg}[\ca]$.
\end{lem}
\begin{proof}By Part~i) of Lemma~\ref{lem-omprop} it suffices to
consider the case $\ca=0$ and $g=\matr0{-1}10$. Inspection of
\eqref{omr-prj} shows that
\[ (\Prj{2-r}h)(t)\=\int_{z_0}^0 \Bigl( \frac{z-t}{i-t}\Bigr)^{r-2}\,
 F  (z)\, dz\]
extends holomorphically to $\CC\setminus p$, where $p$ is path from
$z_0$ to~$0$. Since we can take this path as a geodesic half-line, we
have a holomorphic extension to a $\{0\}$-excised neighbourhood. Hence
$h\in \dsv{2-r}{\om,\wdg}[0]$.

To show that $h\in \dsv{2-r}\infty$ we need to show that $h(t)$ has
Taylor expansions of any order at $0$ valid on a region $\{t\in
\CC\;:\; \im t\leq 0,\; |t|<\e \}$ for some $\e>0$.

We can assume that the path of integration approaches $0$ vertically,
and hence
\[ h(t) \= -i \int_0^\e (iy-t)^{r-2}\, F(iy)\,dy + \text{ a
contribution in $\dsv{v,2-r}\om$}\,.\]
The contribution in $\dsv{v,2-r}\om$ is automatically in
$\dsv{v,2-r}\infty$, so we consider only the integral. For $y\in
(0,\e]$, $|t|\leq \e$ and $\im t\leq 0$ we have
\[(iy-t)^{r-2} \= e^{\pi i r/2}\, y^{r-2}\, \bigl(
1+it/y\bigr)^{r-2}\,,\]
with $\re it/y\geq 0$. Taylor expansion of the factor $(1+\nobreak
it/y)^{r-2}$ is not completely standard, since $it/y$ is unbounded
for the values of $t$ and $y$ under consideration. We use the version
of Taylor's formula in Lang \cite[\S6, Chap.~XIII]{La93}. It shows
that the error term in the Taylor expansion of order $N-1$ of
$(1+\nobreak q)^a$ is
\[ \oh_N \biggl( \int_0^1 (1-x)^{N-1}\,(1+x q)^{a-N}\, q^N\, dx\biggr)
\= \oh_N(q^N)\,,\]
if $N>\re a$. (The subscript $N$ indicates that the implicit constant
may depend on~$N$.) For sufficiently large $N$ this leads to
\[ \bigl(1+it/y\bigr)^{r-2} \= \sum_{n=0}^{N-1}\binom{r-2}n\, i^n \,
t^n \, y^{-n}
+\oh_N\bigl(t^N y^{-N}\bigr)\,\]
and hence
\begin{align*} \int_0^\e (iy-t)^{r-2} \, F(iy)\, dy
&\= e^{\pi i r/2} \sum_{n=0}^{N-1}\binom{r-2}n i^n t^n \int_0^\e
y^{r-2-n}\, F(iy)\, dy\\
&\qquad\hbox{}
+ \oh_N \biggl( \int_0^\e y^{r-2-N} \, t^N\, F(iy)\, dy\biggr)\,.
\end{align*}

The exponential decay of $F$ implies that all integrals converge, and
we obtain a Taylor expansion of the integral of any order that is
valid for $t\in \lhp\cup\RR$ near~$0$.
\end{proof}

\rmrk{Parabolic cohomology and mixed parabolic cohomology} For $z_0\in
\uhp$ the cocycle $\ps^{z_0}_F$ for a cusp form $F$ takes values in
$\dsv{v,2-r}\om$. The next result shows that $\ps_F^{\ca}$ is a
parabolic cocycle in a larger module, and relates both cocycles.

\begin{prop}\label{prop-cu-inf}
Let $r\in \CC$.
\begin{enumerate}
\item[i)] For each cusp $\ca$ of $\Gm$ and each $F\in \cusp r(\Gm,v)$
the cocycle $\ps_F^\ca$ defined in \eqref{psparb} is a parabolic
 cocycle in $\hpar^1(\Gm;\dsv{v,2-r}{\fsn,\infty,\wdg})$.
\item[ii)] Associating to $F\in \cusp r(\Gm,v)$ the cohomology class
$[\ps^\ca_F]$ defines a linear map\ir{rinf}{\coh r \infty}
\be\label{rinf}
\coh r \infty : \cusp r(\Gm,v) \longrightarrow
\hpar^1(\Gm;\dsv{v,2-r}{\fsn,\infty,\wdg})\,.\ee
\item[iii)] $\coh r \om\, \cusp r(\Gm,v)
\,\subset\,
\hpar^1(\Gm;\dsv{v,2-r}\om,\dsv{v,2-r}{\fsn,\infty,\wdg})$.
\item[iv)] The following diagram is commutative:
\bad \xymatrix{ \cusp r(\Gm,v) \ar[r]^(.35){\coh r \om} \ar[rd]_{\coh
r \infty}
& \hpar^1(\Gm;\dsv{v,2-r}\om,\dsv{v,2-r}{\fsn,\infty,\wdg})
\ar@{^{(}->}[r] \ar[d]
& H^1(\Gm;\dsv{v,2-r}\om)
\\
& \hpar^1(\Gm;\dsv{v,2-r}{\fsn,\infty,\wdg})
} \ead
The vertical arrow denotes the natural map associated to the inclusion
$ \dsv{v,2-r}\om \subset \dsv{v,2-r}{\fsn,\infty,\wdg} $.
\end{enumerate}
\end{prop}

\rmrke For $r\in \ZZ_{\geq 2}$, the linear maps $\coh r \om$ and $\coh
r \infty$ take values in the much smaller $\Gm$-module
$\dsv{v,2-r}\pol$.

\begin{proof}We split the integral in \eqref{psparb} as
$-\int^{\gm^{-1}\ca}_{z_1} + \int_{z_1}^\ca$ for any $z_1\in \uhp$,
and find with Lemma~\ref{lem-inf} that $\ps^\ca_F \in
\dsv{v,2-r}{\fs,\infty,\wdg} \cap \dsv{v,2-r}\om[\ca,\gm^{-1}\ca]
\subset \dsv{v,2-r}{\fsn,\infty,\wdg}$. So $\ps^\ca_F \in
Z^1(\Gm;\dsv{v,2-r}{\fsn,\infty,\wdg})$.

Like in the proof of Proposition~\ref{prop-cohrom}, replacing the cusp
$\ca$ by another cusp means adding a coboundary in
$B^1(\Gm;\dsv{v,2-r}{\fsn,\infty,\wdg})$. We have
$\ps^\ca_{F,\pi_{\ca}}=0$, and hence for a cusp $\eta$ there is $p\in
\dsv{v,2-r}{\fsn,\infty,\wdg}$ such that
\[ \ps^\ca_{F,\pi_\eta} \= \ps^\eta_{F,\pi_\eta} +
p|_{v,2-r}(\pi_\eta-1)\,\in\, 0+
\dsv{v,2-r}{\fsn,\infty,\wdg}|_{v,2-r}(\pi_\eta-1)\,.\]
So $\ps^\ca_F$ is a parabolic cocycle, and $F\mapsto [\ps^\ca_F]$
defines a linear map $\coh r \infty$ as in Part~ii).

For $F\in \cusp r(\Gm,v)$ and $z_0\in \uhp$ we have for each cusp
$\ca$ of~$\Gm$
\[ \ps^{z_0}_{F,\pi_{\ca}} \= \ps^\ca_{F,\pi_{\ca}} + h
|_{v,2-r}(\pi_{\ca}-1)\,,\]
with $h= \int_{z_0}^\ca \om_r(F;\cdot,z)$. With Lemma~\ref{lem-inf} we
have $\ps^{z_0}_{F,\pi_{\ca}} \in
\dsv{v,2-r}{\fsn,\infty,\wdg}|_{v,2-r}(\pi_{\ca}-\nobreak 1)$. This
gives Parts iii)
and~iv).
\end{proof}

\subsection{The theorem of Knopp and Mawi}\label{sect-KM}
Suppose that $\infty$ is a cusp of~$\Gm$, and that $\Gm_\infty$ is
generated by $T= \matc1101 $. (This can be arranged by conjugation in
$\SL_2(\RR)$.) The involution $\Ci$ in \eqref{Ci} gives a parabolic
cocycle $\Ci \ps^\infty_F$ of the form
\be\label{Knopp-coc} (\Ci \ps^\infty_{F,\gm})(w) \= \overline{
\int_{\gm^{-1}\infty}^\infty
(z-\bar w)^{r-2}\, F(z)\, dz } \= \int_{\gm^{-1}\infty}^\infty \bigl(
\bar z-w\bigr)^{\bar r-2}\, \overline{F(z)}\, d\bar z \,. \ee
This describes Knopp's cocycle \cite[(3.8)]{Kn74}. In that paper the
weight $r$ is real and the multiplier system $v$ unitary, so $\bar
v=v^{-1}$. (Actually, in \cite{Kn74} the multiplier system for $F$ is
called $\bar v$, and the weight is called $r+2$.)

The values of $\Ci \psi^\infty_F$ are in the space $\Ci
\dsv{2-r}{\fsn,\infty}$ which is contained in the space
\ir{Pkr}{\Pcal=\Ci \dsv{2-r}{-\infty}}
\be\label{Pkr} \Pcal\;:=\;\Ci \dsv{2-r}{-\infty} \= \Bigl\{ \ph \in
\hol(\uhp)\;:\; \exists_{A\in\RR}\; \ph(z) = \oh\bigl( y^{-A}\bigr)+
\oh\bigl(|z|^A\bigr) \Bigr\}\ee
(polynomial growth), which is invariant under the action $|_{\bar
v,\bar r}$ of $\Gm$. (The notation $\Pcal$ is taken
from~\cite{Kn74}.)

Knopp \cite{Kn74} conjectured that the map $F\mapsto
[\Ci\ps_F^\infty]$ gives a bijection $\cusp r(\Gm,v)
\rightarrow \hpar^1(\Gm,\Pcal)$, and proved this for $r\in
\RR\setminus(0,2)$. He also gives a proof, by B.A.\,Taylor, that
 $\hpar^1(\Gm;\Pcal)= H^1(\Gm;\Pcal)$. In \cite{KM10} Knopp and Mawi
prove the isomorphism for all weights $r\in \RR$ and unitary
multiplier systems $v$. Transforming their result to the lower
half-plane we obtain the following theorem:
\begin{thm}\label{thm-KM}{\rm(Knopp, Mawi) } Let $v$ be a unitary
multiplier system on $\Gm$ for the weight $r\in \RR$. Then
\be \label{Knopp}\cusp r (\Gm,v)\;\cong \;
H^1(\Gm;\dsv{v,2-r}{-\infty})
\;\cong\; \hpar^1(\Gm; \dsv{v,2-r}{-\infty})\,.\ee
\end{thm}
In combination with the, not yet proven, Theorems \ref{THMac}
and~\ref{THMcc} we obtain the following commuting diagram, valid for
weights $r\in \RR\setminus \ZZ_{\geq 2}$ and unitary multiplier
systems:
\badl{KM-diag} \xymatrix{ A_ r(\Gm,v) \ar@{^{(}->}[r]^(.4){\coh r \om}
& H^1(\Gm;\dsv{v,2-r}\om) \ar[d] \\
\cusp r(\Gm,v) \ar[r]^(.4)\cong \ar@{^{(}->}[u]&
H^1(\Gm;\dsv{v,2-r}{-\infty})
} \eadl

This implies that there is a complementary subspace $X$ giving a
direct sum decomposition $A_r(\Gm,v) = \cusp r(\Gm,v) \oplus X$, such
that for $F\in X$ the cocycle $\ps^{z_0}_F$ becomes a coboundary in
$Z^1(\Gm;\dsv{v,2-r}{-\infty})$. Then there is $H\in
\dsv{v,2-r}{-\infty}$ such that $H|_{v,2-r}\gm - H =
\psi^{z_0}_{F,\gm}$ for all $\gm\in \Gm$, in other words
\[ \int_{\gm^{-1}z_0}^{z_0} (z-t)^{r-2}\, F(z)\,dz \=
(H|_{v,2-r}\gm)(t)-H(t)\,.\]

\begin{rmk}\label{rmk-hol-ahol}
The operator $\Ci$ can also be applied to the linear map $\coh r \om
$. Thus we have two $\RR$-linear maps from automorphic forms to
cohomology:
\badl{hol-ahol} \coh r \om &: A_r(\Gm,v) \rightarrow
H^1(\Gm;\dsv{v,2-r}\om)\,,\\
\Ci \coh r \om &: A_r(\Gm,v) \rightarrow H^1(\Gm;\Ci
\dsv{v,2-r}\om)\,. \eadl
The second map is antilinear.

These two maps become interesting in the case $r\in \ZZ_{\geq 2} $
with a real-valued multiplier system $v$. Then $\dsv {v,2-r}\om$ and
$\Ci\dsv{v,2-r}\om$ have a nonzero intersection,
namely~$\dsv{v,2-r}\pol$.\end{rmk}

\subsection{Modular group and powers of the Dedekind eta-function} The
\il{mgp}{modular group}modular group
\il{Gmodgen}{$\Gmod$}$\Gmod=\SL_2(\ZZ)$ is generated by $T=\matc1101$
and \il{Smat}{$S=\matr0{-1}10$}$S=\matr0{-1}10$. In the quotient
$\overline{\Gm(1)}=\SL_2(\ZZ)$ the relations are $\bar S^2=1$ and
$(\bar S\bar T)^3=1$. There is a one-parameter family of multiplier
systems parametrized by $r\in \CC\bmod12\ZZ$, determined
by\ir{vr-mod}{v_r}
\be \label{vr-mod}v_r( T)\=e^{\pi i r/6}\,, \qquad v_r(S) \= e^{-\pi i
r/2}\,. \ee
It can be used for weights $p\equiv r\bmod 2$. The complex power
$\eta^{2r}$ of the \il{etapow}{powers of the Dedekind
eta-function}\il{Def}{Dedekind eta-function}Dedekind eta-function can
be chosen in the following way:\ir{eta2r}{\eta^{2r}}
\be\label{eta2r} \eta^{2r}(z) \;:=\; e^{2r\,\log\eta(z)}\,,\quad
\log\eta(z) \= \frac{\pi i}{12} - \sum_{n\geq 1}\s_{-1}(n)\, e^{2\pi
i n z}\,.
  \ee
It defines $\eta^{2r}\in A_r\bigl(\Gmod,v_r\bigr)$. The Fourier
expansion at the cusp~$\infty $ has the form
\be \eta^{2r}(z) \= \sum_{k\geq 0} p_k(r)\, e^{2\pi i (12k+r)z/12}\,,
\ee
where the \il{pk-eta}{$p_k(r)$}$p_k(r)$ are polynomials in $r$ of
degree $k$ with rational coefficients. These polynomials have
integral values at each $r\in \frac12\ZZ$. For $\re r>0$ we have
$\eta^{2r}\in \cusp r\bigl(\Gmod,v_r\bigr)$, and the parabolic
cocycle $ \ps^\infty_{\eta^{2r}}$ given by
\be \ps_{\eta^{2r},\gm}^\infty(t) \= \int_{\gm^{-1}\infty}^\infty
(z-t)^{ r -2}\, \eta^{2r}(z) \, dz\,.\ee
Convergence is ensured by the exponential decay of $\eta^{2r}(z)$ as
$y\uparrow \infty$, and by the corresponding decay at other cusps by
the invariance of $\eta^{2r}$ under $|_{v_r,r}\gm$.

Since $T\infty=\infty$ we have $\ps_{\eta^{2r},T}^\infty=0$. The
cocycle $\ps_{ \eta^{2r}}^\infty$ is determined by its value on the
other generator
\be \ps^\infty_{\eta^{2r},S} (t) \= \int_0^\infty( z-t)^{r-2}\,
\eta^{2r}(z) \, dz\,. \ee
We have $\ps_{\eta^{2r},S}^\infty \in
\dsv{v_r,2-r}{\om,\wdg}[0,\infty] \subset \dsv{v_r,2-r}{\fsn,\wdg}$.
This function is called the \il{pfeta2r}{period function of powers of
the $\eta$-function}\emph{period function} of $\eta^{2r}$. The
relations between $\bar S$ and $\bar T$ imply
\be \ps_{\eta^{2r},S}^\infty |_{v_r,2-r} S \= -
\ps_{\eta^{2r},S}^\infty\,,\qquad \ps_{\eta^{2r},S}^\infty
|_{v_r,2-r}\bigl(1+ST+STST)\=0\,,\ee
which is equivalent to
\be \label{perfun}
\ps_{\eta^{2r},S}^\infty|_{v_r,2-r}S\= -\ps_{\eta^{2r},S}^\infty
\,,\quad \ps_{\eta^{2r},S}^\infty \=
\ps_{\eta^{2r},S}^\infty|_{v_r,2-r}\bigl( T
+TST\bigr)\,.\ee

Let us put\ir{Irs}{I(r,s)}
\be\label{Irs}
I(r,s) \;:=\; \int_0^\infty y^s \, \eta^{2r}(iy)\,\frac {dy}y. \ee
The decay properties of $\eta^{2r}$ imply that this function is
holomorphic in $(r,s)$ for $\re r>0$ and $s\in \CC$.

The reasoning in the proof of Lemma~2.5 gives that for a given $\e>0$
and $t\in \lhp$ with $|t|<\e$ we have
\begin{align*} i  \int_0^\e (iy-t)^{r-2} \eta^{2r}(iy) \, dy \= e^{\pi i
r/2} \sum_{n=0}^{N-1} \binom{r-2}n i^n t^n \int_0^\e
y^{r-2-n}\eta^{2r}(iy)\, dy + \oh_n(t^N)
\end{align*}
for all sufficiently large $N$. The integral over $(\e,\infty)$ can be
computed by direct insertion of the Taylor series for $(iy-\nobreak
t)^{r-2}$. Since $t^N=\oh(\e^N)$, this leads to the following
equality for the period function:
\badl{pf-Iint} \ps^\infty_{\eta^{2r},S}&(t)
\= e^{\pi i r/2} \sum_{n\geq 0} i^n \,\binom{r-2}n\, I(r,r-1-n)
\, t^n\,, \eadl
for $\re r>0$, $s\in\CC$ and $t\in \lhp$ near~$0$.

For a real weight $ r >0$ one has the estimate $p_k(r) = \oh(k^{r/2})$
from the fact that $\eta^{2r}$ is a cusp form. For $\re s>1+ r/12$
the integral $I(r,s)$ can be expressed in terms of the
\il{Lser}{$L$-series}$L$-series\ir{L-etap}{L(\eta^{2r},s)}
\badl{L-etap} L(\eta^{2r},s) &\= \sum_{k\geq
0}\frac{p_k(r)}{(r/12+k)^s}\,,\\
I(r,s) &\= (2\pi)^{-s} \, \Gf(s)\, L(\eta^{2r},s)
\,. \eadl
Usually one  defines the analytic continuation of $L$-functions by the
expressing it in the period integral~\eqref{pf-Iint}.
\medskip

If $\re r\leq 0,$ $\ps_{\eta^{2r}}^{z_0}$ is defined only with a base
point $z_0\in \uhp$. For instance, the case $r=0$ gives the constant
function $1 = \eta^0\in A_0\bigl(\Gmod,1\bigr)$, with the trivial
multiplier system $v_0=1$, for which
\be\label{1-psi} \ps^{z_0}_{\eta^0,\gm}(t)
\= \frac1{\gm^{-1}z_0-t} - \frac1{z_0-t}\,. \ee
It can be checked by a direct computation that $\ps^{z_0}_{1,\gm}-
\ps^{z_1}_{1,\gm} = b|_{1,2}(\gm-\nobreak 1)$, with $b(t) =
\frac1{z_0-t}-\frac1{z_1-t}$.

We use this to find a substitute for the cocycle
$\ps^\infty_{\eta^0}$. The rational function
$b_\infty(t)=\frac1{z_0-t}$ is an element of
$\dsv{1,2}{\om,\wdg}[\infty]$. Subtracting the coboundary $\gm\mapsto
b_\infty|_{2}(\gm-\nobreak 1)$ from $\ps^{z_0}_{\eta^0}$ gives the
parabolic cocycle $\tilde \ps \in
Z^1\bigl(\Gmod;\dsv{1,2}{\fsn,\wdg}\bigr)$ given on $\gm=\matc
abcd\in \Gmod$ by
\be\label{ps-r=0} \tilde\ps_\gm (t) \= \frac{-c}{ct+d}\,. \ee
This cocycle $\tilde \ps$ is parabolic, since $\tilde\ps_T=0$. It
gives the period function $\tilde\ps_S = \frac{-1}t$ in
$\dsv{1,2}{\fsn,\wdg}$. It is in the subspace of rational functions,
hence one calls it a \il{rpf}{rational period function}\emph{rational
period function}. In \S\ref{sect-prfhl} we will return to this
example.

\subsection{Related work}\label{sect-lit2}
Much of the work on the relation between automorphic forms and
cohomology is done for integral weights at least~$2$. The association
of cocycles to automorphic forms is stated clearly in 1957 by
Eichler, \cite[\S2]{Ei57}. Eichler gives the integral
in~\eqref{psiz0def}, and notes \cite[(17),\S2]{Ei57} that for cusp
forms the cocycles have the property that we now call parabolic.

The idea can be found earlier in the literature. As pointed out
in~\cite{DIT10}, Poincar\'e mentions already in 1905
\cite[\S3]{Poinc} the repeated  antiderivative of automorphic forms
and polynomials measuring the non-invariance. Also Cohn \cite{Co56}
mentions this relation in the main theorem, for modular forms of
weight~$4$.

Shimura \cite{Sh59} studies the relation between cusp forms and
cohomology groups with the aim of obtaining a lattice in the space of
cusp form such that the quotient is an abelian variety. He discusses
real and integral structures in the cohomology groups.\smallskip

Since then the relation between automorphic forms and cohomology has
been studied in numerous papers, of which we here mention
Manin~\cite{Ma73}.

The use of the space of rational functions for cocycles associated to
modular forms originates in Knopp \cite{Kn78}. Kohnen and Zagier
\cite{KZ84} used it for period functions on the modular group. In
\cite{KZ84} the concept of mixed parabolic cohomology seems to be
arising. See also \cite{Za91}.


\section{The image of automorphic forms in cohomology}
\label{sect-image}
The main goal of this section is to show that
\be \coh r \om A_r(\Gm;v) \;\subset\;
\hpar^1(\Gm;\dsv{v,2-r}\om,\dsv{v,2-r}{\fsn,\wdg})\,. \ee
This will contribute to the proof of Theorem~\ref{THMac} (which will
be completed in Subsection~\ref{sect-recap-THMac}). In Subsection
\ref{sect-prfTHMcc} we will describe, under assumptions on $r$
 and~$v$, and based on the truth of Theorem~\ref{THMac}, the images
$\coh r \om \cusp r(\Gm,v)$ and $\coh r \om M_r(\Gm,v)$. This gives
Theorem~\ref{THMcc}.

We start in Subsection~\ref{sect-mpcg} with a simple lemma, with which
we immediately can prove some of the isomorphism in
Theorem~\ref{THMiso} on page~\pageref{THMiso}.

\subsection{Mixed parabolic cohomology groups}\label{sect-mpcg}To show
that $\ps\in Z^1(\Gm;\dsv{v,2-r}\om)$ is a parabolic cocycle in
$\zpar^1(\Gm;\dsv{v,2-r}\om,W)$ for some $\Gm$-module $W\subset
\dsv{v,2-r}\fs$ we have to find for each cusp $\ca$ of $\Gm$ an
element $h_{\ca}\in W$ such that
\be \label{parb-eq}
\ps_{\pi_{\ca}} \= h_{\ca}|_{v,2-r}(\pi_{\ca}-1)\,. \ee
The following result gives the position of the singularities of the
solutions.

\begin{lem}\label{lem-sing-inv}If $h\in \dsv{2-r}\fs$ satisfies
$\ld^{-1}\,h|_{2-r}\pi- h \in \dsv{2-r}\om$ for a parabolic element
$\pi\in \SL_2(\RR)$ and $\ld \in \CC^\ast$, then $\bsing h \subset
\{\ca\}$, where $\ca$ is the unique fixed point of~$\pi$.
\end{lem}
\begin{proof}Each parabolic element $\pi\in \SL_2(\RR)$ is conjugate
in $\SL_2(\RR)$ to $T=\matc 1101$ or to $T^{-1}$. With \eqref{ncj} we
can transform the hypothesis in both cases to $\ld^{-1} \, h|_{2-r}T
- h \in \dsv{2-r}\om$. If $h$ has singularities in $\RR$ put them in
increasing order: $\xi_1<\xi_2<\cdots$. Then
$\xi_1-1=T^{-1} \xi_1\in \bsing \bigl(h|_{2-r}T\bigr)$, and cannot be
canceled by a singularity of $h$. So a singularity can occur only
at~$\infty$, and only at $\ca$ in the original situation.
\end{proof}

\begin{prop}\label{prop-parb*0}Let $r\in \CC$. Then
\begin{align*}
\hpar^1(\Gm;\dsv{v,2-r}\om,\dsv{v,2-r}\fsn)&\=
\hpar^1(\Gm;\dsv{v,2-r}\om,\dsv{v,2-r}\fs) \,,\\
\hpar^1(\Gm;\dsv{v,2-r}\om,\dsv{v,2-r}{\fsn,\wdg})&\=
\hpar^1(\Gm;\dsv{v,2-r}\om,\dsv{v,2-r}{\fs,\wdg}) \,,\\
\hpar^1(\Gm;\dsv{v,2-r}\om,\dsv{v,2-r}{\fsn,\smp})&\=
\hpar^1(\Gm;\dsv{v,2-r}\om,\dsv{v,2-r}{\fs,\smp}) \,,\\
\hpar^1(\Gm;\dsv{v,2-r}\om,\dsv{v,2-r}{\fsn,\infty})&\=
\hpar^1(\Gm;\dsv{v,2-r}\om,\dsv{v,2-r}{\fs,\infty}) \,.
\end{align*}
\end{prop}
\begin{proof}If $\ps\in \zpar^1(\Gm;\dsv {v,2-r}\om,\dsv{v,2-r}\fs)$
then we have for each cusp $\ca$ an element $h\in \dsv{v,2-r}\fs$
such that $h|_{v,r}(\pi_\ca-\nobreak 1)= \ps_{\pi_\ca}$. This is the
situation considered in Lemma~\ref{lem-sing-inv}, so $h\in
\dsv{v,2-r}\om[\ca]$. Hence $\ps \in
\zpar^1(\Gm;\dsv{v,2-r}\om,\dsv{v,2-r}\fsn)$. The same argument is
valid for the other cases.
\end{proof}

\subsection{The parabolic equation for an Eichler
integral}\label{sect-pe}
\begin{defn}
\label{def-ldper}We call a function $F$ on a subset of $\CC$ that is
invariant under horizontal translations,
\il{ldper}{$\ld$-periodic}\emph{$\ld$-periodic} if it satisfies
$F(t+\nobreak 1) \= \ld \, F(t)$ for all $t$ in its domain.
\end{defn}
\rmrk{Example} For an automorphic form $F\in A_r(\Gm,v)$ and a
cusp~$\ca$, the relation $F|_{ {v, r} } \pi_\ca = v(\pi_\ca)\,
F$ implies that the translated function $F|_r
\sigma_{\ca}$ is $v(\pi_{\ca})$-periodic.\medskip

\rmrk{Parabolic difference equation}\il{parbeq}{parabolic difference
equation} We take an arbitrary holomorphic $\ld$-periodic
function~$E$ on $\uhp$. It has an absolutely convergent Fourier
expansion
\be \label{Eexp} E(z) \= \sum_{n\equiv\al\bmod 1} a_n\, e^{2\pi i n
z}\ee
on~$\uhp$ with $\ld=e^{2\pi i \al}, \al\in \CC$. In the next
subsections we aim to find functions $h$ such that
\be \label{peq} \ld^{-1} \, h(t+1) - h(t) \= \int_{z_0-1}^{z_0}
(z-t)^{r-2}\, E
(z)\, dz \ee
at least for $t\in \lhp \cup\RR$, and to get information concerning
its behavior near~$\infty$.

\subsection{Asymptotic behavior at infinity}It will be useful to
understand the behavior of $\Prj{2-r}h$ at $\infty$ for solutions $h$
of~\eqref{peq}. For functions $f$ on $\RR$ we understand in 
these notes
\il{sim1}{$\sim$}$f(t)\sim \sum_{n\geq k}c_n\, t^{-n}$ to mean $f(t) =
\sum_{n=k}^{N-1} c_n\, t^{-n} + \oh(t^{-N})$ as $t\rightarrow
\pm\infty$ for all $N\geq k$. So $f(t) \sim 0 $ means $f(t) =
\oh(t^{-N})$ for all $N\in \ZZ_{\geq 0}$.

For elements $f\in \dsv{2-r}\infty$ we know that there are
 coefficients $b_n$ such that
\[ (\Prj{2-r}f)(t) \sim \sum_{n\geq0} b_n \, t^{-n} \]
as $t$ approaches $\infty$ through $\lhp\cup\RR$. So we have surely
this behavior as $t$ approaches $\infty$ through~$\RR$.

\begin{lem}\label{lem-ldpr}Let $r\in \CC$ and $\ld\in \CC^\ast$, and
suppose that $f\in \dsv{2-r}\om[\infty]$ is $\ld$-periodic. We
consider asymptotic expansions of $(\Prj{2-r}f)(t)=(i-t)^{2-r}\,f(t)$
of the type
\be \label{per-exp} (\Prj{2-r}f)(t) \;\sim\; \sum_{n\geq k} b_n \,
t^{-n}
\qquad\text{for some }k\in \ZZ\,.\ee
\begin{enumerate}
\item[i)] If $f$ satisfies \eqref{per-exp} for $t\uparrow \infty$ as
well as for $t\downarrow-\infty$ (with the same coefficients $b_n$),
then
\begin{enumerate}
\item[a)] $f$ is a constant function if $\ld=1$ and $r\in \ZZ$. In
this case $r\geq k+2$.
\item[b)] $f=0$ in all other cases.
\end{enumerate}
\item[ii)] Let $\e\in \{1,-1\}$. Suppose that $f$ satisfies
\eqref{per-exp} as $\e t\uparrow \infty$. Then
\begin{enumerate}
\item[a)] if $\ld=1$ and $r\in \ZZ_{\geq k+2}$,  then  $f$
is a constant function;
\item[b)] else if $|\ld|=1$, then $f=0$;
\item[c)] else $f(t)\sim0$ as $\e t\uparrow\infty$.
\end{enumerate}
\end{enumerate}
\end{lem}
\begin{proof}Consider $f\in
\dsv{2-r}\om[\infty]$ that is $\ld$-periodic
and has an expansion \eqref{per-exp} as $t\uparrow\infty$, or
$t\downarrow-\infty$, or both. If the expansion is non-zero it has
the form $b_n \,t^{-n} + b_{n+1}\, t^{-n-1}+\cdots$ where
$b_n\neq 0$. Insertion in
\[(\Prj{2-r}f)(t+1) \= \ld \, \Bigl( \frac{1-(i-1)/t}
{1-i/t}\Bigr)^{2-r}\,(\Prj{2-r}f)(t)\]
gives
\[ \ld \, b_n \= b_n\,,\quad \ld \, \bigl(b_{n+1} - (r-2)\, b_n) \=
b_{n+1}-n\, b_n\,.\]
This is impossible with $b_n\neq 0$ if $\ld\neq 1$. If $\ld=1$ it is
possible if $n= r-2\geq k$, and has solutions corresponding to a
constant function $f(t)=c$, and
\[ (\Prj{2-r} f)(t)\= (i-t)^{2-r}\, c\,.\]
This shows that non-zero expansions occur only in the Case~a) in
Part~i).

In this case we set $f_0(t)=f(t)-c$. We set $f_0=f$ otherwise. Then
$f_0\in \dsv{2-r}\om[\infty]$ is $\ld$-periodic with expansion
$(\Prj{2-r}f_0)(t)\sim 0$. In other words, $(\Prj{2-r}f_0)(t) =
\oh(t^{-N})$ for any order $N\in \ZZ_{\geq 0}$,
and then the same holds for $f_0(t)$. To show that $f=0$ in Case~b)
of~i), we notice that, as a $\ld$-periodic function, $f_0$ has a
Fourier expansion and the estimate $f(t) = \oh(t^{-N})$ holds for
each Fourier term, which is of the form $c_n \, e^{2\pi i n t}$ with
$e^{2\pi i n}=\ld$. For expansion in both directions,
$|t|\rightarrow\infty$, this implies that all Fourier terms vanish,
and hence $f=0$. This finishes the proof of Part~i).

For a one-sided expansion, say $t\rightarrow\infty$, there might be
Fourier terms that satisfy $e^{2\pi i n t} \sim 0$ as $t\uparrow
\infty$, namely if $\im n<0$. This possibility and the same
possibility as $t\downarrow-\infty$ are excluded by the assumption
$|\ld|=1$ in Part~ii)b). Without this assumption, $f(t)\sim 0$.
\end{proof}

\subsection{Construction of solutions}\label{sect-constr-peq}We break
up the Fourier expansion \eqref{Eexp} in three parts, according to
$\re n>0$, $\re n<0$ and $\re n=0$.

\rmrk{Cuspidal case}
\begin{lem}\label{lem-peq-cusp}Suppose that the Fourier expansion
\eqref{Eexp} has the form
\[ E(z)
\= \sum_{n\equiv\al\bmod 1,\; \re n>0} a_n\, e^{2\pi inz}\,,\]
with $\al\in \CC$, $\ld=e^{2\pi i \al}$.
\begin{enumerate}
\item[i)] If $r\in \CC\setminus\ZZ_{\geq 2}$, then there is a unique
$h\in \dsv{2-r}{\om,\infty,\wdg}[\infty]$ satisfying~\eqref{peq}.
\item[ii)]If $r\in \ZZ_{\geq 2}$ then \eqref{peq} has solutions in
$\dsv{2-r}\pol$.
\begin{enumerate}
\item[a)] If $\ld=e^{2\pi i \al}\neq 1$, then there is a unique
solution.
\item[b)] If $\ld=1$, then the solutions of \eqref{peq} in
$\dsv{2-r}\pol$ are unique up to addition of a constant.
\end{enumerate}
\end{enumerate}
\end{lem}
\begin{proof}Lemma~\ref{lem-inf} states that
\be\label{h0-cusp} h_0(t) \;:=\; \int_{z_0}^\infty (z-t)^{r-2}\,
E(z)\, dz \ee
defines $h_0\in \dsv{2-r}{\om,\infty,\wdg}[\infty]$. If we take a
vertical path of integration, then $h_0$ is a holomorphic function on
$\CC\setminus \left( z_0+i[0,\infty)\right)$.

\twocolwithpictr{
\quad Let us consider $t\in \CC$ with $\im t<\im z_0$. The integral
over the closed path sketched in Figure~\ref{fig-rh} on the right
equals zero for all $a>0$, and due to the exponential decay of~$E$
the limit as $a\rightarrow\infty$ of the integrals over the sides
depending on~$a>0$ exist. Hence we get
\begin{align*} &\int_{z_0-1}^\infty (z-t)^{r-2}\, E(z) \, dz
\\
&\quad\;=\; \int_{z_0-1}^{z_0}(z-t)^{r-2}\, E(z) \, dz + h_0(t)\,.
\end{align*}
}{ \label{fig-rh} \setlength\unitlength{.8cm}
\begin{picture}(5,7)(-1,0)
\put(0,0){\line(1,0){4}}
\put(1,1){\circle*{.1}}
\put(2.5,1){\circle*{.1}}
\put(1,6){\circle*{.1}}
\put(2.5,6){\circle*{.1}}
\put(2.6,6.1){$z_0+ia$}
\put(-.5,6.1){$z_0-1+ia$}
\put(0,.6){$z_0-1$}
\put(2.6,.6){$z_0$}
\thicklines
\put(1,1){\line(1,0){1.5}}
\put(1,1){\line(0,1){5}}
\put(1,6){\line(1,0){1.5}}
\put(2.5,1){\line(0,1){5}}
\end{picture}
}

\noindent
Like in Part~iii) of Lemma~\ref{lem-omprop} this gives
\begin{align*}
\ld^{-1}\, h_0 (t+1)&\= \ld^{-1}\, \int_{z_0}^\infty (z-t-1)^{r-2}\,
E(z)\, dz \= \int_{z_0-1}^\infty (z-t)^{r-2}\, E(z)\, dz\\
&\= h_0 (t) + \int_{z_0-1}^{z_0} (z-t)^{r-2}\, E(z)\, dz \,.
\end{align*}
This relation extends holomorphically to all $t\in \CC$ outside the
region determined by $\re z_0-\nobreak 1\leq \re t\leq \re z_0$ and
$\im t\geq \im z_0$. So $h_0$ is a solution of \eqref{peq} in
$\dsv{2-r}{\om,\infty,\wdg}[\infty]$.

Let $h$ be another solution in $\dsv{2-r}{\om,\infty,\wdg}[\infty]$.
Then $p=h-h_0$ is a $\ld$-periodic function in $\dsv{2-r}\infty$, and
hence $\Prj {2-r}p$ has an expansion as in Lemma~\ref{lem-ldpr}, with
$k\geq 0$. If $r\in \CC \setminus \ZZ_{\geq2}$ then Part~i) a) and
Part~ii) of Lemma~\ref{lem-ldpr} implies that $p=0$ so that we have
proved Part~i) of this lemma.

To prove Part~ii) let $r\in \ZZ_{\geq2}.$ It is clear from the
integral that $h_0\in \dsv{2-r}\pol$ if $r\in \ZZ_{\geq 2}$. If
$\lambda=1 $ it reduces to the case b) in Part~i)
of Lemma~\ref{lem-ldpr}. So $p$ is a non-zero constant. This handles
Part~ii)b) of the present lemma. If $\lambda \neq 1,$ it reduces to
the case a) in Part~ii) of Lemma~\ref{lem-ldpr} so that $p=0.$ This
gives Part a) in Part~ii) of the present lemma.
\end{proof}

\rmrk{Exponentially increasing part}
\begin{lem}\label{lem-expgr}Suppose that the Fourier expansion
\eqref{Eexp} has the form
\be\label{negexp} E(z) \= \sum_{n\equiv\al\bmod 1,\; \re n<0} a_n\,
e^{2\pi inz}\,,\ee
with $\al\in \CC$, $\ld=e^{2\pi i \al}$.
\begin{enumerate}
\item[i)] Equation \eqref{peq} has solutions
$h\in\dsv{2-r}{\om,\wdg}[\infty]$, among which occurs a solution for
which $\Prj {2-r}h(t)$ has an asymptotic expansion as
 $t\uparrow\infty$ of the form $\sum_{n\geq 0}c_n t^{-n}$.
\item[ii)] Let $|\ld|=1$ and $E\neq 0$. For none of these solutions
$h$ do we have an asymptotic expansion of the form $\Prj{2-r}h(t)
\sim \sum_{n\geq k} q_k\, t^{-n}$ valid for $t\uparrow\infty$ and for
$t\downarrow-\infty$ with the same coefficients.
\end{enumerate}
\end{lem}

\begin{proof}We cannot use the integral in~\eqref{h0-cusp}, since $E$
has exponential growth on $\uhp$.\vskip.2ex
\twocolwithpictl{\label{fig-pd}
\setlength\unitlength{.9cm}
\begin{picture}(5.5,3.5)(0,-2)
\put(0,0){\line(1,0){5}}
\put(.6,.04){$\RR$}
\put(3,1){\circle*{.1}}
\put(3.1,.8){$z_0$}
\thicklines
\put(3,1){\vector(0,-1){3}}
\end{picture}
}{ The convergence of $E(z)$ in $\uhp$ implies good growth for its
Fourier coefficients. This growth then implies that $E(z)$ can be
defined on $\CC$ with exponential decay as $\im z \downarrow-\infty$.
So we use a path of integration as in Figure~\ref{fig-pd}.}
\vskip.6ex

In this way we  obtain  a holomorphic function
$h_{\mathrm{ri}}$ on the region $\re t>\re z_0$ given by
\be h_{\mathrm{ri}}(t) \= \int_{z_0-i[0,\infty)} (z-t)^{r-2}\, E(z)\,
dz\,,
\ee
and satisfying \eqref{peq} for these values of~$t$.

\twocolwithpictl{\label{deform}
\setlength\unitlength{.9cm}
\begin{picture}(5.5,4.5)(0,-2)
\put(0,0){\line(1,0){5}}
\put(.6,.04){$\RR$}
\put(3,1){\circle*{.1}}
\put(3.1,.8){$z_0$}
\thicklines
\put(3,1){\line(0,1){.2}}
\put(2,1.2){\oval(2,2)[t]}
\put(1,1.2){\vector(0,-1){3.6}}
\end{picture}
}{\quad Deforming the integral as in Figure~\ref{deform} we get the
holomorphic continuation of $h_{\mathrm{ri}}$ to a larger region. In
this way we get the continuation to the region $\CC\setminus \bigl(
z_0+i[0,\infty)\bigr)$. By analytic continuation, the
extension satisfies \eqref{peq} on the region $\CC\setminus \bigl(
z_0 + [-1,0]+i[0,\infty)\bigr)$. We normalize the factor $(z-\nobreak
t)^{r-2}$ by requiring that $\frac\pi2\leq \arg(z-t) \leq
\frac{3\pi}2$ if $\re t >\re z_0$ and
$z$ is on the path of integration. }

We have
\be \label{hprj-i}
(\Prj{2-r}h_{\mathrm{ri}})(t) \= \int_{z_0} \Bigl( \frac{z-t}{i-t}
\Bigr)^{r-2}\, E(z)\, dz\,,\ee
over a path of integration starting at $z_0$ going down to $\infty$,
adapted to $t$, and normalized by $\arg \bigl( \frac{z-t}{i-t}\bigr)
\rightarrow 0$ as $t\uparrow \infty$. We
consider the asymptotic behavior of $(\Prj{2-r}h_{\mathrm{ri}})(t)$
as $t\uparrow\infty$ through~$\RR$. The exponential decay of $E$ as
$\im z \downarrow-\infty$ implies that for a fixed large $t$ in $\RR$
the contribution of the integral over $\im z< - \frac12 t$ can be
estimated by $\oh(e^{-\e t})$ as $t\rightarrow\infty$, with $\e>0$
depending on $ E$, $\al$ in \eqref{negexp}, and $z_0$. We insert the
Taylor expansion of order $N$ of $\bigl( \frac{z-t}{i-t}\bigr)^{r-2}$
in $\frac 1t$ into the remaining part of the integral, and find an
expansion starting at $k=0$, but only as $t\uparrow\infty$. In this
way we obtain the second statement in Part~i).

Actually, if we apply the same reasoning to the integral
in~\eqref{hprj-i} for $t\downarrow-\infty$ we get the same expansion,
with the same coefficients. However, that is not an expansion of
$h_{\mathrm{ri}}$, but of another solution $h_{\mathrm{le}}$, which
we can define in the following way.

\twocolwithpictr{\quad An equally sensible choice is the path of
integration sketched in Figure~\ref{fig-pr}. Now the path has to be
chosen so that $t$ is to the left of it, and below it if $\re t>\re
z_0 $. This defines another solution $h_{\mathrm{le}}\in
\dsv{2-r}{\om,\wdg}[\infty]$ of \eqref{peq}. The normalization of the
corresponding  integrand  for $\Prj{2-r}h_{\mathrm{le}}$
is also by $\arg \bigl( \frac{z-t}{i-t }\bigr) \rightarrow 0$ as
$t\uparrow \infty$.

\quad As indicated above, $h_{\mathrm{le}}(t)
$ has an asymptotic expansion as $t\downarrow-\infty$, with the same
coefficients as in the expansion of $h_{\mathrm{ri}}(t)$ as
$t\uparrow\infty$.}{\label{fig-pr}
\setlength\unitlength{.9cm}
\begin{picture}(5.5,6)(0,-3.5)
\put(0,0){\line(1,0){5}}
\put(.6,.04){$\RR$}
\put(3,1){\circle*{.1}}
\put(2.6,.8){$z_0$}
\thicklines
\put(3,1){\line(0,1){.5}}
\put(3.5,1.5){\oval(1,1)[t]}
\put(4,1.5){\vector(0,-1){4.5}}
\end{picture}
}

\twocolwithpictl{\label{fig-2p}
\setlength\unitlength{.9cm}
\begin{picture}(5,5.2)(0,-3)
\put(0,0){\line(1,0){5}}
\put(2,0){\circle*{.15}}
\put(2.1,.1){$t$}
\thicklines
\put(2,1){\line(0,1){.5}}
\put(1.5,1.5){\oval(1,1)[t]}
\put(1,1.5){\vector(0,-1){4.5}}
\put(2.5,1){\oval(1,1)[t]}
\put(3,1){\line(0,-1){4}}
\put(3,-3){\vector(0,1){1}}
\end{picture}
}{ Both $h_{\mathrm{ri}}$ and $h_{\mathrm{le}}$ are solutions of
\eqref{peq} for $E$ as in~\eqref{negexp}. We have for $t\in \RR$
\begin {align*}\bigl(\Prj{2-r}&( h_{\mathrm{ri}}(t)
-h_{\mathrm{le}})\bigr)(t) \\
&\;=\; \int_D \Bigl( \frac{z-t}{i-t} \Bigr)^{r-2}\, E(z)\,
dz\,,\end{align*}
over a path of integration indicated in Figure~\ref{fig-2p}.

This integral is holomorphic in $r\in \CC$. For $\re r>1$ it can be
computed by deforming the path of integration to the vertical
half-line downward from~$t$. This leads to the following result: }
\be\label{hr-hell}\bigl(\Prj{2-r}( h_{\mathrm{ri}}
-h_{\mathrm{le}})\bigr)(t) \= \frac{(2\pi)^{2-r}\, e^{\pi i
r/2}}{\Gf(2-r)}\,
(i-t)^{r-2} \sum_{n\equiv \al\bmod 1,\, \re( n)<0}
\frac{a_n}{(-n)^{r-1}}\, e^{2\pi i n t}\,. \ee
This difference gives a $\ld$-periodic function
$H=h_{\mathrm{ri}}-h_{\mathrm{le}}$. Moreover, the difference is
holomorphic in $r\in \CC$. So the equality is valid for all $r\in
\CC$.

Now let $|\ld|=1$. Suppose that $h\in \dsv{2-r}{\om,\wdg}$ is a
solution of \eqref{peq} with a two-sided asymptotic expansion. We
have $h=h_{\mathrm{ri}} + p_r=h_{\mathrm{le}}+p_\ell$ with
$\ld$-periodic $p_r,p_\ell\in \dsv{2-r}{\om,\wdg}[\infty]$.

Suppose that the difference $p_r=h-h_{\mathrm{ri}}$ has an asymptotic
expansion as $t\uparrow\infty$. Part~ii) of Lemma~\ref{lem-ldpr}
shows that $p_r$ is constant (and zero in most cases). Similarly
$p_\ell=h-h_{\mathrm{le}}$ is constant. So
$h_{\mathrm{ri}}-h_{\mathrm{le}}$ is constant. However, in
\eqref{hr-hell} we see that this implies
$h_{\mathrm{ri}}-h_{\mathrm{le}}=0$ by the assumption $\re n<0$
in~\eqref{negexp}. Then all $a_n$ vanish and $E=0$.
\end{proof}

\rmrk{Remaining Fourier term}We are left with the multiples of $
e^{2\pi i n z}$ with $\re n=0$. So $n=i\,\im \al \equiv\al\bmod 1$.
\begin{lem}\label{lem-pim}
Suppose that $E(z) = e^{-2\pi z n}$ with $\re n=0$, $n\neq 0$. Then
Equation~\eqref{peq} has solutions in $\dsv{2-r}{\om,\wdg}[\infty]$.
\end{lem}
\begin{proof}We can still find a direction in which $E(z)$ decays
exponentially.
\twocolwithpictr{In the case $\im n<0$ we choose a path as indicated
in Figure~\ref{fig-C-}. We can proceed as in the proof of
Lemma~\ref{lem-expgr}. } {\setlength\unitlength{.9cm} \label{fig-C-}
\begin{picture}(5.5,2.5)
\put(0,0){\line(1,0){5}}
\put(.6,.04){$\RR$}
\put(3,1){\circle*{.1}}
\put(3.1,.8){$z_0$}
\thicklines
\put(3,1){\line(0,1){.5}}
\put(2.5,1.5){\oval(1,1)[rt]}
\put(2.5,2){\vector(-1,0){2}}
\end{picture}
}
\twocolwithpictl{\setlength\unitlength{.9cm} \label{fig-C+}
\begin{picture}(5.5,2.3)
\put(0,0){\line(1,0){5}}
\put(.6,.04){$\RR$}
\put(3,1){\circle*{.1}}
\put(3.1,.8){$z_0$}
\thicklines
\put(3,1){\line(0,1){.5}}
\put(3.5,1.5){\oval(1,1)[lt]}
\put(3.5,2){\vector(1,0){1}}
\end{picture}
}{For $\im n>0$ we use the path in Figure~\ref{fig-C+}. Again,we can
proceed as in the proof of Lemma~\ref{lem-expgr}.\qedhere}
\end{proof}

\begin{lem}\label{lem-peq-cst}Let $r\in \CC$, $\ld=1$ and $E(z) = 1$.
\begin{enumerate}
\item[i)] Equation \eqref{peq} has the following solution in
$\dsv{2-r}{\om,\wdg}[\infty]$:
\be\label{h0-expl} h(t) \= \begin{cases}
(1-r)^{-1}\,(z_0-t)^{r-1}&\text{ if }r\neq 1\,,\\
-\log(z_0-t)&\text{ if }r=1\,,
\end{cases}
\ee
where we choose in both cases $-\frac\pi2< \arg(z_0-\nobreak
t)<\frac{3\pi}2$.
\item[ii)] For $r\neq 1$ this is the unique solution in
$\dsv{2-r}{\om,\wdg}[\infty]$ for which $\Prj{2-r}h$ has an
asymptotic expansion valid for $t\uparrow\infty$ and for
$t\downarrow-\infty$.
\item[iii)] If $r=1$ there are no solutions in
$\dsv{2-r}{\om,\wdg}[\infty]$ that have a two-sided asymptotic
expansion at~$\infty$ in the projective model.
\end{enumerate}
\end{lem}
\begin{proof}Part~i) can be checked by a computation of the integral
in~\eqref{peq}. For $r\neq 1$ it is seen that $(\Prj{2-r}h)(t)$ has a
two-sided asymptotic expansion of the form $(\Prj{2-r}h)(t)\sim
\sum_{n\geq
-1} c_n t^{-n}$. For $r=1$ this solution clearly has no such
expansion.

Any other solution is of the form $h+p$ with a $1$-periodic function
$p$. If it has a two-sided asymptotic expansion it is zero, by
Part~i) of Lemma~\ref{lem-ldpr}. This gives Part~ii) of the present
lemma. For Part~iii) one can check that no $1$-periodic function can
produce logarithmic behavior at~$\infty$.
\end{proof}

\rmrk{Example} Only for a constant function $E$ we have given an
explicit formula for a solution $h$. It is possible to express
solutions for the other cases in terms of sums of incomplete
gamma-functions.

\il{Deta-pe}{Dedekind eta-function}\il{etapow1}{powers of the Dedekind
eta-function}For the powers of the Dedekind eta-function we get for
$\re r>0$ a solution of the form
\badl{etahr} h(t)&\= -i e^{\pi i r/2}\,(2\pi)^{1-r}\sum_{k\geq 0}
p_k(r)\,\bigl(r/12+k)^{1-r}\\
 &\qquad\hbox{} \cdot
 e^{2\pi i (12 k+r)t/12}\, \Gf\Bigl( r-1,2\pi i
 (r/12+k)(t-z_0)\Bigr)\,, \eadl
with the \il{icgf}{incomplete gamma-function}incomplete gamma-function
\be \Gf(a,u) \= \int_u^\infty v^{a-1}\, e^{-v}\, dv \= e^{-u}
\int_{x=0}^\infty (u+x)^{a-1}\, e^{-x}\,dx\,. \ee
The incomplete gamma-function is well defined on $\CC\setminus
(-\infty,0]$. That suffices for~\eqref{etahr} if $\re r>0$ and $t\in
\lhp$.

If $\re r\leq 0$, the same formulas can be used for the terms in the
Fourier expansion with $\re \frac r{12}+k>0$. For the remaining terms
with $k+\frac r{12}\neq 0$ the choices in this subsection lead also
to the same expression with incomplete gamma-functions, but now
interpreted with a choice of a suitable branch of the multivalued
extension. For $k+\frac r{12}=0$ we can use the formula for $r\neq 1$
in Lemma~\ref{lem-peq-cst}.

\subsection{Image of automorphic forms in the analytic cohomology}
Now we can take a step towards the proof of Theorem~\ref{THMac}:
\begin{thm}\label{thm-imafwdg}
For all $r\in \CC$ and all multiplier systems for the weight~$r$:
\be \coh r \om\, A_r(\Gm,v) \;\subset\;
\hpar^1(\Gm;\dsv{v,2-r}\om,\dsv{v,2-r}{\fsn,\wdg})\,. \ee
\end{thm}
\begin{proof}As explained in \S\ref{sect-pe} we have to solve, for
each cusp $\ca$ of $\Gm$, Equation \eqref{parb-eq} with an element of
$\dsv{v,2-r}{\fsn,\wdg}$. Lemma~\ref{lem-sing-inv} shows that such
solutions always satisfy $\bsing h \subset\{\ca\}$, hence if we have
a solution in $\dsv{v,2-r}{\fs,\wdg}$ it is in
$\dsv{v,2-r}{\fsn,\wdg}$.

By conjugation, the task is equivalent to solving \eqref{peq} with $E$
replaced by $F|_r \sigma_{\ca}$. The Fourier series of $E$, which is
the Fourier expansion of $F$ at the cusp~$\ca$, is split up as a sum
of two or three terms. The existence of solutions is obtained in the
 Lemmas \ref{lem-peq-cusp}--\ref{lem-peq-cst}.
\end{proof}

\subsection{Proof of Theorem~\ref{THMcc}}\label{sect-prfTHMcc}
Assuming that Theorem~\ref{THMac} has been proved, we now prove
Theorem~\ref{THMcc}. We use the results concerning asymptotic
expansions in \S\ref{sect-constr-peq}. There we have seen that we
need $|\ld|=1$ to get satisfactory results. Hence we impose the
assumptions of real weight and unitary multiplier system, which are
the same assumptions as in the Theorem of Knopp and Mawi. See
Theorem~\ref{thm-KM}.

\begin{proof}[Proof of Theorem~\ref{THMcc} on the basis of
Theorem~\ref{THMac}] In Proposition~\ref{prop-cu-inf} we saw that the
space $\coh r \om \, \cusp r(\Gm,v)$ is contained in
$\hpar^1(\Gm;\dsv{v,2-r}\om,\dsv{v,2-r}{\fsn,\infty,\wdg})$ for all
$r\in \CC$. Part~i) of the Lemmas \ref{lem-peq-cusp}
and~\ref{lem-peq-cst} show that $\coh r \om \, M_r(\Gm,v)$ is
contained in
$\hpar^1(\Gm;\dsv{v,2-r}\om,\dsv{v,2-r}{\fsn,\smp,\wdg})$ for $r\in
\CC\setminus\{1\}$, since for $h$ as in~\eqref{h0-expl} the function
$\Prj{2-r}h(t)$ has an asymptotic expansion at~$\infty$ starting at
$t^1$
(``$k=-1$'').

Let $r \in \RR \smallsetminus \ZZ_{\geq 2}$. Under Theorem~\ref{THMac}
any class in $\hpar^1(\Gm;\dsv{v,2-r}\om,\dsv{v,2-r}{\fsn,\wdg})$ is
of the form $\coh r \om F$ for some $F\in A_r(\Gm,v)$. We want to
show that $F$ satisfies the following:
\[\renewcommand\arraystretch{1.4} \begin{array}{|c|c|c|}\hline
&\text{Part~i)}&\text{Part~ii)}\\ \hline
&r\not\in \ZZ_{\geq 2}&r\not\in \ZZ_{\geq 1}\\
\text{wish:}&F\in \cusp r(\Gm,v)&F\in M_r(\Gm,v) \\ \hline
\end{array}
\]

We consider a cusp~$\ca$ of~$\Gm$, and put $E=F|_r
  \s_\ca$.  The assumptions on $F$ imply that there is $h_\ca$ in
  $\dsv{2-r}{\om,\infty,\exc}[\infty]$, respectively
  $\dsv{2-r}{\om,\smp,\exc}[\infty]$ such that
  $$ h_\ca|_{2-r}\bigl(1-v(\pi_\ca)^{-1}\, T\bigr)\,(t)
   = \int_{z_0-1}^{z_0}(z-t)^{r-2} \, E(z)\, dz\,. $$
  we write  $E= E_c +E_0 + E_e$ by taking the Fourier terms with $
  n>0$, $ n=0$ and $ n<0$, respectively. Note that
  $|v(\pi_\ca)|=1$, hence the Fourier term orders are real.

We take $h_c$ provided by Lemma~\ref{lem-peq-cusp}, $h_0$ by
Lemma~\ref{lem-peq-cst}, and $h_e$ by Lemma~\ref{lem-expgr}. Then
$h_\ca|_{2-r}\s_\ca = h_c+h_0+h_e+p$, with a $v(\pi_{\ca})$-periodic
element $p$ in $\dsv{2-r}{\om,\wdg}[\infty]$. Table~\ref{tab-ae}
gives information on the asymptotic behavior, where we use the
 definitions of $\dsv{2-r}{\om,\infty,\wdg}$ and
$\dsv{2-r}{\om,\wdg}$, and Lemmas \ref{lem-peq-cusp},
\ref{lem-expgr}, and~\ref{lem-peq-cst}.
(The $c_n$ in the table depend on the function.)
\begin{table}[ht]
\[ \renewcommand\arraystretch{1.4} \begin{array}{|c|c|c|}\hline
&\text{Part~i)}&\text{Part~ii)}\\ \hline
&r\not\in \ZZ_{\geq 2}&r\not\in \ZZ_{\geq 1}\\
&h_\ca|\s_\ca \in \dsv{2-r}{\om,\infty,\wdg}[\infty]& h_\ca|\s_\ca \in
\dsv{2-r}{\om,\smp,\wdg}[\infty]\\
\hline
\text{$2$-sided:}& \bigl(\Prj{2-r}h_\ca|\s_\ca\bigr)(t) \sim
\sum_{n\geq 0}c_n \, t^{-n}& \bigl(\Prj{2-r}h_\ca|\s_\ca\bigr)(t)
\sim \sum_{n\geq -1}c_n \, t^{-n}\\ \hline
\text{$2$-sided:}&\multicolumn{2}{|c|}{
(\Prj{2-r}h_c)(t)\sim \sum_{n\geq 0}c_n \, t^{-n}} \\ \hline
\text{as }t\uparrow\infty\,:&\multicolumn{2}{|c|}{
(\Prj{2-r}h_e)(t) \sim \sum_{n\geq 0}c_n \, t^{-n} } \\ \hline
v(\pi_\ca)\neq 1 & \multicolumn{2}{|c|}{h_0=0} \\ \hline
v(\pi_\ca)=1& \multicolumn{2}{|c|}{(\Prj{2-r}h_0)(t) \sim \sum_{n\geq
-1} c_n t^{-n}\quad\text{ if }r\neq 1} \\
&\text{ no asymptotic expansion if $r=1$}&\\ \hline
\end{array}
\]
\caption{}\label{tab-ae}
\end{table}

Let $r \in \RR \smallsetminus \ZZ_{\geq 1}$. Note that the
$v(\pi_\ca)$-periodic function $p$ has an asymptotic expansion as
$t\uparrow
\infty$, and is $\oh(t)$.  Part~ii) of Lemma~\ref{lem-ldpr}
implies that either $p=0$ or $p$ is a non-zero constant and $r\in
\ZZ_{\geq 1}$.
Since the latter case is impossible, we deduce that $p=0$.

We  conclude that $h_e=h_\ca|_{2-r}\s_\ca-h_c-h_0$ has a two-sided
expansion. Part~ii) of Lemma~\ref{lem-expgr} shows that $h_e=0$, and
hence $E_e=0$. So $F|_{2-r}\s_\ca=E_c$, and hence $F$ behaves like an
element of $\cusp r(\Gm,v)$ at the cusp~$\ca$. Since $\ca$ was chosen
arbitrarily, this finishes the proof of both parts under the
assumption $r\neq 1$.

For Part~i) we have still to consider $r=1$. If we work modulo
functions with an asymptotic expansion in powers of $t^{-1}$ as
$t\uparrow \infty$, the $v(\pi_\ca)$-periodic function $p$ has to
compensate for the possible logarithmic behavior of $h_0$ given in
Lemma~\ref{lem-peq-cst}. The logarithmic term is growing as
$t\uparrow\infty$ and the periodic function is bounded (since
$|v(\pi_\ca)|=1$), so this is impossible, and $h_0=0$.

Now we proceed as above, with asymptotic expansions starting at~$t^0$.
For $h_e$ this rules out the constant function, and we arrive again
at $h_e$=0, and hence $F|_{2-r}\s_\ca = E_c$.
\end{proof}

\begin{rmk}Since we have used the unitarity of the multiplier
system~$v$ only for $|v(\pi_\ca)|=1$, Theorem~\ref{THMcc} is still
true under the assumption that $|v(\pi)|=1$ for all \emph{parabolic}
$\pi\in \Gm$, without such an assumption concerning hyperbolic
elements.
\end{rmk}

\subsection{Related work}\label{sect-lit3}Pribitkin \cite[Theorem
1]{Pr5} uses integrals along paths like those
in~\S\ref{sect-constr-peq}.

Proposition~\ref{prop-parb*0} is analogous
to~\cite[Proposition~10.3]{BLZm}. Here we use the explicit integrals
in \S\ref{sect-constr-peq}, since we want to handle complex weights.


\section{One-sided averages}\label{sect-osa}
In \S\ref{sect-pe} we considered the parabolic equation with an
Eichler integral as the given function. We now take the right hand
side to be more general, and use the \il{osa}{one-sided
average}\emph{one-sided averages} given by\ir{avpm}{\av {T,\ld}^+,\;
\av{T,\ld}^-}
\badl{avpm} \bigl( \av{T,\ld} ^+ g\bigr)(t) &\;:=\; \sum_{n\geq 0}
\ld^{-n}\, g(t+n)\,,\\
\bigl( \av{T,\ld} ^- g\bigr)(t) &\;:=\; -\sum_{n\leq -1} \ld^{-n}\,
g(t+n)\, ,\eadl
where $\ld \in \CC^\ast$, and where the subscript $T$ refers to
$T:t\mapsto t+1$.

\subsection{One-sided averages with absolute
convergence}\label{sect-osa-ac}
If one of the series in \eqref{avpm} converges absolutely then
$h=\av{T,\ld}^\pm g$ provides a solution of the equation
\be\label{hfe}
h(t) - \ld^{-1}\, h(t+1) \= g(t)\,.
\ee

\begin{prop}\label{prop-osaa}Let $\ld\in \CC^\ast$, and suppose that
$g$ represents an element of $\dsv{2-r}\om$.
\begin{enumerate}
\item[a)] the one-sided average $\av{T,\ld}^\pm g$ converges
absolutely if one of the following the conditions is satisfied:
\be\label{avc1} \begin{array}{|r|c|c|c|}\hline
\ld & |\ld|>1 & |\ld|=1\text{ and }\re r<1 & |\ld|<1 \\ \hline
\av{T,\ld} ^+ \ph & \text{convergent} & \text{convergent} &
\text{undecided}
\\ \av{T,\ld} ^- \ph & \text{undecided}
&\text{convergent}&\text{convergent}
 \\
 \hline
\end{array}
\ee
\item[b)] The average defines a holomorphic function $\av{T,\ld}^\pm
g$ on a region\ir{Depspm}{D_\e^\pm}
\be\label{Depspm} D_\e^\pm \;:=\; \Bigl\{ z\in \CC\;:\; y<\e\text{ or
}\pm x >\e^{-1} \Bigr\}\,,\ee
for some $\e\in(0,1)$.
\item[c)]$\av{T,\ld}^\pm g$  satisfy  \eqref{hfe}.
\item[d)] $\av{T,\ld}^\pm g$  represent  an element of $
\dsv{2-r}\om[\infty]$.
\end{enumerate}
\end{prop}

\begin{rmk}\label{2interpr} We can interpret the phrase \emph{$g$
represents an element of $\dsv{2-r}\om$} in two ways, and we will
have reasons to use both interpretations.
\begin{enumerate}
\item[a)] $g$ is a holomorphic function on an $\{\infty\}$-excised
neighbourhood, and $\Prj{2-r}g$ is holomorphic on a neighbourhood of
$\infty$ in $\proj\CC$;
\item[b)] $g \in C^2(\CC)$ has a holomorphic restriction to an
$\{\infty\}$-excised neighbourhood, and $\Prj{2-r}g$ is holomorphic on
a neighbourhood of $\infty$ in $\proj\CC$.
\end{enumerate}
Under the first interpretation, $\av{T,\ld}^\pm g$ is a holomorphic
function on $D_\e^\pm$, where $\e$ depends on the domain of $g$. In
the second interpretation, we have $g\in C^2(\CC)$ such that
$\Prj{2-r}g$ is holomorphic on a neighbourhood of $\lhp\cup\proj\RR$
in $\proj\CC$, and $\av{T,\ld}^\pm g$ is in $C^2(\CC)$ and
holomorphic on $D_\e^\pm$.
\end{rmk}

\begin{proof}[Proof of proposition~\ref{prop-osaa}] Representatives of
the projective model $\Prj{2-r}\dsv{2-r}\om$ are holomorphic on
$\proj\CC\setminus K$ for some compact set $K\subset \uhp$, which is
contained in a set of the form $[-\e^{-1},\e^{-1}]\times
[i\e,i\e^{-1}]$ for some small $\e\in(0,1)$. To get a representative
in the space $\dsv {2-r}\om$ itself, we have to multiply by
$(i-\nobreak t )^{r-2}$. So we work with functions
$g$ that are holomorphic on
\[ \CC \setminus \Bigl([-\e^{-1},\e^{-1}]\times i[\e,\infty)
\Bigr)\,.\]

For $t$ in a compact region $ V $ anywhere in  $ \CC$,
there is a tail of the series with $z+n$ in the region where $g$ is
holomorphic. Moreover, $g(t)= \oh\bigl(|t|^{ \re
r-2}\bigr)$ as $|t|\rightarrow
\infty$. So the tail converges absolutely on $ V$, 
and represents a holomorphic function on the interior of~$ V
$. If $g\in C^2(\uhp)$, then the remaining
terms give a $C^2$ contribution, and $g\in C^2(\CC)$. If we take $V
\subset D_\e^\pm$ the whole series may be taken as
the tail, and we get the holomorphy of the one-sided average on
$D_\e^\pm$.

Note that the form of the set $D_\e^\pm$ implies that $\av{T,\ld}^\pm
g$ represents an element of $\dsv{2-r}\om[\infty]$,
but not necessarily of $\dsv{2-r}{\om,\wdg}[\infty]$.
\end{proof}

\rmrke The relation \eqref{hfe} between $\av{T,\ld}^\pm g$ and $g$
implies this relation for the elements in $\dsv{2-r}\om[\infty]$ and
$\dsv{2-r}\om$ that they represent.

\begin{prop}\label{prop-avT-Fe}Let $r\in \ZZ_{\leq 0}$ and
$\ld=e^{2\pi i \al}$ with $\al\in \RR$. Suppose that $g$ is a
representative of an element of~$\dsv{2-r}\om$ of the type~b) in
Remark~\ref{2interpr}.
\begin{enumerate}
\item[a)]Then\ir{avTdef}{\av{T,\ld}}
\be\label{avTdef}
\av {T,\ld} g \;:=\; \av {T,\ld}^+ g - \av{T,\ld}^- g
\ee
defines a $\ld$-periodic element of $C^2(\CC)$.
\item[b)]There is $\e\in
(0,1)$ such that the function $\av{T,\ld}g$ is holomorphic on two
regions, with Fourier expansions of the following form:
\be \av {T,\ld} g \;(z)\=\begin{cases}\sum_{m\equiv \al(1),\; m>0}
a_m^{\mathrm{up}}\, e^{2\pi i m z}
&\text{ on }\bigl\{ z\in \uhp\;:\; y>\e^{-1}\bigr\}\,,\\
\sum_{m\equiv \al(1),\; m<0}  a_m^{\mathrm{down}}\, e^{2\pi i m z}
&\text{ on }\bigl\{z\in \CC\;:\; y<\e\bigr\}\,.
\end{cases}
\ee
\end{enumerate}
\end{prop}

\begin{proof}The function $(\Prj{2-r}g)(z)=
(i-\nobreak z)^{2-r}\, g(z)$ represents an element of the projective
model of $\dsv{2-r}\om$. Since $r\in \ZZ_{\leq 0}$ the function $g$
itself is holomorphic on a neighbourhood of $\infty$ in $\proj\CC$,
and has a zero of order at least $2-r$ at~$\infty$. Since $|\ld|=1$,
both series $\av{T,\ld}^+ g$ and $\av{T,\ld}^- g$ converge absolutely
on~$\CC$, by Proposition~\ref{prop-osaa}.

These functions are now holomorphic on a set of the form
\[ \CC\setminus\left( (-\infty,\e^{-1}]\times
i[\e,\e^{-1}]\right)\,,\text{ respectively } \CC\setminus\left(
(-\e^{-1},\infty]\times i[\e,\e^{-1}]\right)\,. \]
So $(\av{T,\ld}g)(z) = \sum_{n\equiv\al(1)} \ld^{-n} g(z+\nobreak
n)$ defines a
$\ld$-periodic function on $\CC$ that is holomorphic on the two
regions in the proposition. On both regions the Fourier coefficients
are given by integral
\[ \int_{\im z=v} e^{-2\pi i m z}\,g(z)\, dz\,,\]
representing the coefficients $a_m^{\mathrm{up}}$ if $v>\e^{-1}$, and
the coefficients $a_m^{\mathrm{down}}$ if $v<\e$. The coefficients
can differ on both regions. The integral is invariant under changes
in $v$ in the corresponding interval. Since $g(z)=\oh\bigl(|z|^{\re
r-2}\bigr)$ as $|z|\rightarrow\infty$ through~$\CC$, the integral
satisfies $a_m = o\bigl( e^{2\pi m v})$ for fixed~$v$. So
$a_m^{\mathrm{up}}=0$ for $m\leq0$, and $a_m^{\mathrm{down}}=0$ for
$m\geq 0$.
\end{proof}

In \S\ref{sect-abg-iw} we will use the following result:
\begin{lem}\label{lem-h-Av+-} Let $r\in \ZZ_{\leq0}$. Suppose that $g$
is a representative of an element of $\dsv {2-r} \om$, of the type~a)
in Remark~\ref{2interpr}. Suppose that  $h$ represents an element of
$\dsv{2-r}{\om,\exc}[\infty]$ such that $h-h|_{2-r}T$ and $g$ represent the same
element of $\dsv {2-r} \om$. 

Then there are $1$-periodic $p_+,p_-\in \hol(\CC)$ such that for all sufficiently
small $\e>0$
\badl{havTp}
h(z) &\= (\av {T,1}^+ g)(z)+p_+(z)&&\text{ if }\im(z)<\e\text{ or }\re (z)>\e^{-1}\,,\\
h(z) &\= (\av {T,1}^- g)(z)+p_+(z)&&\text{ if }\im(z)<\e\text{ or }\re (z)<-\e^{-1}\,.
\eadl
\end{lem}

\begin{proof}
Proposition~\ref{prop-osaa} shows that we are in the domain of
absolute convergence of $\av{T,1}^\pm g$, and that these averages are
holomorphic on a set $D_\e^\pm$ with some $\e\in
(0,1)$. Now the weight is an integer, and the factor $( i -\nobreak
z)^{r-2}$ is non-zero and holomorphic on $\proj\CC
\setminus\{i,\infty\}$. The function $z \mapsto g(z+\nobreak n)$ is
holomorphic outside  the smaller region 
\[ [-\e^{-1}-n,\e^{-1}-n]\times i[\e,\e^{-1}]\,.\]
Hence the averages $\av{T,1}^\pm g$ are holomorphic on
\[ \CC\setminus\left( (-\infty,\e^{-1}]\times
i[\e,\e^{-1}]\right)\,,\text{ respectively } \CC\setminus\left(
[-\e^{-1},\infty)\times i[\e,\e^{-1}]\right)\,. \]

The function $h$ is holomorphic on an $\{\infty\}$-excised neighbourhood. So
after adaptation of $\e>0$ on a region 
\[ \CC \setminus \bigl[-\e^{-1},\e^{-1}\bigr]
\times i \,\bigl[ \e,\infty\bigr)\,.\]

On $0<\im(z)=y<\e$ the  functions $h$, $\av{T,1}^+ g$ and $\av{T,1}^-
g$ satisfy the same relation,
hence there are $1$-periodic $p_+$ and $p_-$ on this region that
satisfy~\eqref{havTp}. The relations between $h$ and the averages
extend by holomorphy to the half-plane $y<\e$ in~$\CC$, and the
$1$-periodic functions $p_+$ and $p_-$ extend holomorphically to
$y<\e$.

The relation $p_+=h-\av{T,1}^+g$ extends $p_+$ to a region
\[ \bigl\{z\in \CC\;:\; \im(z)<\e\bigr\}\,\cup\, \bigl\{z\in \CC\;:\;
\re(x) >\e^{-1}\bigr\}\,.\]
Then by $1$-periodicity $p_+$ has a $1$-periodic holomorphic extension
to~$\CC$. 

The case of $p_-$ and $\av{T,1}^- g$ goes similarly.
\end{proof}

\subsection{Analytic continuation of one-sided averages} To obtain the
one-sided averages with $|\ld|=1$ on representatives of $\dsv
{2-r}\om$ for more values of~$r$, we use that the space of the
projective model $\Prj{2-r}\dsv{2-r}\om$ does not depend on~$r$. The
representatives in the projective model are holomorphic functions $h$
on a neighbourhood $\Om$ of $\lhp\cup\proj\RR$ in $\proj\CC$. For a
fixed $h$ the function $g_r:=\Prj{2-r}^{-1}h$ represents an element
of $\dsv{2-r}\om$ for each $r\in \CC$, ie.
\be\label{reprgr} g_r(t) \= (i-t)^{r-2}\,h(t)\qquad(t\in \Om)\,.
\ee
In this subsection we work with the interpretation~a) in
Remark~\ref{2interpr}.

\begin{lem}\label{lem-osa-Hurw}Let $|\ld|=1$, and let $g_r$ be as
defined in~\eqref{reprgr}. Let $\e\in
(0,1)$ be such that $ \av{T,\ld}^\pm \, g_r$ is holomorphic
on the set $D_\e^\pm$ in~\eqref{Depspm}  for $\re r<1$.
\begin{enumerate}
\item[a)]The function $(r,z)\mapsto \bigl(\av{T,\ld}^\pm g_r\bigr)(z)$
extends as a holomorphic function on $\left( \CC\setminus \ZZ_{\geq
1}\right) \times D_\e^\pm$.
\item[b)] If $h(\infty)=0$ then $(r,z)\mapsto \bigl(\av{T,\ld}^\pm\,
g_r\bigr)(z)$
can be extended holomorphically to the slightly larger region $\left(
\CC\setminus \ZZ_{\geq 2}\right)
\times D_\e^\pm$.
\end{enumerate}
\end{lem}

\begin{proof}We use the \il{HuLe}{Hurwitz-Lerch
zeta-function}\emph{Hurwitz-Lerch zeta-function}\ir{HLzt}{H(s,a,z)}
\be \label{HLzt} H(s,a,z) \= \sum_{n\geq 0} e^{2\pi i an}\,
(z+n)^{-s} \,,\ee
which converges absolutely and is holomorphic in $(s,z)$ for $\im
a\leq 0$, $z\in \CC\setminus \ZZ_{\leq 0}$ and $\re s>1$. Kanemitsu,
Katsurada and Yoshimoto \cite[Theorem 1*]{KKY} give the holomorphy of
$(s,z)\mapsto H(s,a,z)$ on $\left(\CC\setminus\{1\}\right)
\times \{z\in\CC\;:\; \re z>0\}$ with a first order singularity at
$s=1$ if $a\in \ZZ$, and no singularity in $s$ at all otherwise. With
\[ H(s,a,z) \= \sum_{n=0}^{m-1} e^{2\pi i an}\, (z+n)^{-s} +
H(s,a,z+m)\]
for each $m\in \ZZ_{\geq 1}$, we obtain holomorphy in $z\in
\CC\setminus
(-\infty,0]$. (Lagarias and Li \cite{LaLi} study the continuation in
three variables. Here we need only the continuation in $(s,z)$.)

The function $h$  in \eqref{reprgr}  is holomorphic on a
neighbourhood of $\proj\RR$ in~$\proj\CC$, and hence has a convergent
power series expansion on a neighbourhood of~$\infty$:
\be h(z) = \sum_{k=0}^\infty \tilde a_k \,{z}^{-k}\,.
\ee
This implies that we have for $z\in D_\e^\pm$
\be\label{phi-e1} g_r(z) \= \sum_{k=0}^{N-1} \, a_k(r) \,
(z -i)^{r-2-k} + g_{r,N}(z)\,, \ee
with $g_{r,N}(z) = \oh(z^{r-2-N})$ as $ z \rightarrow\infty$
through~$D_\e^\pm$, uniformly for $r$ in compact sets in~$\CC$. The
$a_k(r)$ are polynomials in~$r$. We take $\arg
(z-\nobreak i) \in \bigl(-\frac {3\pi}2,\frac\pi 2\bigr)$. The
one-sided averages of $g_{r, N }$ converge absolutely, and provide
holomorphic functions in $(r,z)$ on $\re r<N+1 $ and $z\in D^\pm_\e$.
For the remaining finitely many terms we have a sum of
\begin{align*}
 a_k(r)\, &\ld^{-n}\, (z+n-i)^{r-2-k} \\
 &\= a_k(r)\, e^{-2\pi i n \al}\,
\begin{cases}
(z-i+n)^{r-2-k}&\text{ for }n\geq 0\text{ and } \re z>\e^{-1}\,,\\
e^{-\pi i(r-2-k)}\,(i-z+|n|)^{r-2-k}& \text{ for }n\leq -1\text{ and
}\re z<-\e^{-1}\,,
\end{cases}
\end{align*}
where $\al\in \CC$ has been chosen so that $\ld=e^{2\pi i \al}$. In
this way we obtain:
\badl{Av-Hurw} \bigl(\av{T,\ld}^+ g_r\bigr)(z) &\= \sum_{k=0}^{N-1}
a_k(r)\, H(k+2-r,-\al,z-i)
+ \bigl(\av{T,\ld}^+ g_{r,N}\bigr)(z)\,,\\
\bigl(\av{T,\ld}^- g_r\bigr)(z) &\= \sum_{k=0}^{N-1} a_k(r)\, \ld\,
e^{\pi i(k-r)}\, H(k+2-r,\al,1+i-z)\\
&\qquad\hbox{}
+ \bigl(\av{T,\ld}^- g_{r,N}\bigr)(z)\,. \eadl
The function $(r,z)\mapsto H(k+2-r,-\al,z-i)$ is meromorphic on the
region $r\in \CC$, $z\in i + \CC\setminus
(-\infty,0]$, with a singularity  at  $r=k+1$. The function
$(r,z)\mapsto
(\av {T,\ld}^+ g_{r,N})(z)$ is holomorphic on the region $\re r<1+N$,
$z\in D_\e^+$. So $(r,z)\mapsto (\av {T,\ld}^+ g_r)(z)$ extends
meromorphically to the region $\re r< 1 +N$, $z\in
D_\e^+$, and its singularities can occur only  at  $r=a$ with $a\in
\{1,\ldots,N\}$. The case of $\av{T,\ld}^- g_r$ goes similarly.
\end{proof}

\begin{prop}\label{prop-osa1}Let $r\in \CC$, $\ld \in \CC^\ast$, and
let $g$ represent an element of $\dsv{2-r}\om$.
\begin{enumerate}
\item[i)] There are well-defined one-sided averages \il{osa1}{$\av
{T,\ld}^\pm$}$\av {T,\ld}^\pm g$ holomorphic on $D_\e^\pm$, as in
\eqref{Depspm}, for some $\e\in(0,1)$ depending on~$g$, under the
following conditions
\be
\begin{array}{|c|c|c|c|c|}\hline
 & |\ld|>1 &\multicolumn{2}{|c|}{  |\ld|=1 } & |\ld|<1\\
 && (\Prj{2-r}g)(\infty)=0&(\Prj{2-r}g)(\infty)\neq 0& \\
 \hline
 \av{T,\ld}^+ g& r\in \CC & r\in \CC\setminus \ZZ_{\geq 2} & r\in
 \CC\setminus\ZZ_{\geq 1} & \\
\av{T,\ld}^- g & & r\in \ZZ_{\geq 2}& r\in \CC\setminus 
\ZZ_{\geq 1} & r\in \CC \\ \hline
\end{array}
\ee
\item[ii)] These one-sided averages satisfy $\av{T,\ld}^\pm g
|_{2-r}(1-\ld^{-1}\,T) = g$.
\item[iii)] If $g=g_r=\Prj{2-r}^{-1}h$ as in \eqref{reprgr}, and
$|\ld|^{\pm 1}\geq 1$, then $r\mapsto \av{T,\ld}^\pm g_r$ is a
meromorphic function on~$\CC$.
\end{enumerate}
\end{prop}

\begin{proof}
Each representative $g$ of an element of $\dsv{2-r}\om$ is of the form
$g_r=\Prj{2-r}h$ for some holomorphic function on a neighbourhood of
$\lhp\cup\proj\RR$ in $\proj\CC$. If $|\ld|\neq 1$,
Proposition~\ref{prop-osaa} gives the convergence of one of the
averages and the relation in Part~ii). The convergence is
sufficiently quick to have holomorphy in~$r$.

Let $|\ld|=1$. Proposition~\ref{prop-osaa} gives convergence of both
averages for $\re r<1$, and Lemma~\ref{lem-osa-Hurw} provides the
 meromorphic continuation to $\CC$, with singularities only in the
 points indicated in Part~i). The relation in Part~ii) stays valid by
 analytic continuation.
\end{proof}

\begin{rmk}If $|\ld|=1$, $\ld\neq 1$, the proof of
Lemma~\ref{lem-osa-Hurw} can be adapted to give holomorphy of
$\av{T,\ld}^\pm\, g_r$ in $r\in \CC$. We can strengthen the
statements in Cases i)
and~iii) of Proposition~\ref{prop-osa1} as well.
\end{rmk}

\rmrk{Asymptotic behavior}To get the asymptotic behavior of $
 \av{T,\ld }^\pm g_r(t)$ as $\pm \re t\rightarrow\infty$, we use the
 following result:

\begin{prop}\label{prop-Ka}\il{HLzt-as}{Hurwitz-Lerch zeta-function,
asymptotic behavior} \emph{(Katsurada, \cite{Ka98}) } Let $s,a\in
\CC$, $\im a\leq 0$. There are $b_k(\ld,s)\in \CC$ such that for each
$K\in \ZZ_{\geq 0}$ we have as $|z|\rightarrow\infty$ on any region
$\dt-\pi \leq \arg z \leq \pi-\dt$ with $\dt>0$
\be\label{H-as} H\bigl(s,a,\frac12+z\bigr) \= \frac{\e(\ld)}{1-s} \,
z^{1-s} + \sum_{k=0}^{K-1} \, b_k(\ld,s)\, z^{-k-s} + \oh\bigl(
|z|^{-\re s-K}\bigr)\,, \ee
with $\ld=e^{2\pi i a}$, $\e(\ld)=1$ if $\ld=1$ and $\e(\ld)=0$
otherwise. The coefficients $b_k$ satisfy
\be\label{bk-rel} \ld^{-1}\, b_k(\ld^{-1},s) \= (-1)^{k+1}\,
b_k(\ld,s)\,. \ee
The first three coefficients are as follows:
\be \renewcommand\arraystretch{1.3}
\begin{array}{|c|c|c|}\hline
& \ld=1 & \ld \neq 1 \\ \hline
b_0(\ld ,s) &0 &\frac 1{1-\ld}
\\
b_1(\ld,s) & - \frac s{48} & - \frac s4 \frac{1+\ld}{(1-\ld)^2}\\
b_2(\ld,s) & 0 & \frac{s(s+1)(1+6\ld+\ld^2)}{48(1-\ld^3)}
\\ \hline
\end{array}
\ee
\end{prop}

\begin{proof}This is a direct consequence of~\cite[Theorem 1]{Ka98},
applied with $\al=\frac12$.

We have $H(s,a,z)=\Phi(a,z,s)$ with Katsurada's $\Phi$. This
gives~\eqref{H-as} with
\[ b_k(\ld,s) \= \frac{(-1)^{k+1}}{(k+1)!}\, B_{k+1}\bigl( \tfrac 12,
\ld\bigr)\; (s)_k\,,\]
where the $B_k$ are generalized Bernoulli polynomials, given by
\[ \sum_{k\geq 0} B_k(x,y) \,\frac{z^k}{k!} \=
\frac{ze^{xz}}{ye^z-1}\,.\]
Relation \eqref{bk-rel} follows from
\[ \frac {z e^{ z/2}}{y e^z-1} \= \frac z {ye^{z/2}- e^{-z/2}} \= y^{-1}\,
\frac{-z}{y^{-1}e^{-z/2}-e^{z/2}}\,. \qedhere\]
\end{proof}

\begin{prop}\label{prop-av-as}\il{1sa-as}{one-sided average,
asymptotic behavior}Let $r\in \CC$, and let $g$ be a
representative of an element of $\dsv{2-r}\om$.
\begin{enumerate}
\item[i)]  Let $|\ld|=1$ and suppose that $r\in \CC$ is such that
$\av{T,\ld}^+\, g$ and $\av{T,\ld}^-\, g$ exist.
 \begin{enumerate}
\item[a)]There are coefficients $c_k$ depending on $\ld$, $r$ and on
the coefficients of the expansion of $\Prj{2-r}g$ at~$\infty$, such
that for each $M\in \ZZ_{\geq 0}$ we have:
\be\label{avtexp}
(\av{T,\ld}^ \pm\, g)(t) \= (it)^{r-2} \, \sum_{k=-1}^{M-1} c_k \,
t^{-k}
+ \oh\bigl( |t|^{r-2-M}\bigr)\ee
as $|t|\rightarrow\infty$ with $\pm \re t\geq 0$, $\im t\leq 0$.
\item[b)] If $g(t) = (it)^{r-2}\, \bigl( a_0+a_1\,
t^{-1}+\cdots\bigr)$
near~$\infty $, then
\be\renewcommand\arraystretch{1.3}
\begin{array}{|c|c|c|}\hline
& \ld=1& \ld \neq 1 \\ \hline
c_{-1} &\frac {a_0}{r-1} & 0
\\ c_0 & \frac {a_1}{r-2} & \frac{ \ld\, a_0}{\ld-1}\\
c_1 & \frac{ a_2}{r-3}+\frac{(r-2)\,a_0}{48} & \frac{\ld \,
a_1}{\ld-1}
 + \frac{(r-2)\,\ld\,(\ld+1)\, a_0}{4(\ld-1)^2}
\\\hline
\end{array}
\ee
\end{enumerate}
\item[ii)]
Let $|\ld|\neq 1$. If $\pm$ is such that $|\ld|^{\pm 1}>1$, then
there are coefficients $c_k$ such that for each $M\in \ZZ_{\geq 0}$
\be\label{avtexp1}
(\av{T,\ld}^ \pm\, g)(t) \= (it)^{r-2} \, \sum_{k=0}^{M-1} c_k \,
t^{-k}
+ \oh\bigl( |t|^{r-2-M}\bigr)\ee
as $|t|\rightarrow\infty$ with $\pm \re t\geq 0$, $\im t\leq 0$.
\end{enumerate}
\end{prop}

\rmrke It is remarkable that if both one-sided averages exist, then
the coefficients in both expansions are the same, although
$\av{T,\ld}^+\, g$ and $\av{T,\ld}^ -\, g$ have
in general no reason to be equal.

\begin{proof}
It suffices to consider large values of~$M$. We take $M >
\re r +1$. If $g(t) = \oh \bigl( |t|^{r-3-M}\bigr)$ we have
$(\av{T,\ld}^\pm\, g)(t) = \oh\bigl( |t|^{r-2-M}\bigr)$, which only
influences the error term. So the explicit terms in the asymptotic
expansion are determined by the part
\[ (it)^{r-2} \sum_{j=0}^M a_j\, t^{-j}\]
of the expansion of $g$ at~$\infty$. We consider for 
$0\leq j \leq M$
functions $g_j$ representing elements of $\dsv{2-r}\om$ for which
$g_j(t) =
(it)^{r-2}\, t^{-j} + \oh \bigl( |t|^{r-3-M}\bigr)$ as
$t\rightarrow\infty$.

In Part~i) we have $|\ld|=1$. 
In~\eqref{phi-e1} we took $t-i$ as the variable. Now $t-\frac12$ is
more convenient. We put $\ld=e^{2\pi i \al }$, and have, modulo terms
that can be absorbed into the error term:
\begin{align*} ( \av{T,\ld}^+ \, g_j)(1/2+t)&\;\equiv\; -e^{\pi
ir/2}\, H(2+j-r,-\al,1/2+t)\,,\\
(\av{T,\ld}^ -\, g_j)(1/2+t) &\;\equiv\; e^{-\pi i r/2} \, (-1)^j
\,\ld\, H(2+j-r, \al ,1/2-t)\,.
\end{align*}
With Proposition~\ref{prop-Ka} this gives
\begin{align*}
(\av{T,\ld}^ +\, g_j)&(1/2+t) \;\equiv\; (it)^{r-2}\, \Bigl(
\frac{\e(\ld)}{r-j-1}\, t^{1-j} + \sum_{k=0}^{M-j}
b_k(\ld^{-1},2+j-r) \, t^{-k-j}\Bigr)\,,
\displaybreak[0]\\
(\av{T,\ld}^ -\, g_j)&(1/2+t) \;\equiv\; e^{-\pi i r/2}(-1)^j \ld
\;\Bigl( \frac{\e(\ld)}{r-1-j}\, (-t)^{r-1-j} \\
&\qquad\hbox{} + \sum_{k=0}^{M-j} b_k(\ld,2+j-r)\, (-t)^{-k-j+r-2}
\Bigr)
\displaybreak[0]\\
&\= (it)^{r-2} \; \Bigl( \frac{\e(\ld)}{r-1-j} \, t^{1-j} -
\sum_{k=0}^{M-j} (-1)^k \, \ld\, b_k(\ld,2+j-r)\, t^{-k-j} \Bigr)\,.
\end{align*}
(In the last step we have used that $\ld=1$ if $\e(\ld)\neq 0$.)

For $g$ with expansion $(it)^{r-2}\sum_{j\geq 0} a_j \, t^{-j}$
near~$\infty$ this leads to an expansion as in~\eqref{avtexp}, with
coefficients $c_\ell^\pm$ of the form
\[ c^\pm_{-1} \= \frac{\e(\ld)}{r-1}\, a_0\,,\]
and for $\ell\geq 0$
\begin{align*}
c_\ell^+ &\= \frac{\e(\ld)}{r-\ell-2} a_{\ell+1} + \sum_{j=0}^\ell
b_{\ell-j}(\ld^{-1},2+j-r)\, a_j\,,
\displaybreak[0]
\\
c_\ell^- &\=\frac{\e(\ld)}{r-\ell-2} a_{\ell+1} - \sum_{j=0}^\ell
(-1)^{\ell-j} \, \ld\, b_{\ell-j}(\ld,2+j-r)\, a_j\,.
\end{align*}
Relation~\eqref{bk-rel} shows that $c_\ell^+ = c_\ell^-$.

In part~ii) the factor $\ld^{-n}$ takes care of the convergence
of the one-sided average.
The asymptotic behavior  follows directly
from the behavior of $g(t)$ near~$\infty$.
\end{proof}

\begin{lem}\label{lem-osa-pe}Let $|\ld|=1$.
\begin{enumerate}
\item[i)] Let $r\in \RR\setminus \ZZ_{\geq 1}$. 
The following statements concerning $\ph \in
\dsv{2-r}\om$ are equivalent:
\begin{enumerate}
\item[a)] There is a representative $g$ of $\ph$ for which
$\av{T,\ld}^+ g $ and $\av{T,\ld}^- g$ represent the same element of
$\dsv{2-r}\om[\infty]$.
\item[b)] There is a function $h$ representing an element of
$\dsv{2-r}{\om,\smp}[\infty]$ such that
$h|_{2-r}(1-\nobreak\ld^{-1}\, T)$ represents~$\ph$.
\end{enumerate}
If these statements hold, then $\av{T,\ld}^+g$, $\av{T,\ld}^- g$ and
$h$ represent the same element of~$\dsv{2-r}{\om,\smp}[\infty]$.
\item[ii)] Let $r\in \RR\setminus \ZZ_{\geq 2}$. The following
statements concerning $\ph \in \dsv{2-r}\om$ are equivalent:
\begin{enumerate}
\item[a)] There is a representative $g$ of $\ph$ such that
$\Prj{2-r}g(\infty)=0$, and for which $\av{T,\ld}^+ g $ and
$\av{T,\ld}^- g$ represent the same element of
$\dsv{2-r}\om[\infty]$.
\item[b)] There is a function $h$ representing an element of
$\dsv{2-r}{\om,\infty}[\infty]$ such that
$h|_{2-r}(1-\nobreak\ld^{-1}\, T)$ represents~$\ph$.
\end{enumerate}
If these statements hold, then $\av{T,\ld}^+g$, $\av{T,\ld}^- g$ and
$h$ represent the same element of~$\dsv{2-r}{\om,\infty}[\infty]$.
\end{enumerate}
\end{lem}

\begin{proof}Let $g$ be a representative of $\ph$ as in one of the
statements~a). Then $\av{T,\ld}^+ g(z) = \av{T,\ld}^-g(z)$ for $y<\e$
for some $\e\in
(0,1)$. Let us call this function $f$. It is holomorphic on a
neighbourhood of $\lhp\cup\RR$ in~$\CC$, and
Proposition~\ref{prop-av-as} shows $(\Prj{2-r}f)(z)$ has an
asymptotic expansion as $z\rightarrow\infty$ through $\lhp\cup\RR$
required in Definition~\ref{sav-ss} for representatives of elements
of $\dsv{2-r}{\om,\smp}[\infty]$. This gives b) in Part~i). If we
have the additional condition $(\Prj{2-r}g)(\infty)=0$, the
asymptotic expansion starts at $k=0$ instead of $k=-1$, and we
conclude that $f$ represents an element of
$\dsv{2-r}{\om,\infty}[\infty]$. This concludes the proof of
a)$\Rightarrow$b) in both parts.

Let $h$ as b) be given. With any representative $g$ of $\ph$, we have
also $\av{T,\ld}^+ g$ and $\av{T,\ld}^- g$ in $\dsv{2-r}\om[\infty]$
satisfying the same relation. So $h-\av{T,\ld}^\pm g$ is a
$\ld$-periodic function on a neighbourhood of~$\RR$, with a one-sided
asymptotic expansion of the type \eqref{avtexp} as $\pm\re
z\rightarrow \infty$. Hence this $\ld$-periodic function is zero 
by Lemma~\ref{lem-ldpr}, and the three functions $h$, $\av{T,\ld}^+
g$ and $\av{T,\ld}^- g$ are
the same on a neighbourhood of $\lhp\cup\RR$ in~$\CC$, and represent
the same element of $\dsv{2-r}\om[\infty]$. That gives a) in Part~i).
For Part~ii) we note that the fact that $h$ represents an element of
$\dsv{2-r}{\om,\infty}[\infty]$ implies $(\Prj{2-r}g)(\infty)=0$.
\end{proof}

\subsection{Parabolic cohomology groups} With the one-sided averages
we can prove some of the isomorphisms in Theorem~\ref{THMiso} on
page~\pageref{THMiso}.

\begin{prop}\label{prop-rid-of-exc}Let $r\in \RR$, and let $v$ be a
unitary multiplier system.
\begin{enumerate}
\item[i)] If $r\not\in \ZZ_{\geq 2}$ then
\[ \hpar^1(\Gm;\dsv{v,2-r}\om,\dsv{v,2-r}{\fsn,\infty,\wdg}) \=
\hpar^1(\Gm;\dsv{v,2-r}\om,\dsv{v,2-r}{\fsn,\infty})\,.\]
\item[ii)]The codimension of
$\hpar^1(\Gm;\dsv{v,r}\om,\dsv{v,r}{\fsn,\smp,\wdg})$ in
$\hpar^1(\Gm;\dsv{v,r}\om,\dsv{v,r}{\fsn,\smp})$ is finite if $r=1$,
and zero if $r\not\in \ZZ_{\geq 1}$.
\end{enumerate}
\end{prop}

\begin{proof}Let  $\ps \in
Z^1(\Gm;\dsv{v,2-r}\om,\dsv{v,2-r}{\fsn,\infty})$. For
cusps $\ca$ in a (finite) set of representatives of the $\Gm$-orbits
of cusps we consider $h_\ca\in \dsv{v,2-r}{\fsn,\infty}$ such that
$h_\ca |_{v,2-r}(1-\nobreak\pi_\ca) =
\ps_{\pi_\ca}\in\dsv{v,2-r}\om$. After conjugation, we are in the
situation of Part~ii)b) of Lemma~\ref{lem-osa-pe} with
$\ld=v(\pi_\ca)$. Since the conditions on~$r$ and $\ld$ in that lemma
are satisfied, we have
$h=\av{T,\ld}^+\ps_{\pi_\ca}=\av{T,\ld}^-\ps_{\pi_\ca}$ near
$\lhp\cup\RR$. Since $\av{T,\ld}^\pm \ps_{\pi_\ca}$ is holomorphic on
$D_\e^+ \cup D_\e^-$ for some $\e>0$, with $D_\e^\pm$ as
in~\eqref{Depspm}, the function $h$ is holomorphic on a
$\{\infty\}$-excised neighbourhood, hence $h\in
\dsv{2-r}{\om,\infty,\wdg}[\infty]$, and $h_\ca \in
\dsv{2-r}{\om,\infty,\wdg}[\ca]\subset \dsv{2-r}{\fsn,\infty,\wdg}$.

The other case goes similarly, except if $r=1$ and $v(\pi_\ca)=1$. If
$\Prj{1}(\ps_{\pi_\ca})(\infty)=0$ then Proposition~\ref{prop-av-as}
implies that the starting term of the asymptotic expansion
\eqref{avtexp} satisfies $k\geq 0$, and $h_\ca$ is in
$\dsv{v,1}{\om,\infty}[\ca]$, and the same reasoning applies. Since
the number of cuspidal orbits is finite, this imposes conditions on
the cocycles determining a subspace of finite codimension.
\end{proof}

\begin{prop}\label{prop-acpc}If $r\in \CC\setminus \ZZ_{\geq 1}$, then
\be
\hpar^1(\Gm;\dsv{v,2-r}\om,\dsv{v,2-r}\fs)\=H^1(\Gm;\dsv{v,2-r}\om)
\,. \ee

If $r=1$, then the space $\hpar^1(\Gm;\dsv{v,1}\om,\dsv{v,1}\fs)$ has
finite codimension in the space $H^1(\Gm;\dsv{v,1}\om)$.
\end{prop}

\rmrke So for all $r\not\in \ZZ_{\geq 2}$ the space
$\hpar^1(\Gm;\dsv{v, 2-r}\om,\dsv{v, 2-r}\fs)$ has finite codimension
in $H^1(\Gm;\dsv{v, 2-r}\om)$.
\smallskip

We prepare the proof of Proposition~\ref{prop-acpc} by a lemma.
\begin{lem}\label{lem-*solve}Let $r\in \CC$ and $\ld\in \CC^\ast$.
Then
\bad
&\dsv{2-r}\om \,\subset\,
\dsv{2-r}\om[\infty]|_{2-r}(1-\ld^{-1}T)&\text{ if }& r\not\in
\ZZ_{\geq 1}\,,\\
&\dim \biggl( \dsv 1 \om \bigm/ \Bigl( \dsv 1 \om \cap \bigl( \dsv 1
\om[\infty]|_1(1-\ld^{-1}T) \bigr) \Bigr) \biggr)\;\leq\; 1&\text{ if
}&r=1\,. \ead
In the case $r=1$ an element $\ph\in \dsv 1 \om$ is in $\dsv 1
\om[\infty]|_1(1-\nobreak\ld^{-1}T)$ if $\ld\neq 1$ or if
$\Prj1\ph(\infty)=0$.
\end{lem}

\begin{proof}Proposition~\ref{prop-osa1} shows that if $r\not\in
\ZZ_{\geq 1}$ or if $\ld\neq 1$, we can use at least one of the
one-sided averages to show that $\dsv{2-r}\om$ is contained in
$\dsv{2-r}\om[\infty]|_{2-r}(1-\nobreak\ld^{-1}T)$. If $r=1$ and
$\ld=1$ we have to restrict ourselves to a subspace of
$\dsv1\om[\infty]$ of codimension~$1$.
\end{proof}

\begin{proof}[Proof of Proposition~\ref{prop-acpc}.]The inclusion
$\subset$ follows from the definition of parabolic cohomology. To
prove the other inclusion we consider a cocycle $\ps\in
Z^1(\Gm;\dsv{v,2-r}\om)$ and need to show that for a representative
$\ca$ of each $\Gm$-orbit of cusps there is $h\in \dsv{v,2-r}\fs$
 such that $h|_{v,2-r}(1-\nobreak\pi_\ca)=\ps_{\pi_\ca}$. By
 conjugation this can be brought to $\infty$, into the situation
 considered in Lemma~\ref{lem-*solve}. Since there are only finitely
 many cuspidal orbits, we get for $r=1$ a subspace of finite
 codimension.
\end{proof}

\subsection{Related work}\label{sect-lit4}
Knopp uses one-sided averages in \cite[Part IV]{Kn74}, attributing the
method to B.A.\,Taylor (non-published). In \cite{BLZm} the one-sided
averages are an important tool, defined in Section~4, and used in
Sections 9 and~12.


\part{Harmonic functions}

\section{Harmonic functions and cohomology}\label{sect-hf}

\subsection{The sheaf of harmonic functions}\label{sect-shf}
By associating to open sets $U\subset\uhp$ the vector space
$\sharm_r(U)$ of $r$-harmonic functions on~$U$ (as defined in
Definition~\ref{haf-def}) we form the \il{shf}{sheaf of $r$-harmonic
functions}\il{sharm}{$\sharm_r$}\emph{sheaf $\sharm_r$ of
$r$-harmonic functions} on~$\uhp$.

The \il{shad1}{shadow operator}shadow operator \il{sh1}{$\shad_r$}$\shad_r$
in~\eqref{shad} determines a morphism of sheaves $\sharm_r
\rightarrow \hol_\uhp$, where \il{holuhp}{$\hol_\uhp$}$\hol_\uhp$
denotes the sheaf of holomorphic functions on~$\uhp$, and leads to an
exact sequence
\be\label{shad-seq}
0\rightarrow \hol_\uhp \rightarrow \sharm_r
\stackrel{\shad_r}\rightarrow \hol_\uhp \rightarrow 0\,.\ee
The maps $\shad_r: \sharm_r(U) \rightarrow \hol(U)$ are antilinear for
the structure of vector spaces over~$\CC$. The surjectivity of $\shad
_r$ follows from classical properties of the operator $\partial_{\bar
z}$. (It suffices to solve locally $\partial_{\bar z} h = \ph$ for
given holomorphic $\ph$. See, eg., H\"ormander
\cite[Theorem~1.2.2]{Ho}.)

\rmrk{Actions}Let $r\in \CC$. For each $g\in \SL_2(\RR)$ the operator
$|_r g$ gives bijective linear maps $\sharm_r (U)
\rightarrow \sharm_r
(g^{-1}U)$ and $\hol_\uhp(U)
\rightarrow\hol_\uhp(g^{-1}U)$. For sections $F$ of $\sharm_r$
\be \Dt_r \bigl( F|_r g\bigr) \= (\Dt_r F)|_r g\,,\qquad \shad_r\bigl(
F|_r g\bigr) \= \bigl(\shad_r F \bigr)|_{2-\bar r} g\,. \ee

If $v$ is a multiplier system for $\Gm$ for the weight~$r$, then we
have also the actions $|_{v,r}$ of $\Gm$ on $\sharm_r$ and $|_{\bar
v,2-\bar r}$ on $\hol$. With these actions $\sharm_r$ and $\hol$ are
$\Gm$-equivariant sheaves.

\subsection{Harmonic lifts of automorphic forms}\label{sect-prfhl}In
this subsection we will prove Theorem~\ref{THMhl}.

\rmrk{Example} The image in $
H^1\bigl(\Gm(1);\dsv{1,2}{\fsn,\exc}\bigr)$ of $\coh0\om 1\in
H^1\bigl(\Gm(1);\dsv{1,2}\om\bigr)$ can be represented by the cocycle
$\tilde \ps$ in~\eqref{ps-r=0}, given by
$\tilde\ps_\gm(t)=\frac{-c}{ct+d}$. The cocycle $\Ci \tilde\ps$ in
the functions on $\uhp$, obtained with the involution $\Ci$
in~\eqref{Ci} can be written as $\gm\mapsto h|_{1,2}(1-\gm)$, with
the $2$-harmonic function $h(z) = \frac i{2y}$.

The \il{holEis}{holomorphic Eisenstein series in weight
$2$}holomorphic Eisenstein series of weight $2$\ir{E2}{E_2,\;
E_2^\ast}
\be\label{E2} E_2(z) \= 1-24\sum_{n\geq 1}\sigma_1(n)\, e^{2\pi i n z}
\ee
is not a modular form. Adding a multiple of $h$ we get $E_2^\ast =
\frac{6i}\pi\, h + E_2 $, which is a harmonic modular form in $
\harm_2\bigl(\Gm(1),1\bigr)$. The function $E_2^\ast$ is a
$2$-harmonic lift of the constant function $\frac 3\pi$. See
Definition~\ref{hldef}. Furthermore, we have
\[ \Ci\tilde\ps_\gm \= \frac\pi{6i}\, E_2|_{1,2}(\gm-1)\,.\]
Since $E_2$ has polynomial growth near the boundary $\proj\RR$ of
$\uhp$, we conclude that with $b=\frac{\pi i}6\, \Ci E_2\in
\dsv{1,2}{-\infty}$ we obtain $b|_{1,2}(\gm-\nobreak
1)=\tilde\ps_\gm$. So the class $\coh0\om 1$ becomes trivial under
the natural map to $H^1\bigl(\Gm(1);\dsv{1,2}{-\infty}\bigr)$.

\rmrk{Alternative description of cocycles} Generalizing this example,
we first use the differential form $\om_r(F;\cdot,\cdot)$
in~\eqref{omr} to describe the cocycle $\ps^{z_0}_F$
in~\eqref{psiz0def} in an alternative way.  We recall the
involution $\Ci $ in~\eqref{Ci}.

\begin{lem}Let $r\in \CC$ and $F\in A_r(\Gm,v)$. We put for $t\in
\lhp$:\ir{QF}{Q_F}
\be\label{QF} Q_F(t) \;:=\; \int_{z_0}^{\bar t} \om_r(F;t,z)\,.\ee
\begin{enumerate}
\item[i)] The function $Q_F$ on $\lhp$ satisfies
\be \label{QF-psi}
Q_F|_{v,2-r}(\gm-1) \= \ps_{F,\gm}^{z_0}\qquad\text{ for each }\gm\in
\Gm. \ee
\item[ii)] The corresponding function $\Ci Q_F$ on $\uhp$ is
$(2-\nobreak\bar r)$-harmonic, and
\be\label{shad Q} \shad_{2-\bar r}\, \Ci \, Q_F (z) \= 2^{r-1}\,
e^{\pi i (r-1)/2}\, F(z)\,. \ee
\end{enumerate}
\end{lem}

\begin{proof}For Part~i) we use Part~iii) of Lemma~\ref{lem-omprop}.
For $\gm=\matc abcd \in \Gm$
\begin{align*} Q_F|_{v,2-r} \gm\,(t) &\= v(\gm)^{-1}\, (c t+d)^{r-2}\,
Q_F(\gm t)
\\
&\= \int_{z_0}^{\gm \bar t} \om_r(F;\cdot,z)
|_{v,2-r}\gm\,(t) \= \int_{\gm^{-1}z_0}^{\bar t}
\om_r(F;t,z)\,.\end{align*}

For Part~ii) we note that the function $\Ci Q_F$ on $\uhp$ satisfies
\begin{align*}
 \overline{\Ci Q_F (z)} &\= \int_{\tau= z_0}^z (\tau-\bar z)^{r-2}\,
 F(\tau)\, d\tau\,,
 \displaybreak[0]\\
\shad_{2-\bar r}\, \Ci Q_F(z) &\= 2i\, y^{2- r} \,
\frac\partial{\partial
z} \overline{ \Ci Q_F(z) } \= 2i \, y^{2- r} \, (z-\bar z)^{r-2}\,
F(z)\\
&\= 2^{r-1}\, e^{\pi i (r-1)/2}\, F(z)\,.
\end{align*}
Since the image of $\Ci Q_F$ under the shadow operator is holomorphic,
the function $\Ci Q_F$ is in $\sharm_{2-\bar r}(\uhp)$.
\end{proof}

\begin{proof}[Proof of Theorem~\ref{THMhl}] The theorem states the
equivalence of two statements concerning an automorphic form $F\in
A_r(\Gm,v)$. Statement a) means that the cocycle $\ps^{z_0}_F$ is a
coboundary in $B^1(\Gm;\dsv{v,2-r}{-\om})$. This is equivalent to
\begin{enumerate}
\item[a')] $\exists \Ph \in \hol(\lhp)\; \forall \gm\in \Gm\;:\;
\Ph|_{v,2-r}(\gm-1)
\= \ps_{F,\gm}^{z_0}$.
\end{enumerate}
Statement b) amounts to the existence of a $(2-\nobreak\bar
r)$-harmonic lift of~$F$, and is equivalent to
\begin{enumerate}
\item[b')] $\exists H \in \harm_{2-\bar r}(\Gm,\bar v)\;:\; \shad_r H
= F$.
\end{enumerate}
We relate these two statements by a chain of intermediate equivalent
statements s1)--s6).

\begin{enumerate}
\item[s1)] $\exists \Ph \in \hol(\lhp)\; \forall \gm\in \Gm\;:\;
\Ph|_{v,2-r}(\gm-1)
\= Q_F|_{v,2-r}(\gm-1)$.
\end{enumerate}
Relation~\eqref{QF-psi} implies the equivalence of a') and~s1).

We rewrite s1) as follows:
\begin{enumerate}
\item[s2)] $\exists \Ph \in \hol(\lhp)\; \forall \gm\in \Gm\;:\;
(\Ph-Q_F)|_{ v,2- r}\gm = \Ph-Q_F$.
\end{enumerate}

Functions on $\lhp$ and $\uhp$ are related by the involution $\Ci$
in~\eqref{Ci}, which preserves holomorphy. So s2) is equivalent to
the following statement:
\begin{enumerate}
\item[s3)] $\exists M \in \hol(\uhp)\; \forall \gm\in \Gm\;:\;
(M-\Ci Q_F)|_{\bar v,2-\bar r}\gm = M-\Ci Q_F$.
\end{enumerate}

The holomorphy of $M$ is equivalent to the vanishing of $\shad_{2-\bar
r} M$. Hence s3) is equivalent to the following statement:
\begin{enumerate}
\item[s4)] There is a function $M$ on $\uhp$ such that $\forall \gm\in
\Gm\;:\;
(M-\Ci Q_F)|_{\bar v,2-\bar r}\gm = M-\Ci Q_F$ and $\shad _ {2-\bar r}
M=0$.
\end{enumerate}

Now we relate functions $H$ and $M$ on $\uhp$ by $H=M-\Ci Q_F$. With
\eqref{shad Q} this shows that s4) is equivalent to the following
statement:
\begin{enumerate}
\item[s5)] There is a function $H$ on $\uhp$ such that $\forall \gm\in
\Gm\;:\; H|_{\bar v2-\bar r}\gm = H$ and $\shad _{2-\bar r} H = -
2^{r-1} e^{\pi i(r-1)/2} F$.
\end{enumerate}

The statement that $\shad_r H$ is holomorphic is equivalent to the
statement that $H$ is $(2-\nobreak\bar r)$-harmonic. Hence we get the
equivalent statement:
\begin{enumerate}
\item[s6)] $\exists H\in \harm_{2-\bar r}(\Gm,\bar v)\;:\;
\shad_{2-\bar r}H =- 2^{r-1} e^{\pi i(r-1)/2} F$.
\end{enumerate}

Up to replacing $H$ by a non-zero multiple, statement~s6) is
equivalent to statement~b').
\end{proof}

\rmrke \ The $r$-harmonic function $Q_F$ in~\eqref{QF} describes
the cocycle $\ps_F^{z_0}$ by the relation~\eqref{QF-psi}.
Theorem~\ref{THMhl} relates the existence of a holomorphic function
also describing $\ps_F^{z_0}$ to the existence of a $r$-harmonic
lift. One may call such holomorphic functions \il{ai}{automorphic
integral}\emph{automorphic integrals}. In the work of Knopp
\cite{Kn74} and others there is the additional requirement that
automorphic integrals are invariant under~$T$.

\rmrk{Consequences} Kra's result \cite[Theorem 5]{Kra69a} is
equivalent to the statement that $H^1(\Gm;\dsv{1,2-r}{-\om})=\{0\}$
for even weights~$r$. So we have the following direct consequence of
Theorem~\ref{THMhl}.

\begin{cor}\label{cor-Kra}Let $r\in 2\ZZ$ and let $v$ be the
trivial multiplier system. Then each automorphic form in $A_r(\Gm,1)$
has a harmonic lift in $\harm_{2-r}(\Gm,1)$.
\end{cor}

A bit more work is needed for the following consequence of
Theorem~\ref{THMhl}.

\begin{thm}
\label{thm-c-hl}Let $v$ be a unitary multiplier system for the weight
$r\in \RR$. If each cusp form in $\cusp r
(\Gm,v)$ has a $(2-\nobreak r)$-harmonic lift, then each unrestricted
holomorphic automorphic form in $ A_r(\Gm,v)$ has a $(2- r)$-harmonic
lift.
\end{thm}

\begin{proof}Comparing our results with the Theorem of Knopp and Mawi
\cite{KM10}, reformulated as Theorem~\ref{thm-KM} above, we noted
that the diagram \eqref{KM-diag} shows that for real $r$ and unitary
$v$ we can decompose $A_r(\Gm,v)
= \cusp r(\Gm,v)\oplus X$, where $X$ is the kernel of the composition
\[ A_r(\Gm,v) \stackrel{\coh r \om}\rightarrow
H^1(\Gm;\dsv{v,2-r}\om)\rightarrow H^1(\Gm;\dsv{v,2-r}{-\infty})\,.\]
So all elements of this space $X$ have $(2-\nobreak \bar r)$-harmonic
lifts, which are $(2-\nobreak r)$-harmonic lifts, since here the
weight is real. So if one can lift cusps forms, one can lift all
elements of $A_r(\Gm,v)$.
\end{proof}

\subsection{Related work}\label{sect-lit5}
Knopp \cite[\S V.2]{Kn74} discussed the question how far the module has
to be extended before a cocycle attached to an automorphic forms
becomes a coboundary.

The relation between harmonic automorphic forms, automorphic integrals
and cocycles for the shadow is mentioned by Fay on p.~145 of
\cite{Fay77}.

Bruinier and Funke \cite{BF4} explicitly considered the shadow
operator and the question whether
\il{hlc}{harmonic lift}harmonic lifts exist. Existence of harmonic
lifts is often shown with help of real-analytic Poincar\'e series
with exponential growth, introduced by Niebur \cite{Ni73}. For
instance, Bringmann and Ono \cite{BO7} (cusp forms for $\Gm_0(N)$
weight $\frac 12$), Bruinier, Ono and Rhoades \cite{BOR8}
(integral weights at least $2$), Jeon, Kang and Kim \cite{JKK}
(weight $\frac32$, exponential growth), Duke, Imamo\ovln{g}lu,\ and
T\'oth \cite{DIT14} (weight~$2$). The approach in \cite{Br-hl}
(modular forms of complex weight with at most exponential growth) is
similar; it uses no Poincar\'e series but similar meromorphic
families. Bruinier and Funke \cite[Corollary 3.8]{BF4} use Hodge
theory for the the existence of $r$-harmonic lifts, and Bringmann,
Kane and Zwegers \cite[\S3, \S5]{BKZ} explain how to employ
holomorphic projection for this purpose.

The harmonic lifts are related to ``mock automorphic forms''. For a
given unrestricted holomorphic automorphic form $F\in A_r(\Gm,v)$ it
is relatively easy to write down a harmonic function $C$ such that
$\shad_{2-\bar r}\,C=F$. The function $\Ci Q_F$ in~\eqref{QF} is an
example. Any holomorphic function $M$ such that $M+C$ is a harmonic
automorphic form may be called a \il{maf}{mock automorphic form}mock
automorphic form.  The function $E_2$ in \eqref{E2} is a well
known example.
In the last ten years a vast literature on mock automorphic
forms has arisen. For an overview we mention \cite{Fo10, Za09}.

It should be stressed that our Theorem~\ref{THMhl} concerns the
existence of harmonic lifts and of automorphic integrals. An
enjoyable aspect of the theory is the large number of mock modular
forms with a explicit, number-theoretically nice description, often
with weights $\frac12$ and $\frac32$, related to functions on the
Jacobi group. See, for instance, \cite{CL10, CK12, CL12}). This leads
to explicit harmonic lifts, and via Theorem~\ref{THMhl} to the
explicit description of cocycles as coboundaries.


\section{Boundary germs}\label{sect-bg}

To complete the proof of Theorem~\ref{THMac} we have to show that each
cohomology class in
$\hpar^1(\Gm;\dsv{v,2-r}\om,\dsv{v,2-r}{\fsn,\wdg})$ is of the form
$\coh r \om F$ for some unrestricted holomorphic automorphic form. To
do this, we use the spaces of ``analytic boundary germs'', in
Definition~\ref{sharmbdef}. This allows us to define, for $r \in \CC
\smallsetminus \ZZ_{\geq 2}$, $\Gm$-modules isomorphic to
$\dsv{v,2-r}\om$ and $\dsv{v,2-r}{\fsn,\wdg}$, consisting of germs of
functions. These germs are sections of a sheaf on the common boundary
$\proj\RR$ of $\lhp$ and $\uhp$. Using these isomorphic modules we
will be able, in Section~\ref{sect-caf}, to complete the proof of
Theorem~\ref{THMac}.

\subsection{Three sheaves on the real projective line}
\subsubsection{The sheaf of real-analytic functions on $\proj\RR$ }
Recall that $\hol$ denotes the sheaf of holomorphic functions
on~$\proj\CC$.

\begin{defn}\label{Vom}
For each open set $I\subset \proj\RR$ we define the sheaf $\V{2-r}\om$
by\ir{Vtmrom}{\V{2-r}\om}
\be \label{Vtmrom}
\V{2-r}\om(I) \;:=\; \indlim \hol(U)\,, \ee
where $U$ runs over the open neighbourhoods of $I$ in $\proj\CC$. The
 operator $|^\prj_{2-r}g$ in~\eqref{prjact} gives a linear bijection
 $\V{2-r}\om(I)
 \rightarrow \V{2-r}\om(g^{-1}I)$ for each $g\in \SL_2(\RR)$.
\end{defn}

The sections in $\V{2-r}\om(I)$ for $I$ open in $\proj\RR$ are
holomorphic on some neighbourhood of $I$ in~$\proj\CC$, and hence have
a real-analytic restriction to~$I$. Conversely, any real-analytic
function on $I$ is locally given by a convergent power series, and
hence extends as a holomorphic function to some neighbourhood of~$I$.
So we can view $\V{2-r}\om$ for each $r$ as the \il{srafP}{sheaf of
real-analytic functions}sheaf of real-analytic functions
on~$\proj\RR$, provided with the operators $|^\prj _{2-r}g$ with
$g\in \SL_2(\RR)$.

The space of global sections $\V {2-r}\om(\proj\RR)$ contains a copy
of $\dsv{2-r}\om$. Indeed, the map
\il{Prj1}{$\Prj{2-r}$}$(\Prj{2-r}\ph)(t)
= (i-\nobreak t)^{2-r}\, \ph(t)$ induces the injection
$$\Prj{2-r}:\dsv{2-r}\om \rightarrow \V{2-r}\om(\proj\RR)$$ that
intertwines the operators $|_{2-r}g$ and $|^\prj_{2-r}g$ for $g\in
\SL_2(\RR)$. It further induces a morphism of $\Gm$-modules
$\Prj{2-r}:\dsv{v,2-r}\om \rightarrow \V{v,2-r}\om(\proj\RR)$ and an
injective map from $\dsv{2-r}\om[\xi_1,\ldots,\xi_n]$ into
$\V{2-r}\om \left( \proj\RR\setminus \{\xi_1,\ldots,\xi_n\}\right)$.


\subsubsection{The sheaf of harmonic boundary germs } We recall that
$\sharm_r$ denotes the sheaf of $r$-harmonic functions on
$\uhp$.\il{bg}{boundary germ}

\begin{defn}\label{hbg} For open $I\subset\proj\RR$\ir{bdg}{\bg_r}
\be\label{bdg} \bg_r(I)\;:=\; \indlim \sharm_r(U\cap\uhp)\,,\ee
where $U$ runs over the open neighbourhoods of $I$ in $\proj\CC$. The
induced sheaf $\bg_r$ on $\proj\RR$ is called the sheaf of
\il{hbg}{$r$-harmonic boundary germs}\emph{$r$-harmonic boundary
germs}.

 The operator $|_r g$ with $g\in \SL_2(\RR)$ induces linear bijections
$\bg_r(I) \rightarrow \bg_r(g^{-1}I)$ for open $I\subset\proj\RR$.
\end{defn}
We identify $\sharm_r(\uhp)$ with its image in~$\bg_r(\proj\RR)$.

\subsubsection{The sheaf of analytic boundary germs}We now turn to the
boundary germs that are most useful for the purpose of this paper.
\begin{defn}\label{sharmbdef}
Let $r\in \CC$.
\begin{enumerate}
\item[i)] Consider the real-analytic function $f_r$ on
$\uhp\setminus\{i\}$ given by\ir{frdef}{f_r}
\be\label{frdef}
f_r(z) \;:=\; \frac{2i}{z-i}\, \Bigl(\frac{\bar z-i}{\bar
z-z}\Bigr)^{r-1} \,. \ee
\item[ii)] For open $U\subset \proj\CC$ we define \ir{sharmb}{\sharmb
r}
\be\label{sharmb}
\sharmb r(U):=\bigl\{F\in \sharm_r(U\cap\uhp) \;:\; F/f_r \text{ has a
real-analytic continuation to }U\bigr\}\,.\ee
\item[iii)] For open sets $I\subset\proj\RR$ we define\ir{Wdef}{\W r
\om}
\be\label{Wdef} \W r \om(I)\;:=\; \indlim \sharmb r(U)\,,\ee
where $U$ runs over the open neighbourhoods of $I$ in~$\proj\CC$. This
defines a subsheaf $\W r \om$ of~$\bg_r$, called the sheaf of
\il{abg}{analytic boundary germs}\emph{analytic boundary germs}.
\end{enumerate}
\end{defn}

\begin{rmk}The function $f_r$ is analogous to the function $z\mapsto
\bigl( \frac{4y}{|z+i|^2}\bigr)^s$ (or to $w\mapsto
(1-\nobreak |w|^2\bigr)^s$ in the disk model) in \cite{BLZ13},
Definition 5.2 and the examples after Equation~(5.9).

The function $f_r$ can be written as $c q(z)^{1-r}\, \frac{z+i}{z-i}\,
(z+\nobreak i)^{2-r}$, where
\il{q}{$q(\cdot)$}$q(z)= \frac
y{|z+i|^2}$ is a real-valued real-analytic function on
$\proj\CC\setminus \{-i\}$ with $\proj\RR$ as its zero set, and $c$
is some factor in~$\CC^\ast$.

The motivation for our choice of $f_r$ is that it is the right choice
to make Proposition~\ref{prop-restr} below work.
\end{rmk}

\begin{lem}\label{lem-WactG}
Let $r\in \CC$. For each $g\in \SL_2(\RR)$ the operator $|_r g$
induces a linear bijection $\W r \om(I)
\rightarrow \W r \om(g^{-1}I)$.
\end{lem}

\begin{proof}Let $g=\matc abcd\in\SL_2(\RR)$. We have
\be\label{frquot}
 (cz+d)^{-r} \frac{f_r(gz)}{f_r(z)} \=(a-ic)^{r-2}\;
\frac{z-i}{z-g^{-1}i}\;\Bigl( \frac{\bar z-i}{\bar
z-g^{-1}i}\Bigr)^{r-1}\,.
\ee
To see this up to a factor depending on the choice of the arguments is
just a computation. Both sides of the equality are real-analytic in
$z\in\uhp \setminus\{i,g^{-1}i\}$ and in $g$ in the dense open set
$G_0\subset\SL_2(\RR)$, defined in~\eqref{G0}. The equality holds for
$g=\mathrm{I}$, hence for $g\in G_0$. Elements in
$\SL_2(\RR)\setminus G_0$ are approached with $c\downarrow 0$, and
$a$ and $d$ tending to negative values. The argument conventions
\S\ref{sect-actions} and Proposition~\ref{prop-opprj} are such that
both $(cz+\nobreak d)^{-r}$ and $(a-\nobreak ic)^{r-2}$ are
continuous under this approach.

Suppose that $F\in \sharmb r(U)$ represents a section in $\W r
\om(I)$, with $I=U\cap\proj\RR$ and $A=F/f_r$ real-analytic on
$U\subset \proj\RR$. Then we have
\[F|_r g ( z) = f_r(z)\;(cz+d)^{-r}\, \frac{f_r(gz)}{f_r(z)}\,
A(gz)\,,\]
which is of the form $f_r$ times a real-analytic function 
near~$g^{-1}I$.
\end{proof}

\rmrk{Examples}
\itmi Consider the $r$-harmonic function
\be \label{yex}
F(z)\=y^{1-r}\ee
on $\uhp$. Since
\[ A(z) = \frac{F(z)}{f_r(z)} \= \frac{ z-i}{2i}\, \Bigl( \frac{\bar
z-i}{-2i} \Bigr)^{1-r}\]
extends as a real-analytic function to $U=\CC\setminus \{i\}$, the
function $F$ is in $\sharmb r \left( \CC\setminus \{i\}\right)$. It
is not in $\sharmb r\left(\proj\CC\setminus \{i\}\right)$, since $A$
is not given by a convergent power series in $1/z$ and $1/\bar z$ on
a neighbourhood of $\infty$ in~$\proj\CC$. So the function $F$
represents an element of $\W r \om(\RR)$.

\itm For $r \in \CC \smallsetminus \ZZ_{\geq 2}$ and $\mu \in
\ZZ_{\geq 0}$
 \ir{mex}{\M{r,\mu}}
\be \label{mex} \M{r,\mu}(z) \;:=\;f_r(z)\, \Bigl(
\frac{z-i}{z+i}\Bigr)^{\mu+1}\, \hypg21\Bigl(
1+\mu,1-r;2-r;\frac{4y}{|z+i|^2} \Bigr)\,. \ee
At this moment we only state that $\M{r,\mu}$ is $r$-harmonic on
$\uhp\setminus\{i\}$, and postpone giving arguments for this
statement till~\S\ref{sect-pol}. The function $z\mapsto \frac
{4y}{|z+i|^2} = 1-\Bigl|\frac{z-i}{z+i} \Big |^2$ is real-analytic on
$\proj\CC\setminus \{-i\}$, with value $0$ on~$\proj\RR$. Since the
hypergeometric function is holomorphic on the unit disk in~$\CC$,
this implies that $\Bigl(
\frac{z-i}{z+i}\Bigr)^{-\mu-1}\,\M{r,\mu}(z)$ is real-analytic on
$\proj\CC\setminus\{i,-i\}$, and hence $\M{r,\mu}$ is in $\sharmb
r\left(\proj\CC\setminus\{i,-i\}\right)$, and represents an element
of $\W r \om(\proj\RR)$.

\itm Let $\re r>0$. Then $\eta(z)^{2r} = e^{2r\log\eta(z)}$ is a cusp
form of weight~$r$ for the modular group $\Gmod$. The
function\ir{hmodint}{ $powers of the Dedekind eta-function$ }
\be\label{hmodint} \Ph(z) \= \int_0^{i\infty} \eta^{2r}(\tau)\,
\frac{2i}{z-\tau}\,\Bigl(\frac{\bar z-\tau}{\bar
z-z}\Bigr)^{r-1}\,d\tau\,,\ee
defines an $r$-harmonic function on $\uhp\setminus i(0,\infty)$, if we
take the path of integration along the geodesic from $0$ to~$\infty$.
(To check the harmonicity one may apply $\xi_r$; this gives a
holomorphic function.)
Deforming the path of integration leads to other domains. Such a
change in the function does not change the $r$-harmonic boundary germ
in $\W r \om(\RR)$ it represents.


\subsection{Relation between the sheaves of harmonic boundary and
analytic boundary germs} The sheaf $\W r \om$ is related to the
simpler sheaf $\V{2-r}\om$ by the important \it restriction morphism
\rm that we will define now.
\begin{prop}\label{prop-restr}
 Let $r\in \CC$. There is a unique morphism of sheaves
\il{restr}{$\rsp_r$}$\rsp_r : \W r \om \rightarrow \V{2-r}\om$ with
the following property: If $f\in \W r \om(I)$ for an open set
$I\subset\proj\RR$ is represented by $F\in \sharmb r(U)$, and $\rsp_r
f\in \V {2-r}\om(I)$ is represented by $\ph \in
\hol(U_0)$ for some open neighbourhood $U_0$ of $I$ in~$\proj\CC$,
then the real-analytic function $F/f_r$ on $U$ is related to $\ph$ by
\[ \bigl( F/f_r\bigr)(t) \= \ph(t)
\qquad \text{ for }t\in I\,.\]

This is called the
\il{rm}{restriction morphism}\emph{restriction morphism} and is
compatible with the actions of $\SL_2(\RR)$:
\be\label{res-g} \rsp_r(f|_r g) = (\rsp_r f)|^\prj_{2-r}
g\qquad\text{for }f\in \W r \om(I)\text{ and }g\in \SL_2(\RR)\,. \ee
\end{prop}

\begin{proof} Let $F\in \sharmb r(U)$ for some neighbourhood $U$ of $I$
in~$\proj\CC$. Then $A:=F/f_r$ on $U\cap\uhp$ extends as a
real-analytic function to $U$. If we replace $F$ by another
representative $F_1\in \sharmb r
(U_1)$ of the same element of $\W r \om(I)$, then $F_1$ and $F$ have
the same restriction to $U_2\cap \uhp$ for a connected neighbourhood
$U_2\subset U\cap U_1$ of $I $ in~$\proj\CC$. Since $U_2$ is
connected, the functions $A$ and $A_1$ extend uniquely to~$U_2$, and
hence to $I\subset U_2$. We thus obtain a well-defined function on
$I$ which has further a holomorphic extension to some neighbourhood
$U_0$ of $I$ in~$\proj\CC$ since it is real-analytic on~$I$. Hence it
represents an element of $\V {2-r}\om(I)$ that is uniquely determined
by the element $f\in \W r \om(I)$ represented by~$F$.

This defines $\rsp_r : \W r \om(I)\rightarrow \V {2-r}\om(I)$. We have
compatibility with the restriction maps associated to $I_1\subset I$,
and hence obtain a morphism of sheaves.

Let $g = \matc abcd\in \SL_2(\RR)$. {}From relation \eqref{frquot} we
see that the map $\W r \om(I)
\rightarrow \W r \om(g^{-1}I)$ determined by $F\mapsto F|_r g$ sends
 $F/f_r$ to the real-analytic function
\begin{align*} \Bigl( \frac{F|_r g}{f_r}\Bigr)(z) &\=
\frac{(cz+d)^{-r}\, F(gz)}{f_r(z)} \=
(cz+d)^{-r}\,\frac{f_r(gz)}{f_r(z)}\,\bigl(F/f_r)(gz)\\
&\= (a-ic)^{r-2}\, \frac{z-i}{z-g^{-1}i}\,\Bigl( \frac{\bar z-i}{\bar
z-g^{-1} i} \Bigr)^{1-r} \, (F/f_r)(gz)
\,.\end{align*}
For $z=t\in I$, this equals $((F/f_r)|_{2-r}^\prj g)(t)$ by
\eqref{prjact}. Thus, we obtain~\eqref{res-g}.
\end{proof}

\rmrk{Illustration}Let $I$ be an interval in $\proj\RR$, and let $U$
be an open neighbourhood of $I$ in $\proj\CC$. For representatives $F$
of $f\in \W r \om(I)$ a representative $\ph$ of the image $\rsp_r
f\in \V{2-r}\om$ is obtained by a sequence of extensions and
restrictions. See Figure~\ref{fig-ext-res}.
\begin{figure}[ht]
{\small\[\setlength\unitlength{.7cm}\begin{array}{ccccccc}
F & \stackrel{\mathrm {ext}}\rightarrow& F/f_r & \stackrel{\mathrm
{res}}\rightarrow&
(F/f_r)|_I = \ph|_I& \stackrel{\mathrm {ext}}\rightarrow & \ph\\
\begin{picture}(2.5,2.5)(-1.25,-1.25)
\put(-1.25,0){\line(1,0){2.5}}
\put(0,0){\oval(2,2)[b]}
\thicklines
\put(0,0){\oval(2,2)[t]}
\put(-1,0){\line(1,0){2}}
\end{picture}&&
\begin{picture}(2.5,2.5)(-1.25,-1.25)
\put(-1.25,0){\line(1,0){2.5}}
\put(.3,.3){\vector(0,-1){.6}}
\put(-.3,.3){\vector(0,-1){.6}}
\thicklines
\put(0,0){\oval(2,2)[b]}
\put(0,0){\oval(2,2)[t]}
\end{picture}&&
\begin{picture}(2.5,2.5)(-1.25,-1.25)
\put(-1.25,0){\line(1,0){2.5}}
\put(0,0){\oval(2,2)[b]}
\put(0,0){\oval(2,2)[t]}
\put(-.3,.6){\vector(0,-1){.4}}
\put(.3,.6){\vector(0,-1){.4}}
\put(-.3,-.6){\vector(0,1){.4}}
\put(.3,-.6){\vector(0,11){.4}}
\thicklines
\put(-1,0){\line(1,0){2}}
\end{picture}&&
\begin{picture}(2.5,2.5)(-1.25,-1.25)
\put(-1.25,0){\line(1,0){2.5}}
\put(0,0){\vector(0,1){.4}}
\put(0,0){\vector(0,-1){.4}}
\thicklines
\put(-1,0){\line(1,1){1}}
\put(-1,0){\line(1,-1){1}}
\put(1,0){\line(-1,1){1}}
\put(1,0){\line(-1,-1){1}}
\end{picture}\\
\text{$r$-harmonic}&& \text{real-analytic}&& \text{real-analytic}&&
\makebox[0cm][c]{holomorphic on}\\
\text{on $U\cap\uhp$}&& \text{on $U$}&& \text{on $I$}&&
\makebox[0cm][c]{some neighbourhood}
\end{array}\]}
\caption{$\rsp_r$ as a sequence of extensions and
restrictions.}\label{fig-ext-res}
\end{figure}

\rmrk{Examples}
\itmi The restriction of $F(z) = y^{1-r}$ in~\eqref{yex} is \[(\rsp_r
F)(t)= \frac{t-i}{2i}\Bigl(\frac{t-i}{-2i}\Bigr)^{1-r} =-\Bigl(
\frac{t-i}{-2i}\Bigr)^{ 2-r}\,.\]

\itm For $\M{r,\mu}$ in~\eqref{mex} we use that $\frac{4 y
}{|z+i|^2}=0$ on $\proj\RR$ and that $\hypg21(\cdot,\cdot;\cdot;0)=1$
to obtain
\be \label{rspM} (\rsp _r \M {r,\mu})(t) =
\Bigl(\frac{t-i}{t+i}\Bigr)^{\mu+1}\,.
\ee


\subsection{Kernel function for the map from automorphic forms to
boundary germ
cohomology}\label{sect-Kr}\begin{prop}\label{prop-Kr}For $r\in \CC$
let $K_r$ be the \il{Krkf}{kernel function}function on
$\left(\uhp\times\uhp\right)\setminus(\text{diagonal})$ given
by:\ir{Kq}{K_r(\cdot;\cdot)}
\be\label{Kq}
K_r(z;\tau) \;:=\; \frac{2i}{z-\tau}\,\Bigl( \frac{\bar z-\tau}{\bar
z-z}\Bigr)^{r-1}\,. \ee

For each $z\in \uhp$ the function $K_r(z;\cdot)$ is holomorphic on
$\uhp\setminus \{z\}$, and for each $\tau\in \uhp$ the function
$K_r(\cdot;\tau)$ is $r$-harmonic on $\uhp\setminus\{ \tau\}$ and
represent an element of $\W r \om(\proj\RR)$ with restriction
\be \label{res-Kr}
\bigl(\rsp_r K_r(\cdot;\tau)\bigr)(t) \= \Bigl(
\frac{\tau-t}{i-t}\Bigr)^{r-2}\,. \ee
For each $g\in \SL_2(\RR)$ it satisfies
\be\label{Kq-inv} K_r(\cdot;\cdot)\, |_r g \otimes |_{2-r} g \= K_r\,.
\ee
\end{prop}

\rmrke In \eqref{Kq-inv} we use $K_r(\cdot;\cdot)|_r g \otimes
|_{2-r}g (z,\tau)=
(cz+d)^{-r}\,(c\tau+d)^{r-2}\, K_r(gz;g\tau)$ for $g=\matc abcd$.

\begin{proof}The shadow operator, defined in \eqref{shad}, gives
\be \label{shad-Kr}
(\shad_r K_r(\cdot;\tau)\bigr)(z) \= (\bar r-1) \, \Bigl( \frac{z-\bar
\tau}{2i} \Bigr)^{\bar r-2}\,. \ee
The result is holomorphic in $z$, hence $z\mapsto K_r(z;\tau)$ is
$r$-harmonic.

The quotient
\be\label{Kr/fr}
K_r(z;\tau)/f_r(z) \= \frac{i-z}{\tau-z}\, \Bigl( \frac{\tau-\bar
z}{i-\bar z}\Bigr)^{r-1} \,. \ee
extends real-analytically in $z$ across $\proj\RR$, \ to
$U_\tau=\proj\CC \setminus\left( \{\tau\} \cup p \right)$, where $p$
is a path in $\lhp$ from $-i$ to $\bar \tau$. So $K_r(\cdot;\tau)\in
\sharmb r(U_\tau)$ represents an analytic boundary germ on $\proj\RR$.
On $\proj\RR$ the values of $z$ and $\bar z$ coincide, and we find
the restriction in~\eqref{res-Kr}.

For the equivariance in \eqref{Kq-inv} we check by a computation
similar to the computation in the proof of Lemma~\ref{lem-WactG} that
\[ (cz+d)^{-r}\,(c\tau+d)^{r-2}\, K_r(gz;g\tau) \=
K_r(z;\tau)\,.\qedhere\]
\end{proof}

\rmrke The restriction $\rsp_r K_r(\cdot;\tau)$ in \eqref{res-Kr}
gives a function in $\Prj{2-r}\dsv{2-r}\om$. Since it is convenient
to work with $\dsv{2-r}\om$ itself, we introduce the following
operator:

\begin{defn}
We set \ir{rsl}{\rs_r}
\be\label{rsl} \rs_r \;:=\; \Prj{2-r}^{-1}\,\rsp_r \ee
\end{defn}
So $(\rs_r f)(t)\=(i-\nobreak t)^{r-2}\,
(\rsp_r f)(t)$, and we find
\be \bigl(\rs_r K_r(\cdot;\tau)\bigr)(t) \= (\tau-t)^{r-2}\,.
\ee
This shows that the kernel $K_r(\cdot;\cdot)$ is analogous to the
 kernel function $(z,t)\mapsto
(z-\nobreak t)^{r-2}$ in the Eichler integral. For fixed $\tau\in
\uhp$ the representative
\[ z\mapsto K_r(z;\tau)\]
of an element of $\W r \om(\proj\RR)$ is sent by the restriction map
to the representative
\[ t\mapsto (z-t)^{r-2}\]
of an element of $\dsv{2-r}\om$.

\begin{defn}\label{def-cz0}Let $F\in A_r(\Gm,v)$. We put for $z_0\in
\uhp$ and $\gm\in \Gm$:\ir{cz0-def}{c^{z_0}_F}
\be\label{cz0-def} c^{z_0}_{F,\gm}(z) \;:=\; \int_{\gm^{-1}z_0}^{z_0}
K_r(z;\tau)\, F(\tau)\, d\tau\,. \ee
\end{defn}

\begin{prop}\label{prop-bcoh}
 Let $r\in \CC$, $z_0\in \uhp$, and $F\in A_r(\Gm,v)$.
\begin{enumerate}
\item[i)] The map $\gm\mapsto c^{z_0}_{F,\gm}$ defines a cocycle
$c^{z_0}_F \in Z^1\bigl( \Gm;\W r \om(\proj\RR)\bigr)$.
\item[ii)] The cohomology class \il{bcoh}{$\bcoh r \om$}$\bcoh r\om F
:= \bigl[c^{z_0}_F \bigr]$ in $H^1\bigl(
\Gm;\W{v,r}\om(\proj\RR)\bigr)$ does not depend on the base
point~$z_0\in \uhp$.
\item[iii)] With the natural map $ H^1(\Gm; \dsv{v,2-r} \om)
\rightarrow H^1\bigl(\Gm;\V {2-r}\om
(\proj\RR)\bigr)$ corresponding to $\Prj{2-r}:\dsv{v,2-r}\om
\rightarrow \V{v,2-r}\om$, the following diagram commutes:
\bad \xymatrix{ A_r(\Gm,v) \ar[r]^(.45){\coh r \om}
\ar[drr]^(.65){\bcoh r \om}
& H^1(\Gm; \dsv{v,2-r} \om) \ar[r]^(.45){ \Prj{2-r}}
& H^1\bigl(\Gm;\V {v,2-r}\om (\proj\RR)\bigr)
\\
&& H^1\bigl(\Gm;\W {v,r} \om(\proj\RR)\bigr)\ar[u]_{\rsp_r} } \ead
\end{enumerate}
\end{prop}

\begin{proof}Proposition~\ref{prop-Kr} shows that the differential
form $K_r(z;\tau)\, F(\tau)\, d\tau$ has properties analogous to
those of $\om_r(F;t,\tau)$ in \S\ref{sect-af-coh}. The proof of Parts
i)
and~ii) goes along the same lines as the proof of
Proposition~\ref{prop-cohrom}. For Part~iii) use~\eqref{res-Kr}.
\end{proof}

Thus, we see that the map $\coh r \om$ to cohomology in
Theorem~\ref{THMac} is connected to the map $\bcoh r \om$ to boundary
germ cohomology by the restriction map $\rsp r$. However, in
Theorem~\ref{THMac} the basic module is $\dsv{v,2-r}\om$ and not the
larger module $\V{2-r}\om(\proj\RR)$. We need to study the boundary
germs more closely, in order to identify inside
$\W{v,r}\om(\proj\RR)$ a smaller module that can play the role
of~$\dsv{v,2-r}\om$.

\subsection{Local study of the sheaf of analytic boundary germs}

\subsubsection{Positive integral weights}At many places in this
section positive integral weights require separate treatment. A
reader wishing to avoid these complications may want to concentrate
on the general case of weights in $\CC\setminus \ZZ_{\geq 1}$.

The next definition will turn out to be relevant for weights in
$\ZZ_{\geq 1}$ only.

\begin{defn}\label{defn-sharmh}For $U$ open in $\proj\CC$
let\ir{sharh}{\sharmh r(U)}
\badl{sharh}\sharmh r(U)&\;:=\; \bigl\{ F\in \hol(U\cap\uhp)\cap
\sharmb r(U)\;:\;\\
&\quad F \text{ has a holomorphic extension to } U \bigr\}\,.\eadl
\end{defn}

\begin{lem}\label{lem-hW}Let $r\in \ZZ_{\geq 1}$. For open sets
$U\subset \proj\RR$ such that $U\cap\proj\RR\neq\emptyset$ and
$-i\not\in U$ the restriction to $U\cap\uhp$ of $F\in \hol(U)$ is in
$\sharmh r(U)$ if one of the following conditions is satisfied:
\begin{enumerate}
\item[a)] $U\subset \CC$,
\item[b)] $\infty\in U$ and $F$ has at $\infty$ a zero of order at
least~$r$.
\end{enumerate}
\end{lem}

\begin{proof}$F$ is holomorphic on $U\cap\uhp$, hence $r$-harmonic on
$U\cap\uhp$. For $z\in U\cap\uhp$:
\begin{align}
F(z)/f_r(z) &\= \frac1{2i}\, F(z)\, (z-i)\,(\bar z-i)^{1-r}\,(\bar
z-z)^{r-1} \label{hw1}
\\
&\= \frac1{2i}\,\bigl( z^r F(z) \bigr)\,(1-i/z)\, (1-i/\bar z)^{1-r}\,
(1/2-1/\bar z)^{r-1}\,. \label{hw2}
\end{align}
Equality~\eqref{hw1} shows that $F/f_r$ is real-analytic on $U
\setminus \{\infty,-i\}=U\setminus\{\infty\}$. If $\infty\in U$ then
\eqref{hw2} shows that it is also real-analytic on some neighbourhood
of~$\infty$.
\end{proof}

\subsubsection{Local structure} We return to the sheaves $\V {2-r}\om$
and $\W r \om$, in Definitions \ref{Vom} and~\ref{Wdef}.

The sections of $\V{r-2}\om$ are holomorphic on neighbourhoods of open
sets $I\subset \proj\RR$, and are locally at $x\in \RR$ given by a
power series expansion in $z-x$ converging on some open disk with
center~$x$. At $\infty$ we have a power series in $z^{-1}$. The
real-analytic functions $A={F/f_r}$ corresponding to representatives
of sections of $\W r \om$ are also given by a power series near $x\in
I$, now in two variables, $z-x$ and $\bar z-x$, which also
converges on a disk around $x$. At $\infty$ we have a power series
expansion in $1/z$ and $1/\bar z$.

With the operators $|_r g$ for $g\in \SL_2(\RR)$ we can construct
isomorphisms between the stalks of $\W r \om$. So for a local study
it suffices to work with a disk around~$0$. A problem is that the
points $i$ and $-i$ play a special role in the function $f_r$. Hence
it is better not to use arbitrary elements of $\SL_2(\RR)$ to
transport points of $\proj\RR$ to~$0$, but to use
\il{k(.)}{$k(\th)=\matr{\cos\th}{\sin\th}{-\sin\th}{\cos\th}$}$k(\th)
=\matr{\cos\th}{\sin\th}{-\sin\th}{\cos\th}\in \SO(2)$.

We denote disks around $0$ by\ir{Dp}{D_p}
\be\label{Dp} D_p \= \bigl\{ z\in \CC\;:\; |z|<p\bigr\}\,, \ee
where we take $p\in (0,1)$ to have $\pm i \not\in D_p$. All points of
$\proj\RR$ are uniquely of the form $k(\th)\,0=\tan\th$ with $\th \in
\RR\bmod \pi\ZZ$. All $k\, D_p\subset\proj\CC$ with $k\in \SO(2)$
do not contain $\pm i$; in general they are Euclidean disks in~$\CC$.
The sets $k\, D_p$ are invariant under complex conjugation.

\begin{prop}\label{prop-tPloc} Suppose that the set $U\subset\proj\CC$
 is of the form $U=k D_p$ with $k\in \SO(2)$, $0<p<1$.
\begin{enumerate}
\item[i)] {\rm Restriction. } Let $r\in \CC$. If $F\in \sharm_r^b(U)$
then the restriction $\rsp_r F$ extends as a holomorphic function
on~$U$.
\item[ii)]
\begin{enumerate}
\item[a)] If $r\in \CC\setminus \ZZ_{\geq 1}$ then $\sharmh r(U)
=\{0\}$.
\item[b)] If $ r=1$, then $\sharmb 1(U)= \sharmh 1(U)$.
\item[c)] If $r\in \ZZ_{\geq 1}$, then $\sharmh r(U) \subset \sharmb
 r(U)$.
\end{enumerate}
\item[iii)]\begin{enumerate}
\item[a)]If $r\in \CC\setminus \ZZ_{\geq 2}$, then the restriction map
\il{res1}{$\rsp_r$}$\rsp_r:\sharm_r^b(U)
\rightarrow \hol(U)$ is bijective.
\item[b)] If $r\in \ZZ_{\geq 2}$ then the following sequence is exact:
\[ 0 \rightarrow \sharmh r(U) \rightarrow \sharmb r ( U) \stackrel
{\rsp_r} \rightarrow \Prj{2-r}\dsv{2-r}\pol\rightarrow 0\,, \]
where the last space has to be interpreted as the space of functions
on $U$ that extend to $\proj\CC\setminus \{i\}$ as elements of the
projective model of $\dsv{2-r}\pol$.
\end{enumerate}
\item[iv)] {\rm Shadow operator. } If $F\in \sharm_r^b(U)$, then the
holomorphic function $ \xi_r F \in \hol(U\cap\uhp)$, defined in
\eqref{shad}, extends holomorphically to $U$ and satisfies
\be\label{ah-rc-coupl} (\shad_r F) (z) \= (\bar r-1)\, \Bigl(
\frac{z+i}{2i}\Bigr)^{\bar r-2}\, \overline{\rsp_r F(\bar z)}
\qquad(z\in U\cap\uhp)\,. \ee
\end{enumerate}
\end{prop}

\rmrks \itmi Part~i) shows that for small disks $U$ the restriction of
an element of $\sharmb r(U)$ is represented by a function defined on
the whole disk~$U$, not just on some unspecified neighbourhood of
$U\cap\proj\RR$.

\itm The shadow operator and the restriction morphism were defined in
different ways. Part~iv) shows that they are related.

\begin{proof}We start with proving the statements for $U=D_p$, and
will at the end derive the general case.

\rmrk{ Part i), $r\in \CC\setminus \ZZ_{\geq 1}$} First we consider
the case when $r \in \CC \smallsetminus \ZZ_{\geq 1}$. We'll show
that for any $F\in \sharmb r(D_p)
 \subset \sharm_r(D_p\cap\nobreak\uhp)$
(a representative of) $\rsp_r F$ is holomorphically continuable to
$D_p.$

We note that we have $F(z)=y^{1-r}\,B(z)$ where
$$B(z)=\frac{2i}{z-i}\allowbreak\, \Bigl( \frac{\bar z-i}{-2i}
\Bigr)^{r-1}\allowbreak\, \frac{F(z)}{f_r(z)}.$$
Since $F\in \sharm_r^b(D_p)$, $B$ is real-analytic on $D_p$ and there
are coefficients $b_{n,m}$ such that
\be\label{Bexp} B(z) \= \sum_{n,m\geq 0} b_{n,m} z^n \bar z^m \ee
converges absolutely on a disk $D_{p_1}$ where $0<p_1\leq p$. We
define
\[b(z)\;:=\; 2i y^r \partial_{\bar z}\bigl(y^{1-r} B(z)\bigr)\= 2i\, y
\, B_{\bar z}(z)+(r-1)\, B(z)\quad\text{ on }D_p\,.\]
The condition of $r$-harmonicity on $F$ is equivalent to the
antiholomorphicity of $b.$ On $D_{p_1}$ we have, by~\eqref{Bexp}, the
expansion
\bad b(z) &\= \sum_{m\geq 0} (r-1-m)\, b_{0,m}\, \bar z^m
\\
&\qquad\hbox{}
+ \sum_{m\geq 0,\,n\geq 1} \Bigl( (m+1)\, b_{n-1,m+1} + (r-1-m)\,
b_{n,m} \Bigr)\, z^n \bar z^m\,. \ead
The antiholomorphy implies that for each $n\geq 1$ the coefficient of
$z^n\bar z^m$ has to vanish. If $r\not\in \ZZ_{\geq 1}$ this gives
the relation
\be\label{bnm} b_{n,m} \= \frac{(1-r)_m}{m!}\, b_{n+m,0}
\ee
(with the \il{Poch}{Pochammer symbol}Pochhammer symbol
\il{Ps}{$(a)_m$}$(a)_m = \prod_{j=0}^{m-1}
(a+j)$), and
\be\label{aexp} b(z) \= -\sum_{m\geq 0}\frac{(1-r)_{m+1}}{m!}\,
b_{m,0}\, \bar z^m\qquad(z\in D_{p_1})\,. \ee
This is the power series of an antiholomorphic function on~$D_p$,
hence it converges absolutely on $D_p$.

Now, if $\phi$ is a representative of $\rsp_r F$, then, for $t\in
(-p,p)$,
\begin{eqnarray}\label{ph1exp} \ph(t) \=
-\Bigl(\frac{t-i}{-2i}\Bigr)^{2-r} B(t)&=&
-\Bigl(\frac{t-i}{-2i}\Bigr)^{2-r} \sum_{n,m\geq 0} \frac{(1-r)_m}{m!}
\, b_{n+m,0}\, t^{n+m} \nonumber \\
&=& -\Bigl(\frac{t-i}{-2i}\Bigr)^{2-r} \sum_{\ell\geq 0}
\frac{(2-r)_\ell}{\ell!}\, b_{\ell,0} t^\ell \,. \end{eqnarray}
Here we use the well-known formula $\sum_{m=0}^\ell \frac{(1-r)_m}{m!}
= \frac{(2-r)_\ell}{\ell!}$. A comparison of the absolutely
convergent series \eqref{aexp} and \eqref{ph1exp} shows that
\eqref{ph1exp} also converges absolutely on $D_p$. This implies that
the restriction gives a map $\sharm_r^b(D_p)
\rightarrow \hol(D_p)$.

\rmrk{Part i), $r\in \ZZ_{\geq 1}$} Secondly we show that in the case
$r\in \ZZ_{\geq 1}$ the same conclusion holds. In this case
 \eqref{bnm} is valid for $m \leq r-1$. For $m\geq r$ we get
 successively $b_{n,m}=0$. Since the corresponding Pochhammer symbols
vanish, expansion~\eqref{aexp} stays valid. We get the same estimate
for $b_{n+m,0}$ and arrive at
\be \label{ph1expi}
\ph(t) \= -(\frac{t-i}{-2i})^{2-r} \begin{cases}
\sum_{\ell= 0}^{r-2} (1-r)_\ell\; (\ell!)^{-1}\, b_{\ell,0}\,
t^\ell&\text{ if }r\geq 2\,,\\
\sum_{n=0}^\infty b_{n,0}\,(n!)^{-1}\, t^n&\text{ if }r=1\,,
\end{cases}
\ee
on~$D_p$. This completes the proof of Part~i) for $U=D_p$.

\rmrk{ Part ii) a) and c)} Let $F\in \hol(U)$, and suppose that its
restriction to $U\cap\uhp$ is in $\sharm_r(U)$, then
\[ F(z)/f_r(z) \= \frac{z-i}{2i} \, \Bigl(\frac{\bar z-i}{-2i} \Bigr)
F(z)\, y^{r-1}\,.\]
If $r \in \CC \smallsetminus \ZZ_{\geq 1}$, then the presence of the
factor $y^{1-r}$ shows that this can be real-analytic near $0$ only
if $F=0$. This implies Part~a). If $r\in \ZZ_{\geq 2}$, all factors
are real-analytic. This implies Part~c).

\rmrk{Part ii) b)} If $r=1$, all $b_{n,m}$ with $m\geq 1$ vanish.
Hence $B(z)$ and $F(z)=y^{1-1}\, B(z)$ are holomorphic. This gives
Part~ii)b).

\rmrk{Part iii) a), surjectivity} In the case when $r \in \CC
\smallsetminus \ZZ_{\geq 2}$, take $\ph \in \hol(D_p)$, which is
represented by an absolutely convergent power series
\be \ph(t) \= \sum_{\ell\geq 0} a_\ell\, t^\ell\,.\ee
Hence $a_\ell= \oh(c^{-\ell})$ for all $c\in
(0,p)$. We put
\be\label{bnm1} b_{n,m} \= \frac{(1-r)_m\, (n+m)!}{m!\, (2-r)_{n+m}}\,
a_{n+m}\,, \ee
and define $B$ by \eqref{Bexp} with these coefficients~$b_{n,m}$. The
factor $\frac{(1-r)_m\, (n+m)!}{m!\,
(2-r)_{n+m}}$ has polynomial growth in $m$ and $n$. We arrive at
absolute convergence of the power series \eqref{Bexp} on $|z|<c$ for
all $c<p$. Hence $B$ is real-analytic on~$D_p$. The structure of the
$b_{n,m}$ shows that $b(z)=2i \, y B_{\bar z} +
(r-\nobreak 1)\, B$ is antiholomorphic, hence $F:=y^{1-r}\,B$ on
$D_p\cap \uhp$ is in $\sharm_r^b(D_p)$ and $\rsp_r F
(t) = - \Bigl( \frac{t-i}{-2i}\Bigr)^{r-2}\allowbreak\, \ph(t)$. This
shows the surjectivity if $r\not\in \ZZ_{\geq 2}$.

\rmrk{Part iii) b), surjectivity}In the case of $r\in \ZZ_{\geq 2}$,
Equation~\eqref{ph1expi} shows that the restriction map has the
 projective model of $\dsv{2-r}\pol$ as its image, since we can freely
choose the $b_{\ell,0}$ with $\ell\leq r-2$. This gives the
surjectivity in the exact sequence in Part~iii)b).

\rmrk{Part iii) a),b) injectivity} We suppose that $\rsp_r F=0$ for
$F\in \sharm_r^b(D_p)$. Then $\ph=0$, which by \eqref{ph1exp} implies
$F=0$ if $r \in \CC \smallsetminus \ZZ_{\geq 2}$. Thus we have the
injectivity in Part~iii)a), which completes the proof of~iii)a).

For $r\in \ZZ_{\geq 2}$ we have $b_{\ell,0}=0$ for $\ell \leq r-2$. So
\[ B(z) \= \sum_{\ell\geq r-1}b_{\ell,0}\, z^\ell \sum_{m=0}^{r-1}
\frac{(1-r)_m}{m!} (\bar z/z)^m \= \sum_{\ell\geq r-1}
b_{\ell,0}z^{\ell+1+1-r} (2iy)^{r-1} \,. \]
So the kernel consists of the functions $F$ on $D_p\cap \uhp$ with
expansion
\[ (2i)^{r-1} \sum_{\ell\geq 0} b_{\ell+r-1,0} \, z^\ell\,,\]
which are just the restrictions to $U\cap\uhp$ of the holomorphic
functions on~$D_p$. This completes also the proof of Part~iii)b).

\rmrk{Part iv)} We find by a direct computation for $z\in U\cap \uhp$
\[ (\shad_r F)(z) \= \frac4{\bar z+i}\, \Bigl(
\frac{z+i}{2i}\Bigr)^{\bar r-1}\,\biggl( \frac{1-\bar
r}{2i}\,\frac{\bar z+i}{z+i}\,\overline{A(z)} + y\, \overline{A_{\bar
z}(z)} \biggr)\,,\]
where $A(z)=F(z)/f_r(z)$, analytically continued. 
This shows that $\shad_r F$ extends as a real-analytic function
to~$U$. We know that it is holomorphic on $U\cap\uhp$, hence on~$U$.
It is determined by its values for $x\in U\cap\RR$:
\begin{align*} (\shad_r F)(x) &\= -2i \Bigl(
\frac{x+i}{2i}\Bigr)^{\bar r-2}\, \biggl( \frac{1-\bar r}{2i}\,
\overline{A(x)} + 0 \biggr)\\
&\= (\bar r-1) \, \Bigl( \frac{ x+i}{2i} \Bigr)^{\bar r-2}\,
\overline{A(x)}\,.
\end{align*}
On $U\cap\RR$ we have $A = \rsp_r F$, hence
\[ (\shad_r F)(x) \= (\bar r-1) \, \Bigl( \frac{x+i}{2i} \Bigr)^{\bar
r-2}\, \overline{(\rsp_r F)(x)}\,,\]
which extends to an equality of holomorphic functions on~$U$, which
is~\eqref{ah-rc-coupl}.
\smallskip

\rmrk{Shifted disks} To prove the proposition for $kD_p$ with general
$k=k(\th)$ ($\th\in \bigl( - \frac\pi2,\frac\pi 2\bigr]$) we note
that the bijective operator $|^\prj_{2-r}k : \hol(D_p) \rightarrow
\hol(k^{-1}D_p)$ preserves holomorphy. This together with the
bijective operator $|_r k : \sharmb r(D_p)\rightarrow \sharmb
r(k^{-1}D_p)$, and Relation~\eqref{res-g}
imply Parts i), ii) and~iii).

To prove (iv) for general $kD_p$, we first apply (iv) to $F|_r k \in
\sharmb r(D_p)$ to get, for $z \in D_p\cap\CC$:
\be\label{ah-rc-coupla} (\shad_r (F|_r k)) (z) \= (\bar r-1)\, \Bigl(
\frac{z+i}{2i}\Bigr)^{\bar r-2}\, \overline{\rsp_r (F|_r k)(\bar z)}.
\ee
Upon an application of \eqref{xi-g} and \eqref{res-g}, this becomes,
with $k=\matr a{-c}ca$:
\badl{ah-rc-couplb}
(\shad_r F) (kz) &\=(cz+a)^{2-\bar r}\, (\bar r-1)\, \Bigl(
\frac{z+i}{2i}\Bigr)^{\bar r-2}\, \overline{\rsp_r (F|_r k)(\bar
z)}\\
&\= (cz+a)^{2-\bar r}\, (\bar r-1)\, \Bigl( \frac{z+i}{2i}\Bigr)^{\bar
r-2}\, \overline{ (\rsp_r F)|^\prj_{2-r}k \; (z) }\\
&\= (cz+a)^{2-\bar r}\, (\bar r-1)\, \Bigl( \frac{z+i}{2i}\Bigr)^{\bar
r-2}\, (a+ic)^{\bar r-2}\\
&\qquad\hbox{} \cdot
 \overline{\Bigl(\frac{\bar z-i}{\bar
 z-k^{-1}i}\Bigr)^{2-r}}\,\overline{ (\rsp_r F)(k\bar z)}\\
&\= (\bar r-1)\, \Bigl( \frac{kz+i}{2i}\Bigr)^{\bar r-2}
\,\,\overline{ (\rsp_r F)(k\bar z)}\,. \eadl
We used that $k^{-1}i=i$, and that $(a+ic)\,(z+i)=(cz+a)(kz+i)$. Since
the conjugate of a holomorphic function on $\lhp$ is holomorphic on
$\uhp$, this proves the statement.
\end{proof}

Proposition~\ref{prop-tPloc} gives a rather precise description of the
local relation between the sheaves $\V{2-r}\om$ and $\W r \om$. The
next theorem ties this together to a global statement, which will
turn out to be crucial in Sections \ref{sect-hws} and~\ref{sect-caf}.

\begin{thm}\label{thm-VWiso}
\begin{enumerate}
\item[i)] If $r\in \CC\setminus \ZZ_{\geq 2}$ the morphism of sheaves
$\rsp_r : \W r \om \rightarrow\V{2-r}\om$ is an isomorphism.
\item[ii)] For $r\in \ZZ_{\geq 1}$ we define the subsheaf $\Whol r
\om$ of $\W r \om$ by\ir{Whol}{\Whol r \om}
\be\label{Whol}
\Whol r \om(I) \;:=\;
\indlim \sharmh r(U) \,, \ee
where $U$ runs over the open neighbourhoods in $\proj\CC$ of open sets
$I$ in~$\proj\RR$.
\begin{enumerate}
\item[a) ] If $r=1, \Whol 1 \om = \W 1 \om$ .
\item[b)] If $r\in \ZZ_{\geq2}$, the following sequence is exact:
\be\label{es-VWD} 0 \rightarrow \Whol r \om \rightarrow \W r \om
\stackrel{\rs_r}\rightarrow \dsv{2-r}\pol \rightarrow 0\,. \ee
The space $\dsv{2-r}\pol$ is interpreted as a constant sheaf
on~$\proj\RR$.
\end{enumerate}
\end{enumerate}
\end{thm}

\begin{proof}
Proposition~\ref{prop-tPloc} gives the corresponding statements on
sets $U$ near all points of $\proj\RR$, giving all statements in the
theorem on the level of stalks.
\end{proof}

\subsection{Related work}\label{sect-lit6}In \cite{BLZm} the analytic
boundary germs form the essential tool to prove the surjectivity of
the map from Maass forms of weight zero to cohomology considered
there. This gave the motivation to study these boundary germs 
for  themselves, in the paper~\cite{BLZ13}. In the
introduction of~\cite{BLZ13} (``Further remarks'', p.~111) it is
indicated that the boundary germs have been studied in the much wider
context of general symmetric spaces.

One finds the isomorphism analogous to the isomorphism in Part~i) of
Theorem~\ref{thm-VWiso} in \cite[\S5.2]{BLZ13}. There the isomorphism
is approached in two ways: by power series expansions and by
integrals. In the proof of Proposition~\ref{prop-tPloc} we have used
only power series.


\section{Polar harmonic functions}\label{sect-pol}
The subject of this section may seem slightly outside the line of
thought of the previous sections. It has its interest on itself, and
it provides more examples of $r$-harmonic function that do or do not
represent analytic boundary germs. The main reason to discuss it is
in the case $r\in \ZZ_{\geq 2}$. Though Theorem~\ref{thm-VWiso} leads
directly to spaces of analytic boundary germs isomorphic to the
spaces $\dsv{2-r}\om$ for $r\in \CC\setminus\ZZ_{\geq 2}$, the
situation is \ less clear for $r\in \ZZ_{\geq 2}$. With polar
harmonic functions we will arrive in \S\ref{sect-def-E} at a
satisfactory definition for all $r\in \CC$.

\subsection{Polar expansion}\label{sect-pol-exp}
The map \il{pc}{$w=\frac{z-i}{z+i}$}$z\mapsto w(z) :=\frac{z-i}{z+i}$
with inverse $w\mapsto z(w) :=i\frac{1+w}{1-w}$ gives a bijection
between the upper half-plane $\uhp$ and the unit disk in~$\CC$. We
write a continuous function $F$ on~$\uhp$ in the form $F(z)
= \Bigl( \frac{2i}{z+i}\Bigr)^r \,P\bigl(w(z)\bigr)$. This has the
advantage that the transformation $F \mapsto F|_r
\matr{\cos\th}{\sin\th}{-\sin\th}{\cos\th}$ with
$-\frac\pi2<\th<\frac\pi2$ corresponds to sending $P$ to the function
$w\mapsto e^{i r \th}\,\allowbreak P\bigl(e^{2i\th}w\bigr)$.

We put\ir{pol-term}{F(\mu;z)}
\badl{pol-term} F(\mu;z)& \;:=\; \Bigl(
\frac{2i}{z+i}\Bigr)^r\,\frac1\pi \int_{-\pi/2}^{\pi/2} e^{-2i\mu\th}
P(e^{2i\th}w)\, d\th\\
&\= \frac1\pi \int_{-\pi/2}^{\pi/2} e^{-i(2\mu+r)\th}\, \Bigl( F|_r
\matr{\cos\th}{\sin\th}{-\sin\th}{\cos\th} \Bigr)(z)\, d\th\qquad
(\mu\in \ZZ)\,. \eadl
In the first expression we see a coefficient of the Fourier expansion
of the function $\th \mapsto P(e^{2i\th}w)$. Thus we have a
convergent representation
\be F(z) \= \sum_{\mu\in \ZZ} F(\mu;z)\,,\ee
the \il{polexp}{polar expansion}\emph{polar expansion}. If we do not
work on the whole of $\uhp$, but on an annulus
$c_1<\Bigl|\frac{z-i}{z+i}\Bigr|<c_2$, we can proceed similarly.

We use this in particular for $r$-harmonic functions $F$. {}From the
second expression in \eqref{pol-term} we see that $F(\mu;\cdot)$ is
$r$-harmonic, since the operators $|_r g$ with $g\in \SL_2(\RR)$
preserve $r$-harmonicity and we can exchange the order of
differentiation and integration. The terms $F(\mu;\cdot)$
can be written in the form $ F(\mu,z) \= \Bigl( 2i/(z+\nobreak
i)\Bigr)^r\allowbreak \,\bigl(w/\bar w\bigr)^\mu\allowbreak\,
p_\mu\bigl( |w|^2\bigr)$, for some function $p_\mu$ on $[0,\infty)$.
With some computations one can obtain an ordinary differential
equation the $p_\mu$, which turns out to be related to the
hypergeometric differential equation, with a two-dimensional solution
space. This leads to the following $r$-harmonic functions, all
depending holomorphically on $r$ in a large subset
of~$\CC$.\ir{Prmu}{\P{r,\mu}}
\ir{Mrmu}{\M{r,\mu}}\ir{Hrmu}{\H{r,\mu}}
\begin{align}
\label{Prmu}
\P{r,\mu}(z) &\= \Bigl( \frac{2i}{z+i}\Bigr)^r \, w^\mu \= \Bigl(
\frac{2i}{z+i}\Bigr)^r \, \Bigl( \frac{z-i}{z+i}\Bigr)^\mu \,,
\quad\mu\in \ZZ\,,\\
\label{Mrmu}
\M{r,\mu}(z) &\= \begin{cases}
f_r(z)\, \Bigl( \frac{z-i}{z+i}\Bigr)^{\mu+1}\,
\hypg21\Bigl(1+\mu,1-r;2-r;\frac{4y}{|z+i|^2}\Bigr)
&\text{ if }\mu\geq 0\,,
\\
f_r(z) \;\frac{z-i}{z+i}\;\Bigl( \frac{\bar z+i}{\bar z-i}
\Bigr)^{-\mu}\, \hypg21\Bigl( 1-\mu-r,1;2-r;\frac{4y}{|z+i|^2}\Bigr)
&\text{ if }\mu\leq 0\,,
\end{cases}
\\
\label{Hrmu}
\H{r,\mu}(z) &\= f_r(z)\,\frac{z-i}{z+i}\, \Bigl( \frac{\bar z+i}{\bar
z-i} \Bigr)^{-\mu}\, \hypg21\Bigl( 1-\mu-r,1;1-\mu;
\Bigl|\frac{z-i}{z+i}\Bigr|^2\Bigr), \quad \mu\leq -1.
\end{align}
The function $\P{r,\mu}$ is holomorphic, hence $r$-harmonic. Checking
the $r$-harmonicity of $\M{r,\mu}$ and $\H{r,\mu}$ requires work, for
which there are several approaches:
\begin{enumerate}
\item[a)] Carry out the computation for the differential equation for
$p_\mu$, transform it to a hypergeometric differential equation, and
check that the hypergeometric functions in \eqref{Mrmu}
and~\eqref{Hrmu} are solutions.
\item[b)] Check by a direct computation (for instance with formula
manipulation software like Mathematica) that the shadow operator
sends the functions to the holomorphic functions indicated in
Table~\ref{tab-PMH}, thus establishing $r$-harmonicity.
\item[c)] Transform the problem to the universal covering group of
$\SL_2(\RR)$, and use the remarks in~\S\ref{app-pol-fcts} in the
Appendix.
\end{enumerate}

\begin{table}[ht]
{\small \renewcommand\arraystretch{1.6}
\[\begin{array}{|c|c|c|c|}\hline
f={}& \P{r,\mu}& \M{r,\mu} & \H{r,\mu} \\ \hline
f\in&\sharm_r(\uhp) \;\; (\mu \geq 0)& \sharm_r(\uhp\setminus\{i\})&
\\
& \sharm_r(\uhp\setminus \{i\})\;\;(\mu<0)&
& \sharm_r(\uhp)\;\;(\mu \leq -1) \\ \hline
\shad_r f & 0 & (\bar r-1)\Bigl( \frac{2i}{z+i}\Bigr)^{2-\bar r}
\Bigl( \frac{z-i}{z+i}\Bigr)^{-\mu-1}&
-\mu \Bigl( \frac{2i}{z+i}\Bigr)^{2-\bar r} \Bigl(
\frac{z-i}{z+i}\Bigr)^{-\mu-1}\\
\hline
f\text{ reprs.~elt.}&\text{if }r\in\ZZ_{\geq 1}& \text{if }r\in
\CC\smallsetminus \ZZ_{\geq2}\text{ or if}&
\\
\text{of } \W r \om(\proj\RR)&
&r\in \ZZ_{\geq 2}\text{ and }1-r\leq \mu \leq -1
 &
\\ \hline
\rsp_r f & 0\;\;(r\in \ZZ_{\geq 1})& \Bigl(
\frac{t-i}{t+i}\Bigr)^{\mu+1}&
\\ \hline
\end{array}
\]} \caption{Properties of polar $r$-harmonic functions.}
\label{tab-PMH}
\end{table}
Most of these facts follow directly from the formulas, and the
properties of the hypergeometric function. We note the following:
\begin{itemize}
\item The factor $\frac{4y}{|z+i|^2}$ is real-analytic on
$\proj\CC\setminus \{-i\}$ with zero set $\proj\RR$. It has values
 between $0$ and $1$ on $\uhp$, reaching the value~$1$ only at $z=i$.
Since the hypergeometric functions are holomorphic on $\CC\setminus
[1,\infty)$, the definition shows that $\M{r,\mu}\in \sharmb r \bigl(
\proj\CC\setminus\nobreak\{i,-i\}\bigr)$. To investigate the
behavior of $\M{r,\mu}$ at $i$ we note that its shadow has a
singularity at $z=i$ if $\mu\geq 1$, so $\M{r,\mu}$ cannot be
 real-analytic at~$i$ for $\mu\geq 0$.
\item The functions $\P{r,\mu}$ and $\M{r,\mu}$ are linearly
independent for $r \in \CC \smallsetminus \ZZ_{\geq 1}$.
\item The Kummer relation \cite[\S2.9, (33)]{EMOT} implies
\be\label{HMP} \H{r,\mu} \= \frac\mu{1-r}\M{r,\mu}
+\frac{|\mu|!}{(1-r)_{|\mu|}}\, \P{r,\mu} \quad (\mu\leq -1) \,.\ee
{}From the singularity of $\P{r,\mu}$ at $i$ and the fact that
$\H{r,\mu}\in \sharm_r(\uhp)$ we see that $\M{r,\mu}$ has a
singularity at $i$ for $\mu\leq
-1$ as well.
\item If $r_0\in \ZZ_{\geq 2}$ the meromorphic function $r\mapsto
\M{r,\mu}$ has in general a first order singularity at $r=r_0$ with a
non-zero multiple of $\P{r_0,\mu}$ as the residue. However, if
$1-r_0\leq \mu \leq
-1$ it turns out to be holomorphic at $r=r_0$. So under these
conditions $\M{r_0,\mu}$ is well defined.
\end{itemize}

\begin{prop}\label{prop-Ei}
$\sharm_r(\uhp)\,\cap\, \W r \om(\proj\RR) \= \{0\}$ for $r\in
\CC\setminus \ZZ_{\geq 1}$.
\end{prop}

\begin{proof}Let $F\in\sharm_r(\uhp)\,\cap\, \W r \om(\proj\RR)$. With
\eqref{frquot}  and the second line of \eqref{pol-term} we have
\[ F(\mu;z) \=\int_{-\pi/2}^{\pi/2}\, A(\th,z)\, d\th\,,\]
with a function $A$ that is real-analytic in $(\th,z)$ with
$\th$ in a neighbourhood of $[-\frac\pi2,\frac\pi2]$ in $\CC$, and $z$
in a neighbourhood of $\proj\RR$ in~$\proj\CC$. So all terms
$F(\mu;z)$ in the polar expansion of $F$ also represent elements of
$\W r \om(\proj\RR)$.  Consulting Table~\ref{tab-PMH} we conclude
for $r\in \CC \smallsetminus \ZZ_{\geq 1}$ that $F(\mu;\cdot)$ is a
multiple of $\M{r,\mu}$.

On the other hand, as mentioned above, since $F$ is $r$-harmonic in
$\uhp$, so is $F(\mu; z)$ for each $\mu$. Again from table 2 we see
that $F(\mu, \cdot)$
should be a multiple of $\P{r,\mu}$ for $\mu\geq 0$ and a multiple of
$\H{r,\mu}$ if $\mu\leq
-1$.

Hence all terms in the polar expansion of $F$ vanish, and $F=0$.
\end{proof}

\subsection{Polar expansion of the kernel function} The kernel
function $K_r(z;\tau)=\frac{2i}{z-\tau} \allowbreak\,\Bigl(\frac{\bar
z-\tau}{\bar z-z} \Bigr)^{r-1}$ in \S\ref{sect-Kr} gives, for a fixed
$\tau\in \uhp$, rise to two polar expansions in $z$, on the disk
$\Bigl|\frac{z-i}{z+i}\Bigr|<\Bigl|\frac{\tau-i}{\tau+i}\Bigr|$ and
on the annulus
$1>\Bigl|\frac{z-i}{z+i}\Bigr|>\Bigl|\frac{\tau-i}{\tau+i}\Bigr|$.

\begin{prop}\label{prop-polexp}\il{Kr1}{$K_r(z;\tau)$}
\begin{enumerate}
\item[i)] Consider $z,\tau$ satisfying
$\bigl|\frac{z-i}{z+i}\bigr|>\bigl|\frac{\tau-i}{\tau+i}\bigr|$.
\begin{enumerate}
\item[a)] If $r\in \CC\setminus \ZZ_{\geq 2}$, then
\be K_r(z;\tau)
\= \sum_{\mu \leq -1} \frac{(2-r)_{-\mu-1}}{(-\mu-1)!}\,
\P{2-r,-\mu-1}(\tau)\, \M{r,\mu}(z)\,. \ee
\item[b)] If $r\in \ZZ_{\geq 2}$, then
\be \label{iwK} K_r(z;\tau)
\= \sum_{\mu=1-r}^{-1} (-1)^{-\mu-1}\, \binom{r-2}{-\mu-1}\,
\P{2-r,-\mu-1}(\tau)\, \M{r,\mu}(z)+p_r(z;\tau)\,, \ee
with\ir{prkernel}{p_r(z;\tau)}
\be \label{prkernel}
p_r(z;\tau) \;:=\; \frac{2i}{z-\tau}\, \Bigl(\frac{
\tau-i}{z-i}\Bigr)^{r-1}\,. \ee
\end{enumerate}
\item[ii)] Consider $z,\tau$ satisfying
$\bigl|\frac{z-i}{z+i}\bigr|<\bigl|\frac{\tau-i}{\tau+i}\bigr|$. For
all $r\in \CC$:
\be K_r(z;\tau) \= -\sum_{\mu\leq -1} \frac{(1-r)_{-\mu}}{(-\mu)!}\,
\P{2-r,-\mu-1}(\tau)\, \H{r,\mu}(z)- \sum_{\mu\geq 0}\,
\P{2-r,-\mu-1}(\tau)\, \P{r,\mu}(z)
\,. \ee
\item[iii)] For $r\in \ZZ_{\geq 2}$:
\bad y^{1-r} &\= \sum_{\mu=1-r}^1 \frac{(1-r)_{-\mu}}{(-\mu)!}\,
\H{r,\mu}(z) + \Bigl( \frac{2i}{z+i}\Bigr)^{r-1}\\
&\= - \sum_{\mu=1-r}^{-1} \frac{(2-r)_{-\mu-1}}{(-\mu-1)!}\,
\M{r,\mu}(z)+ \Bigl( \frac{2i}{z-i} \Bigr)^{r-1} \,. \ead
\end{enumerate}
\end{prop}

\begin{proof}
We use the coordinates $w=\frac{z-i}{z+i}$ and
$\xi=\frac{\tau-i}{\tau+i}$, and put
\[ X_+ \;:=\;\bigl\{(z,\tau)\in \uhp^2 \;:\;|\xi|<|w|<1\bigr\}\,, X_-
\;:=\; \bigl\{(z,\tau)\in \uhp^2 \;:\;|w|<|\xi|<1 \bigr\}\,.\]

\rmrk{General expansion}  The function $K_r(\cdot;\cdot)$ has
polar expansions on both regions, that have the following form
on~$X_\pm$:
\[ K_r(z;\tau) \= \sum_{\mu\in \ZZ} A_\mu^\pm(r,\tau)\,
F_\mu^\pm(r,z)\,.\]
We consider this first for $r\in \CC\setminus \ZZ_{\geq 2}$. The 
fact that $K_r(\cdot;\tau)$ on $X_+$ represents an element
of $\W r \om(\proj\RR)$ implies that we can take $F_\mu^+ =
\M{r,\mu}$ and where the fact that $K_r(\cdot;\tau)$ has no
singularity on $X_-$ implies that we can take $F_\mu^-=\H{r,\mu}$ for
$\mu \leq -1$ and $F_\mu^- = \P{r,\mu}$ for $\mu \geq 0$.

The invariance relation
\be \label{Kr-inv}K_r(\cdot;\cdot)|_r g \otimes |_{2-r} g \= K_r
\quad\text{ for each }g\in \SL_2(\RR)\ee
in \eqref{Kq-inv}, applied with $g=k(\th)$ for small $\th$ implies
that $A_\mu^\pm(r,\tau)$ transforms under $|_r k(\th)$ by
$e^{-i(r+2\mu)\th}$, hence it has the form $A_\mu^\pm(r,\tau)
= d_\mu^\pm(r)\, \P{2-r,-\mu-2}(\tau)$, for some quantity
$d_\mu^\pm(r)$.

For a given $z\in \uhp$ the function $K_r(z;\cdot)$ has only a
singularity in the upper half-plane at $\tau=z$. Since
$\P{2-r,-\mu-1}(\tau)$ has a singularity at $\tau=i$ if $-\mu-1\leq
0$, we have $d^+_\mu(r)=0$ for $\mu \geq 0$. Thus, we have the
following:
\bad \text{on $X_+$:}& & K_r(z;\tau) &\= \sum_{\mu\leq -1}d^+_\mu(r)
\, \P{2-r,-\mu-1}(\tau)\, \M{r,\mu}(z)\,,\\
\text{on $X_-$:}&& K_r(z;\tau) &\= \sum_{\mu\leq -1}d^-_\mu(r) \,
\P{2-r,-\mu-1}(\tau)\, \H{r,\mu}(z) \\
&&&\qquad\hbox{}
+ \sum_{\mu\geq0}d^-_\mu(r) \, \P{2-r,-\mu-1}(\tau)\, \P{r,\mu}(z)\,.
\ead
To both sides of both equations we apply the shadow operator~$\xi_r$.
With  \eqref{shad-Kr},  Table~\ref{tab-PMH}, 
and the fact that $\P{r,\mu}$ is holomorphic, we get
\begin{align*}
(\bar r-1)\,& \Bigl( \frac{z-\bar \tau}{2i}\Bigr)^{\bar r-2} \\
&\= \sum_{\mu\leq -1} \overline{ \P{2-r,-\mu-1}(\tau)}\, \Bigl(
\frac{2i}{z+i}\Bigr)^{2-\bar r}\, \Bigl( \frac{z-i}{z+i}
\Bigr)^{-\mu-1} \cdot
\begin{cases}
\overline{d_\mu^+(r)}\, (\bar r-1)&\text{ on }X_+\,,\\
\overline{d_\mu^-(r)}\, (-\mu)&\text{ on }X_-\,.
\end{cases}
\end{align*}
We consider this for $(z,\tau)\in X_+$ with $\xi$ near~$0$, and for
$(z,\tau)\in X_-$ with $w$ near~$0$. Then we can rewrite the 
left hand side of the equation so that the main factor is
\begin{align*} (\bar r-1)\, & (1-w\bar \xi)^{\bar r-2} \= (\bar r-1)\,
\sum_{a\geq 0} \frac{(2-\bar r)_a}{a!}\, w^a \, \bar \xi^a\\
&\= (\bar r-1)\, \sum_{\mu\leq -1} \frac{(2-\bar
r)_{-\mu-1}}{(-\mu-1)!}\, \Bigl( \frac{\bar \tau+i}{\bar
\tau-i}\Bigr)^{-\mu-1}\, \Bigl( \frac{z-i}{z+i}\Bigr)^{-\mu-1}\,.
\end{align*}
So for $\mu \leq -1$ we have
\be d_\mu^+(r) \= \frac{(2-r)_{-\mu-1}}{(-\mu-1)!}\,,\qquad d_\mu^-(r)
\= - \frac{(1-r)_{-\mu}}{(-\mu)!}\,. \ee

\rmrkn{Part i)a) }is now clear.

\rmrk{Part i)b)} The integrals in \eqref{pol-term} with
$F(z) = K_r(z;\tau)$ show that each term
$A_\mu^\pm(r,\tau)\, F_\mu^\pm(r,z)$ in the polar expansion is
holomorphic in~$r$. So we can handle the case $r_0\in \ZZ_{\geq 2}$ 
by a limit argument. For $1-r_0\leq \mu \leq
-1$ with $r_0\in \ZZ_{\geq 2}$, the function $\M{r,\mu}$ has a
holomorphic extension to $r=r_0$ (remarks to Table~\ref{tab-PMH}).
This gives the sum in \eqref{iwK}. With \eqref{HMP} the terms with
$\mu \leq -r_0$ can be written as
\[ \frac{(1-r)\, (1-r)_{-\mu-1}}{\mu\, (-\mu-1)!}\,
\P{2-r,-\mu-1}(\tau) \, \H{r,\mu}(z) + \P{2-r,-\mu-1}(\tau)\,
\P{r,\mu}(z)\,.\]
The first of these terms has limit $0$ as $r\rightarrow r_0$. The
second term leads to a series with $p_r(z;\tau)$ as its sum.

\rmrk{Part ii)} On $X_-$ we have
\[ K_r(z;\tau) \= -\sum_{\mu\leq 0} \frac{(1-r)_{-\mu}}{(-\mu)!}\,
\P{2-r,-\mu-1}(\tau)\, \H{r,\mu}(z)
+ \sum_{\mu\geq 0} d_\mu^-(r) \, \P{2-r,-\mu-1}(\tau)\,
\P{r,\mu}(z)\,, \]
with still unknown $d_\mu^-(r) $ for $\mu\geq 0$. In the coordinates
$w$ and $\xi$ this becomes:
\badl{ieq} &\frac{(1-w)^r \,(1-\xi)^{2-r}\,
(1-\bar w \xi)^{r-1} }{(w-\xi)\,(1-|w|^2)^{r-1}}\\
&\quad\= - \sum_{\mu \leq -1} \frac{(1-r)_{-\mu}}{(-\mu)!}\,
(1-\xi)^{2-r}\, \xi^{-\mu-1}\, \H{r,\mu}(z) \\
&\quad\hbox{} + \sum_{\mu\geq 0} d_\mu^-(r)\,
(1-\xi)^{2-r}\,\xi^{-\mu-1}\, (1-w)^r \, w^\mu\,. \eadl
We divide by $(1-\nobreak \xi)^{2-r}$. The remaining quantity on the
left has the following expansion:
\[ \frac{(1-w)^r }{(1-|w|^2)^{r-1}} \sum_{a,b\geq 0} (-1) w^a
\xi^{-a-1} \, \binom{r-1}b\, (-1)^b \bar w^b \xi^b\,. \]
Let $n\geq 1$. The coefficient of $\xi^{-n}$ in this expansion is
\[
-\sum_{b\geq 0} \,w^{b+n-1} \, \binom{r-1}{b} (-\bar w )^b \= -w^{n-1}
\,(1-|w|^2)^{r-1}\,. \]
Hence $d_\mu^{-1}(r)=-1$ for $\mu\geq 0$. By holomorphic
extension from $r\in \CC\setminus\ZZ_{\geq 2}$ to $r\in \CC$, this
gives Part~ii).

\rmrk{Part iii)} Let $r\in \ZZ_{\geq 2}$. The equality \eqref{ieq}
  divided by $(1-\nobreak\xi)^{2-r}$ becomes
\begin{align*} \frac{(1-w)^r\,(1-\bar w
\xi)^{r-1}}{(w-\xi)\,(1-|w|^2)^{r-1}} & \= -
\sum_{\mu=1-r}^{-1}\frac{(1-r)_{-\mu}}{(-\mu)!} \, \xi^{-\mu-1}\,
\H{r,\mu}(z)
- (1-w)^r \xi^{-1} \,(1-w/\xi)^{-1}\,.
\end{align*}
Now we let $\xi$ tend to $1$. We obtain:
\[ -(1-w)^{r-1}\, (1-\bar w)^{r-1}\,(1-|w|^2)^{1-r}\=
- \sum_{\mu=1-r}^{-1}\frac{(1-r)_{-\mu}}{(-\mu)!} \, \H{r,\mu}(z) -
(1-w)^{r-1}\,.\]
In terms of the coordinate $z$ this is
\[ - y^{1-r} \= - \sum_{\mu=1-r}^{-1}\frac{(1-r)_{-\mu}}{(-\mu)!} \,
\H{r,\mu}(z) - \Bigl(\frac{2i}{z+i}\Bigr)^{r-1}\,,\]
which gives the first expression for $y^{1-r}$ in Part~iii). We use
\eqref{HMP} to obtain the second expression.\end{proof}

\subsection{Related work}\label{sect-lit7}
The polar expansion generalizes the power series expansion in
$w=\frac{z-i}{z+i}$ for holomorphic functions on the upper
half-plane. When dealing with $r$-harmonic functions a
straightforward generalization leads to the functions in
Table~\ref{tab-PMH}. Proposition~\ref{prop-polexp} is analogous
to~\cite[Proposition 3.3]{BLZ13}.


\part{Cohomology with values in analytic boundary germs} We turn to
the proof of the surjectivity in Theorem~\ref{THMac} and the proof of
Theorem~\ref{THMaci}, by relating cohomology with values in
$\dsv{v,2-r}\om$ to cohomology in modules $\esv{v,r}\om \subset \W
{v,r}\om(\proj\RR)$. Section~\ref{sect-hws} gives the definition of
these modules.

We use a description of cohomology that turned out to be useful in the
analogous result for Maass forms, in~\cite{BLZm}. This description of
cohomology is based on a tesselation of the upper half-plane. See
Section~\ref{sect-tesscoh}.

Theorem~\ref{thmbg} describes the relation between holomorphic
automorphic form and boundary germ cohomology. This theorem
immediately implies the surjectivity in Theorem~\ref{THMac}. For the
 weights in $\ZZ_{\geq 2}$ work has to be done, in
Section~\ref{sect-abg-iw}, to prove Theorem~\ref{THMaci}.


\section{Highest weight spaces of analytic boundary
germs}\label{sect-hws} This section serves to define the modules
$\esv{v,r}\om$ to take the place of the modules $\dsv{v,2-r}\om$.

\subsection{Definition of highest weight space}\label{sect-def-E}The
cases $r\in \ZZ_{\geq 2}$, and $r\in\CC\setminus \ZZ_{\geq 2}$ are
dealt with separately.

\subsubsection{Weight in $\CC\setminus \ZZ_{\geq 2}$} Part~i) of
Theorem~\ref{thm-VWiso} points the way how to treat this case. It
states that $\rsp_r: \W r \om(\proj\RR) \rightarrow \V r
\om(\proj\RR)$ is bijective. For $\ph\in\dsv{2-r}\om$
\be\label{rs-1} \rs_r^{-1} \ph \= (\rsp_r )^{-1} \Prj{2-r}\ph \;\in\;
\W r \om(\proj\RR)\,,
\ee
where we use that $(\rs_r f)(t) = (i-t)^{r-2}\,(\rsp_r f)(t)$.
See~\eqref{rsl}. For $\ph \in \dsv{2-r}\om[\xi_1,\ldots,\xi_n]$ we
can proceed similarly to get $\rs_r^{-1}\ph\in \W r
  \om[\xi_1,\ldots,\xi_n]$.

\begin{defn}\label{esv-def-gen}For $r\in \CC\setminus \ZZ_{\geq 2}$ we
define\ir{esv-gen}{\esv r \om}
\be\label{esv-gen} \esv r \om \;:=\; \rs_r^{-1}\dsv{2-r}\om,\quad \esv
r{\om,\wdg}[\xi_1,\ldots,\xi_n] \;:=\;
\rs_r^{-1}\dsv{2-r}{\om,\wdg}[\xi_1,\ldots,\xi_n]
\ee
\il{esv-gen[]}{$ \esv r {\om,\wdg}[\xi_1,\ldots,\xi_n]$}for each
finite set $\{\xi_1,\ldots,\xi_n\}\subset\proj\RR$. 
\end{defn}

\rmrk{Weight $1$} The case $r=1$ is special. The restriction morphism
is given by $(\rsp_1 F)(t)=\frac1{2i}\,(t-\nobreak i)\, F(t)$, and
hence $(\rs _1 F)(t)=\frac i2 \, F(t)$. This gives the following
equalities:
\be \label{wt1ED}
\esv 1 \om \= \dsv 1 \om\,,\qquad \esv 1
{\om,\wdg}[\xi_1,\ldots,\xi_n] = \dsv 1 {\om,\wdg}[\xi_1,\
\ldots,\xi_n]\,.
\ee

\rmrk{Characterization with series}The projective model
$\Prj{2-r}\dsv{2-r}\om$ consists of the holomorphic functions on some
neighbourhood of $\lhp\cup\proj\RR$ in $\CC$. So it consists of the
functions
\[ t\mapsto \sum_{\mu\leq 0} c_\mu\, \Bigl(
\frac{t-i}{t+i}\Bigr)^\mu\]
with coefficients that satisfy $c_\mu= \oh\bigl(e^{-a|\mu|}\bigr)$ as
$|\mu|\rightarrow \infty$, for some $a>0$ depending on the domain of
the function. Table~\ref{tab-PMH} in \S\ref{sect-pol-exp} gives
$\bigl( \frac{t-i}{t+i}\bigr)^\mu =\rsp_r M_{r,\mu-1}$. Hence we have
for $r\in \CC\setminus\ZZ_{\geq 2}$
\be\label{esv-pe-gw} \esv r\om \= \biggl\{ \sum_{\mu \in \ZZ_{\leq
-1}} c_\mu \, \M{r,\mu} \;:\; c_\mu = \oh\bigl(e^{-a|\mu|}\bigr)\text{
for some }a>0\biggr\}\,. \ee

\rmrk{Highest weight spaces}We call $\dsv{2-r}\om$ and $\esv r \om$
\emph{highest weight spaces}. The use of this terminology is
explained in \S\ref{sect-hwssp} in the Appendix.
\medskip

\subsubsection{Case $r\in \ZZ_{\geq 2}$} We note that representatives
of elements of $\dsv r\om$ are holomorphic functions on a
neighbourhood of $\lhp\cup\proj\RR$ in $\proj\CC$ that have at
$\infty$ a zero of order at least~$r$, since \eqref{Prj} shows that
the elements of $\dsv r \om$ are of the form $t\mapsto (i-\nobreak
t)^{-r}\cdot(\text{holo.\ at $\infty$})$.
These functions represent also elements of $\W r \om(\proj\RR)$.

\begin{defn}\label{esv-def-iw}For $r\in \ZZ_{\geq 2}$ we
define:\ir{esv-iw}{\esv r \om}\il{esv-iw[]}{$ \esv r
{\om,\wdg}[\xi_1,\ldots,\xi_n]$}
\badl{esv-iw} \esv r \om &\;:=\; \dsv r \om + \sum_{\mu=1-r}^{-1}\CC\;
\M{r,\mu} \,,\\
\esv r {\om,\wdg}[\xi_1,\ldots,\xi_n]&\;:=\; \dsv r
{\om,\wdg}[\xi_1,\ldots,\xi_n]+ \sum_{\mu=1-r}^{-1}\CC\; \M{r,\mu}
\,, \eadl
for finite sets $\{\xi_1,\ldots,\xi_n\}\subset\proj\RR$. 
\end{defn}

\rmrke This defines $\esv r \om$ as a subspace of $\W r
\om(\proj\RR)$, and $\esv r{\om,\wdg}[\xi_1,\ldots,\xi_n]$ as a
subspace of $\W r \om \left(
\proj\RR\setminus\{\xi_1,\ldots,\xi_n\}\right)$.

\rmrk{Comparison with weight $1$} If we apply the formulas in
\eqref{esv-iw} with $r=1$, the sum over $\mu$ is empty, and we get
back~\eqref{wt1ED}.

\rmrk{Characterization with series}The elements of the projective
model $\Prj{2-r}\dsv r \om$ are the functions of the form
$h(t)=\sum_{\mu\leq 0} d_\mu \,\Bigl(\frac{t-i}{t+i}\Bigr)^\mu$ with
$d_\mu = \oh\bigl( e^{-a|\mu|}\bigr)$ for some $a>0$. In view of
\eqref{Prj} and \eqref{Prmu} $f=\Prj r^{-1}h$ has an expansion of the
form
\[ (z-i)^{-r} \sum_{\mu\leq 0}d_\mu \,\Bigl(\frac{t-i}{t+i}\Bigr)^\mu
\= \sum_{\mu \leq -r} \, c_\mu \, \P{r,\mu}(z)\,,\]
with the $c_\mu$ satisfying the same estimate. This leads to
\be\label{esv-pe-iw}
\esv r \om \=\biggl\{ \sum_{\mu \leq -r} c_\mu\, \P{r,\mu}(z) +
\sum_{\mu=1-r}^{-1} c_\mu \, \M{r,\mu}\;:\; c_\mu = \oh\bigl(
e^{-a|\mu|}\bigr) \text{ for some }a>0\biggr\}\,, \ee
which is similar to~\eqref{esv-pe-gw}.
\smallskip

In the following result we use the subsheaf
\il{Whol1}{$\Whol r \om$}$\Whol r \om$, defined in~\eqref{Whol}. Its
sections are represented by holomorphic functions, contained in the
kernel of the restriction morphism~$\rs_r$.

\begin{prop}\label{prop-iw-esv}Let $r\in \ZZ_{\geq 2}$.
\begin{enumerate}
\item[i)] $\dsv r \om$ is a subspace of $\Whol r \om(\proj\RR)$
invariant under the operators $|_r g$ with $g\in \SL_2(\RR)$, and
$\dsv r {\om,\wdg}[\xi_1,\ldots,\xi_n]|_r g = \dsv r
{\om,\wdg}[g^{-1}\xi_1\ldots, g^{-1}\xi_n]$ for all $g\in
\SL_2(\RR)$.
\item[ii)] $\esv r \om$ is a subspace of $\W r \om(\proj\RR)$
invariant under the operators $|_r g$ with $g\in \SL_2(\RR)$, and
$\esv r {\om,\wdg}[\xi_1,\ldots,\xi_n]|_r g = \esv r
{\om,\wdg}[g^{-1}\xi_1.\ldots, g^{-1}\xi_n]$ for all $g\in
\SL_2(\RR)$.
\item[iii)] The following sequences are exact:
\bad &0 \rightarrow \dsv r \om \rightarrow \esv r \om
\stackrel{\rs_r}\rightarrow \dsv{2-r}\pol\rightarrow0\,,\\
 &0 \rightarrow \dsv r {\om,\wdg}[\xi_1,\ldots,\xi_n] \rightarrow \esv
 r{\om,\wdg}[\xi_1,\ldots,\xi_n] \stackrel{\rs_r}\rightarrow
\dsv{2-r}\pol\rightarrow0\,. \ead
\end{enumerate}
\end{prop}

\begin{proof}For Part~i) we use the definitions in~\S\ref{sect-sav},
applied with $r$ instead of $2-r$. In particular, elements of $\dsv r
\om$ are represented by holomorphic functions with at $\infty$ a zero
of order at least~$r$. In that way we see that all elements of $\dsv
r \om$ are sections in $\Whol r \om$, and similarly for $\dsv r
{\om,\exc}[\xi_1,\ldots,\xi_n]$.

For Part~ii) there remains to show that $\M{r,\mu}|_r g \in \esv r
\om$. 
Relation~\eqref{iwK} in Proposition~\ref{prop-polexp} expresses
$K_r(\cdot;\tau)$ as a linear combination of the $M_{r,\mu}$ in $\esv
r \om$ and an explicit kernel $p_r(\cdot;\tau)$. Since
$P_{2-r,-\mu-1}(\tau)$ is essentially equal to $\Bigl(
\frac{\tau-i}{\tau+i}\Bigr)^{-\mu-1}$, we can invert the relation,
and express each $\M{r,\mu}$ with $1-r\leq \mu\leq -1$ as a linear
combination of $K_r(\cdot;\tau_i)-p_r(\cdot;\tau_i)$ for $r-1$
elements $\tau_i\in \uhp$. The invariance relation \eqref{Kr-inv}
implies that $K_r (\cdot;\tau_i)|_r
g$ is a multiple of $K_r(\cdot;g^{-1}\tau_i)$, which is in $\esv
r \om$ by an application
of~\eqref{iwK}. The contribution of $p_r(\cdot;\tau_i)$ is in $\dsv
r \om $, which is invariant under the operators $|_r g$.

The exactness of the sequences in Part~iii) follows directly from the
fact that $\rs_r$ vanishes on $\dsv r
\om$ and the relations   $\bigl(\rs _r \M{r,\mu}\bigr)(t)=(i-\nobreak
t)^{r-2}\,\allowbreak \Bigl(
\frac{t-i}{t+i}\Bigr)^{\mu+1}$, with \eqref{rspM} and~\eqref{rs-1}. 
\end{proof}

\subsection{General properties of highest weight spaces of analytic
boundary germs} In the previous subsection we have chosen spaces
$\esv r {\om,\wdg}[\xi_1,\ldots,\xi_n]$ of boundary germs for all
finite subsets $\{\xi_1,\ldots,\xi_n\}$ of $\proj\RR$, and the space
$\esv r \om$, which we call also
\il{esv[]}{$\esv r{\om,\wdg}[]$}$\esv r{\om,\wdg}[]$. The following
proposition lists properties that these system have in common for all
$r\in \CC$. In Section~\ref{sect-caf} we shall work on the basis of
these properties.

\begin{prop}\label{prop-Esys}
The systems of spaces in Definitions \ref{esv-def-gen}
and~\ref{esv-def-iw} have the following properties:
\begin{enumerate}
\item[i)] $\esv r{\om,\wdg}[\xi_1,\ldots,\xi_n]\subset \W r
\om\bigl(\proj\RR\setminus \{\xi_1,\ldots,\xi_n\}\bigr)$ consists of
boundary germs represented by functions in $\sharmb r (U)$ where $U$
is open in $\proj\CC$ so that $U\cup\lhp$ is a
$\{\xi_1,\ldots,\xi_n\}$-excised neighbourhood.
\item[ii)]\begin{enumerate}
\item[a)]If $\{\xi_1,\ldots,\xi_n\}\subset \{\eta_1,\ldots,\eta_m\}$,
then
\[\esv r{\om,\wdg}[\xi_1,\ldots,\xi_n]\subset \esv
r{\om,\wdg}[\eta_1,\ldots,\eta_m]\,.\]
\item[b)] If
$\{\xi_1,\ldots,\xi_n\}\cap\{\eta_1,\ldots,\eta_m\}=\emptyset$, then
\[\esv r{\om,\wdg}[\xi_1,\ldots,\xi_n]\cap \esv
r{\om,\wdg}[\eta_1,\ldots,\eta_m]\=\esv r \om\,.\]
\end{enumerate}
\end{enumerate}
{\rm With the inclusion relation ii)a) we define\ir{esv*}{\esv
r{\fs,\wdg}}
\be\label{esv*} \esv r{\fs,\wdg} := \indlim \esv
r{\om,\wdg}[\xi_1,\ldots,\xi_n]\,,
\ee
where $\{\xi_1,\ldots,\xi_n\}$ runs over the finite subsets
of~$\proj\RR$.}
\begin{enumerate}
\item[iii)] $\esv r{\om,\wdg}[\xi_1,\ldots,\xi_n]|_r g = \esv r
{\om,\wdg}[g^{-1}\xi_1,\ldots,g^{-1}\xi_n]$ for each $g\in
\SL_2(\RR)$.
\item[iv)] The function $z\mapsto
\int_{z_1}^{z_2}K_r(z;\tau)\,f(\tau)\, d\tau$ represents an element
of $\esv r \om$ for all $z_1,z_2\in \uhp$ and each holomorphic
function $f$ on~$\uhp$.
\item[v)] If $F\in \sharm_r(\uhp)$ represents an element of $\esv r
\om$, then $F=0$.
\item[vi)] If $F\in \sharmb r
(U)$ represents an element of $\esv r {\om,\wdg}[\xi_1,\ldots,\xi_n]$
then its shadow $\shad_r F\in \hol(U\cap\uhp)$ extends
holomorphically to~$\uhp$.
\item[vii)] Let $\ld\in \CC^\ast$. Suppose that $f\in \esv
r{\om,\wdg}[\infty]$ has a representative $F$ that satisfies:
\begin{enumerate}
\item[a)] $F\in \sharm _r(U\cap\uhp)$ for some neighbourhood $U$ of
$\proj\RR$ in $\proj\CC$,
\item[b)] the function $z\mapsto \ld^{-1}\, F(z+\nobreak 1)-F(z)$ on
$\uhp\cap U\cap T^{-1}U$ represents an element of~$\esv r \om$,
\end{enumerate}
then $f=p+g$ with an element $g\in \esv r \om$ and a $\ld$-periodic
element $p\in \esv r{\om,\wdg}[\infty]$.
\end{enumerate}
\end{prop}

\rmrks
\itmi In Property~a) it is not always possible to choose the
representative so that the set $U$ contains~$\lhp$.
Moreover, the property does not state that all functions in $\sharmb
r(U)$ with $U$ as indicated represent elements of $\esv r
{\om,\wdg}[\xi_1,\ldots,\xi_n]$.

\itm Condition~a) in Part~vii) is strong. In general representatives
of an element of $\esv r {\om,\wdg}[\infty]$ are $r$-harmonic only on
an $\{\infty\}$-excised neighbourhood.

\begin{proof}We consider the various parts of the theorem, often
separately for the general case $r\in \CC\setminus \ZZ_{\geq 2}$ and
the special case $r\in \ZZ_{\geq 2}$.

\rmrk{{\rm a.} Part~i)} Let $r \in \CC \smallsetminus \ZZ_{\geq 2}$.
  An element $F\in \esv r{\om,\wdg}[\xi_1,\ldots,\xi_n]$ is determined
by $h=\rsp_r F$ in the projective model of $\dsv
{2-r}{\om,\wdg}[\xi_1,\ldots,\xi_n]$. So $h$ is holomorphic on a
$\{\xi_1,\ldots,\xi_n\}$-excised neighbourhood~$U_0$. Each point
$\xi\in \proj\RR\setminus \{\xi_1,\ldots,\xi_n\}$ is of the form
$k_\xi\cdot0$ with $k\in \SO(2)$. We choose $p(\xi)\in (0,1)$ such
that $k_\xi\; D_{p(\xi)} \subset U_0$. Then Part~iii)a) of
Proposition~\ref{prop-tPloc} implies that
$\rsp_r  \sharmb r(k_\xi\, D_{p(\xi)})  = \hol(k_\xi\, D_{p(\xi)})$.
So $F\in \sharmb r (U)$, with 
\[ U \;:=\; \bigcup_{\xi\in \proj\RR\setminus\{\xi_1,\cdots,\xi_n\}}
k_\xi\; D_{p(\xi)}\,.\]
However, not for all choices of the $p_\xi$ the set
$U\cup\lhp$ is an excised neighbourhood.\smallskip

We return to the choice of the $p_\xi \in (0,1)$. By conjugation
by an element of $\SO(2)$ we arrange that all $\xi_j$ are in~$\RR$.
\vskip.3em
\twocolwithpictr{\quad We recall that near each of the points $\xi_j$
a $\{\xi_1,\ldots,\xi_n\}$-excised neighbourhood looks like a full
neighbourhood of $\xi_j$ minus the sector between two geodesic
half-lines with end-point $\xi_j$. }{
\setlength\unitlength{1cm}
\begin{picture}(6,2)(0,-1)
\put(0,0){\line(1,0){6}}
\put(1.85,-.5){$\xi_j$}
\put(4.85,-.5){$\xi_{j'}$}
\put(.4,.2){$U_0$}
\put(3.2,.2){$U_0$}
\put(5.7,.2){$U_0$}
\put(3.8,-.6){$U_0$}
\thicklines
\put(2,0){\line(0,1){1}}
\qbezier(2,0)(2,.6)(2.5,.8)
\qbezier(5,0)(5,.6)(4.5,.8)
\qbezier(2.5,.8)(3.5,.6)(4.5,.8)
\qbezier(5,0)(5,.6)(5.5,.8)
\end{picture}
}\vskip.8ex
\twocolwithpictl{
\setlength\unitlength{1cm}
\begin{picture}(6,2)(0,-1)
\put(0,0){\line(1,0){6}}
\put(1.85,-.5){$\xi_j$}
\put(4.85,-.5){$\xi_{j'}$}
\put(2.3,0){\circle{.6}}
\put(2.5,0){\circle{1}}
\put(1.5,0){\circle{1}}
\put(1.7,0){\circle{.6}}
\put(1.4,0){\circle{1.2}}
\thicklines
\put(2,0){\line(0,1){1}}
\qbezier(2,0)(2,.6)(2.5,.8)
\qbezier(5,0)(5,.6)(4.5,.8)
\qbezier(2.5,.8)(3.5,.6)(4.5,.8)
\qbezier(5,0)(5,.6)(5.5,.8)
\end{picture}
}{Since $\xi_j\in \RR$ those geodesic half-lines are
parts of euclidean circles with their center on $\RR$, or vertical
euclidean lines. This implies that there is a small $\e>0$ such that
for all $\xi\in \RR$ with $0<|\xi-\nobreak\xi_j|<\e$ the open
euclidean disk around $\xi$ with radius $|\xi-\nobreak \xi_j|$ is
contained in $U_0$. \hspace{\fill}If $\e$ is
}\vskip.3em%
\noindent
sufficiently small these euclidean disks are of the form
$k_{\xi'}\, D_{p(\xi')}$, with in general $\xi'\neq\xi$.

In this way we can choose for  all $\xi$
sufficiently near to~$\xi_j$  the value 
$p_\xi\in(0,1)$ in such a way that $\xi_j$ is in the closure of
$k_\xi\, D_{p(\xi)}$.  We see
that $U$ is near $\xi_j$ a full neighbourhood of $\xi_j$ minus
the sectors between two geodesics half-lines at $\xi_j$ in the upper
and the lower half-plane. So $\lhp\cup U$ is an excised
neighbourhood provided we take the $p_\xi$ appropriately.

\rmrk{{\rm b.} Part~i) for $r\in \ZZ_{\geq 2}$}Elements of $\dsv r
{\om,\wdg}[\xi_1,\ldots,\xi_n]$ are already represented by functions
of the desired form. The functions $\M{r,\mu}$ are in $\sharmb
r\bigl( \proj\CC\setminus \nobreak\{i,-i\}\bigr)$.

\rmrk{{\rm c.} Part~ii)}Immediate from the corresponding property of
$\dsv p {\om,\wdg}$, with $p\in \{r,2-\nobreak r\}$.

\rmrk{{\rm d.} Part~iii)}Immediate from Part~i) of
Proposition~\ref{prop-sasi} and \eqref{res-g} if $r\not\in \ZZ_{\geq
2}$, and from Proposition~\ref{prop-iw-esv}, Part~ii), if $r\in
\ZZ_{\geq 2}$.

\rmrk{{\rm e.} Part~iv) for $r\not\in \ZZ_{\geq 2}$}Integration over
$\tau$ in a compact set in~$\uhp$ preserves the property that
$K_r(\cdot;\tau)$ represents an element of $\W r \om(\proj\RR)$, and
commutes with taking the restriction. Applying $\rs _r$ gives the
integral $\int_{z_1}^{z_2}
(z-\nobreak \tau)^{r-2}\, F(\tau)\, d\tau$, which has a value in
$\dsv{2-r}\om$.

\rmrk{{\rm f.} Part~iv) for $r\in \ZZ_{\geq 2}$}Equation \eqref{iwK}
in Proposition~\ref{prop-polexp} expresses $K_r(\cdot;\tau)$
as a linear combination of the $\M{r,\mu}$ in~$\esv r \om$ and an
explicit kernel $p_r(\cdot;\tau)$.
In the integral of the terms with $\M{r,\mu}$ only the coefficient
depends on $\tau$. Hence the result is a multiple of $\M{r,\mu}$. The
kernel $p_r(\cdot;\tau)$ is in $\dsv r \om$ by the description in
Part~i) of Proposition~\ref{prop-polexp}, and it stays there under
integration with respect to~$\tau$.

\rmrk{{\rm g.} Part~v) for $r\not\in \ZZ_{\geq 1}$}See
Proposition~\ref{prop-Ei}.

\rmrk{{\rm h.} Part~v) for $r=1$}Let $F\in \sharm_1(\uhp)$ represent
an element of $\esv 1 \om(\proj\RR)$. So $\rs_1 F(z)=\frac{z-i}{2i}\,
F(z)$ is holomorphic on a neighbourhood $U$ of $\lhp\cap\proj\RR$ in
$\proj\CC$. Then $F$ itself is holomorphic on $\uhp\cap U
\setminus\{i\}$, hence $F$ is holomorphic on $\uhp$ since it is
already real-analytic. Thus, $F\in \hol(\proj\CC)$ with a zero
at~$\infty$, hence $F=0$.

\rmrk{{\rm i.} Part~v) for $r\in \ZZ_{\geq 2}$} For $r\in \ZZ_{\geq
2}$, we consider $F=F_0+\sum_{\mu=1-r}^{-1} c_\mu \,\M{r,\mu}\in
\sharm_r(\uhp)$ with $F_0$ representing an element of $\dsv r \om$
and $c_\mu \in \CC$. Since the $\M{r,\mu}$ are $r$-harmonic on
$\uhp\setminus \{i\}$ the function $F_0$ is holomorphic on
$\proj\CC\setminus\{i\}$ with at $\infty$ a zero of order at least
$r$.

To investigate the singularity of $F_0$ at $i$ we use Kummer
  relation~\eqref{HMP}, which relates $\M{r,\mu}$, $\H{r,\mu}$ and
  $\P{r,\mu}$. Since $\H{r,\mu}$ is $r$-harmonic on $\uhp$, the
singularity at $i$ of $\sum_{\mu=1-r}^{-1} c_\mu \M{r,\mu}$ is the
same as that of
\badl{F1f} F_1(z) &\= \sum_{\mu=1-r}^{-1}
\frac{(-\mu-1)!}{(2-r)_{-\mu-1}}\, c_\mu \,\P{r,\mu}\\
& \= \sum_{\mu=1-r}^{-1} \frac{(-\mu-1)!}{(2-r)_{-\mu-1}}\, c_\mu \,
\Bigl( \frac{2i}{z+i}\Bigr)^r\, \Bigl( \frac{z-i}{z+i}\Bigr)^\mu
\,.\eadl
This leads to a holomorphic function $F_0+F_1$ on
$\proj\CC\setminus\{-i\}$, with at $z=-i$ a pole of order at most
$r-1$. At $\infty$ the function $F_0$ has a zero of order at least
$r$. The same holds for $F_1$ by the factor $(z+\nobreak i)^{-r}$ in
\eqref{F1f}. So the number of zeros is larger than the number of
poles, and we conclude that $F_0+F_1=0$. However, $F_1=-F_0$ has to
be in $\dsv r \om$, in particular, it has to be holomorphic at
$z=-i$. Inspection of \eqref{F1f} shows that successively $c_{1-r},
c_{2-r},\ldots$ have to vanish. So $F=0$.

 \rmrk{{\rm j. }Part~vi)} The representative $F$ is defined on $U\cap
 \uhp$, where the open set $U \subset \proj\RR$ contains
 $\proj\RR\setminus \{\xi_1,\ldots,\xi_n\}$. Part~iv) of
Proposition~\ref{prop-tPloc} implies that there is an open set $U_1
\subset U$, still containing
$\proj\RR\setminus\{\xi_1,\ldots,\xi_n\}$ (obtained as the union of
sets $k \; D_p$) such that on $U_1$ the shadow $\shad _r F (z)
$ is a holomorphic multiple of $a(z) = \overline{(\rsp_r F)(\bar z)}$.
  So the domain of $a$ is some neighbourhood $U_2$ of
  $\proj\RR\setminus\{\xi_1,\ldots,\xi_n\}$ in $\proj\CC$. Since $\rs_r
F \in \dsv{v,2-r}\fs$ the functions $\rs_r F$  and $\rsp_r
F$ are holomorphic on~$\lhp$. Hence the domain of the holomorphic
function~$a$ contains $\overline{\lhp}=\uhp$.

\rmrk{{\rm k. }Part~vii)}See \S\ref{sect-cpEsys}.\qedhere\end{proof}

\subsection{Splitting of harmonic boundary germs, Green's form} We
discuss now a splitting of the space of global sections of
$r$-harmonic boundary germs. We shall use this to prove Part~vii) in
Proposition~\ref{prop-Esys} in the case that $r \in \CC
\smallsetminus \ZZ_{\geq 1}$. To obtain the
\il{splbg}{splitting of harmonic boundary germs}splitting we use the
Green's form for harmonic functions and the resolvent kernel.

\begin{thm}\label{thm-bgdecomp}If $r\in \CC\setminus \ZZ_{\geq 1}$
then\ir{bg-decomp}{\bg_r}
\be\label{bg-decomp} \bg_r(\proj\RR) \= \sharm_r(\uhp)\;\oplus\; \W r
\om(\proj\RR)\,. \ee
\end{thm}
Since we have already Proposition~\ref{prop-Ei}, we need only to show
that $\bg_r(\proj\RR)
= \sharm_r(\uhp)+ \W r \om(\proj\RR)$.

\rmrk{Resolvent kernel}We put\ir{Qr-def}{Q_r(\cdot,\cdot)}
\be \label{Qr-def} Q_r(z_1,z_2)
\= \M{r,0}\Bigl( \frac{z_2- \re z_1}{\im z_1} \Bigr)\,, \ee
with the $r$-harmonic function $\M{r,0}\in
\sharm_r\bigl(\uhp\setminus\{i\}\bigr)$ in \eqref{Mrmu}. So
$Q_r(z_1,z_2)$ is defined on $\uhp\times\uhp\setminus
(\text{diagonal})$. It is called the \il{fsresk}{free space resolvent
kernel}\emph{free space resolvent kernel} \il{rk}{resolvent kernel}.
It is a special case of the resolvent kernel that inverts the
differential operator $\Dt_r-\ld$ on suitable functions. See, eg.,
\cite[\S3, Chap.~XIV]{La75}.

The following properties can be checked by a computation, but are more
easily seen on the universal covering group, as we explain
in~\S\ref{app-rk}.
\begin{align}
\label{qh2}
\Dt_r Q_r(z_1,\cdot) &\= 0\,,\\
\label{nhe} 4 y^2 \partial_z\partial_{\bar z} Q_r(\cdot,z_2) +
2iry\,\partial_{\bar z}Q_r(\cdot,z_2) + r\, Q_r(\cdot,z_2) &\= 0\,,\\
\label{Qinv}
\text{for }\matc abcd \in \SL_2(\RR):\quad (cz_1+d)^r\,(c
z_2+d)^{-r}\, Q_r(g \, z_1,g\, z_2)&\= Q_r(z_1,z_2)\,.
\end{align}
The $r$-harmonic function $z_2\mapsto Q_r(z_1,z_2)$ represents an
element of $\W r \om(\proj\RR)$.

\rmrk{Green's form}For $f_1,f_2\in C^\infty(U)$, with $U\subset \uhp$,
we define the \il{Gf}{Green's form}\emph{Green's
form}\ir{Gr-f}{[\cdot,\cdot]_r}
\be\label{Gr-f} \bigl[f_1,f_2\bigr]_r \= \Bigl( \partial_z f_1 + \frac
r{z-\bar z} f_1\Bigr) f_2 \, dz + f_1 (\partial_{\bar z} f_2)\, d\bar
z\,. \ee
This is a $1$-form on~$U$, which satisfies
\be \bigl[f_1|_r g, f_2|_{-r} g\bigr]_r \= [f_1,f_2]_r \circ g
\quad\text{on }g^{-1}U
\qquad \text{for }g\in \SL_2(\RR)\,. \ee
If $f_1$ is $r$-harmonic on $U$ and if $f_2$ satisfies the
differential equation in~\eqref{nhe} on~$U$, then $[f_1,f_2]_r$ is a
closed differential form on~$U$.
(These results can be checked by some computations.)

\rmrk{Cauchy-like integral formula}
\begin{prop}\label{prop-psC}Let $r\in \CC\setminus\ZZ_{\geq 1}$. Let
$U$ be an open set in~$\uhp$, and let $C$ be a positively oriented
simple closed curve in $U$ such that the region~$V$ enclosed by $C$
is contained in~$U$. Then for each $F\in \sharm_r(U)$:
\[ \int_C\bigl[F,Q_r(\cdot,z')\bigr]_r \=
\begin{cases} 2\pi i(1-r) F(z')&\text{ if }z'\in V\,,\\
0&\text{ if }z'\in \uhp\setminus(C\cup V)\,.\end{cases}\]
\end{prop}

\begin{proof}In this result the kernel $Q_r$ and the Green's form are
combined to give for $r$-harmonic functions $F\in \sharm_r(\uhp)$ the
closed differential form
\[ \bigl[ F, Q_r(\cdot,z')\bigr]_r(z)\]
on $\uhp\times\uhp\setminus(\text{diagonal})$. It satisfies for $g=
\matc abcd \in \SL_2(\RR)$
\be (c z'+d)^{-r} \bigl[F,Q_r(\cdot,gz')\bigr]_r(gz) \= \bigl[F|_r
g,Q_r(\cdot,z')\bigr]_r(z)\,.
\ee
Hence it suffices to establish the relation for $z'=i$. The proof
proceeds along the same lines as the proof of~\cite[Theorem
3.1]{BLZ13}.

We have
\bad Q_r(z,i) &\= \frac{2iy}{i-\bar z}\, \Bigl(
\frac{z+i}{2i}\Bigr)^{r-1} \, \hypg21
\Bigl(1,1-r;2-r;\frac{4y}{|z+i|^2} \Bigr)
\\
&\= 2(r-1)\,\log|z-i|+\oh(1)\qquad\text{ as }z\rightarrow i\,,\\
\partial_{\bar z} Q_r(z,i)
&\= \frac {r-1}{\bar z+i} + \oh(\log |z-i|)
\qquad \text{ as }z\rightarrow i\,. \ead
The integral of the term with $dz$ in $\bigl[F,Q_r(\cdot,i)]$ over a
circle around $i$ with radius $\e$ is $\oh(\e \,\log\e) = o(1)$ as
$\e\downarrow0$. The other term gives
\[ \int_{\ph=0}^{2\pi} \, \bigl( F(i) + \oh(\e) \bigr)\, \Bigl(\frac
{r-1} {\e \, e^{-i\ph}} + \oh(\log\e)\Bigr)\, (-i\e\, e^{-i\ph})\,
d\ph \= 2\pi i (1-r) \, F(i) + o(1)\,.\qedhere \]
\end{proof}

We first illustrate a possible use of Proposition~\ref{prop-psC} in an
example, and will after that complete the proof of
Theorem~\ref{thm-bgdecomp}.\vskip.4ex

\twocolwithpictr{\quad Let $r \in \CC \smallsetminus \ZZ_{\geq 1}$.In
the situation sketched in Figure~\ref{fig-UC}, the integral
\[\frac1{2\pi i(1-r)} \int_C \bigl[ F,Q_r(\cdot,z')\bigr]\]
represents a function of $z'$ on the regions inside and outside the
simple positively oriented closed path~$C$. According to
Proposition~\ref{prop-psC} the resulting function inside $C$ is equal
to $F$, and outside $C$ we get the zero function. }{
\setlength\unitlength{1cm}\label{fig-UC}
\begin{picture}(6,3)
\put(0,0){\line(1,0){6}}
\put(3,1.5){\oval(4,2)}
\put(4,2){$U$}
\put(2.8,2.2){$C$}
\put(1.68,1.9){\vector(0,-1){.4}}
\thicklines
\put(2.3,1.8){\circle{1.1}}
\end{picture}
}
\twocolwithpictl{\label{hole}\setlength\unitlength{1cm}
\begin{picture}(6,3)
\put(0,0){\line(1,0){6}}
\put(3,1.5){\oval(4,2)}
\put(4,2){$U$}
\put(2.3,1.4){\circle{.6}}
\put(2.1,1.3){\circle*{.01}}
\put(2.1,1.4){\circle*{.01}}
\put(2.1,1.5){\circle*{.01}}
\put(2.2,1.2){\circle*{.01}}
\put(2.2,1.3){\circle*{.01}}
\put(2.2,1.4){\circle*{.01}}
\put(2.2,1.5){\circle*{.01}}
\put(2.2,1.6){\circle*{.01}}
\put(2.3,1.2){\circle*{.01}}
\put(2.3,1.3){\circle*{.01}}
\put(2.3,1.4){\circle*{.01}}
\put(2.3,1.5){\circle*{.01}}
\put(2.3,1.6){\circle*{.01}}
\put(2.4,1.2){\circle*{.01}}
\put(2.4,1.3){\circle*{.01}}
\put(2.4,1.4){\circle*{.01}}
\put(2.4,1.5){\circle*{.01}}
\put(2.4,1.6){\circle*{.01}}
\put(2.5,1.3){\circle*{.01}}
\put(2.5,1.4){\circle*{.01}}
\put(2.5,1.5){\circle*{.01}}
\put(3.5,1){$C$}
\thicklines
\put(1.6,.8){\line(1,0){2}}
\put(1.6,.8){\vector(1,0){1}}
\put(3.6,.8){\line(-1,1){1.5}}
\put(2.1,2.3){\line(-1,-1){.5}}
\put(1.6,.8){\line(0,1){1}}
\end{picture}
}{\quad
The situation is different if we let $C$ run around a hole in $U$. Now
the integral defines an $r$-harmonic function $F_{\!i}$ on the region
inside $C$ (including the hole), and a function $F_{\!o}$ outside
$C$.

\quad Since the differential form is closed, we can deform the path of
integration inside $U$, thus obtaining extensions of $F_{\!i}$ and
$F_{\!o}$ to overlapping regions inside $U$. On the intersection of
the domains Proposition~\ref{prop-psC} implies $F_{\!i}-F_{\!o}=F$. }

\begin{proof}[Completion of the proof of Theorem~\ref{thm-bgdecomp}]
Let $F$ represent an element of $\bg_r(\proj\RR)$. So $F\in
\sharm_r(U)$ for any open $U\subset \uhp$ that contains a region of
the form $1-\e<\Bigl|\frac{z-i}{z+i}\Bigr|<1$. The disk
$\Bigl|\frac{z-i}{z+i}\Bigr|\leq 1-\e$ will play the role of the
hole in Figure~\ref{hole}.

For a positively oriented simple closed curve $C$ in $U$ we have two
$r$-harmonic functions:
\bad F_{\!i}(z') &= \frac{1}{2\pi i(1-r)} \int_C
\bigl[F,Q_r(\cdot,z')\bigr]_r&&\text{ for }z'\in \uhp\text{ inside
}C\,,\\
F_{\!o}(z')&= \frac{1}{2\pi i(1-r)}\int_C
\bigl[F,Q_r(\cdot,z')\bigr]_r&&\text{ for }z'\in \uhp \text{ outside
}C\,. \ead
By moving the path closer and closer to the boundary $\proj\RR$ of
$\uhp$ we obtain that $F_{\!i}\in\sharm_r(\uhp)$. Further $F_{\!o}$
is $r$-harmonic on a region $U'\subset U$ that contains the
intersection with $\uhp$ of a neighbourhood of $\proj\RR $ in
$\proj\CC$. The function $Q_r(z,\cdot)$ represents an element of
$\W r \om(\proj\RR)$ for each $z\in \uhp$. This property is preserved
under integration. Hence $F_{\!o}$ represents an element of $\W r
\om(\proj\RR)$.

If $z'$ is in the intersection of the domains of $F_i$ and $F_o$, then
we apply the integral representation with different paths, and get
$F(z') = F_{\!i}(z')- F_{\!o}(z')$ by Proposition~\ref{prop-psC}.
This gives the desired decomposition.

Together with Proposition \ref{prop-Ei}, this implies the theorem.
\end{proof}

\subsection{Periodic harmonic functions and boundary germs} In
Definition~\ref{def-ldper} we introduced the concept of
$\ld$-periodic functions, for $\ld\in \CC^\ast$.
\il{ldperbg}{$\ld$-periodic boundary germ}We use this terminology also
for boundary germs satisfying $f|_p T = \ld\, f$, with $T=\matc1101$.
(The transformation does not depend on the weight~$p$.

\begin{lem}\label{lem-per-harm}
Put
\ir{Frndef}{\F{r,n}}
\be\label{Frndef} \F{r,n}(z)\;:=\; e^{2\pi i n z} \, y^{1-r}\,
\hypg11\bigl(1-r;2-r;4\pi n y\bigr)
\qquad(r\in \CC\setminus\ZZ_{\geq 2},\; n\in \CC)\,.\ee
{\rm (For $r=1$ we have $\F{1,n}(z)=e^{2\pi i n z}$.)}
\begin{enumerate}
\item[i)] For $r\in \ZZ_{\geq 2}$ the $\ld$-periodic elements $\ph\in
\dsv{2-r}\pol$ form the one-dimensional subspace of constant
functions if $\ld=1$, and are zero otherwise.
\item[ii)] If $F\in \sharmb r(U)$ represents a $\ld$-periodic element
of $\esv r {\om,\wdg}[\infty]$ then it has an $r$-harmonic extension
as an element of $\sharm_r(\uhp)$, and is given by a Fourier
expansion
\be \label{harm-Four}
F(z)\= \begin{cases}
\sum_{n\equiv\al(1)} c_n \, \F{r,n}(z)&\text{ if }r\in \CC\setminus
\ZZ_{\geq 2}\,,\\
\sum_{n\equiv \al(1)}c_n \, e^{2\pi i n z}
+ a_0\, y^{1-r}&\text{ if }r\in \ZZ_{\geq 2}\,,
\end{cases}
\ee
where the coefficients $c_n$ satisfy $c_n=\oh\bigl( e^{-b\,|\re
n|}\bigr)$ as $|\re n|\rightarrow\infty$, and where $a_0\in \CC$ is
equal to zero unless $\ld=1$.
\item[iii)]Let $r\in \CC\setminus \ZZ_{\geq 1}$. If $F\in
\sharm_r(\uhp)$ represents a $\ld$-periodic element of $\W
r\om(\RR)$, then $F\in \esv r {\om,\wdg}[\infty]$. {\rm (Hence the
statements in Part~ii) apply to~$F$.)}
\end{enumerate}
\end{lem}

\begin{proof}Since $\dsv{2-r}\pol$ consists of polynomials of degree
at most $r-2$, Part~i) is immediately clear.

For Parts ii) and iii) we consider first a neighbourhood $U$ of $\RR$
in $\CC$ and $F\in \sharmb r(U)$ that represents an element of $\W r
\om(\RR)$. Since $F$ is $\ld$-periodic the set $U$ contains a
strip $-\e<\im z<\e$ for some $\e>0$.

The $\ld$-periodic, $r$-harmonic function $F$ on $U\cap\uhp$ is given
by a Fourier expansion
\[ F(z) \= \sum_{n\equiv\al(1)} e^{2\pi i n x} \, f_n(y)\]
that is absolutely convergent on~$\uhp$. Since the operator $\Dt_r$
defining $r$-harmonicity commutes with translation $z\mapsto z+u$
with $u\in \RR$, all Fourier terms have the form $e^{2\pi i n x}\,
f_n(y)$ \, and are also $r$-harmonic. Hence they are in a
two-dimensional solution space.

We use that the condition $F\in \sharmb r
(U)$ is inherited by the Fourier terms. The multiples of $\F{r,n}$ are
in $\sharmb r(U)$ for some neighbourhood $U$ of $\RR$ in~$\CC$. If $r
\in \CC \smallsetminus \ZZ_{\geq 1}$, a linearly independent
element of the solution space is the holomorphic function $z\mapsto
e^{2\pi i n z}$, which does not represent an analytic boundary germ.
Thus we get for $r\not\in \ZZ_{\geq 1}$ a Fourier expansion of the
form indicated in~\eqref{harm-Four}.

For $r\in \ZZ_{\geq 1}$ we consider the functions $F(z) = e^{2\pi i n
z} \, g(y)$ for which $\shad _r F$ is
holomorphic. In all cases we can take $g(y)$ constant, and obtain the
multiples of $e^{2\pi i n z}$.

If $g$ is not constant, the condition that $F\in \sharmb r(U)$ leads
to $g(y) = y^{1-r}\, a(y)$ with a real-analytic function $a $ on a
neighbourhood of $y=0$ in~$\RR$. We find that we should have
\[ \frac{r-1}{2i}\, a(y) + y\, a'(y) \= c\, e^{4\pi n y}\]
with some $c\in \CC$. If $r=1$ this is possible only with $c=0$, and
then $a'(y)=0$. So we do not get more than indicated
in~\eqref{harm-Four}.

For $r\in \ZZ_{\geq 2}$ we take the restriction:
\[ \rsp_r \bigl(e^{2\pi i n z} \, y^{1-r}\, a(y)\bigr)(t) \=
-(-2i)^{r-2}\, (t-i)^{2-r}\, e^{2\pi i n t}\, a(0)\,.\]
This should be a $\ld$-periodic element of $\dsv{2-r}\pol$, which can
be non-zero only if $\ld=1$, and then is a constant function. This
leads to the term $a_0\, y^{1-r}$ in~\eqref{harm-Four}.\smallskip

For Part~ii) we suppose that $F$ represents an element of $\esv r
{\om,\exc}[\infty]$. This is an assumption in Part~ii). Part~i) in
 Proposition~\ref{prop-Esys} implies that $U\cap\uhp$ contains a
region of the form
\[ \bigl\{ z\in \uhp \;:\; |\re z|>\e^{-1}\bigr\} \cup \bigl\{ z\in
\uhp\;:\; \im z<\e\bigr\}\]
for some $\e>0$. The relation $F(z+\nobreak 1)=\ld\, F(z)$ allows us
to find a real-analytic continuation of $F$ to all of~$\uhp$. So
under the assumptions of Part~ii) we have $U=\CC$. In Part~iii) it is
given that $U$ contains $\uhp$. So now we know only that $U$ contains
all $z\in \CC$ with $\im z>-\e$ for some $\e>0$.\smallskip

In both parts we have the expansion \eqref{harm-Four} for all $z\in
\uhp$. This leads to information concerning the coefficients.

For $r \in \CC \smallsetminus \ZZ_{\geq 1}$, we quote
from~\cite[\S4.1.1]{Sl60} the asymptotic behavior of the confluent
hypergeometric series:
\be\label{F11as}
\hypg11\bigl( 1-r;2-r;t\bigr) \;\sim\; \begin{cases}
(1-r)\, t^{-1}\, e^t
&\text{ as }\re t\rightarrow\infty\,, \nonumber \\
\Gf(2-r)\, (-t)^{r-1}
&\text{ as }\re t \rightarrow -\infty\,.
\end{cases}
\ee
The absolute convergence of the Fourier expansion of $F$ implies the
estimate of the coefficients~$c_n$.

This gives, for $r\in \CC \smallsetminus \ZZ_{\geq 1}$  the growth of
the coefficients, and finishes the proof of Part~ii)
for $r\in \CC \smallsetminus \ZZ_{\geq 1}$ Moreover, dividing by
$y^{1-r}$ we get a Fourier expansion converging on all of $\CC$, and
representing a real-analytic function on~$\CC$. That implies that
$F(z)/f_r(z)$ is real-analytic on $\CC$, hence $F\in \sharmb r(\CC)$,
which shows that $F$ represents an element of $\esv r
{\om,\exc}[\infty]$, by Part~i) of Proposition~\ref{prop-Esys}. This
gives Part~iii).\smallskip

We are left with Part~ii) for $r\in \ZZ_{\geq 1}$. For $r=1$ we have
$(\rsp_1 F)(z) = \frac{z-i}{2i}\, F(z)$ in $\dsv1{\om,\wdg}[\infty]$.
Hence $F$ has a holomorphic extension to $\CC$. This extension is
still given by a convergent Fourier expansion, which should be the
expansion $\sum_n c_n\, \F{1,n}(z) = \sum_n c_n\, e^{2\pi i n z}$.
This convergence on $\CC$ implies the estimate of the coefficients.

Finally, if $r\in \ZZ_{\geq 2}$, then the term $\sum_n c_n\, e^{2\pi i
n z}$ is holomorphic, and hence is in $\dsv r {\om,\wdg}[\infty]$.
Again, we get convergence on all of~$\CC$. This ends the proof of
Part~ii).
\end{proof}

\subsection{Completion of the proof of
Proposition~\ref{prop-Esys}}\label{sect-cpEsys}
\begin{proof}[Proof of Part~vii) for $r
\in\CC\smallsetminus\ZZ_{\geq1}$.] The function $F\in
\sharm_r(U\cap\uhp)$ represents an $r$-har\-monic boundary germ $f\in
\bg_r(\proj\RR)$. According to Theorem~\ref{thm-bgdecomp} we
have a unique decomposition $f=P
+g$, with $P\in \sharm_r(\uhp)$ identified with the boundary germ it
represents, and $g\in \W r \om(\proj\RR)$ with representative $G=F-P$
in $\sharm_r(U\cap \uhp)$. Since $G$ represent an element of $\W r
\om(\proj\RR)$ it is an element of $\sharmb r
(U_1)$ for some neighbourhood $U_1\subset U$ of $\proj\RR$
in~$\proj\CC$.

Now
\[ f|_r(\ld^{-1} T- 1) \= P|_r(\ld^{-1} T- 1) +g |_r(\ld^{-1} T-
1)\,.\]
The left hand side is in $\esv r \om \subset \W r \om(\proj\RR)$ by
condition~b) in the assumption. So the direct sum in
\eqref{bg-decomp} shows that $\ld^{-1} \, P|_r T=P$.

Since $P=F-G $ represents an element of $ \esv r{\om,\wdg}[\infty]+ \W
r \om(\proj\RR)
\subset \W r \om(\RR)$ we can apply Part~iii) of
Lemma~\ref{lem-per-harm} to~$P$ and conclude that $P\in \esv r
{\om,\wdg}[\infty]$. Then $G = F-P$ represents an element
\begin{align*}g\in \esv r{\om,\wdg}[\infty] \cap \W r \om(\proj\RR)
&\= (\rs_r)^{-1} \Bigl(\dsv{2-r}{\om,\wdg}[\infty] \cap \Prj{2-r}^{-1}
\V{2-r}\om(\proj\RR)\Bigr)\\
&\= \rs_r^{-1} \dsv{2-r}\om \= \esv r \om\,. \qedhere\end{align*}
\end{proof}

The proof for $r\in \ZZ_{\geq 1}$ requires some preparation.

\begin{lem}\label{lem-Gz}Let $r\in \ZZ_{\geq 1}$. Suppose that $F$
representing an element of $\dsv r {\om,\wdg}[\infty]$ satisfies the
conditions in Part~vii) in Proposition~\ref{prop-Esys}. Then there is
a $\ld$-periodic function $P\in \hol(\CC)$ such that $G=F-P$ is
holomorphic on a neighbourhood of $\lhp\cup\proj\RR$ in $\proj\CC$,
and satisfies $G(\infty)=0$.
\end{lem}

\begin{proof}Condition~a) in Part~vii) of Proposition~\ref{prop-Esys}
tells us that $F$ is $r$-harmonic on $\uhp\setminus K$ for some
compact set $K\subset \uhp$. Since we have the additional information
that $F$ represents an element of $\dsv r {\om,\wdg}[\infty]$ it is
holomorphic on a neighbourhood of $\lhp\cup\RR$ in~$\CC$. Hence $F\in
\hol(\CC\setminus\nobreak K)$.

We define a holomorphic function $P$ on~$\CC$ by
\be\label{P} P(z) \= \frac1{2\pi i}\int_{|\tau|=c}\,
F(\tau)\,\frac{d\tau}{\tau-z}\,,\ee
where $c$ is chosen larger than $|z|$, and in such a way that $K$ is
enclosed by the path of integration. The function $z\mapsto
\ld^{-1}\, F(z+\nobreak 1)-F(z)$ is holomorphic on  $\CC \setminus
\left( K \cup T^{-1}K\right)$.  It represents an element of $\esv r
\om$ by Assumption~b), hence it represents an element of $\dsv r
\om$, and  is holomorphic on $\proj\CC \setminus \left( K \cup
T^{-1}K\right)$, with  at $\infty$ a
zero of order at least~$r$. Since $r\geq 1$, Cauchy's theorem implies
that
\[ \frac1{2\pi i} \int_{|\tau|=c} \Bigl( \ld^{-1}\,
F(\tau+1)-F(\tau)\Bigr)\,\frac{d\tau}{\tau-z}\=0\,,\]
for all sufficiently large~$c$. This implies that
$\ld^{-1}P(z+\nobreak 1)=P(z)$ for all $z\in \CC$.

Take $G=F-P$. With \eqref{P} we find for all sufficiently large $c$
and $|z|<c<c_1$
\begin{align*} \frac1{2\pi i}&\int_{|\tau|=c} G(\tau) \,
\frac{d\tau}{\tau-z} \= P(z) - \frac1{2\pi
i}\int_{|\tau_1|=c_1}\frac1{2\pi i}\int_{|\tau|=c}
\frac{d\tau}{(\tau_1-\tau)\,(\tau-z)}\, F(\tau_1)\, d\tau_1
\\
&\=P(z) - P(z)\=0\,.
\end{align*}
Insertion of the Laurent expansion $G(\tau) = \sum_{k\in \ZZ} b_k\,
\tau^k$ of $G$ at $\infty$ into the integral shows that $b_k=0$ for
$k\geq 0$. So the function $G$ is holomorphic on a neighbourhood of
$\infty$ with a zero at~$\infty$.
\end{proof}

\begin{proof}[Proof of Part~vii) for $r=1$]We have $\esv 1 \om=\dsv 1
\om$ and $\esv 1 {\om,\wdg}[\infty]=\dsv 1{\om,\wdg}[\infty]$; see
\eqref{wt1ED}. We apply Lemma~\ref{lem-Gz}, and use that a first
order zero at~$\infty$ suffices to conclude that~$G$ represents an
element of~$\dsv 1 \om$.\end{proof}

\begin{proof}[Proof of Part~vii) for $r\in \ZZ_{\geq2}$]We have  $H
\in\sharm_r(\uhp\setminus K)$ with a compact set
$K\subset\uhp$, for which we can arrange that $i\in K$.

We write $F=F_0+m$, with $F_0$ representing an element of $\dsv
r{\om,\wdg}[\infty]$ and $m=\sum_{\mu=1-r}^{-1} c_\mu \, \M{r,\mu}\in
\esv r \om$, with the $c_\mu$ in~$\CC$. Then $F_0$ is $r$-harmonic on
$\uhp\setminus K$, and is holomorphic on an $\{\infty\}$-excised
neighbourhood. It also satisfies Condition~b)
in Part~vii) of Proposition~\ref{prop-Esys}, so we can apply
Lemma~\ref{lem-Gz}. This gives $F_0 = G+P$, with a $\ld$-periodic
holomorphic function $P$ on~$\CC$, and $G$ holomorphic on a
neighbourhood of $\lhp\cup\proj\RR$ in $\proj\CC$ with a zero
at~$\infty$. For $G$ to represent an element of $\dsv r \om$ we would
need a zero at $\infty$ of order at least~$r$.

The function $G$ shares with $F_0$ the property that $z\mapsto
\ld^{-1} \, G(z+\nobreak 1)-G(z)$ represents an element of $\esv r
\om$, even an element of $\dsv r \om$ by holomorphy. Insertion of
this property in the power series of $G$ at~$\infty$ shows that the
zero of $G$ at~$\infty$ has order at least $r-1$. 

If $\ld=1$, we cannot reach order~$r$. In this case we write
$G=G_0+c_0\, R_r$, where
\[ R_r(z) \= \Bigl( \frac{2i}{z-i}\Bigr)^{r-1}\,,\]
and where $c_0$ is chosen such that $G_0$ has a zero at~$\infty$ of
order at least~$r$. So $G_0$ represents an element of~$\dsv r \om$.

Part~iii) of Proposition~\ref{prop-polexp} shows that $R_r(z)
= y^{1-r} + m_1$, with $m_1\in \esv r \om$. Working modulo elements of
$\esv r \om$ we have
\[ F \= F_0+m \;\equiv\; G+P \= G_0+c_0\, R_r +P \;\equiv\; c_0\,
y^{1-r} + c_0\,m_1+P \;\equiv\; c_0\, y^{1-r}+P\,. \]
The function $c_0\, y^{1-r}+P$ represents a $\ld$-periodic element of
$\esv r {\om,\wdg}[\infty]$.
\end{proof}

\subsection{Related work}\label{sect-lit8}
In \cite{BLZm} the analytic cohomology groups have values in the space
$\V s \om(\proj\RR)$ which is isomorphic to the space $\W s
\om(\proj\RR)$ by the results in~\cite[\S5.2]{BLZ13}. In the present
context we work with the subspace $\dsv{v,2-r}\om$ of
$\Prj{2-r}^{-1}\V{v,2-r}\om(\proj\RR)$, and we have to do more work
to determine a submodule of analytic boundary germs related to
$\dsv{v,2-r}\om$.

For weights $r \in \CC \smallsetminus \ZZ_{\geq 1}$ we have the
isomorphism in Part~i) of Theorem~\ref{thm-VWiso}, which points the
way to the definition of $\esv{v,r}\om$. The power series approach in
that theorem is more complicated than that in~\cite[Proposition
5.6]{BLZ13}. For weights $r \in \ZZ_{\geq 1}$ we defined
$\esv{v,r}\om$ so that it satisfies the properties in
Proposition~\ref{prop-Esys}.

Property~vii) in Proposition~\ref{prop-Esys} is similar to~\cite[Lemma
9.23]{BLZm}. It will be essential for the proof of
Theorem~\ref{sect-THMbg} in~\S\ref{sect-caf}. The proof of this
property in the case $r \in \CC \smallsetminus \ZZ_{\geq 1}$ follows
the proof in~\cite{BLZm}. It uses the boundary germ splitting in
Theorem~\ref{thm-bgdecomp} which is similar to~\cite[Proposition
5.3]{BLZ13}. The resolvent kernel $Q_r$ in \eqref{Qr-def}, the
Green's form in \eqref{Gr-f} and the Cauchy-like result
Proposition~\ref{prop-psC} have their examples in
(3.8), (3.13), Theorem~3.1 in~\cite{BLZ13}, respectively.

To prove property~vii) in Proposition~\ref{prop-Esys} for positive
integral weights we had to find other methods, which were inspired by
the use of hyperfunctions and the Poisson transformation (\S2.2 and
\S3.3 in~\cite{BLZ13}). In these notes we avoid the explicit use of
hyperfunctions and the Poisson transformation.

\section{Tesselation and cohomology}\label{sect-tesscoh}

Up till now we worked with the standard description of group
cohomology, recalled in~\S\ref{sect-coh-mpc}. For the boundary germ
cohomology we turn to the description of cohomology that turned out
to be useful in~\cite{BLZm}. We use the concepts and notations of
those notes, and do not repeat a complete discussion. We invite the
reader to have a quick look at \cite[\S6.1--3]{BLZm}, where the
approach is explained for cocompact discrete groups, and then to
consult \cite[\S11]{BLZm} for the case of groups with cusps.

\subsection{Tesselations of the upper half-plane}\label{sect-tess}The
\il{tess}{tesselation} tesselations that we use are called ``of type
\textbf{Fd}'' in~\cite{BLZm}. They are based on the choice of a
suitable fundamental domain for $\Gm\backslash\uhp$.

\rmrk{Tesselation for the modular group}With the standard choice of
the fundamental domain $\fd$ for $\Gmod\backslash\uhp$, a part of the
tesselation looks as in Figure~\ref{fig-modtess}.
\begin{figure}[ht]
\[
\newcommand\hy{Y}
\setlength{\unitlength}{2.3cm}
\begin{picture}(3,3.2)(-1.5,-.3)
\put(0,1){\circle*{.07}}
\put(-.5,.866){\circle*{.07}}
\put(.5,.866){\circle*{.07}}
\put(-.5,2.1){\circle*{.07}}
\put(.5,2.1){\circle*{.07}}
\put(-.05,1.05){$i$}
\put(-.1,2.2){$f_\infty$}
\put(.2,.8){$e_2$}
\put(-.3,.8){$S e_2$}
\put(.55,1.5){$e_1$}
\put(-.95,1.5){$T^{-1}e_1$}
\put(.55,2.3){$e_\infty$}
\put(-.95,2.3){$T^{-1}e_\infty$}
\put(.55,1.92){$P_\infty$}
\put(-.1,2.48){$V_\infty$}
\put(-.1,1.5){$\fd_\hy$}
\put(-1.5,0){\line(1,0){3.2}}
\put(-1.5,.866){\line(0,1){1.7}}
\put(1.5,.866){\line(0,1){1.7}}
\qbezier(-1.5,.866)(-1.3,1)(-1,1)
\qbezier(-.5,.866)(-.7,1)(-1,1)
\qbezier(.5,.866)(.7,1)(1,1)
\qbezier(1.5,.866)(1.3,1)(1,1)
\qbezier(0,0)(0,.45)(.5,.866)
\qbezier(1,0)(1,.45)(1.5,.866)
\qbezier(-1,0)(-1,.45)(-.5,.866)
\qbezier(0,0)(0,.45)(-.5,.866)
\qbezier(-1,0)(-1,.45)(-1.5,.866)
\qbezier(1,0)(1,.45)(.5,.866)
\qbezier(0,0)(0,.333)(.333,.333)
\qbezier(.333,.333)(.4,.333)(.5,.289)
\qbezier(0,0)(0,.333)(-.333,.333)
\qbezier(-.333,.333)(-.4,.333)(-.5,.289)
\qbezier(1,0)(1,.333)(.667,.333)
\qbezier(.667,.333)(.6,.333)(.5,.289)
\qbezier(1,0)(1,.333)(1.333,.333)
\qbezier(1.333,.333)(1.4,.333)(1.5,.289)
\qbezier(-1,0)(-1,.333)(-.667,.333)
\qbezier(-.667,.333)(-.6,.333)(-.5,.289)
\qbezier(-1,0)(-1,.333)(-1.333,.333)
\qbezier(-1.333,.333)(-1.4,.333)(-1.5,.289)
\put(-.5,.289){\line(0,1){.577}}
\put(.5,.289){\line(0,1){.577}}
\put(1.5,.289){\line(0,1){.577}}
\put(-1.5,.289){\line(0,1){.577}}
\put(-1.5,2.1){\line(1,0){3}}
\qbezier(0,.476)(.2,.476)(.225,.317)
\qbezier(0,.476)(-.2,.476)(-.225,.317)
\qbezier(-1,.476)(-.8,.476)(-.775,.317)
\qbezier(-1,.476)(-1.2,.476)(-1.225,.317)
\qbezier(1,.476)(1.2,.476)(1.225,.317)
\qbezier(1,.476)(.8,.476)(.775,.317)
\put(-.04,-.15){$0$}
\thicklines
\put(-.5,.866){\line(0,1){1.7}}
\put(.5,.866){\line(0,1){1.7}}
\qbezier(-.5,.866)(-.3,1)(0,1)
\qbezier(.5,.866)(.3,1)(0,1)
\put(-.5,2.1){\line(1,0){1}}
\put(.5,2.2){\vector(0,1){.3}}
\put(-.5,2.2){\vector(0,1){.3}}
\put(.5,2.1){\vector(-1,0){.6}}
\put(.5,1.3){\vector(0,1){.3}}
\put(-.5,1.3){\vector(0,1){.3}}
\put(.2,.983){\vector(4,-1){.17}}
\put(-.2,.983){\vector(-4,-1){.17}}
\end{picture}
\]
\caption{Sketch of a tesselation for the modular group, based on the
standard fundamental domain. } \label{fig-modtess}
\end{figure}
The tesselation \il{te}{$\tess$}$\tess$ is obtained by taking all
$\Gmod$-translates of the \il{funddom}{fundamental domain}fundamental
domain \il{fd}{$\fd$}$\fd$ divided in a \il{cutr}{cuspidal
triangle}cuspidal triangle \il{Vca}{$V_\ca$}$V_\infty$ and a compact
part~\il{fdY}{$\fd_Y$}$\fd_Y$. The set of \il{fate}{face of a
tesselation}faces is \il{Xtess}{$X^\tess_i$}$X^\tess_2=\bigl\{
\gm^{-1}V_\infty, \,\gm^{-1}\fd_Y\;:\; \gm\in\Gmod \bigr\}$. In the
boundary $\partial_2
\fd$ of the fundamental domain there are oriented edges $e_\infty$
from a point \il{Pinf}{$P_\ca$}$P_\infty=\frac12+iY$ (with some
$Y>1$) to the cusp~$\infty$, and compact edges $e_1$ from $e^{\pi
i/3}$ to $P_\infty$ and $e_2$ from $i$ to $e^{\pi i/3}$. There is
also the horizontal edge $f_\infty$ from $P_\infty$ to
$T^{-1}P_\infty$. These four edges generate a set $X_1^\tess$ of
oriented \il{edte}{edge of a tesselation}edges freely over
$\overline\Gmod=\{\pm 1\}\backslash\Gmod$. If $e\in X_1^\tess$ then
the same edge with the opposite orientation is written as $-e$. We
follow the convention used in~\cite{BLZm} to include in $X_1^\tess$
only one of the two oriented edges corresponding to a given
unoriented edge. The points $i$, $e^{\pi i/3}$, $P_\infty$ of $\uhp$
together with the cusp $\infty$ generate over $\Gmod$ the set
$X_0^\tess$ of vertices, but not freely over $\overline{\Gmod}$,
since $i$ and $e^{\pi i/3}$ are fixed by subgroups of
$\overline\Gmod$ of orders $2$ and $3$ respectively. The subgroup of
$\overline\Gmod$ fixing $P_\infty$ consists only of $ 1$, and the
group $\Gmod_\infty$ fixing $\infty$ is infinite.

We define the subsets \il{XiY}{$X^{\tess,Y}_i$}$X_i^{\tess,Y}$
consisting of all elements that are compact in $\uhp$. So
$X_0^{\tess,Y}$ is generated by $i$, $e^{\pi i/3}$, and~$P_\infty$;
$X_1^{\tess,Y}$ by $e_1$, $e_2$ and $f_\infty$; and $X_2^{\tess,Y}$
by $\fd_Y$.

\rmrk{General groups}In general, the fundamental domain $\fd$ is
chosen in such a way that its closure in $\uhp\cup\proj\RR$ contains
only one cusp of $\Gm$ from each $\Gm$-orbit of cusps. The
fundamental domain is the union of a compact part $\fd_Y$ and a
number of cuspidal triangles \il{Va}{$V_\ca$}$V_\ca$, for the cusps
$\ca$ in the closure of~$\fd$. Each $V_\ca$ has vertices $\ca$,
$P_\ca$ and $\pi_\ca^{-1}P_\ca$, and a boundary consisting of edges
\il{eca}{$e_\ca$}$e_\ca\in X_1^\tess$ from $P_\ca$ to~$\ca$,
$\pi_\ca^{-1}e_\ca$, and
\il{fca}{$f_\ca$}$f_\ca\in X_1^{\tess,Y}$ from $P_\ca$ to
$\pi_\ca^{-1}P_\ca$. So each of these cuspidal triangles looks the
same as the triangle $V_\infty$ for the modular group.

\subsection{Resolutions based on a tesselation}\label{sect-rbot} The
tesselation $\tess$ gives rise to $\Gm$-mod\-ules
\il{Fitess}{$F_i^\tess$}\il{FitessY}{$F_i^{\tess,Y}$}$F_i^\tess
:=\CC[X_i^\tess]\supset F_i^{\tess,Y}:=\CC[X_i^{\tess,Y}]$, which are
considered as \emph{right} modules, by \il{x|g}{$(x)|\gm$}$(x)|\gm =
(\gm^{-1}x)$. There are the obvious \il{bo}{boundary
operators}boundary operators \il{pai}{$\partial_i$}$\partial_i
: \CC[X_i^\tess] \rightarrow \CC[X_{i-1}^\tess]$ that satisfy
$\partial_i
\CC[X_i^{\tess,Y}]\,\subset\, \CC[X_{i-1}^{\tess,Y}]$.

For the modular group:
\[ \partial_2 (V_\infty) \= (e_\infty)|(1-T)-(f_\infty)\,,\quad
\partial_2 (\fd_Y) \= (e_1)|(1-T)+(e_2)|(1-S) +(f_\infty)\,.\]

This leads to complexes $\bigl(F_\cdot^\tess\bigr)
\supset \bigl(F_\cdot^{\tess,Y}\bigr)$ of $\Gm$-modules. It turns out
(\cite[\S11.2]{BLZm}) that for right $\Gm$-modules $V$ that are vector
spaces over $\CC$ the cohomology of the resulting complex
$\hom_{\CC[\Gm]}(F^{\tess,Y}_\pnt,V)$ is canonically isomorphic to
the group cohomology $H^\cdot(\Gm;V)$.
In working with this description of cohomology it is often useful to
identify a $\CC[\Gm]$-homomorphism $F_i^\tess=
\CC[X_i^\tess]\rightarrow V$ with the corresponding map
$c:X_i^\tess\rightarrow V$, which satisfies $c(\gm^{-1}x) = c(x)|\gm$
for all $\gm\in \Gm$, $x\in X_i^\tess$.

We use the complex $\bigl( F_\pnt^{\tess}\bigr)$ to describe the
\il{mpc1}{mixed parabolic cohomology}mixed parabolic cohomology. The
\il{mpcc}{mixed parabolic cochain}mixed parabolic cochains are
defined by\ir{CiFtess}{C^i(F^\tess_\pnt;V,W)}
\badl{CiFtess} C^i(F^\tess_\pnt;V,W) \= \Bigl\{ c:X_i^\tess\rightarrow
W\; &:\; c(x)\in V\text{ if }x\in X_i^{\tess,Y}\,,\\
&\quad c(\gm^{-1}x)=c(x)|\gm\text{ for all }\gm\in \Gm\Bigr\}\,. \eadl
A derivation can be defined by \il{deri}{$d$, $d^i$}$d^i c (x) =
(-1)^{i+1} \, c(\partial_{i+1}x)$ for $x\in X_i^\tess$. We often write
$d$ instead of $d^i$.

The space \il{ZiFtess}{$Z^i(F^\tess_\pnt;V,W)$}$Z^i(F^\tess_\pnt;V,W)$
of \il{mpcoc1}{mixed parabolic cocycles}mixed parabolic cocycles is
defined as the kernel of $d^i : C^i(F^\tess_\pnt;V,W)\rightarrow
C^{i+1}(F^\tess_\pnt;V,W)$ and the subspace of \il{mpcob1}{mixed
parabolic coboundary}mixed parabolic coboundaries
\il{BiFtess}{$B^i(F^\tess_\pnt;V,W)$}$B^i(F^\tess_\pnt;V,W)$ as
$d^{i-1}C^{i-1}(F^\tess_\pnt;V,W)$ if $i\geq 1$ and as the zero
subspace if $i=0$. Then the cohomology groups of the complex,
\be\label{mpchi} Z^i(F^\tess_\pnt;V,W) \bigm/
B^i(F^\tess_\pnt;V,W)\,,\ee
are for $i=1$ isomorphic to the mixed parabolic cohomology groups
$\hpar^1(\Gm;V,W)$ in Definition~\ref{mpcg}. In \cite[\S11.3]{BLZm}
the mixed parabolic cohomology groups
\il{hpar1mpc}{$\hpar^i(\Gm;V,W)$}$\hpar^i(\Gm;V,W)$ are defined as
the spaces in~\eqref{mpchi} for all~$i$.

In particular for $i=1$ we have the following commutative diagram for
$\CC[\Gm]$-modules $V\subset W$:
\[\xymatrix{ Z^1(\Gm;V)/B^1(\Gm;V) &
Z^1(F^{\tess,Y}_\cdot;V)/B^1(F^{\tess,Y}_\cdot;V)
\ar[l]_{\cong}\\
Z^1(\Gm;V,W)/B^1(\Gm;V) \ar@{^{(}->}[u]&
Z^1(F^{\tess}_\cdot;V,W)/B^1(F^{\tess}_\cdot;V,W)
\ar@{^{(}->}[u] \ar[l]_{\cong} }\]
The isomorphic spaces in the top row give two isomorphic descriptions
of $H^1(\Gm;V)$, and the two spaces in the bottom row of
$\hpar^1(\Gm;V,W)$.

The conditions on the tesselations are such that the action of
$\bar\Gm=\{\pm 1\}\backslash\Gm$ on $X_1^\tess$ and $X_2^\tess$ is
free on finitely many elements. So for $i\geq 1$ the cochains $c\in
C^i(F^\tess_\pnt;V,W)$ are determined by their values on a
$\CC[\Gm]$-basis of $\CC[X_i^\tess]$. For the modular group, $c\in
C^1(F^\tess_\pnt;V,W)$ is completely determined by $c(e_1)$, $c(e_2)$
and $c(f_\infty)$ in $V$, and $c(e_\infty)\in W$, and $b\in
C^2(F^\tess_\pnt;V,W)$ is determined by $c(\fd_Y)\in V$ and
$c(V_\infty)\in W$. For $i=0$ there are in general no bases over
$\CC[\bar \Gm]$. The fact that each cusp is fixed by an infinite
subgroup of $\Gmod$ makes the difference between parabolic cohomology
and standard group cohomology. Points of $X_0^{\tess,Y}$ may be fixed
by non-trivial finite subgroups of $\bar\Gm$. As long as we work with
$\Gm$-modules that are vector spaces over $\CC$ this is not
important. For $\overline{\Gmod}$ it suffices if we can divide by $2$
and $3$ in the modules that we use.\smallskip

For a cocycle $c\in Z^1(F^\tess_\pnt;V,W)$ the value $c(p)$ on a cycle
$p\in \ZZ[X_1^\tess]$ corresponding to a path from $P_1$ to $P _2$
(both in $X_0^\tess$) does not depend on the choice of the path along
edges in $X_1^\tess$, only on the end-points $P_1$ and $P_2$. So we
can write $c(p)=c(P_1,P_2)$, and view $c$ as a function on
$X_0^\tess\times X_0^\tess$. In general $c(p)\in W$. It satisfies
\badl{crel} c(P_1,P_3) &\= c(P_1,P_2)+c(P_2,P_3)&&\text{ for }P_j\in
X_0^\tess\,,\\
c(\gm^{-1}P_1,\gm^{-1}P_2)&\= c(P_1,P_2)|\gm&&\text{ for }\gm\in
\Gm\,,\; P_j\in X_0^\tess\,. \eadl
If both $P_1$ and $P_2$ are in $X_0^{\tess,Y}$, then the path can be
chosen in $\ZZ[X_0^{\tess,Y}]$, and hence $c(p)
\in V$.

Now choose a base point $P_0 \in X_0^{\tess,Y}$. Then $\ps_\gm=
c(\gm^{-1}P_0,P_0)$ is in $V$ for each $\gm\in \Gm$. It turns out to
define a group cocycle $\ps\in Z^1(\Gm;V)$. It is even a mixed
parabolic group cocycle in $\zpar^1(\Gm;V,W)$. Let us check this in
the situation of the modular group, with the tesselation discussed
above. Then
\begin{align*} \ps_T &\= c(T^{-1}P_0,P_0)\=
c(T^{-1}P_0,T^{-1}P_\infty)+ c(T^{-1}P_\infty,T^{-1}\infty) \\
&\quad \hbox{}
+ c(\infty,P_\infty) + c(P_\infty,P_0)
\= \bigl( -c(e_\infty)+c(P_\infty,P_0)\bigr)\bigm|( 1-T )\,\in\,
W|(1-T)\,.
\end{align*}
This computation shows that the presence of $\infty$ as a vertex of
the tesselation forces parabolicity of the cocycle. (We use $|$ to
denote the action of $\Gm$ on the $F_i^\tess=\CC[X_i^\tess]$, as well
as in the modules $V$ and~$W$.)

On can check that this association $c\mapsto \ps$ sends coboundaries
to coboundaries and that taking a different base point $P_0$ does not
change the cohomology class. The map $c\mapsto \psi$ is an easy way
to describe the canonical isomorphism between the description of
cohomology with a tesselation and the standard description of group
cohomology.

\subsection{Cocycles attached to automorphic forms}\label{sect-af-bgc}
To describe the linear maps \il{coh1}{$\coh r \om$}$\coh r \om:
A_r(\Gm,v)\rightarrow H^1(\Gm;\dsv{v,2-r}\om)$ and \il{bcoh1}{$\bcoh
r \om$}$\bcoh r \om: A_r(\Gm,v)\rightarrow H^1(\Gm;\esv{v,r}\om)$ in
Theorem~\ref{THMac} and Proposition~\ref{prop-bcoh} in the approach
to cohomology based on a tesselation $\tess$ we use for an
unrestricted holomorphic automorphic form $F\in A_r(\Gm,v)$ the
cocycles $\ps_F\in Z^1(F_\pnt^{\tess,Y};\dsv{v,2-r}\om)$ and $c_F \in
Z^1(F_\pnt^{\tess,Y};\esv{v,r}\om)$ given on edges $x\in
X_1^{\tess,Y}$ by\ir{psF}{\ps_F}\ir{cF}{c_F}
\begin{align}\label{psF}
\ps_F (x;t) &\;:=\; \int_{\tau\in x} \om_r(F;t,\tau) &\=&
\int_{\tau\in x}
(\tau-t)^{r-2}\, F(\tau)\, d\tau\,,\\
\label{cF}
c_F(x;z) &\;:=\; \int_{\tau \in x} K_r(z;\tau)\,F(\tau)\, d\tau&\=&
\int_{\tau \in x} \frac{2i}{z-\tau}\,\Bigl( \frac{\bar z-\tau}{\bar
z-z}\Bigr)^{r-1}\, F(\tau)\,d\tau\,.
\end{align}
The orientation of the edge $x$ determines the direction of the
integration. We use the boundary germ cohomology in the next section,
and hence will work with the cocycle~$c_F$. Property~iv) in
Proposition~\ref{prop-Esys} implies that $c_F$ has values in $\esv r
\om$.

\rmrk{Notations} Let $\esv r {\fsn,\wdg} =\indlim \esv
r{\om,\wdg}[\ca_1,\ldots,\ca_n]$, where $\{\ca_1,\ldots,\ca_n\}$ runs
over the finite subsets of cusps of~$\Gm$.

The space $\esv r{\fs,\wdg}$, defined in Proposition~\ref{prop-Esys},
is invariant under the operators $|_r g$ with $g\in \SL_2(\RR)$, but
these operators act in $\esv r {\fsn,\wdg}$ only if $g$ maps cusps to
cusps. By \il{esvfsn}{$\esv{v,r}\ast$}$\esv{v,r}\om$,
$\esv{v,r}{\fsn,\wdg}$, and $\esv{v,r}{\fs,\wdg}$ we denote the
$\Gm$-modules for the action $|_{v,r}$ on the corresponding spaces.

\rmrk{Image of $\bcoh r \om$} For edges in $X_1^\tess\setminus
X_1^{\tess,Y}$ the integration does not make sense, unless $F$
happens to be a cusp form. To extend $c_F$ to $X_1^\tess$ we need to
define $c_F(e_\ca)$ for each cusp $\ca$ of~$\Gm$ such that
$c_F(\partial_2 V_\ca)=0$, in the notation of~\S\ref{sect-tess}.

For weights $r\in \CC\setminus \ZZ_{\geq 2}$ and the highest weight
spaces in Definition~\ref{esv-def-gen}, Theorem~\ref{thm-imafwdg}
implies $\bcoh r \om A_r(\Gm,v) \subset \hpar^1(\Gm;\esv
{v,r}\om,\esv{v,r}{\fsn,\wdg})$. For $r\in \ZZ_{\geq
2}$ we will see in \S\ref{sect-imaf} that not all automorphic forms
give rise to mixed parabolic cocycles with values in the analytic
boundary germs.

\subsection{ Derivatives of $L$-functions}\label{sect-Lf}
In the introduction we mentioned that derivatives of $L$-functions can
be related to cocycles. We illustrate this here by an example.

Let $f$ be a newform of weight $2$ for $\Gamma=\Gamma_0(N)$ such that
$L_f(1)=0$ (under the assumption that $f$ is even for the Fricke
involution). Set
$$u(z)=\log(\eta(z) \eta(Nz)), \qquad z \in \mathfrak H.$$
Then, as shown in \cite{Go95},
\begin{equation}
L_f'(1)=\frac{1}{\pi} \int_0^{\infty} f(iy)u(iy)dy. \label{L}
\end{equation}
This integral, though reminiscent of a period integral, has an
integrand that is far from $\Gamma$-invariant and thus does not give
a cocycle. To address this problem, we first note that, by the
defining formula of $u(z)$, the RHS of \eqref{L} equals the value at
$r=0$ of the derivative
$$\frac{d}{dr} \left ( \frac{1}{\pi} \int_0^{\infty} f(iy) \left (
\eta(iy) \eta(i N y) \right )^r dy \right )\biggr|_{r=0}.$$
This integral is still not $\Gamma$-invariant but now it can be
formulated in terms of cocycles considered in these notes.

Set
$$f_r(z)=f(z) (\eta(z) \eta(Nz))^r\,.$$
This is a cusp form of weight $2+r$ for $\Gamma$ depending
holomorphically on $r$ on a neighbourhood of $0$ in~$\CC$. The
  corresponding multiplier system $w_r$ is also holomorphic in~$r$.

We refine the tesselation in Figure~\ref{fig-modtess} so that the
geodesic from $0$ to $\infty$ is a sum of edges, forming a path $p\in
\ZZ[X_1^\tess]$. (Then the faces $V_\infty$ and $\fd_Y$ are each
divided into two faces.) We have for any automorphic form~$F$ the
value
$$ \ps_F(p;t)=\int_{\tau\in p} \om_r(F;t,\tau)\,.$$
Applying this to $f_r$ defined above, we obtain $ \ps_{f_r}(p;\cdot)
\in \dsv{w,-r}{\om,\infty}[0,\infty]$, since $f_r$ is a cusp form. In
particular
\be \ps_{f_r}(p;0) \= i\, e^{\pi ir/2}\,\int_0^\infty f(iy)\, e^{ r\,
u(iy)}\,y^r\, dy\,.
\ee
 With the change of variables $y\mapsto 1/Ny$, using the invariance of
 $f_r$ under the Fricke involution, this can be seen to be equal to
$$ - i\,N^{-r/2}\, e^{\pi i r}\, \int_0^\infty f(iy)\, e^{r\,u(iy)}\,
dy\,.$$
With Goldfeld's result we obtain the following relation between the
cocycle $\ps_{f_r}$ and the $L$-function:
\bad \ps_{f_r}(p;0) &\= -i\, L_f(1)
 + r\, \Bigl(\frac {\pi i}2 \,\log N\, L_f(1) -\pi^2\, L_f(1) - \pi
 i\,L'_f(1)\Bigr) \\
 &\qquad\hbox{}+\oh(r^2)\quad(r\rightarrow0)\\
&\=-\pi i r\, L'_f(1) + \oh(r^2)
\quad(r\rightarrow0) \,.\ead

\subsection{Related work}\label{sect-lit9}The general approach to
group cohomology via an arbitrary projective resolution is well
known. See for instance, in Brown~\cite{Bro82}, Chap.~III, \S1, for
the definition, and Chap.~I, \S5, for the standard complex. Also more
topologically oriented complexes are well known; see for instance
\cite[\S4, Chap.~1]{Bro82}. In \cite{BLZm} the tesselations of the
upper half-plane based on a fundamental domain of the discrete group
in question turned out to be useful.

\section{Boundary germ cohomology and automorphic
forms}\label{sect-caf}

\subsection{Spaces of global representatives for highest weight
spaces}\label{sect-gr}
Property~i) in Proposition~\ref{prop-Esys} shows that elements of
$\esv r \om$ are represented by elements of $\sharmb r(U)$, hence by
$r$-harmonic functions on $U\cap\uhp$. Property~v) shows that if
$F\in \sharmb r(U)$ is non-zero, then $U\supset\uhp$ is impossible.
For the cohomological manipulations in this section it is desirable
to have spaces of representatives that are defined on~$\uhp$. If
non-zero, these functions cannot be $r$-harmonic everywhere
on~$\uhp$.

\begin{defn}\label{Gr-def}
We define the spaces $\Gr r \om$, $\Gr r \fs$  and $\Gr r \fsn$
of functions on $\uhp$:\ir{Gromfs}{\Gr r \om,\; \Gr r \fs,\; \Gr r \fsn} 
\begin{align}
\label{Gromfs}
\Gr r \om \;:=\; \Bigl\{& F \in C^2(\uhp) \;:\; \text{there exists an
open neighbourhood $U$}\\
\nonumber
&\text{ of $\proj\RR$ in $\proj \CC$ such that } F|_{U\cap\uhp} \text{
is in } \sharmb r (U)
\\
\nonumber
&\text{ and represents an element of }\esv r \om \Bigr\}\,,
\displaybreak[0]\\
\Gr r {\fs,\wdg} \;:=\; \Bigl\{& F \in C^2(\uhp) \;:\; \text{there
exists an excised neighbourhood $U$}\\
\nonumber
&\text{such that }F|_{U\cap\uhp} \text{ is in } \sharmb r (U)
\text{ and represents}\\
\nonumber
&\text{an element of }\esv r {\fs,\wdg} \Bigr \}\,,
\displaybreak[0]\\
\Gr r {\fsn,\wdg} \;:=\; \Bigl\{& F \in C^2(\uhp) \;:\; \text{there
exists an excised neighbourhood $U$}\\
\nonumber
&\text{with excised set consisting of cusps, such that }F|_{U\cap\uhp}
\\
\nonumber
&\text{is in } \sharmb r (U)
\text{ and represents an element of }\esv r {\fsn,\wdg} \Bigr \} \,.
\end{align}
The operators $|_r g$ with $g\in \SL_2(\RR)$ act in $\Gr r \om$ and
$\Gr r {\fs,\wdg}$, and in $\Gr r {\fsn,\wdg}$ if $g\in \Gamma$.
\il{Grvr}{$\Gr{v,r}\ast$}By $\Gr{v,r}\om$, $\Gr{v,r}{\fs,\wdg},$
$\Gr{v,r}{\fsn,\wdg}$ we denote the corresponding $\Gm$-modules with
the action $|_{v,r}$.
\end{defn}

\rmrks
\itmi This definition formalizes for $\esv r \om$ what we did
informally for $\dsv{2-r}\om$ in Remark~\ref{2interpr},~b).

\itm While, by Part~v) of Proposition~\ref{prop-Esys},
\be\label{H-Gom-0} \sharm_r(\uhp) \,\cap\, \Gr r \om \= \{0\}\,, \ee
the space $\sharm_r(\uhp) \cap \Gr r {\fs,\wdg}$ contains non-zero
elements, for instance the functions $\F{r,n}$ in~\eqref{Frndef}.

\begin{defn}\label{Nrdef}
We define \il{Nr}{$\N r\om $, $\N r{\fs,\wdg}$}$\N r \om$, $\N r
{\fs,\wdg}$, or $\N r {\fsn,\wdg}$ as the kernels of the natural maps
$\Gr r \om \rightarrow \esv r \om$, $\Gr r{\fs,\wdg}\rightarrow\esv
r{\fs,\wdg}$, or $\Gr r {\fsn,\wdg}\rightarrow\esv r{\fsn,\wdg}$
which assign to $F$ the boundary germ represented by it.
\end{defn}

\begin{prop}\label{prop-NG}
\begin{enumerate}
\item[i)] $\N r \om$, and $\N r {\fs,\wdg}$ are invariant under the
operators $|_r g$ with $g\in \SL_2(\RR)$, and the action $|_{v,r}$
makes $\N r {\fsn,\wdg}$ into a $\Gm$-module $\N {v,r}{\fsn,\wdg}$.
\item[ii)] The space $\N r \om$ is the space $C_c^2(\uhp)$ of the
twice differentiable compactly supported functions on $\uhp$, and $\N
r {\fs,\wdg}$, respectively $\N r {\fsn,\wdg}$, is the space of the
twice differentiable functions on~$\uhp$ with support contained in a
set $\uhp\setminus U$ where $U$ is an excised neighbourhood of
$\proj\RR$, which in the case of $\N r {\fsn,\wdg}$ has an excised
set consisting of cusps.
\item[iii)] The diagram of $\Gm$-equivariant maps
\[ \xymatrix{ 0 \ar[r] & \N {v,r} \om \ar[r] \ar@{^{(}->}@<-1ex>[d] &
\Gr{v,r} \om \ar[r] \ar@{^{(}->}@<-1ex>[d] & \esv {v,r} \om \ar[r]
\ar@{^{(}->}[d]
&0
\\
0 \ar[r] & \N {v,r} {\fs,\wdg} \ar[r] & \Gr {v,r} {\fs,\wdg} \ar[r] &
\esv {v,r} {\fs,\wdg} \ar[r] &0 } \]
commutes. The rows are exact sequences.
\end{enumerate}
\end{prop}

\begin{proof}
Part~i) follows directly from Definition~\ref{Nrdef}.

For Part~ii), suppose that $F$ is in the kernel of $\Gr r
\om\rightarrow \esv r \om$, $\Gr r {\fs,\wdg}\rightarrow\esv
r{\fs,\wdg}$, or $\Gr r {\fsn,\wdg}\rightarrow\esv r{\fsn,\wdg}$,
then $F=0$ on a set $U\cap \uhp$ with $U$ a neighbourhood of
$\proj\RR$ in~$\proj\CC$, or $U\cap\uhp=U_0\cap \uhp$ for an excised
neighbourhood $U_0$ of~$\proj\RR$. In the former case $\uhp \setminus
U$ is relatively compact in~$\uhp$, hence $F$ has compact support. In
the latter case $F$ is zero on an excised neighbourhood intersected
with~$\uhp$.

For the exactness in Part~iii) we need to prove the surjectivity of
the linear maps $\Gr r \om\rightarrow \esv r \om$ and $\Gr r
{\fs,\wdg} \rightarrow \esv r {\fs,\wdg}$. The commutativity of the
diagram is clear.

We start with a representative $F\in \sharm_r^b(U)$ of an element of
$\esv r \om$, respectively $\esv r{\fs,\wdg}$, where $U$ is a
neighbourhood of $\proj\RR$ in $\proj\CC$, respectively contained in
an excised neighbourhood $U_0$ of $\proj\RR$ such that $U\cap\uhp =
U_0\cap\uhp$. We take smaller sets $U_1\subset U_2 \subset U$ such
that $U_2$ is a neighbourhood of the closure of $U_1$ and $U$ is a
neighbourhood of the closure of $U_2$, and consider a cut-off function
 $\ps\in C^2(\uhp)$ equal to $1$ on $U_1$ and equal to $0$ on
$\uhp\setminus U_2$. Then $z\mapsto\ps(z)\, F(z)$, extended by $0$,
is an element of $\Gr r \om$, respectively $\Gr r{\fs,\wdg}$,
representing the same boundary germ as~$F$.
\end{proof}

\begin{lem}\label{lem-singTe}Let $\ld\in \CC^\ast$. If $h\in \esv
r{\fs,\wdg}$ and $\ld^{-1} h|_r T - h \in \esv r \om$, then $h\in
\esv r{\om,\wdg}[\infty]$.
\end{lem}
\begin{proof}In the same way as for Lemma~\ref{lem-sing-inv}.
\end{proof}

\begin{defn}
\begin{enumerate}
\item[i)] For $f\in \esv r {\fs,\wdg}$ we denote by
\il{bsing1}{$\bsing$}$\bsing f$ the minimal set
$\{\xi_1,\ldots,\xi_n\}$ such that $f\in\esv r
{\om,\wdg}[\xi_1,\ldots,\xi_n]$, and call it the set of
\il{bs1}{boundary singularity}\emph{boundary singularities} of~$f$.
For $F\in \sharmb r (U)$ we denote by $\bsing F$ the set $\bsing f$
for the boundary germ $f$ represented by~$F$.
\item[ii)]For any twice differentiable function~$F$ on~$\uhp$ we
denote by \allowbreak\il{sing}{$\singr$}$\singr F$ the complement of
the maximal open set in~$\uhp$ on which $\Dt_r F=0$, and call it the
set of \il{singu}{singularity}\emph{singularities} of~$F$.
\item[iii)] Analogously we define \il{Gr[}{$\Gr r
{\om,\exc}[\xi_1,\ldots,\xi_n]$}$\Gr r
{\om,\exc}[\xi_1,\ldots,\xi_n]$ as the set of $F\in \Gr r{\fs,\exc}$
representing an element of $\esv r{\om,\exc}[\xi_1,\ldots,\xi_n]$,
and $\bsing _r F$ for $F\in \Gr r {\fs,\exc}$ as the set of boundary
singularities of the element of $\esv r{\fs,\exc}$ that $F$
represents.
\end{enumerate}
\end{defn}

\rmrks \itmi $\bsing F\subset \proj\RR$ and $\singr F\subset \uhp$ for
each $F\in \Gr r{\fs,\wdg}$.

\itm For elements of $\dsv {2-r}\fs$ we dealt only with boundary
singularities, and often called them singularities. For $\esv r
{\fs,\wdg}$ and $\Gr r{\fs,\wdg}$ it is important to distinguish both
types of singularities.

\itm The properties in Proposition~\ref{prop-Esys} imply properties of
sets of boundary singularities. For instance
\be \bsing (f|_r g) \= g^{-1}\,\bsing f\qquad \text{for }g\in
\SL_2(\RR)\,.\ee

\itm If $F\in \Gr r \om$ then $\singr F$ is a compact subset of
$\uhp$, and if $F\in \Gr r {\fs,\wdg}$ then $\singr F$ is contained
in an excised neighbourhood.

\rmrk{$\ld$-periodic elements}
\begin{defn}\label{Irld-def}For $\ld\in \CC^\ast$, put \il{Irld}{$\I r
\ld$}$\I r \ld :=\bigl\{ f\in \esv r {\fs,\wdg}\;:\;f|_r T = \ld \,
f\bigr\}$.
\end{defn}

\begin{lem}\label{lem-ENinv}Let $\ld\in \CC^\ast$.
\begin{enumerate}
\item[i)] Each element of $\I r \ld$ is represented by a unique
$\ld$-periodic function in $\sharm_r(\uhp)
\cap \Gr r {\fs,\wdg}$.
\item[ii)] If $F\in \N r {\fs,\wdg}$ is $\ld$-periodic, then $F=0$.
\end{enumerate}
\end{lem}

\begin{proof}Let $F\in \Gr r {\fs,\wdg}$ represent an element of $\I r
\ld$. Then it represents an element of $\esv r {\om,\wdg}[\infty]$ by
Lemma~\ref{lem-singTe}. \vskip.4em

\twocolwithpictr{\quad{}From Part~i) in Proposition~\ref{prop-Esys} we
see that $F\in \sharm_r(U\cap \uhp)$ for an excised neighbourhood $U$
with excised set $\{\infty\}$, and that $\ld^{-1} F(z+\nobreak
1)=F(z)$ for all $z\in \uhp \cap U \cap T^{-1}U$.

\quad This implies that $F$ can be analytically extended to give a
$\ld$-periodic element of $ \sharm_r(\uhp)$. Since this analytic
extension is determined by its values on a strip $0<\im z<\e$ it is
unique.} { \setlength\unitlength{1cm}
\begin{picture}(6,4.5)(-3,-.6)
\put(-3,0){\line(1,0){6}}
\put(-2,2){$U$}
\put(0,-.5){$U$}
\put(-.5,2){\text{singularities}}
\thicklines
\put(-1,1){\line(0,1){3}}
\put(-1,1){\line(1,0){3}}
\put(2,1){\line(0,1){3}}
\end{picture}
}\vskip.4em If the $\ld$-periodic function $F$ represents an element
of $\N r {\fs,\wdg}$ then $F$ is zero on~$U$, hence the extension is
zero.
\end{proof}

\begin{lem}\label{lem-Ies}Let $\ca$ be a cusp of $\Gm$. Denote by
\il{Gmca}{$\Gm_\ca$}$\Gm_\ca$ the subgroup of $\Gm$ fixing~$\ca$. Then
the following sequence is exact:
\[ 0 \rightarrow \bigl( \N {v,r} {\fs,\wdg}\bigr)^{\Gm_\ca}
\rightarrow \bigl(\Gr {v,r} {\fs,\wdg}\bigr)^{\Gm_\ca} \rightarrow
\bigl(\esv {v,r} {\fs,\wdg}\bigr)^{\Gm_\ca} \rightarrow 0 \;.\]
\end{lem}

\begin{proof}The group $\Gm_\ca$ is generated by $\pi_\ca=g_\ca T
g_\ca^{-1}$. By conjugation we can reduce the statement of the lemma
to the exactness of the sequence one obtains if one takes in the
sequence
\[ 0\rightarrow \N r{\fs,\wdg} \rightarrow \Gr r{\fs,\wdg} \rightarrow
\esv r{\fs,\wdg}\rightarrow 0\]
the kernel of the operator $ |_r (\ld^{-1} T
-\nobreak 1)$ with $\ld=v(\pi_\ca)$. This does not necessarily produce
an exact sequence, but here we get by Lemma~\ref{lem-ENinv} the
sequence
\[ 0\rightarrow 0 \rightarrow \I r \ld \rightarrow \I r \ld\rightarrow
0\,,\]
which is exact. (We identify $\I r \ld$ with the space of the harmonic
representatives in Part~i) of Lemma~\ref{lem-ENinv}.)
\end{proof}

\begin{lem}\label{lem-rem-sing}
Let $\ld\in \CC^\ast$. Suppose that the function $F$ on~$\uhp$
satisfies:
\begin{enumerate}
\item[a)] $F\in \Gr r{\fs,\wdg}$,
\item[b)] $\singr F$ is a compact subset of $\uhp$,
\item[c)] $z\mapsto \ld^{-1}\, F(z+1)- F(z)$ is an element of $\Gr r
\om$.
\end{enumerate}
Then $F = P + G$ with $P\in \I r \ld$ and $G\in \Gr r \om$.
\end{lem}

\begin{proof}The open set $\uhp \setminus \singr F$ is of the form
$U\cap\uhp$ with $U$ a neighbourhood of $\proj\RR$ in $\proj\CC$. The
restriction $f$ of $F$ to $U\cap \uhp$ represents an element of $\esv
r {\fs,\wdg}$.  Assumptions a) and~c) imply  that $f$
represents an element of $\esv r{\om,\wdg}[\infty]$, by
Lemma~\ref{lem-singTe}. Part~vii)
in Proposition~\ref{prop-Esys} implies the existence of $p\in \I r
\ld$ such that $g=f-p \in \esv r \om$. Taking $P$ as the global
representative of $p$ in Part~i) of Lemma~\ref{lem-ENinv} we get a
representative $G:=F-P$ of $g$ in~$\Gr r \om$.
\end{proof}

\subsection{From parabolic cocycles to automorphic
forms}\label{sect-cocaf} Now we start with a mixed parabolic cocycle
and construct a corresponding holomorphic automorphic form.

\begin{prop}\label{prop-alr}
\begin{enumerate}\item[i)]
If $c\in Z^1\bigl(F_\pnt^\tess;\esv{v,r}\om, \esv{v,r}{\fsn,\wdg})$
there is $u([c],\cdot) \in
A_r(\Gm,v)$ that depends only on the cohomology class of $c$, and
$[c]\mapsto u([c],\cdot)$ defines a linear map\ir{alr}{\al_r}
\be\label{alr}
\al_r : \hpar^1(\Gm;\esv {v,r} \om,\esv {v,r} {\fsn,\wdg})
\rightarrow A_r(\Gm,v)\,. \ee
\item[ii)] Let $F\in A_r(\Gm,v)$ such that $\bcoh r \om F \in
\hpar^1(\Gm;\esv {v,r} \om,\esv {v,r} {\fsn,\wdg})$, then
$u([c_F],\cdot)=F$, with $c_F$ as in \eqref{cF}.
\end{enumerate}
\end{prop}

\begin{proof}The proof is almost identical to that of
~\cite[Proposition 12.2]{BLZm}. Table~\ref{tab-cor} compares the
analogous quantities.
\begin{table}[ht]\renewcommand\arraystretch{1.3}
\begin{tabular}{|l|l|}\hline
\multicolumn{1}{|c|}{{\bf holomorphic
forms}}&\multicolumn{1}{|c|}{{\bf Maass forms}}\\ \hline
$\Gm$-module $\esv{v,r}\om $ &$\Gm$-module $\W s \om$ \\
$\Gm$-module $\esv{v,r}{\fsn,\wdg}$ &$\Gm$-module $\W s {\fs,\wdg} $\\
cocycle $c $ & cocycle $\ps$\\
cochain $\tilde c$ & cochain $\tilde\ps$\\
$\Gr{v,r}\om,\; \Gr{v,r}{\fsn,\wdg}$& $\Gr s \om,\; \Gr s
{\fs,\wdg}$\\
$\N{v,r}\om,\; \N{v,r}{\fsn,\wdg}$& $\N s \om,\; \N s {\fs,\wdg}$\\
$u([c],\cdot)$ & $u_\ps$
\\ \hline
\end{tabular}\smallskip
\caption{Correspondence between the quantities in the proof here, for
holomorphic automorphic forms, and the quantities in the proof
of~\cite[Proposition 12.2]{BLZm}, for Maass forms of weight $0$ and
more general invariant eigenfunctions. Here we work with boundary
singularities restricted to the cusps, whereas in~\cite{BLZm} the
singularities were general at first, and had to be reduced to
singularities in cusps by an additional step.}\label{tab-cor}
\end{table}
Instead of repeating the proof, we give below a discussion of the main
ideas in the context of the modular group. There is one complication,
which is not present in~\cite{BLZm}. We handle it in
Lemma~\ref{lem-harmhol}.
\end{proof}

\rmrk{Lift of the cocycle}Let $c\in Z^1(F^\tess_\pnt;\esv
{v,r}\om,\esv{v,r}{\fsn,\wdg})$ be given. Its values $c(x)$ on $x\in
X_1^\tess$ are boundary germs in $\esv{v,e}{\fsn,\exc}$. See the
right column in the diagram in Proposition~\ref{prop-NG}. We want to
lift $c$ to a cochain $\tilde c\in C^1(F_\cdot^\tess;\Gr
{v,r}\om,\Gr{v,r}{\fsn,\exc})$, which involves the central column in
the diagram. For each $x$ in the $\CC[\Gmod]$-basis
$\{e_1,e_2,f_\infty\}$ of $F^{\tess,Y}_1$ we can, according to
Proposition~\ref{prop-NG} choose a representative $\tilde c(x) \in
\Gr {v,r}\om$ of $c(x) \in \esv{v,r}\om$. For $c(e_\infty)$ we can
choose a representative $\tilde c(e_\infty)\in \Gr {v,r}{\fsn,\wdg}$.
Since $\tilde c(e_\infty)|_{v,r}(1-\nobreak T)$ represents
$c(e_\infty)|_{v,r}(1-\nobreak T) = c(f_\infty)\in \esv{v,r}\om$, we
have $\bsing \tilde c(e_\infty)
\subset\{\infty\}$ by Lemma~\ref{lem-singTe}. So $\tilde c $ is
determined by
\badl{lift-c} \tilde c(e_1),\;\tilde c(e_2),\; \tilde c(f_\infty) \in
\Gr{v,r}\om\,,
&\text{ \ representatives of }c(e_1),\;c(e_2),\;c(f_\infty)\,,\\
\tilde c(e_\infty)\in \Gr{v,r}{\om,\wdg}[\infty]\,,&\text{ \
representative of } c(e_\infty)\,. \eadl

For each $x\in \{e_1,e_2,f_\infty\}$ the set $\singr \tilde c(x)$ is
compact. So we can find $R>0$ such that for each of these three edges
the set $\singr \tilde c(x)$ is contained in the $R$-neighbourhood
(for the hyperbolic distance)
of~$x$. Furthermore $\singr \tilde c(e_\infty)$ is contained in the
complement of an excised neighbourhood with excised set $\{\infty\}$,
hence
\[ \singr \tilde c(e_\infty) \subset \bigl\{ z\in \uhp\;:\; |\re
z|\leq \e^{-1}\,,\im z>\e\bigr\}\,,\]
for some $\e>0$.

Since $\tilde c$ is given on a basis of $F_1^\tess$, we can extend it
$\CC[\Gmod]$-linearly, and obtain a cochain \il{tldc}{$\tilde c$
representing $c$}$\tilde c \in
C^1(F^\tess_\pnt;\Gr{v,r}\om,\Gr{v,r}{\fsn,\wdg})$. There is no
reason for the lift $\tilde c$ to be a cocycle. For any $y\in
X_1^\tess$ and any $\gm\in \Gm(1)$ we have $\singr \tilde
c(\gm^{-1}y) = \gm^{-1}\,\singr \tilde c(y)$. So for $y\in
X_1^{\tess,Y}$ the set $\singr \tilde c(y)$ is contained in the
$R$-neighbourhood of $x$. Similarly, $\singr \tilde
c(\gm^{-1}e_\infty)$ is contained in
\be \label{wdg}
\bigl\{ \gm^{-1} z \;:\; |\re z|\leq \e^{-1}\,,\im z>\e\bigr\}\,.\ee
This means that the singularities of any $\tilde c(y)$ cannot be ``too
far'' from the edge $y\in X_1^\tess$.

\rmrk{Construction}We start the construction of an automorphic form.
First we work on a connected set $Z\subset\uhp$ that is contained in
finitely many $\Gmod$-translates of the standard fundamental
domain~$\fd$. We choose a closed path $C\in\ZZ[X_1^\tess]$ encircling
$Z$ once in positive direction. Since $Z$ may contain a translate of
$\fd$ this path may have to go through cusps, as illustrated in
Figure~\ref{fig-Cp}.
\begin{figure}[ht]
\[\setlength\unitlength{1cm}
\newcommand\ci{
 \qbezier(-.2,0)(-.2,.2)(0,.2)
 \qbezier(0,.2)(.2,.2)(.2,0)
}
\newcommand\bgr{
 \qbezier(0,0)(0,.4)(.5,.867)
}
\newcommand\bgl{
 \qbezier(0,0)(0,.4)(-.5,.867)
}
\begin{picture}(12,5.2)(-6,0)
\put(-6,0){\line(1,0){12}}
\put(-.8,2.3){\line(0,1){3}}
\put(.8,2.3){\line(0,1){3}}
\put(0,2.3){\oval(1.6,1.6)[b]}
\put(-.15,3.6){$Z$}
\put(3.6,3){$C$}
\thicklines
\put(.2,0)\ci
\put(.6,0)\ci
\put(1,0)\ci
\put(1.4,0)\ci
\put(1.8,0)\ci
\put(2,0)\bgr
\put(2,0)\bgr
\put(3,0)\bgl
\put(3,0)\bgr
\put(3.5,.867){\vector(0,1){4.4}}
\put(-.2,0)\ci
\put(-.6,0)\ci
\put(-1,0)\ci
\put(-1.4,0)\ci
\put(-1.8,0)\ci
\put(-2,0)\bgl
\qbezier(-2.5,.867)(-2.7,1)(-3,1)
\qbezier(-3.5,.867)(-3.3,1)(-3,1)
\put(-3.5,.867){\line(0,1){4.4}}
\end{picture}
\]
\caption{}\label{fig-Cp}
\end{figure}
We can take the cycle $C$ far away from $Z$, such that $\singr \tilde
c(x) \cap Z=\emptyset$ for all $x$ occurring in the path~$C$.  We
choose $C$ such that for all edges $x$ occurring in $C$ the
$R$-neighbourhoods of edges in $X_1^{\tess,Y}$ and the sets
in~\eqref{wdg} do not intersect~$Z$. So for each edge $x$ occurring
in $C$ the set $Z$  is in the region on which $\tilde c(x)$ is
a representative of $c(x)$.

We define for $z\in Z$:\footnote{The factor $\frac1{4\pi}$ differs
from the factor $\frac1{\pi i}$ in~\cite{BLZm}. This is caused by a
difference in the normalization of $K_r(\cdot;\cdot)$ in \eqref{Kq},
and $q_s(\cdot,\cdot)$ in~\cite[(1.4)]{BLZm}. }\ir{uc}{u(C;z)}
\be \label{uc} u(C;z) : = \frac1{4\pi} \, \tilde c(C)\,(z) \,.\ee
So $u(C;z)$ is the sum of contributions $\frac{\pm 1}{4\pi} \tilde
c(x)(z)$ with $x\in X_1^\tess$ occurring in~$C$. Since $C$ is far
away from $Z$ the function $u(C;\cdot)$ is $r$-harmonic on (the
interior of)~$Z$.

\rmrk{Independence of choices}The next step is to get rid of the
choice of the lift $\tilde c$ and of the choice of $c$ in its
cohomology class. This can be done in exactly the same way as
in~\cite[\S7.1 and \S12.2]{BLZm}. The main reasoning is given in
\cite[\S7.1]{BLZm} for the cocompact case. There it is explained that
the definition does not depend on the choice of the cycle $C$,
provided it is far enough from $Z$. For a given $\gm\in \Gm$ we can
take $C$ such that both $C$ and $\gm^{-1}C$ can be used
in~\eqref{uc}. Then the $\Gm$-equivariance of $\tilde c$ implies that
$u(C;\cdot)|_{v,r}\gm = u(C;\cdot)$ on the intersection
$Z\cap \gm^{-1}Z$ for $\gm\in \Gm$. The function is $r$-harmonic on
the interior of~$Z$, by the same argument as in~\cite[\S7.1]{BLZm}.
The independence on $C$ allows us to enlarge the set $Z$, thus ending
up with an element of $\harm_r(\Gm,v)$, which we now can call
$u(\tilde c;\cdot)$.

By the reasoning in~\cite[\S7.1]{BLZm} the $r$-harmonic automorphic
form that we obtained is independent of the choice of the lift
$\tilde c$ of~$c$, and of the choice of $c$ in its cohomology class.
So we may now denote it by $u\bigl( [c];\cdot\bigr)$.

\rmrk{Remaining questions}There are two questions left: (1) Is
$u\bigl([c];\cdot\bigr)$ a \emph{holomorphic} automorphic form? (2)
If $\bcoh r \om F$ happens to be in $\hpar^1(\Gm;\esv
{v,r}\om,\esv{v,r}{\fsn,\wdg})$, what is then the relation between
$F$ and $u\bigl( \bcoh r \om F;\cdot)$?\smallskip

Question (1) does not arise in \cite{BLZm}. The following lemma treats
it, for general cofinite~$\Gm$ with cusps.

\begin{lem}
\label{lem-harmhol}Let $\tilde c\in C^1(F^\tess_\pnt;\Gr
{v,r}\om,\Gr{v,r}{\fsn,\wdg})$ be a lift of $c\in
Z^1(F^\tess_\pnt;\esv{v,r}\om,\esv{v,r}{\fsn,\wdg})$. Suppose that
$C=\sum_j \e_j\, x_j\in \ZZ[X_1^\tess]$ (finite sum, with $\e_j=\pm
1$, $x_j\in X_1^\tess$) is a cycle encircling an open set $Z\in \uhp$
once in positive direction, so that for each $x_j$ the set $Z$ is
contained in the set where $ \tilde c(x_j)$ represents $c(x_j)$. Then
$\tilde c(C)$ is holomorphic on $Z$.
\end{lem}

\begin{proof}
For each $x_j$ the function $\tilde c(x_j)$ represents an element of
$\esv{v,r}{\fsn,\wdg}$ on some set $U_j$ as in Property~i) in
Proposition~\ref{prop-Esys}, and $Z \subset U_j$. By Property~vi) we
know that $\shad_r \tilde c(x_j)$ has a holomorphic extension $h_j\in
\hol(\uhp)$.
\begin{figure}[ht]
\[\setlength\unitlength{1cm}
\begin{picture}(8,4.5)(-4,0)
\put(-4,0){\line(1,0){8}}
\put(-3.1,3){$S_{\!1}$}
\put(2.9,3){$S_{\!2}$}
\put(-2.1,1.2){$S_{\!3}$}
\put(1.9,1.2){$S_{\!4}$}
\put(0,3){\circle{1}}
\put(-.13,2.9){$Z$}
\put(-4,.6){$V_{\!1}$}
\put(3.8,.6){$V_{\!2}$}
\thicklines
\put(-3.5,1.2){\line(0,1){3.3}}
\put(-2.5,1.2){\line(0,1){3.3}}
\put(-3.5,1.2){\line(1,0){1}}
\put(2.5,1.2){\line(0,1){3.3}}
\put(3.5,1.2){\line(0,1){3.3}}
\put(2.5,1.2){\line(1,0){1}}
\qbezier(0,0)(0,2)(-3,2)
\qbezier(0,0)(0,1)(-3,1)
\put(-3,1){\line(0,1){1}}
\qbezier(0,0)(0,2)(3,2)
\qbezier(0,0)(0,1)(3,1)
\put(3,1){\line(0,1){1}}
\end{picture}
\]
\caption{Illustration for the proof of Lemma~\ref{lem-harmhol}. We
take $C= \sum_{j=1}^4 \e_j \, x_j$. The singularities of $\tilde
c(x_j)$ are contained in $S_{\!j}$, and we can take
$U_j=\uhp\setminus S_{\!j}$. The union $V_{\!1}\cup V_{\!2}$ is an
excised neighbourhood, on which $\tilde c(C)$ represents $ c(C)$.
}\label{fig-harmhol}
\end{figure}
Since $C$ is a closed cycle, we have $c( C )=0$ in
$\esv{v,r}{\fsn,\wdg}$. So $\tilde c( C)=0$ on an excised
neighbourhood $V$ on which $\tilde c( C)$ represents  $c(C )$. This
neighbourhood is contained in the intersection of the $U_j$, but will
in general not contain $Z$. See Figure~\ref{fig-harmhol}.

We now know the following:
\begin{align*}
\forall_j:&& \tilde c(x_j) &\=c(x_j)&\text{ on }&U_j \supset Z\,,
\displaybreak[0]\\
\forall _j:&& h_j &\;\in\; \hol(\uhp)\,,
\displaybreak[0]\\
\forall_j:&& 
\shad_r\,  \tilde c(x_j) &\= h_j &\text{ on }&U_j\supset Z\,,
\displaybreak[0]\\
&&\tilde c(C) &\= c(C) \=0&\text{ on }&V\,,
\displaybreak[0]\\
&&\shad_r \tilde c(C) &\= \sum_j \e_j\, \shad_r \tilde c(x_j)
\=0&\text{ on }&V
\displaybreak[0]\\
&& \sum_j \e_j\, h_j &\= \sum_j \e_j\, \shad_r \tilde c(x_j)\=0&\text{
on }&\uhp \text{ (continuation)}
\displaybreak[0]\\
&& \shad _r \tilde c(C) &\=\sum_j \e_j\,\shad_r \tilde c(x_j)\=\sum_j
\e_j \, h_j=0 & \text{ on }&Z
\displaybreak[0]\\
&& u([c];\cdot)&\=
\frac 1{4\pi} \tilde c(C) \text{ is holomorphic on } Z\,. \qedhere
\end{align*}
\end{proof}

This lemma implies directly that $u\bigl([c];\cdot)\in
A_r(\Gm,v)$, thus completing the proof of Part~i) of
Proposition~\ref{prop-alr}.

The other remaining question concerns Part~ii),  to which we apply
Cauchy's formula:
\begin{lem}Suppose that $[c]=\bcoh r \om F\in
\hpar^1(\Gm;\esv{v,r}\om,\esv{v,r}{\fsn,\wdg})$ for some automorphic
form $F\in A_r(\Gm,v)$. Then $u\bigl( [c];\cdot)=F$.
\end{lem}
\begin{proof}By analytic continuation it suffices to show the equality
on some non-empty open set. Let us take $Z$ open and relatively
compact in the interior of the compact face $\fd_Y$ of the
tesselation contained in the fundamental domain.

For $z$ in the interior of $\fd_Y$ we have
\be \label{cbdi} c(\partial_2\fd_Y)(z) \= \int_{\partial_2\fd_Y}
\frac{2i}{z-\tau}\, \Bigl( \frac{\bar z-\tau}{\bar z-z} \Bigr)^{r-1}\,
F(\tau)\, d\tau\,,\ee
as follows from~\eqref{cF}. The factor $\Bigl( \frac{\bar z-\tau}{\bar
z-z} \Bigr)^{r-1}$ is holomorphic as a function of~$\tau$. So the
value of the integral is $4\pi \, F(z)$ for $z$ in the interior of
$\fd_Y$, in particular for $z\in Z$. The hyperbolic distance of $Z$
to $\partial_2\fd_Y$ is larger than some $\e>0$. We can choose the
lift $\tilde c$ of $c$ such that for each $x\in X_1^{\tess,Y}$ the
singularities of $\tilde c(x)$ are contained in the $\e$-neighbourhood
of~$x$. Then $\tilde c(\partial_2
\fd_Y)$ is equal to $c(\partial_2\fd_Y)$ on~Z.\end{proof}

\rmrk{Averages}An alternative to \eqref{uc} is the description of
$u([c];\cdot)$ as an infinite sum, which is a kind of Poincar\'e
series.

\begin{defn}Let $f$ be a continuous function on~$\uhp$ with support
contained in finitely many $\Gm$-translates of a fundamental domain
of $\Gm\backslash\uhp$. We define the
\il{avr}{$\Gm$-average}\emph{$\Gm$-average} of $f$ by
\ir{avG}{\avGm{v,r}}
\be\label{avG}
 \bigl( \avGm {v,r} f \bigr) (z) \;:=\; \sum_{\gm \in \{\pm
 1\}\backslash\Gm} \bigl(f|_{v,r} \gm ) \,(z)\,. \ee
\end{defn}

\rmrks
\itmi We have $|_{v,r}(-\gm) = |_{v,r}\gm$, so it makes sense to sum
over $\{\pm 1\}\backslash\Gm$.

\itm Under the support condition in the definition the sum is locally
finite and defines a continuous function that is invariant for the
action $|_{v,r}$ of~$\Gm$.


\itm To use the average to describe $u([c],\cdot)$, we start with the
exact sequence
\be \label{esNGE}\rightarrow\hpar^1(\Gm;\Gr
{v,r}\om,\Gr{v,r}{\fsn,\wdg}) \rightarrow
\hpar^1(\Gm;\esv{v,r}\om,\esv{v,r}{\fsn,\wdg})\rightarrow
\hpar^2(\Gm;\N {v,r}\om,\N {v,w}{\fsn,\wdg}) \rightarrow \ee
The exactness follows from~\cite[Proposition 11.9]{BLZm}. To see that
the conditions of that theorem are satisfied, we use the diagram in
Part~iii) of Proposition~\ref{prop-NG} and Lemma~\ref{lem-Ies}.

Let $\tilde c \in C^1(F^\tess_\pnt;\Gr{v,r}\om,\Gr{v,r}{\fsn,\wdg})$
be a lift of $c\in
Z^1(F^\tess_\pnt;\esv{v,r}\om,\esv{v,r}{\fsn,\wdg})$. The exact
sequence~\eqref{esNGE} shows that $d\tilde c\in
\hpar^2(\Gm;\N{v,r}\om,\N{v,r}{\fsn,\wdg})$. We apply $d\tilde c$ to
the fundamental domain $\fd$ of $\Gm\backslash\uhp$ underlying the
tesselation $\tess$. So $d\tilde c(\fd) = d\tilde c(\fd_Y)+d\tilde
c(V_\infty)$ in the case of $\Gamma=\Gamma(1)$, and in general
$d\tilde c(\fd) = d\tilde c(\fd_Y) + \sum_\ca d\tilde c(V_\ca)$,
where $\ca$ runs over the cusps in the closure of the fundamental
domain~$\fd$. This implies that $d\tilde c(\fd)\in
\N{v,r}{\fsn,\wdg}$, hence we can apply $\avGm{v,r}$ to it.

\begin{prop}\label{prop-alav}With the notations of
Proposition~\ref{prop-alr}:
\be\label{avGm} u([c],z) \= \frac1{4\pi}
\bigl(\av{\Gm,v,r}d\tilde c(\fd)\bigr)(z)
\= \frac1{4\pi} \sum_{\gm\in \{\pm 1\}\backslash\Gm} \bigl(d\tilde
c(\fd)\bigr)\bigm|_{v,r} \gm\;(z)\,. \ee
\end{prop}

\begin{proof}The proof follows the approach to Propositions 7.1 and
(12.5) in \cite{BLZm}.\end{proof}

\rmrke On first sight it may seem amazing that the sum of translates
of the non-analytic function $d\tilde c(\fd)$ is a holomorphic
function. See the discussion after~\cite[Proposition 7.1]{BLZm}.

\subsection{Injectivity}\label{sect-inj}Proposition~\ref{prop-alr}
gives us a linear map $\al_r$ from mixed parabolic cohomology that is
left inverse to $\bcoh r \om$. It might have a non-zero kernel.

\begin{prop}\label{prop-alr-inj} The linear map $\al_r$ in~\eqref{alr}
is injective.
\end{prop}

\begin{proof}The proof is based on the exact sequence \eqref{esNGE}
and the average in~\eqref{avGm}:
\be\label{isch}
\xymatrix@C=.4cm{ \hpar^1\bigl(\Gm;
\Gr{v,r}\om,\Gr{v,r}{\fsn,\wdg}\bigr)
\ar[r] & \hpar^1\bigl(\Gm; \esv{v,r}\om,\esv{v,r}{\fsn,\wdg}\bigr)
\ar[d]^{\al_r} \ar[r]^\dt & \hpar^2\bigl(\Gm;
\N{v,r}\om,\N{v,r}{\fsn,\wdg}\bigr)
\ar[d]_{[b] \mapsto \avGm {v,r} b(\fd)}\\
& A_r(\Gm,v) \ar@{^{(}->}[r]
&C^2(\uhp)^\Gm_{v,r} } \ee
The vertical map on the right is given by associating to the
cohomology class $[b]$ the average $\avGm {v,r} b(\fd)$. By
$C^2(\uhp)_{v,r}$ we mean the space $C^2(\uhp)$ provided with the
action $|_{v,r}$ of~$\Gm$. The map $\al_r$ is the composition of the
connecting homomorphism $\dt$ and the vertical map. Failure of
injectivity might be caused by $\dt$ and by the average.

Lemma~\ref{lem-kerav-cb} below implies that the vertical map cannot
contribute to the kernel of~$\al_r$. That leaves us with the
connection homomorphism $\dt$. Lemma~\ref{lem-H1G} below gives the
vanishing of $\hpar^1(\Gm;\Gr{v,w}\om,\Gr{v,r}{\fsn,\wdg})$, and
hence the injectivity of~$\dt$.
\end{proof}

\begin{lem}\label{lem-kerav-cb}Let $c \in
Z^1(F_\pnt^\tess;\esv{v,r}\om, \esv{v,r} {\fsn,\wdg})$ and let
$\tilde c$ be a lift of $c$ as in \eqref{lift-c}. If $\avGm{v,r} d
\tilde c(\fd) =\sum_{\gm\in \{\pm 1\}\backslash\Gm} d\tilde
c(\fd)|_{v,r}\gm=0$ then $d\tilde c\in B^2(F_\pnt^\tess; \N {v,r}\om,
\N {v,r}{\fsn,\wdg})$.
\end{lem}

\begin{proof}
The proof is analogous to that of~\cite[Lemma 12.6]{BLZm}. Here we
discuss it in the modular case $\Gm=\Gmod$.

A cocycle $b\in Z^2(F^\tess_\pnt;\N{v,r}\om,\N{v,r}{\fsn,\wdg})$ is
determined by its values on the faces $(V_\infty)$ and $(\fd_Y)$. The
freedom that we have within a cohomology class is to add to $\bigl(
b(V_\infty),b(\fd_Y)\bigr)$ elements of three forms:
(1) $(u,-u)$, with $u\in \N{v,r}\om$ (related to the edge $f_\infty$),
(2) $\bigl( t|_{v,r}(1-\nobreak T),0\bigr)$ with $t\in
\N{v,r}{\fs,\wdg} $ (related to the edge $e_\infty$), (3)
$\bigl( 0,w|_{v,r}(1-\gm)\bigr)$ with $w\in \N{v,r}\om$, $\gm\in \Gm$
(related to the edges $e_1$ and $e_2$). So $b(\fd)=
b(V_\infty)+b(\fd_Y) $ is determined by $b$ up to addition of an
element of $\N{v,r}{\fs,\wdg}|_{v,r}(1-\nobreak T)
+ \sum_{\gm\in \Gm} \N{v,r}\om|_{v,r}(1-\nobreak \gm)$.

The first consequence of this description is that $\avGm{v,r} b(\fd)$
does not depend on the choice of $b$ in its cohomology class.

Now we consider $b=d\tilde c$ as in the lemma. The element $b(\fd_Y)$
is in $\N{v,r}\om \subset C_c^2(\uhp)$. So there is $q>Y$ such that
the support of $b(\fd_Y)$ does not intersect the region
\[ \bigcup_{\gm\in \Gm} \Bigl\{ \gm z\;:\; \im z\geq q\Bigr\}\,.\]
Further,  $b(V_\infty)=\tilde c(e_\infty)|_{v,r}(1-T)-\tilde
c(f_\infty)$
represents the zero element of $\esv{v,r}{\om,\exc}[\infty]$. Hence
$b(V_\infty)$ has support  in a set of the form $\bigl\{ z\in
\uhp\;:\; \im z>\e,\; |\re z| \leq \e^{-1}\bigr\}$ for some $\e>0$.
We deal with $C^2$-functions, and hence we can split off from
$b(V_\infty)$ an element $u\in C_c^2(\uhp)=\N{v,r}\om$ and move it to
$b(\fd_Y)$, by the freedom indicated above. In this way we arrange
that $b(V_\infty)$ has support in the set $\{z\in \uhp\;:\; \im
z>q-1\,,\; |\re z|\leq \e^{-1}\bigr\}$.

We take a partition of unity $\al$ on~$\RR$: $\al\in C_c^2(\RR)$ such
that $\sum_{n\in \ZZ} \al(x+\nobreak n)=1$ for all $x\in \RR$. We
take $\bt\in C^2(0,\infty)$ such that $\bt(y)=1$ for $y\geq q+\dt$
with $\dt>0$ and $\bt(y)=0$ for $y\leq q$, and put $\ch(z) = \al(\re
z)\, \bt(\im z)$. So $\sum_n \ch(z+\nobreak n)=1$ for all $z$ with
$\im z\geq q+\dt$.

The element $b_1 \in C^2(F^\tess_\cdot;
\N{v,r}\om,\N{v,r}{\fs,\exc})$ determined by $b_1(\fd_Y)=0$ and
\[b_1(V_\infty)(z)\= \sum_{n\in
\ZZ}\bigl(b(V_\infty)\,\ch(\cdot+n)\bigr)\bigm|_{v,r}(1-T^{-n})(z)\]
is a coboundary. (Note that the terms in the sum vanish for all but
finitely many~$n$.) We define $\hat b = b-b_1$, which is in the same
cohomology class as~$b$.  For $\im z\geq q+\dt$ 
\begin{align}\nonumber
\hat b(V_\infty)(z)&\=b(V_\infty)(z) - \sum_n \Bigl(
b(V_\infty)(z)\cdot \ch(z+n)-b(V_\infty)(z-n)\cdot \ch(z)
\Bigr)\\
\label{smhbi}
&\= \al(\re z)\, \bt(\im z) \, \sum_{\gm\in \{\pm
1\}\backslash\Gmod_\infty} b(V_\infty)|_{v,r}\gm\,(z)\end{align}

Now we use the assumption that $\avGm{v,r} b(\fd)=0$. {}From our
knowledge of the support $b(V_\infty)$ we conclude
that $\bigl(\avGm{v,r}b(\fd)\bigr)(z)
= \bigl(\avGm{v,r}b(V_\infty)\bigr)(z)$ if $\im z\geq q+\dt$.
Furthermore, for $\im z\geq q+\dt$ the expression in~\eqref{smhbi} is
equal to $\al(x) \allowbreak
\bigl(\avGm{v,r}b(V_\infty)\bigr)(z)$, since  since for $\gm \not \in
\Gm(1)_{\infty}$, the support of  $b(V_{\infty}) | \gm$ 
does not intersect the support of $b(V_{\infty})$. for So $\hat
b(V_\infty)$ vanishes on this domain, hence it has compact support.
So we can move $\hat b(V_\infty)$ to $\hat b(\fd_Y)$.

We are left with a cocycle $\hat b$ given by $\hat b(V_\infty)=0$ and
$\hat b(\fd_y)$ with support not intersecting the region
\[ \bigcup_{\gm\in \Gm} \Bigl\{ \gm z\;:\; \im z\geq q+\dt\Bigr\}\,.\]

We take a $\Gmod$-partition of unity $\ps\in C^\infty(\uhp)$ with
support of $\ps$ contained in the union of finitely many
$\Gmod$-translates of $\fd$. So $\sum_{\gm\in \{\pm 1\}\backslash\Gm}
\ps(\gm z)=1$ for all $z\in \uhp$, and the sum is a finite sum for
all~$z$. We write $f=\hat b(\fd)=\hat b(\fd_Y)$, and know by the
assumption that $\avGm{v,r}f=0$. For $z\in \uhp$:
\begin{align*}
f(z) &\= f(z) - \ps(z) \bigl(\avGm {v,r}f\bigr) (z)
\= \sum_{\gm\in \{\pm 1\}\backslash\Gm} \bigl( f(z)\, \ps(\gm^{-1}z) -
\ps(z)
\, f|_{v,r}\gm\,(z)
\bigr)\\
&\= \sum_{\gm\in \{\pm 1\}\backslash\Gm} \bigl( f\cdot
(\ps\circ\gm^{-1}) \bigr)|_{v,r}(1-\gm)\,.
\end{align*}
For almost all $\gm$ the intersection of the supports of $f$ and
$\psi\circ\gm^{-1}$ have empty intersections. So the sum is finite,
and $\hat b$ is a coboundary.
\end{proof}

\begin{lem}\label{lem-H1G}
$\hpar^1(\Gm;\Gr{v,r}\om, \Gr{v,r}{\fsn,\wdg}) =\{0\}$.
\end{lem}

\begin{proof}Similar to the proof of~\cite[Proposition 12.5]{BLZm}, to
which we refer for the full proof. Table~\ref{tab-cor1} gives a list
of corresponding notations and concepts.
\begin{table}[ht]\renewcommand\arraystretch{1.3}
\begin{tabular}{|l|l|}\hline
\multicolumn{1}{|c|}{{\bf holomorphic
forms}}&\multicolumn{1}{|c|}{{\bf Maass forms wt.~$0$ }}\\ \hline
$\Gm$-module $\Gr{v,r}\om $ &$\Gm$-module $\Gr s \om$ \\
$\Gm$-module $\Gr{v,r}{\fsn,\wdg}$ &$\Gm$-module $\Gr s {\fs,\wdg} $\\
cocycle $c $ & cocycle $\ps$\\
$c(\xi,\xi')$& $\ps_{\xi,\xi'}$
\\ \hline
\end{tabular}\smallskip
\caption{Correspondence with~\cite[Proposition
12.5]{BLZm}.}\label{tab-cor1}
\end{table}

Let $c\in Z^1(F^\tess_\pnt;\Gr{v,r}\om,\Gr{v,r}{\fsn,\wdg})$ be given.
This cocycle induces a map $X_0^\tess\times X_0^\tess \rightarrow
 \esv{v,r}{\fsn,\wdg}$ which we also indicate by $ c$. It has the
properties in~\eqref{crel}. The aim is to show that it is a
coboundary. To do that is suffices to show that the group cocycle
$\gm\mapsto c(\gm^{-1}P_0,P_0)$ is a coboundary for one
base point~$P_0\in X^\tess_0$.

\setcounter{rmrkcnt}{0}
\itm There exists $R>0$ such that $\singr c(x)\subset N_R(x)$ for all
edges $x\in X_1^\tess$. The set $N_R(x)$ is an $R$-neighbourhood of
$x$ for the hyperbolic metric if $x$ is an edge in $X_1^{\tess,Y}$,
and a more general neighbourhood defined in~\cite[(12.2)]{BLZm} if $x$
is an edge going to a cusp.

\itm We prove that $c(\ca,\cb)\in \sharm_r(\uhp)$ for any two cusps
$\ca,\cb$.

Suppose that $z\in \singr c(\ca,\cb)$. The value of $c(\ca,\cb)$ is
the value $c(p)$ for any path in $\ZZ[X_1^\tess]$ from $\ca$ to
$\cb$. We can move the path $p$ away from $z$ in such a way that $z$
is not in $N_R(x)$, in (a), for any of the edges $x$ occurring
in~$p$. So $\singr c(\ca,\cb)=\emptyset$.

\itm By breaking up a path from $\ca$ to $\cb$ at a point $P\in
X_0^{\tess,Y}=X_0^\tess\cap \uhp$ it can be shown that $\singr
c(P,\ca)$ is a compact subset of $\uhp$ for any path $\ZZ[X_1^\tess]$
from $P$ to $\ca\in \cu$.

\itm Now Lemma~\ref{lem-rem-sing} can be applied to the conjugated
element $ F= c(P,\ca)|_r \s_\ca^{-1}$. We note that $F\in \Gr r
{\fs,\wdg}$, for Condition~a), that its singularities are contained
in a compact set, for Condition~b), and that
\[F|_{v,r}(1-\pi_\ca) \= c(P,\ca)|_{v,r}(1-\pi_\ca)
\= c(P,\pi_\ca^{-1}P)\in \Gr {v,r}\om\]
implies Condition~c). The conclusion is that $F=Q_\ca+G$, with $Q_\ca
\in \Gr {v,r}{\fsn,\wdg}$ satisfying $Q_\ca|_{v,r}\pi_\ca = P$, and
$G\in \Gr {v,r}\om$ representing an element of $\esv r \om$. Then use
Lemma~\ref{lem-ENinv} to see that $Q_\ca\in \sharm_r(\uhp)$.

\itm Such an element $Q_\ca$ exists for all cusps $\ca$, and for
$\cb=\gm^{-1}\ca$ we have $Q_\cb = Q_\ca|_{v,r}\gm$.

\itm The transformation properties of the $Q_\ca$ allow us to define
 another cocycle $\hat c$ in the same class as $c$ by taking for
 $x,y\in X_0^\tess$:
\[ \hat c(x,y) \;:=\; c(x,y)
+ \left\{ \begin{array}{ll}Q_x&\text{ if }x\in X_0^{\tess}\setminus
X_0^{\tess,Y}\\
0&\text{ if }x\in X_0^{\tess,Y}
\end{array}
\right\}+\left\{ \begin{array}{ll}-Q_y&\text{ if }y\in
X_0^{\tess}\setminus X_0^{\tess,Y}\,,\\
0&\text{ if }y\in X_0^{\tess,Y}\,.
\end{array}\right\}
\]  It has the property that $\hat c(\ca,\cb)
\in \sharm_r(\uhp) \cap \Gr r \om$ for all cusps $\ca, \cb$. 
Part~v)  of Proposition~\ref{prop-Esys} implies that  $\hat c
(\ca,\cb)=0$ for all cusps. Taking a cusp as
the base point $P_0$, we see that the cohomology class of the cocycle
$\hat c$, and hence of the original cocycle $c$, is zero.\end{proof}

\subsection{From analytic boundary germ cohomology to automorphic
forms}\label{sect-THMbg} We have obtained two linear maps, $\bcoh r
\om$ (Proposition~\ref{prop-bcoh}) and $\al_r$
(Proposition~\ref{prop-alr}):
\badl{bcoh-alr} \xymatrix{ A_r(\Gm,v) \ar[r]^{\bcoh r \om}
& H^1(\Gm;\esv{v,r}\om)
\\
& \hpar^1(\Gm;\esv{v,r}\om,\esv{v,r}{\fsn,\wdg}) \ar[lu]^{\al_r} }
\eadl
 We recall that $\hpar^1(\Gm;\esv{v,r}\om,\esv{v,r}{\fsn,\wdg})\subset
H^1(\Gm;\esv{v,r}\om)$. The following theorem shows the relation
between these maps.

\begin{thm}\label{thmbg}
Let $\Gm$ be a cofinite discrete group of $\SL_2(\RR)$ with cusps. Let
 $r \in \CC$ and let $v$ be a corresponding multiplier system.
\begin{enumerate}
\item[i)] Both linear maps $\bcoh r \om$ and $\al_r$ in
\eqref{bcoh-alr} are injective.
\item[ii)] Define \ir{AE}{A^\E_r(\Gm,v)}
\be\label{AE} A^\E_r(\Gm,v) \;:=\; \bigl(\bcoh r \om\bigr)^{-1}
\hpar^1(\Gm;\esv {v,r}\om,\esv {v,r}{\fsn,\wdg}) \,. \ee

Then the restriction of $\bcoh r \om$ to $A^\E_r(\Gm,v) $ and the
restriction of $\al_r$ to the image $\bcoh r\om A^\E_r(\Gm,v) \subset
\hpar^1(\Gm;\esv{v,r}\om,\esv{v,r}{\fsn,\wdg})$ are inverse to each
other.
\[ \xymatrix{ A^\E_r(\Gm,v) \ar@<.3ex>[r]^{\bcoh r \om}
& \bcoh r \om A^\E_r(\Gm,v) \ar@<.3ex>[l]^{\al_r} } \]
\end{enumerate}
\end{thm}

\begin{proof}Proposition~\ref{prop-alr-inj}
gives the injectivity of~$\al_r$. Suppose that $F\in \ker \bcoh r
\om$. (Then $F\in A_r^\E(\Gm,v)$ by the definition
in~\eqref{AE}.)  Proposition~\ref{prop-alr} ii)
shows that $\al_r \bcoh r \om F = F$, hence $F=\al_r\, 0 = 0$. This
shows that $\bcoh r \om$ is injective. This gives Part~i).

Part~ii) of Proposition~\ref{prop-alr} shows that $\al_r \circ \bcoh r
\om$ is the identity on $A^\E_r(\Gm,v)$. Since $\bcoh r \om :
A^\E_r(\Gm,v)
\rightarrow \bcoh r \om A^\E_r(\Gm,v)$ is surjective Part~ii) follows.
\end{proof}

\subsection{Completion of the proof of Theorem~\ref{THMac} for general
weights}\label{sect-recap-THMac}
We consider $r\in \CC\setminus\ZZ_{\geq 2}$.
Proposition~\ref{prop-cohrom} shows that $\coh r \om : A_r(\Gm,v)
\rightarrow H^1(\Gm;\dsv{r,2-r}\om)$ is a well-defined linear map.
Theorem~\ref{thm-imafwdg} shows that the image is contained in
$\hpar^1(\Gm;\dsv{v,2-r}\om,\dsv{v,2-r}{\fsn,\wdg})$. We have the
following relations:
\[ \xymatrix{ A_r(\Gm;v) \ar[r]^(.38){\coh r \om} &
\hpar^1(\Gm;\dsv{v,2-r}\om,\dsv{v,2-r}{\fsn,\wdg})
\ar@{^{(}->}[r]& H^1(\Gm;\dsv{v,2-r}\om)
 \\
A_r^\E (\Gm, v) \ar@{^{(}->}[u] \ar[r]^(.38){ \bcoh r \om } &
\hpar^1(\Gm;\esv{v,r}\om,\esv{v,r}{\fsn,\wdg})
\ar[u]_{\rs_r}^{\cong} } \]
Definition~\ref{esv-def-gen} of the highest weight spaces of boundary
germs $\esv r \ast$ as isomorphic to the corresponding highest weight
spaces $\dsv{2-r}\ast$ induces an isomorphism in cohomology.
Theorem~\ref{thm-imafwdg} implies that $\coh r \om
A_r(\Gm,v)\subset\hpar^1(\Gm;\dsv{v,2-r}\om,\dsv{v,2-r}{\fsn,\wdg})$,
and hence $\bcoh r \om A_r(\Gm,v)\subset
\hpar^1(\Gm;\esv{v,2-r}\om,\esv{v,2-r}{\fsn,\wdg})$ by 
Proposition~\ref{prop-bcoh}. 
So $A^\E_r(\Gm,v)=A_r(\Gm,v)$. Theorem~\ref{thmbg} then gives the
inverse $\al_r$ of $\bcoh r \om$:
\[ \xymatrix{ A_r(\Gm;v) \ar[r]^(.38){\coh r \om} &
\hpar^1(\Gm;\dsv{v,2-r}\om,\dsv{v,2-r}{\fsn,\wdg})
\ar@{^{(}->}[r]& H^1(\Gm;\dsv{v,2-r}\om)
 \\
A_r (\Gm, v) \ar[u]_{=} \ar@<.3ex>[r]^(.38){\bcoh r \om} &
\hpar^1(\Gm;\esv{v,r}\om,\esv{v,r}{\fsn,\wdg})
\ar[u]_{\rs_r}^{\cong} \ar@<.3ex>[l]^(.62){\al_r} } \]

\subsection{Related work}\label{sect-lit10}As indicated at several
places in this section, we followed closely the approach of
\cite{BLZm}, \S7 and~\S12.2--3.


\section{Automorphic forms of integral weights at least $2$ and
analytic boundary germ cohomology}\label{sect-abg-iw}

In this section we will prove Theorem~\ref{THMaci}, which concerns
automorphic forms with weight $ r \in \ZZ_{\geq 2}$ and analytic
boundary germ cohomology.

Throughout this section we only treat the case of weight $ r \in
\ZZ_{\geq 2}.$

\subsection{Image of automorphic forms in mixed parabolic
cohomology}\label{sect-imaf}
The linear map $\bcoh r \om: A_r(\Gm,v)
\rightarrow H^1\bigl(\Gm;\W {v,r} \om(\proj\RR)\bigr)$ in
Proposition~\ref{prop-bcoh} has image in $H^1(\Gm;\esv r \om)$, by
Property~iv) in Proposition~\ref{prop-Esys}.

\begin{defn}For all $r\in \CC$ we define\ir{A0}{A_r^0(\Gm,v)}
\be\label{A0} A^0_r(\Gm,v) \;:=\; \bigl\{ F\in A_r(\Gm,v)
\;:\;a_0(\ca,F)=0\text{ for all cusps }\ca \text{ with
}v(\pi_\ca)=1\Bigr\}\,. \ee
\end{defn}
See \eqref{FexpF-xi} for the Fourier coefficients $a_n(\ca,F)$ at the
cusp~$\ca$.

The idea is to allow automorphic forms with large growth at the cusps,
but not to allow constant terms in the Fourier expansion.

\begin{prop}\label{prop-A0-a0}Let $r\in \ZZ_{\geq 2}$. For each $F\in
A_r(\Gm,v)$ the following statements are equivalent:
\begin{enumerate}
\item[a)] $\bcoh r \om F\in \hpar^1(\Gm;\esv
{v,r}\om,\esv{v,r}{\fsn,\wdg})$
\item[b)] $F\in A^0_r(\Gm,v)$.
\end{enumerate}
\end{prop}
We will base the proof on the following lemma:

\begin{lem}\label{lem-locwdgbg}Let $r\in \ZZ_{\geq 2}$, $z_0\in \uhp$,
and $\ld=e^{2\pi i \al}$ with $\al\in \CC$. Suppose that the
holomorphic function $E$ on $\uhp$ is given by the Fourier expansion
\[ E(\tau) = \sum_{n\equiv\al(1)}a_n\, e^{2\pi i n \tau}\,.\]
Then there exists $h \in \esv r {\om,\wdg}[\infty]$ such that
\be\label{pbE} \ld^{-1}\,h(z+1) - h(z)
\= \int_{\tau=z_0-1}^{z_0} K_r(z;\tau)\, E(\tau)\, d\tau\,,\ee
if and only if $\ld\neq 1$ or $a_0=0$.
\end{lem}

\rmrke If $\ld\neq 1$, then $n$ in the Fourier expansion does not run
over the integers, and $a_0$ is not defined.

\begin{proof}This is a situation similar to that in
\S\ref{sect-constr-peq}. We can split up the Fourier expansion of
$E$. For the cuspidal part
\[ E_c(\tau) \= \sum_{n\equiv\al(1),\, \re n>0}a_n\, e^{2\pi i n
\tau}\] we can use
\be\label{hcint}
 h_c(z) \= \int_{\tau=z_0}^\infty K_r(z;\tau)\, E_c(\tau)\, d\tau \ee
for $z\in \uhp\setminus \left(z_0+i[0,\infty)\right) $. (In this way
we avoid the singularity at $\tau=z$. See \eqref{Kq}.)

Expression \eqref{iwK} gives, for those $\tau$ that have smaller
hyperbolic distance to~$i$ than~$z$, an expression for $K_r(z;\tau)$
in terms of a linear combination of $\M{r,\mu}$ with $1-r\leq \mu
\leq-1$ and an explicit expression $p_r(z;\tau)$. The factors of the
$\M{r,\mu}$ are holomorphic on~$\proj\CC\setminus\{-i\}$, and
$p_r(z;\tau)$ is meromorphic on $\proj\CC\times\proj\CC$, with
singularity in $\uhp\times\uhp$ given by $\frac1{z-\tau}$. 
By analytic continuation 
\eqref{iwK} is valid for all $(z,\tau)$ of interest in~\eqref{hcint}.

On insertion of~\eqref{iwK} the integrals of the term with
$\M{r,\mu}(z)$ yield, by the exponential decay of $E_c(\tau)$, a
multiple of $\M{r,\mu}$, hence in $\esv r \om$. The term with
$p_r(z;\tau)$ gives
\[ \int_{\tau=z_0}^\infty \frac{2i}{z-\tau} \,
\frac{(\tau-i)^{r-1}}{(z-i)^{r-1}}\, E_c(\tau)\, d\tau \,.\]
It yields a holomorphic function on $\CC\setminus
\left(z_0+i[0,\infty)\right)$, hence the result is an element of
$\dsv r {\om,\wdg}[\infty]$. Together with the multiples of
$\M{r,\mu}$ we obtain an element of $\esv r {\om,\wdg}[\infty]$ with
the desired property.

We proceed similarly with the contribution $E_\infty(\tau)
$ of the part of the Fourier series with $\re n<0$. The path of
integration is replaced by the path used in the proof of
Lemma~\ref{lem-expgr}. If $|\ld|\neq 1$ we can take $\al \in i\RR$,
$\al\neq 0$, and use the paths as in the proof of
Lemma~\ref{lem-pim}. This gives a function~$h_\infty$ satisfying
\[ \ld^{-1}\,h_\infty(z+1)-h_\infty(z)\=\int_{\tau=z_0-1}^{z_0}
K_r(z;\tau)\, E_\infty(\tau)\, d\tau\,.\]

There remains the case of a constant Fourier term, present only if
$\ld=1$. We look for $h\in \esv r {\om,\wdg}[\infty]$ such that
\[ h(z+1)-h(z) \= \int_{\tau=z_0-1}^{z_0} K_r(z;\tau)\, d\tau\,,
\]
which maps under the restriction map to a relation for $\ph=\rs_r h\in
\dsv{2-r}\pol$:
\[ \ph(t+1)-\ph(t) \= \int_{\tau=z_0-1}^{z_0} (\tau-t)^{r-2}\, d\tau
\= \frac{(z_0-t)^{r-1}-(z_0-t-1)^{r-1}}{r-1} \,.\]
(We have used \eqref{res-Kr} and~\eqref{Prj}.)
The right hand side is a polynomial in $t$ with $(-t)^{r-2}$ as the
term of highest degree in~$t$. Any polynomial solution $\ph$ is a
polynomial with degree $r-1$ in~$t$, and hence is not in
$\dsv{2-r}\pol$.\end{proof}

\rmrks
\itmi The function $h=h_c+h_\infty+h_0\in \esv r{\om,\wdg}[\infty]$
constructed in the proof has $\singr h \subset z_0+i[0,\infty)$.

\itm If $\ld=1$ the constant term can be handled by
$\int_{z_0-1}^{z_0} K_r(z;\tau)\, d\tau = h_0(z+1)-h_0(z)$, with
\be h_0(z) \= -2i \,\log(z-z_0)+2i \log y -2i
\sum_{l=1}^{r-1}\binom{r-1}l\; \frac{(z-z_0)^{l}}{l\, (\bar
z-z)^l}\,.\ee
We note that although $h_0$ is an $r$-harmonic function, it does not
represent an analytic boundary germ: $h_0\not\in\W r \om(\RR)$.

\begin{proof}[Proof of Proposition~\ref{prop-A0-a0}] Let $z_0\in
\uhp$, and consider the cocycle $c_F^{z_0}$ in~\eqref{cz0-def}, which
represents the cohomology class $\bcoh r \om F$ of $F\in A_r(\Gm,v)$.
The following statements are equivalent:
\begin{itemize}
\item $\bcoh r \om F \in \hpar^1(\Gm;\esv
{v,r}\om,\esv{v,r}{\fsn,\wdg})$.
\item For each cusp $\ca$ there exists $h\in \esv r{\om,\wdg}[\infty]$
that satisfies the relation \eqref{pbE} with $\ld=v(\pi_\ca)$,
$E=F|_r \s_\ca$, and $z_0$ replaced by $\s_\ca^{-1}z_0$.
\end{itemize}
This gives the proposition.\end{proof}

Relation \eqref{AE} in Theorem~\ref{thmbg} defines a subspace
$A^\E_r(\Gm,v)\subset A_r(\Gm,v)$. We state a direct consequence of
Proposition~\ref{prop-A0-a0}:

\begin{cor}\label{cor-A0E}Let $r\in \ZZ_{\geq 2}$. Then $A^\E_r(\Gm,v)
= A^0_r(\Gm,v)$.
\end{cor}

Next we would like to know that $\bcoh r \om A_r^0(\Gm,v)$ is equal to
$\hpar^1(\Gm;\esv{v,r}\om,\esv{v,r}{\fsn,\wdg})$. This requires quite
some work, carried out in the following subsection.

\subsection{Image of mixed parabolic cohomology classes in automorphic
forms}

\begin{prop}\label{prop-impcA}Let $r\in \ZZ_{\geq 2}$. The linear map
$\al_r: \hpar^1(\Gm;\esv{v,r}\om,\esv{v,r}{\fsn,\wdg} ) \rightarrow
A_r(\Gm,v)$ in Proposition~\ref{prop-alr} has image in
$A^0_r(\Gm,v)$.
\end{prop}

\begin{proof}
Let a cohomology class $[c]\in \hpar^1(\Gm;\esv {v,r} \om,\esv {v,r}
{\fsn,\wdg})$ be given. In the proof of Proposition~\ref{prop-alr}
the image $u=\al_r\bigl([c]\bigr)$ is constructed in \eqref{uc} as
$u(z):=u(C;z)=\frac1{4\pi} \tilde c
(C)(z)$, where the cochain $\tilde c \in C^1(F^\tess_\pnt;\Gr r
\om,\Gr {v,r} {\fsn,\wdg})$ is a lift of $c\in Z^1(F^\tess_\pnt;\esv
{v,r}\om,\esv{v,r}{\fsn,\wdg})$, and where $C\in \ZZ[X_1^\tess]$ is a
path around $z$ adapted to~$\tilde c$. We use a tesselation~$\tess$
and the corresponding resolution $\bigl(F^\tess_i\bigr)
= \bigl( \ZZ[X_i^\tess]\bigr)$ as discussed in~\S\ref{sect-tesscoh}.

We want to show that for each cusp the Fourier term of order
zero of the automorphic form vanish. It suffices to do this
for one representative $\ca$ of each $\Gm$-orbit of cusps for which
$v(\pi_\ca)=1$. (If $v(\pi_\ca)\neq 1$ there is no Fourier term of
order zero at the cusp $\ca$.)
This can be handled for each such cusp separately. By conjugation we
can assume that $\ca=\infty$ and $\pi_\ca=T=\matc 1101$.

After the conjugation, the cuspidal sector $V_\infty$ looks exactly
like that for the modular group, in Figure~\ref{fig-modtess},
\S\ref{sect-tess}. The sector $V_\infty$ is bounded by edges
$e_\infty$ from $P_{\infty}=\frac12+iY$ (for some $Y>0$) to~$\infty$,
$T^{-1}e_\infty$ from $T^{-1}\infty$ to $\infty$, and $f_\infty$ from
$P_\infty$ to~$T^{-1}P_\infty$. By holomorphy it suffices to consider
the Fourier term of order $0$ high up in the cuspidal sector.

The cocycle $c$ satisfies $c(f_\infty)\in \esv r \om$, $c(e_\infty)
\in \esv r{\fs,\wdg}$, and $c(e_\infty)|_r(1-T)=c(f_\infty)$. By
Lemma~\ref{lem-singTe} this implies $c(e_\infty)\in \esv r
{\om,\wdg}[\infty]$. We change the cocycle within its cohomology
class. Definition~\ref{esv-def-iw} shows that $\esv
r{\om,\wdg}[\infty] = \dsv r {\om,\wdg}[\infty]+ \esv r \om$. So
there is $k\in \esv r \om$ such that $c(e_\infty)-k \in \dsv
r{\om,\wdg}[\infty]$. We define $f\in C^0(F^\tess_\pnt;\esv r
\om,\esv r {\fsn,\wdg})$ by taking $f(\gm^{-1}P_\infty)=k|_r \gm$ for
all $\gm\in \Gm$, and $f=0$ on all other $\Gm$-orbits in $X_0^\tess$.
Then $c_1=c-df$ is in the same cohomology class as~$c$, and satisfies
$c_1(e_\infty)\in \dsv r {\om,\wdg}[\infty]$. Replacing $c$ by~$c_1$,
we can assume that the cocycle $c$ now satisfies $c(e_\infty)\in \dsv
r {\om,\wdg}[\infty]$ and then automatically $c(f_\infty)\in \dsv r
\om$ by the cocycle relation.

In the construction of $u(C;z)$ in \eqref{uc} we started with $z$ in a
given set~$Z$ and showed that there are suitable cycles around it.
Here we will take a special cycle $C$ and choose a region $Z$
encircled by it, high up in~$\uhp$.\vskip.3ex
\twocolwithpictl{
\setlength\unitlength{.7cm}
\begin{picture}(8,5)(-3,0)
\put(3.1,3.4){$e_\infty$}
\put(-1.2,.9){$\singr \tilde c(f_\infty)$}
\put(-3,0){\line(1,0){8}}
\put(-.5,2){\oval(1,1)[l]}
\put(3.5,2){\oval(1,1)[r]}
\put(-.5,1.5){\line(1,0){4}}
\put(-.5,2.5){\line(1,0){4}}
\put(3,1.1){\oval(1.6,1.6)[b]}
\put(2.2,1.1){\line(0,1){3.9}}
\put(3.8,1.1){\line(0,1){3.9}}
\put(-.5,3.9){$\singr \tilde c(e_\infty)$}
\thicklines
\put(3,2){\line(0,1){3}}
\put(1,2){\line(1,0){2}}
\end{picture}
}{ Let $\tilde c \in C^1(F^\tess_\cdot;\Gr
{v,r}\om,\Gr{v,r}{\fsn,\exc})$ be a lift of $c$.  Our choice of
the lift $\tilde c(f_\infty)$ of $c(f_\infty)$ in the proof of
Proposition~\ref{prop-alr} implies that  $\singr \tilde
c(f_\infty)$ is compact in $\uhp$. The lift $\tilde
c(e_\infty)$  was  chosen so that $\uhp\setminus\singr
\tilde c(e_\infty)$ is an $\{\infty\}$-excised
neighbourhood. The regions where $\tilde c(e_\infty)$ and $\tilde
c(f_\infty)$ are not holomorphic may be large, and cover the
sector~$V_\infty$. We have drawn the edges $x=e_\infty, f_\infty$
inside the singular set $\singr \tilde c(x)$.   For $\tilde
c(f_\infty)$ this~is}\vskip.6ex

\twocolwithpictr{a consequence of the choice, and for $\tilde
c(e_\infty)$ we can arrange the choice so that $e_\infty \subset
\singr \tilde c(e_\infty)$.

\quad We would like to enclose the set $Z$ on which
to study the function~$u$ by the boundary $\partial_2
V_\infty=e_\infty - T^{-1}e_\infty-f_\infty$. However, the
corresponding sets of singularities may very well overlap, leaving no
space for a region~$Z$. Instead of this, we take the union of a
number of translates $T^{-n}V_\infty$. }{\setlength\unitlength{.75cm}
\begin{picture}(6,5)(-.5,0)
\put(-.5,0){\line(1,0){6}}
\put(3,2){\line(0,1){3}}
\put(2,2){\line(0,1){3}}
\put(1,2){\line(0,1){3}}
\put(4,2){\line(0,1){3}}
\put(5,2){\line(0,1){3}}
\put(0,2){\line(1,0){5}}
\put(2.3,3.4){$\scriptstyle V_\infty$}
\put(1,3.4){$\scriptstyle T^{-1}V_\infty$}
\put(3.2,3.4){$\scriptstyle TV_\infty$}
\put(2.4,1.6){$\scriptstyle f_\infty$}
\thicklines
\put(5,2){\line(0,1){3}}
\put(0,2){\line(1,0){5}}
\put(0,2){\line(0,1){3}}
\end{picture}
}\vskip.6em

We take $k\in \ZZ_{\geq 1}$ large, so that there is a region of
width at least~$2$ between $\singr \tilde c(T^{-k}e_\infty) = T^{-k}
\singr \tilde c(e_\infty)$ and $\singr \tilde c(T^{k}e_\infty) =
T^{k} \singr \tilde c(e_\infty)$. We put
$g_k=\sum_{n=-k}^{k-1}T^{-n}f_\infty$. This leads to the situation in
Figure~\ref{fig-S}.
\begin{figure}[ht]
\[
\setlength\unitlength{.7cm}
\begin{picture}(12,5)(-7,0)
\put(-7,0){\line(1,0){12}}
\put(-6.5,2){\oval(1,1)[l]}
\put(3.5,2){\oval(1,1)[r]}
\put(-6.5,1.5){\line(1,0){10}}
\put(-6.5,2.5){\line(1,0){10}}
\put(3,1.1){\oval(1.6,1.6)[b]}
\put(2.2,1.1){\line(0,1){3.9}}
\put(3.8,1.1){\line(0,1){3.9}}
\put(-3,1.1){\oval(1.6,1.6)[b]}
\put(-3.8,1.1){\line(0,1){3.9}}
\put(3.8,1.1){\line(0,1){3.9}}
\put(-2.2,1.1){\line(0,1){3.9}}
\put(3.9,3.9){$\singr \tilde c(T^k e_\infty)$}
\put(-7.3,3.9){$\singr \tilde c(T^{-k}e_\infty)$}
\put(-7.8,.9){$\singr \tilde c(g_k)$}
\put(0,3.8){$Q$}
\put(5,.8){$R$}
\end{picture}
\]
\caption{Illustration of regions of non-holomorphy. } \label{fig-S}
\end{figure}
There is a connected region $Q$ high up in the upper half-plane of width at
least~$2$ on which $\tilde c(T^{\pm k}e_\infty)$ and $\tilde c(g_k)$
are holomorphic. The region $Q$ is 
disjoint from the connected region $R$ in
the complement of the three singular  sets  that has $\RR$
in its boundary.

We consider the cycle $C = (T^k e_\infty) -
(T^{-k}e_\infty)-(f_\infty^{(k)})\in \ZZ[X_1^\tess]$, with $k$ as
fixed above. It encircles $Q$ once, so $4\pi\, u(z) = \tilde
c(C)\,(z)$ for $z\in Q $. Furthermore, on the $\{\infty\}$-excised
neighbourhood $R$ the function $\tilde c(C)$ is zero, since it
represents $c(C)$, which is zero by the cocycle relation. 

The element $h=\tilde c(e_\infty) \in \Gr{v,r}{\fsn,\exc}$
represents an element of $\dsv r {\om,\exc}[\infty]$ and
$g=\tilde c(f_\infty) \in \Gr{v,r}\om$ represents an element of
$\dsv r \om$. Furthermore the multiplier system $v$ is trivial on
$\pi_\ca$, and we have conjugated $\ca$ to~$\infty$. So we can apply
Lemma~\ref{lem-h-Av+-}. The averages $\av{T,1}^+ g$ and
$\av{T,1}^- g$ are functions in $C^2(\uhp)$ that are holomorphic on
the region $0<\im(z)<\e$ and on $\pm \re(z)>\e^{-1}$ for some
 sufficiently small positive~$\e$. Lemma~\ref{lem-h-Av+-} gives us
$1$-periodic $p_+$ and $p_-$ such that $h = \av{T,1}^\pm g + p_\pm$,
on regions as indicated in the lemma. 

Shifting this with $T^{\pm k}$ we get
\bad \tilde c(T^k e_\infty) &\= \tilde c(e_\infty)|_r T^{-k} \=
(\av{T,1}^-g)|_r T^{-k} +p_-\,,\\
     \tilde c(T^{-k} e_\infty) &\= \tilde c(e_\infty)|_r T^k \=
     (\av{T,1}^+g)|T^k+p_+\,,
\ead
first on regions as indicated in the lemma, and then by analytic continuation to regions containing $Q$ and a strip $0<\im(z)<\e$.

For $z$ in $Q$ or near $\RR$ we have
\begin{align*}
\tilde c(C)(z)
&\= \tilde c(T^ke_\infty)(z) - \tilde c(T^{-k}e_\infty)(z) - \tilde c(g_k)(z)
\displaybreak[0]\\
&\= (\av{T,1}^- g)|_r T^{-k} (z)+p_-(z) 
- (\av{T,1}^+ g)|_r T^k\, (z) -p_+(z)
- \sum_{n=-k}^{k-1} g|_r T^n \,(z)
\displaybreak[0]\\
&\= -\sum_{n\leq -1} g|_r T^{n-k}\,(z) - \sum_{n\geq 0} g|_r T^{n+k}\, (z)
- \sum_{n=-k}^{k-1} g|_r T^n \,(z) +p_-(z)-p_+(z)
\displaybreak[0]\\
&\= -(\av{T,1}g)\,(z) +p_-(z)-p_+(z)\,,
\end{align*}
with $\av{T,1}$ as defined in \eqref{avTdef}.

Next we apply Proposition~\ref{prop-avT-Fe} to $g=\tilde c(f_\infty)$. 
The $1$-periodic function $\av {T,1} g (z)=\sum_{n\in
\ZZ}g(z+n)$  on $\uhp$ is
holomorphic on a region of the form $0<\im z<\e$ and on a region
$y>\e^{-1}$ for some $\e\in
(0,1)$. We denote the holomorphic function on the upper region by $
\av {T,1}^\uparrow \tilde c(f_\infty)$, and the holomorphic function
on the lower region by $\av {T,1}^\downarrow \tilde c(f_\infty)$.
Proposition~\ref{prop-avT-Fe} states that $\av{T,1}^\uparrow \tilde
c(f_\infty)$ has a Fourier expansion with terms of positive order
only, and $\av{T,1}^\downarrow \tilde c(f_\infty)$ a Fourier
expansion with only terms of negative order.

The domain of $\av{T,1}^\downarrow g $ is 
contained in the region~$R$. There we find
\[ 0 \= \tilde c(C)(z) \= p_-(z)-p_+(z) - \bigl(\av {T,1}^\downarrow
g (z)\,.\]
So all Fourier terms of $p_--p_+$ of order $n\geq 0$ vanish. This
holds on $\uhp$, since $p_--p_+$ is holomorphic and $1$-periodic
on~$\uhp$.

If $z\in Q$ then $z$ is in the domain of $\av{T,1}^\uparrow g $, and
\[ 4\pi\, u(C;z) \= \tilde c(C)(z) \= p_-(z)-p_+(z) -
\av{T,1}^\uparrow g 
\,. \]
The function $\bigl( \av {T,1}^\uparrow  g 
\bigr)(z)$ is given by a Fourier expansions with terms of positive
order. The term $p_--p_+$ has a Fourier expansion with terms of
negative order. Hence the Fourier coefficient of $u$ at~$\infty$ of
order~$0$ vanishes.
\end{proof}

\begin{cor}\label{cor-A0imq}Let $r\in \ZZ_{\geq 2}$. Then
$\hpar^1(\Gm;\esv{v,r}\om,\esv{v,r}{\fsn,\wdg})= \bcoh r \om\,
A_r^0(\Gm,v)$.
\end{cor}

\begin{proof} By Proposition~\ref{prop-impcA} we have
$\al_r\hpar^1(\Gm;\esv{v,r}\om,\esv{v,r}{\fsn,\wdg})\subset
A^0_r(\Gm,v)$.

A given class $[c]\in \hpar^1(\Gm;\esv{v,r}\om,\esv{v,r}{\fsn,\wdg})$
has image $\al_r[c]\in A_r^0(\Gm,v)$, and hence $\bcoh r \om \al_r
[c] \in  \bcoh r \om A^0_r(\Gm,v) $. 
Part~ii) of Theorem~\ref{thmbg} implies
that $\al_r \bcoh r \om \al_r [c]= \al_r[c]$, and then $\bcoh r \om
\al_r[c]= [c]$, by the injectivity of $\al_r$ in Part~i)
of that theorem. This proves that
$\hpar^1(\Gm;\esv{v,r}\om,\esv{v,r}{\fsn,\wdg})\subset \bcoh r \om\,
A_r^0(\Gm,v)$.

The other inclusion follows from $(\bcoh r \om)^{-1}\hpar^1(\Gm;\esv
{v,r} \om,\esv{v,r}{\fsn,\wdg})=A^0_r(\Gm,v)$
(Proposition~\ref{prop-A0-a0}).
\end{proof}

\subsection{Exact sequences for mixed parabolic cohomology
groups}\label{sect-es}
\begin{prop}\label{prop-esH}
Let $r\in \ZZ_{\geq 2}$. We put\il{Kvr}{$K_{v,r}$} $K_{1,2}:=
\dsv{1,0}\pol \cong\CC$ (trivial representation), and $K_{v,r}:=\{0\}
$ if $r\in \ZZ_{\geq 3}$ or $v\neq 1$.

The rows in the following commuting diagram are exact. {\rm (We have
suppressed $\Gm$ from the notation.)}
\badl{DED-exs} \xymatrix@C=.5cm{ K_{v,r} \ar[r] \ar@{=}[d] &
H^1(\dsv{v,r}\om) \ar[r] & H^1(\esv{v,r}\om) \ar[r]^{\rs_r} &
H^1(\dsv{v,2-r}\pol) \ar[r] & 0 \\
K_{v,r} \ar[r] & \hpar^1(\dsv{v,r}\om,\dsv{v,r}{\fsn,\wdg}) \ar[r]
\ar@{^{(}->}[u] & \hpar^1(\esv{v,r}\om,\esv{v,r}{\fsn,\wdg})
\ar[r]^{\rs_r} \ar@{^{(}->}[u] & \hpar^1(\dsv{v,2-r}\pol)
\ar@{^{(}->}[u] } \eadl
\end{prop}

\begin{proof}We use the following commuting diagram of $\Gm$-modules
with exact rows:
\badl{ed} \xymatrix@R=.5cm{ 0\ar[r] & \dsv{v,r}\om \ar[r]
\ar@{^{(}->}[d]
& \esv{v,r}\om \ar[r]^{\rs_r} \ar@{^{(}->}[d]
& \dsv{v,2-r}\pol \ar[r] \ar@{=}[d]
& 0
\\
0\ar[r] & \dsv{v,r}{\fsn,\wdg} \ar[r]
& \esv{v,r}{\fsn,\wdg} \ar[r]^{\rs_r}
& \dsv{v,2-r}\pol \ar[r]
& 0 } \eadl
See Part~iii) of Proposition~\ref{prop-iw-esv}. 

The upper row in \eqref{DED-exs} is part of the corresponding
long exact sequence in group cohomology. Since $\Gm$ has
cusps, all groups $H^2(\Gm;V)$ are zero. (See, {\sl e.g.},
\cite[\S11.2]{BLZm}.)
The $\Gm$-invariants of $\dsv{v,2-r}\pol$ are zero, unless $r=2$ and
$v=1$, when $\dsv{1,2-2}\pol$ is the trivial representation. This
gives the exactness of the upper row.

To use \cite[Proposition 11.9]{BLZm} for the lower row, we need also
exactness of
\[ 0 \rightarrow (\dsv{v,r}{\fsn,\wdg})^{\Gm_\ca} \rightarrow
(\esv{v,r}{\fsn,\wdg})^{\Gm_\ca} \stackrel{\rs_r}\rightarrow
(\dsv{v,2-r}\pol)^{\Gm_\ca}\rightarrow 0\]
for each cusp $\ca$ of $\Gm$. Most of the exactness follows from
Part~iii) in Proposition~\ref{prop-iw-esv}. For the surjectivity of
$\rs_r$ we conjugate $\ca$ to $\infty$, and use the Fourier expansion
in Part~ii) of Lemma~\ref{lem-per-harm}. The restriction $\rs_r$
sends the holomorphic contributions to zero. Only if $v(\pi_\ca)=1$
there may be a multiple of $y^{1-r}$; and we note that $\rs_r
y^{1-r}=-(2i)^{r-2}$ spans the $T$-invariants in $\dsv {2-r}\pol$. If
$v(\pi_\ca)\neq 1$ there is no multiple of $y^{1-r}$, and there are
no non-zero elements in $\dsv{2-r}\pol$ on which $T$ acts as
multiplication by $v(\pi_\ca)$.

Proposition 11.9 in \cite{BLZm} gives a long exact sequence of the
corresponding mixed parabolic cohomology groups, of which the lower
row is a part. We use \cite[(11.11)]{BLZm}, which tells us that
$\hpar^0(\Gm;V) $ is the space of
\il{VGm}{$V^\Gm$}\il{invar}{invariants}invariants~$V^\Gm$. Hence we
get $K_{v,r}$ on the left.\end{proof}

\rmrke In the second line of diagram \eqref{DED-exs} we have not
written a terminating $\rightarrow 0$. We did not succeed in proving
this directly, for instance by showing that
\[ \hpar^2(\Gm;\dsv{v,r}\om,\dsv{v,r}{\fsn,\wdg}) \rightarrow
\hpar^2(\Gm;\esv{v,r}\om,\esv{v,r}{\fsn,\wdg})\]
is injective. For unitary multiplier system the surjectivity of
\[\rs_r : \hpar^1(\Gm;\esv{v,r}\om,\esv{v,r}{\fsn,\wdg}) \rightarrow
\hpar^1(\Gm;\dsv{v,2-r}\pol)
\]
is known to hold, by classical results, as we will discuss
in~\S\ref{sect-cocl}.

\subsection{Automorphic forms and analytic boundary germ
cohomology}\label{sect-cmplaci}
We proceed under the assumption $r\in \ZZ_{\geq 2}$. In the diagram
\badl{diaAD} \xymatrix{ \hpar^1(\Gm;\dsv{v,r}\om,\dsv{v,r}{\fsn,\wdg})
\ar[r] & \hpar^1(\Gm;\esv{v,r}\om,\esv{v,r}{\fsn,\wdg})
\ar@<.4ex>[d]^{\al_r\;\; \cong}
\\
A_{2-r}(\Gm,v) \ar[u]_\cong^{\coh {2-r} \om}
& A_r^0(\Gm,v) \ar@<.4ex>[u]^{\bcoh r \om} } \eadl
we use Theorem~\ref{THMac} in the weight $2-r\in \ZZ_{\leq 0}$ to get
the isomorphism $\coh r \om$ on the left. The isomorphisms $\al_r$
and $\bcoh r \om$ on the right follow from Theorem~\ref{thmbg},
Corollaries \ref{cor-A0E} and~\ref{cor-A0imq}. The horizontal arrow
denotes the natural map associated to the inclusions $\dsv{v,r}\om
\subset \esv{v,r}\om$ and $\dsv{v,r}{\fsn,\wdg}\subset
\esv{v,r}{\fsn,\wdg}$. The following
results makes this into a commutative diagram:

\begin{lem}\label{lem-Bol}Let $r\in \ZZ_{\geq 2}$. Let
\il{cr}{$c_r$}$c_r =\frac i2\, \frac1{(r-1)!}$. Let $K_{v,r}$ be as
defined in Proposition~\ref{prop-esH}.

The following diagram commutes and has exact rows:
\bad \xymatrix{ K_{v,r}\ar[r] &
\hpar^1(\Gm;\dsv{v,r}\om,\dsv{v,r}{\fsn,\wdg}) \ar[r]^{\mathrm{id}}
& \hpar^1(\Gm;\esv{v,r}\om,\esv{v,r}{\fsn,\wdg})
\ar@<.4ex>[d]^{\al_r\;\;\cong}
\\
K_{v,r}\ar[r] \ar[u]^{\times (-i/2)}& A_{2-r}(\Gm,v)
\ar[u]^\cong_{\coh{2-r}\om } \ar[r]^{c_r\, \partial_z^{r-1}}
& A_r^0(\Gm,v) \ar@<.4ex>[u]^{\bcoh r \om} } \ead
{\rm By $\mathrm{ id }$ we indicate the homomorphism induced by the
inclusions $\dsv{v,r}\ast \rightarrow \esv{v,r}\ast$.}
\end{lem}

\begin{proof}
Bol's equality $\partial_z^{r-1}
\bigl(F|_{2-r}g) = F^{(r-1)}|_r g$ for $g\in \SL_2(\RR)$, which
appears in \cite[\S8]{Bol}, implies that $c_r\,
\partial_z^{r-1}$ determines a map $A_{2-r}(\Gm,v)\rightarrow
A_r(\Gm,v)$. Since the constant functions are the sole polynomials
that can be automorphic forms, the kernel is $K_{v,r}$. So the lower
row is exact. Proposition~\ref{prop-esH} gives the exactness of the
upper row.

For the commutativity of the left rectangle we assume $r=2$ and $v=1$.
The map $\coh {2-r}\om$ sends the constant function~$1$ to the class
represented by the cocycle $\gm\mapsto\ps_{1,\gm}^{z_0}
(t) = \frac1{t-z_0}-\frac1{t-\gm^{-1}z_0}$, for an arbitrary base
point $z_0\in \uhp$. The constant function $-\frac i2 \in K_{1,2}
\cong
(\dsv{1,0}\pol)^\Gm$ has a lift $t\mapsto
-\frac i2 K_2(t;z_0)$ in $\esv{1,2}\om$, with the kernel function
$K_2$ defined in~\eqref{Kq}. So the connecting homomorphism sends
$\ps_1^{z_0}$ to the cocycle $\ch$ determined by $\ch_\gm =
-\frac i2 K_2(\cdot;z_0)|_{1,2}(1-\nobreak
\gm)$. The kernel function $K_r$ has the invariance property
$K_r(\cdot;\cdot)|_r g \otimes |_{2-r} g
=K_r(\cdot;\cdot)$, in~\eqref{Kq-inv}. For $\gm=\matc abcd\in \Gm$:
\[ \ch_\gm(t) \= -\frac i2 K_2(t;z_0) +\frac i2 (a-cz_0)^0 \,
K_2(t;\gm^{-1}z_0) \= \frac1{t-z_0}-\frac1{t-\gm^{-1}z_0}\=
\ps_{1,\gm}^{z_0}(t)\,. \]

For the commutativity of the second rectangle we start with $F\in
A_{r-2}(\Gm,v)$ and compute its image under the composition $\al_r
\circ \mathrm{id} \circ \coh{r-2}\om$. We use the description of
cohomology with a tesselation as discussed in~\S\ref{sect-tesscoh}.
The cocycle $c$ representing $\coh r \om F$ is determined by
\[ c_F(x;z) \= \int_x (z-t)^{-r}\, F(z)\, dz\,,\]
where $x\subset \uhp$ is an oriented edge in $X_1^{\tess,Y}$.

The function $c_F(x;\cdot)$ is defined on $\CC\setminus x$, and
represents an element of $\dsv{v,2}\om$. Since $\dsv{v,r}\om
\subset\esv{v,r}\om$, the same cocycle represents $\mathrm{id}(\coh r
\om F)$. The image of $\mathrm{ id} (\bcoh r \om F)$ under the map
$\al_r $ in Proposition~\ref{prop-alr} is an automorphic form $u\in
A_r(\Gm,v)$. By analytic continuation it is determined by its value
on the interior $\mathring \fd_Y$ of the face $\fd_Y\in
X_2^{\tess,Y}$. (It is important to use a face that is completely
contained in $\uhp$; otherwise $c(x)$ need not be given by the
integral above for all edges $x$ in the boundary of $\partial_2
\fd_Y$.) We apply Proposition~\ref{prop-alr-inj}. It gives, for $z\in
 \mathring \fd_Y $
\begin{align*}
u(z) &\= \frac1{4\pi} \, c_F(\partial_2 \fd_Y)(z)
 \= \frac1{4\pi} \sum_{x\in
\partial_2
 \fd_Y} c_F(x ; z) \= \frac1{4\pi} \sum_{x\in
\partial_2
\fd_Y} \int_x (\tau-z)^{-r} \, F(\tau)\, d\tau
\displaybreak[0]
\\
&\= \frac1{4\pi}\int_{\partial_2 \fd_Y} (\tau-z)^{-r} \, F(\tau)\,
d\tau \= \frac{2\pi i}{4\pi} \frac1{(r-1)!}\, F^{(r-1)}(z) \=
\bigl(c_r\,
\partial_z^{r-1}
F\bigr)(z)\,.\qedhere
\end{align*}
\end{proof}

\begin{lem}\label{lem-coh-bcoh}Let $r\in \ZZ_{\geq 2}$. The map $\bcoh
r \om : A_r(\Gm,v) \rightarrow H^1(\Gm;\esv{v,r}\om)$ is injective
and the following diagram commutes:
\badl{Adiag} \xymatrix@R=.6cm{ H^1(\Gm;\esv{v,r}\om) \ar[rr]^{\rs_r}
&& H^1(\Gm;\dsv{v,2-r}\pol)\\
\hpar^1(\Gm;\esv{v,r}\om,\esv{v,r}{\fsn,\wdg}) \ar@{^{(}->}[u] \\
A^0_r(\Gm,v) \ar[u]^{\bcoh r \om}_{\cong} \ar@{^{(}->}[r]& A_r(\Gm,v)
\ar@{^{(}->}[luu]_{\bcoh r \om} \ar[ruu]^{\coh r \om} } \eadl
\end{lem}

\begin{proof}
Suppose that $F\in A_r(\Gm,v)$ satisfies $\bcoh r \om F=0$. Then $F$
is in the space $A^\E_r(\Gm,v)$ in~\eqref{AE}, and then $F=0$ by the
injectivity in Part~i) of Theorem~\ref{thmbg}.

The commutativity of the left triangle is a direct consequence of the
definitions. For the right triangle we start with the commutativity
of the diagram in Proposition~\ref{prop-bcoh}, where we can replace
$\dsv{v,2-r}\om$ by $\dsv{v,2-r}\pol$, since $r\in \ZZ_{\geq 2}$.
\[ \xymatrix{ H^1\bigl(\Gm;\W{v,r}\om(\proj\RR)\bigr)
\ar[rr]^{\rsp_r}
&
& H^1\bigl(\Gm;\V{v,2-r}\om(\proj\RR)\bigr)
\\
& H^1(\Gm;\esv{v,r}\om) \ar[lu] \ar[r]^{\rs_r}
& H^1(\Gm;\dsv{v,2-r}\pol) \ar[u]_{\Prj{2-r}}
\\
& A_r(\Gm,v) \ar[ru]_{\coh r \om} \ar[luu]^{\bcoh r \om} \ar[u]_{\bcoh
r \om} } \]
We have chosen the spaces $\esv{v,r}\ast$ in such a way the the image
of $\bcoh r \om$ is in $H^1(\Gm;\esv{v,r}\om)$. {}From
$\rs_r=\Prj{2-r}^{-1}\, \rsp_r$ (Definition \ref{rsl}) and Part~iii) of
Proposition~\ref{prop-iw-esv} it follows that
$\rsp_r \, \bcoh r \om =  \coh r \om$.
\end{proof}

\begin{proof}[Recapitulation of the proof of Theorem~\ref{THMaci}]The
commutativity of various parts of the diagram in~\eqref{Ediag} in
 Theorem~\ref{THMaci} follows from Proposition \ref{prop-esH} and the
Lemmas \ref{lem-Bol} and~\ref{lem-coh-bcoh}.

The exactness of the top row and the second row, in Part~ii), are
given by Proposition~\ref{prop-esH}, which gives also the information
in Part~iii) of the theorem.

The injectivity of $\bcoh r \om : A_r(\Gm,v)
\rightarrow H^1(\Gm;\esv{v,r}\om)$ is shown in
Lemma~\ref{lem-coh-bcoh}, the injectivity of the vertical maps
between cohomology groups follows directly from the definition of
(mixed)
parabolic cohomology. The bijectivity of $\coh{2-r}\om$ is given by
Theorem~\ref{THMac} for weights not in $\ZZ_{\geq 2}$, and the
injectivity of $\bcoh r \om:A_r^0(\Gm,v)\rightarrow
\hpar^1(\Gm;\esv{v,r}\om,\esv{v,r}{\fsn,\wdg})$ is a consequence of
Theorem~\ref{thmbg} and Corollary~\ref{cor-A0E}.
\end{proof}

\subsection{Comparison with classical results}\label{sect-cocl}In the
following part of diagram~\eqref{Ediag} in Theorem~\ref{THMaci}
\badl{essur} \xymatrix{ H^1(\Gm;\esv{v,r}\om) \ar[r]^{\rs_r}
& H^1(\Gm;\dsv{v,2-r}\pol) \ar[r] & 0\\
\hpar^1(\Gm;\esv{v,r}\om,\esv{v,r}{\fsn,\wdg})
\ar[r]^{\rs_r} \ar@{^{(}->}[u] & \hpar^1(\Gm;\dsv{v,2-r}\pol)
\ar@{^{(}->}[u] }\eadl
the absence of an arrow $\rightarrow0$ in the second row is
remarkable. The surjectivity of $\rs_r$ in the top row is a
consequence of the general fact that $H^2(\Gm;V)=\{0\}$ for any
$\Gm$-module for groups $\Gm$ with cusps. See, eg.,
\cite[\S11/2]{BLZm}. In the long exact sequence corresponding to the
diagram in~\eqref{ed} there is a sequel
\[ \hpar^1(\Gm;\dsv{v,2-r}\pol) \rightarrow
\hpar^2(\Gm;\dsv{v,r}\om,\dsv{v,r}{\fsn,\wdg})
\rightarrow \hpar^2(\Gm;\dsv{v,r}\om,\dsv{v,r}{\fsn,\wdg})\]
that may be non-zero. It would be interesting to see that in general
the second row in~\eqref{essur} is surjective.

We review some classical results, under the assumption that the
multiplier system $v$ for weight $r\in \ZZ_{\geq 2}$ is unitary.

The elements of $\dsv{2-r}\pol$ are polynomial functions, and hence
are holomorphic on $\CC$. This space of polynomials is invariant
under the involution $\Ci$ in~\eqref{Ci}. The action is changed
(unless $v$ is real-valued): $\Ci:\dsv{v,2-r}\pol\leftrightarrow
\dsv{\bar v,2-r}\pol$. This induces involutions
\bad\xymatrix{ H^1(\Gm;\dsv{v,2-r}\pol) \ar[r]^\Ci
& H^1(\Gm;\dsv{\bar v,2-r}\pol)
\\
\hpar^1(\Gm;\dsv{v,2-r}\pol) \ar[r]^\Ci \ar@{^{(}->}[u]
& \hpar^1(\Gm;\dsv{\bar v,2-r}\pol) \ar@{^{(}->}[u] }\ead
The linear map $\coh r \om:A_r(\Gm,v)\rightarrow
H^1(\Gm;\dsv{v,2-r}\pol)$ has an antilinear counterpart $\Ci \coh r
\om: A_r(\Gm,v) \rightarrow H^1(\Gm;\dsv{\bar v,2-r}\pol)$, in which
$\Ci\coh r \om F$ is represented by
\[ \gm \mapsto \int_{z=\gm^{-1}z_0}^{z_0} \overline{F(z)} \, (\bar
z-t)^{r-2}\, d\bar z\,.\]

We now look at the classical theory in \cite{HuKn}, where Theorem~1
gives
\be H^1(\Gm;\dsv{v,2-r}\pol)\= \coh r \om M_r (\Gm,v)\,\oplus\, \Ci
 \coh{\bar r}\om \cusp r(\Gm,\bar v)\,.\ee

The restriction of $\coh r \om$ to $M_r(\Gm,v)$ is a multiple of the
map $\bt$ in~\cite{HuKn} and \cite[\S1.3]{Kn74}. It is described by
$(r-\nobreak1)$-fold integration. The construction of $(\Ci \coh
{\bar r}\om)f$ for $f\in \cusp r(\Gm,\bar v)$ is carried out by
forming $ g^\ast \in A_r(\Gm,v)$, and then forming, with the
``supplementary function'', (a multiple of)
$\coh r \om g^\ast$ with the property that $\coh r \om g^\ast$ is a
multiple of $ \Ci \coh{\bar r}\om f$. (The resulting antilinear map
$\cusp r(\Gm,\bar v)\rightarrow H^1(\Gm;\dsv{v,w-r}\pol)$ is called
$\al$ in~\cite{HuKn}. In particular, $\coh r \om g$ is a parabolic
class, in $\hpar^1(\Gm;\dsv{v,2-r}\pol)$. The computations in
\S\ref{sect-constr-peq}, especially Lemma~\ref{lem-peq-cst}, show
that $g\in A_r^0(\Gm,v)$. With Theorem~1 in \cite{HuKn}, we conclude
that $\hpar^1(\Gm;\dsv{v,2-r}\pol)$ is contained in $\coh r \om \,
A_r^0(\Gm,v)$. The diagram in Theorem~\ref{THMaci} implies that
$\hpar^1(\Gm;\dsv{v,2-r}\pol)= \coh r \om \, A_r^0(\Gm,v)$. So
indeed, the classical theory gives us the missing surjectivity, for
unitary multiplier systems.

We note that in \cite{Kn74} the map $\al$ is constructed in a
different way, with automorphic integrals of Niebur \cite{Ni74}. For
the purpose of this subsection the supplementary functions used in
\cite{HuKn} are more useful.

\rmrke Knopp, Lehner and Raji \cite{KLR9} \cite{KR10} \cite{Ra11,Ra13}
have studied cohomology classes associated to \emph{generalized
modular forms} for which the multiplier systems need to satisfy
$|v(\pi)|=1$ only for parabolic $\pi\in\Gm$.

\subsection{Related work}\label{sect-lit11}
In this section we connected the classical results concerning the
relation between automorphic forms and Eichler cohomology to the
boundary germ cohomology in Theorem~\ref{thmbg}.


\part{Miscellaneous}

We have proved Theorems \ref{THMac}--\ref{THMaci} in the introduction,
and some of the isomorphisms in Theorem~\ref{THMiso} in
\S\ref{sect-isocg}. In Sections \ref{sect-isos-pc}
and~\ref{sect-coc-sing} we complete the proof of
Theorem~\ref{THMiso}. In Section~\ref{sect-qaf} we discuss quantum
automorphic forms and their relation to cohomology. We close this
part with Section~\ref{sect-lit}, which gives further remarks on the
literature.


\section{Isomorphisms between parabolic cohomology
groups}\label{sect-isos-pc}


\subsection{Invariants under hyperbolic and parabolic elements} For
$\Gm$-modules $V\subset W$ there is a natural map $\hpar^1(\Gm;V,W)
\rightarrow \hpar^1(\Gm;W)$, which turns out to be an isomorphism in
several cases under consideration. It takes quite some work to sort
this out. As a first step, we consider for parabolic and hyperbolic
elements $\gm\in\Gm$ the spaces
\il{Vinvgm}{$V^\gm$}$V^\gm=\bigl\{v\in V\;:\; v|\gm=v\bigr\}$ of
\il{inv}{invariants}invariants in the $\Gm$-modules $V$ under
consideration.

\rmrk{Parabolic elements} Lemma~\ref{lem-sing-inv} implies that for a
parabolic $\pi \in \Gm$ we have $(\dsv{r,2-r}\fs)^\pi \subset
\dsv{v,2-r}\om[ \ca]$, where $\ca$ is the cusp fixed by~$\pi$.

\begin{lem}\label{lem-par-inv}Let $r\in \CC$, and let $\pi\in \Gm$ be
parabolic. We denote $\ld=v(\pi)$.
\begin{enumerate}
\item[a)] The dimensions of various spaces of invariants are as
follows:
\[\renewcommand\arraystretch{1.4}
\begin{array}{|c|c|c|c|}\hline
& r\not\in \ZZ_{\geq 1}\text{ or }\ld\neq 1 & r=1\text{ and }\ld=1 &
r\in \ZZ_{\geq 2}\text{ and }\ld=1\\ \hline
\dim(\dsv{v,2-r}\fs)^\pi&\infty&\infty&\infty \\
\dim(\dsv{v,2-r}{\fs,\wdg})^\pi&\infty&\infty&\infty \\
\dim(\dsv{v,2-r}{\fs,\smp})^\pi&0&1&1 \\
\dim(\dsv{v,2-r}{\fs,\infty})^\pi&0&0&1\\ \hline
\end{array}
\]
\item[b)] In all cases $(\dsv{v,2-r}{\fs,\infty})^\pi =
(\dsv{v,2-r}{\om})^\pi$, and $(\dsv{v,2-r}\om)^\pi =
(\dsv{v,2-r}\pol)^\pi$ if $r\in \ZZ_{\geq 2}$.
\end{enumerate}
\end{lem}

\begin{proof}Going over to $\pi^{-1}$ if necessary, the element $\pi$
is conjugate in $\SL_2(\RR)$ to $T=\matc1101$. After conjugation we
find that invariance amounts to $\ph(t+\nobreak 1) = \ld \,\ph(t)$,
with $\ld=v(\pi)\in \CC^\ast$. This has solutions given by
$\sum_{n\equiv \al(1)} a_n\, e^{2\pi i n t}$ with $e^{2\pi i
\al}=\ld$. For $\dsv{v,2-r}\fs$ we need convergence on a half-plane
$\im t<\e$ for some $\e>0$. For $\dsv{v,2-r}{\fs,\wdg}$ the
 $\ld$-periodicity of $\ph$ implies that $\ph$ extends holomorphically
 to all of $\CC$, and hence we need convergence on all of $\CC$. In
both cases in Part~i) we get an infinite-dimensional space of
invariants.

In the other parts there is a condition at~$\infty$, which implies
that $(\Prj{2-r}\ph)(t) := (i-\nobreak t)^{2-r}\allowbreak\,\ph(t)$
has an asymptotic expansion of the form $ (\Prj{2-r}\ph)(t) \sim
\sum_{\ell\geq k} c_\ell \, t^{-\ell}$, valid as $t\in \lhp$
approaches~$\infty$. For $\dsv{v,2-r}{\fs,\smp}$ we have $k=-1$, and
for $\dsv{v,2-r}{\fs,\infty}$ and its submodules, $k=0$.

So if $\ph\neq 0$ the expansion starts with $d_n \, t^{r-2-n}
+d_{n+1} \, t^{r-3-n}+\cdots$, where $d_n\neq 0$ and $n\geq k$. We
insert this into the invariance relation. If $\ld\neq 1$, the
starting term shows that $d_n=0$. So for $\ld\neq 1$ no invariants
exist in $\dsv{v,2-r}{\fs,\smp}$ and smaller modules.

If $\ld=1$ then we find from the second term that
$d_n\,\left(r-2-n\right)=0$. So for an invariant the expansion should
start at $n=r-2$. Since $n\geq k$, this leads to $r\in\ZZ_{\geq 1}$
for $\dsv{v,2-r}{\fs,\smp}$, and $r\in \ZZ_{\geq 2}$ for the smaller
modules. Thus we have $\ph(t)=d_{r-2}+ d_{r-1}\, t^{-1} + \cdots$.
There is indeed an easy invariant under these conditions, namely the
constant function $\ph_{\mathrm{cst}}(t)=1$. It is in each of the
modules. 
Then $\ph-d_{r-2}\ph_{\mathrm{cst}}$ is $1$-periodic and $\oh(t^{-1})$,
hence zero. So the dimension of the space of
invariants equals~$1$.
\end{proof}

\rmrke For $\pi=T$ and $v(T)=1$, the proof shows that the invariants
are the constant functions. For other parabolic elements
$\pi=gTg^{-1}$ the spaces of invariants in Parts ii)--iv) of the
lemma have dimension $1$, but need not consist of constant functions.

\rmrk{Hyperbolic elements and closed geodesics} An element $\gm=\matc
abcd \in \SL_2(\RR)$ is \il{hypelt}{hyperbolic}\emph{hyperbolic} if
$a+d>2$. A hyperbolic element $\gm$ of $\SL_2(\RR)$ has exactly two
invariant points in $\proj\CC$, situated on $\proj\RR$, say $\xi$ and
$\xi'$. On the geodesic in $\uhp$ connecting $\xi$ and $\xi'$ the
action of $\gm$ on the points of the geodesic amounts to a shift over
a fixed distance for the hyperbolic metric, which we call
\il{elleta}{$\ell(\gm)$}$\ell(\gm)$. We note that $\ell(\gm^n)
= |n|\, \ell(\gm)$ for $n\in \ZZ$. The image in $\Gm\backslash \uhp$
of that invariant geodesic is a so-called \il{clgeo}{closed
geodesic}closed geodesic, with length~$\ell(\gm)$.

A \emph{hyperbolic subgroup} $\Eta$ of $\Gm$ is a subgroup generated
by a hyperbolic $\gm $ and $-1$. Such a hyperbolic generator $\gm$
is a \il{phe}{primitive hyperbolic}\emph{primitive hyperbolic
element} of~$\Gm$. The inverse $\gm^{-1}$ is the other primitive
hyperbolic element in~$\Eta$. We can conjugate a hyperbolic element
$\gm$ in $\SL_2(\RR)$ to $\matc {p^{1/2}}00{p^{-1/2}}$ with
$p=e^{\ell(\gm)}>1$. This element has $\infty$ as attracting fixed
point, and $0$ as repelling fixed point.

\begin{lem}\label{lem-sing-hyp}
Let $\ld\in \CC^\ast$, and let $\gm\in \SL_2(\RR)$ be hyperbolic. If
$f\in \dsv{2-r}\fs$ satisfies $f|_{2-r}\gm=f$, then $f\in
\dsv{2-r}\om[\xi,\xi']$, where $\xi$ and $\xi'$ are the fixed points
of~$\gm$.
\end{lem}

\begin{proof}Analogous to the proof of Lemma~\ref{lem-sing-inv}.
\end{proof}

To formulate the following result it is convenient to introduce for a
hyperbolic $\gamma\in \Gamma$ the quantity
\il{kp-hyp}{$ \k=\k_{v,2-r}(\gm)$}$\k:= \k_{v,2-r}(\gm)\in \CC$ that
is uniquely determined by
\be \label{kpdef}
e^{\k \,\ell(\gm)} \= v(\gm)\,
e^{(r/2-1)\,\ell(\gm)}\,,\quad\text{and} \quad
-\frac\pi{\ell(\gm)} < \im \k \leq \frac\pi{\ell(\gm)}\,,
\ee
where $\ell(\gm)$ is the length of the periodic geodesic corresponding
to~$\gm$.

\begin{lem}\label{lem-hyp-inv}Let $r\in \CC$, and let $\gm$ be a
hyperbolic element of $\Gm$, corresponding to a closed geodesic in
$\Gm\backslash\uhp$ with length~$\ell(\gm)$.
\begin{enumerate}
\item[a)] With $\k=\k_{v,2-r}(\gm)$ as in \eqref{kpdef} the dimensions
of various spaces of invariants are as follows:
\[ \renewcommand\arraystretch{1.4}
\begin{array}{|c|c|c|c|}\hline
&r\not\in \ZZ \text{ or } & r\in \ZZ_{\geq 0} \text{ and}& \k,r\in \ZZ
\text{ and}\\
&\k\leq -2 \text{ or }\k\geq r &\k\in\{-1,r-1\}&0\leq \k \leq r-2
\\\hline
\dim (\dsv{v,2-r}\fs)^\gm & \infty & \infty & \infty \\
\dim (\dsv{v,2-r}{\fs,\wdg})^\gm & \infty & \infty & \infty \\
\dim (\dsv{v,2-r}{\fs,\smp})^\gm & 0 & 1 & 1\\
\dim(\dsv{v,2-r}{\fs,\infty})^\gm & 0 & 0 & 1 \\ \hline
\end{array}
\]
\item[b)] In all cases $(\dsv {v,2-r}\om)^\gm = (\dsv
{v,2-r}{\fs,\infty})^\gm$, and
$(\dsv{v,2-r}\om)^\gm=(\dsv{v,2-r}\pol)^\gm$ if $r\in\ZZ_{\geq 2}$.
\end{enumerate}
\end{lem}

\begin{proof}
We conjugate $\gm$ in $\SL_2(\RR)$ to $\matc {p^{1/2}}00{p^{-1/2}}$,
where $p=e^{\ell(\gm)}$, which leaves fixed $0$ and $\infty$, and the
geodesic between them. This leads to the equation
\be\label{p-phi-rel}
p^{1-r/2}\, \ph(pt)\=v(\gm)\, \ph(t)\,.\ee

By examining the Fourier expansion of $\ps(x):=\ph(e^x)$, the
solutions in $\dsv{v,2-r}\fs$ can be obtained as
\be \label{it-ser}
\sum_{n\in \ZZ} d_n \, \bigl(it\bigr)^{\alpha +2\pi i
n/\ell(\gm)}\,,\ee
convergent for at least $-\frac\pi2<\arg(it)<\frac\pi 2$, and $\al$ in
the set
\badl{hyp-Ft} E_{v,2-r}(\gm) &\= \Bigl\{\frac r2-1+\frac{\log
v(\gm)+2\pi i m}{\ell(\gm)}\;:\; 
m \in \ZZ \Bigr\}\\
&\= \Bigl\{ \k_{v,2-r}(\gm) + \frac{2\pi i  m}{\ell(\gm)}\;:\; 
m \in \ZZ \Bigr\}\,. \eadl
With the standard choice of the argument, $(it)^\al$ is well defined on
$\CC\setminus i[0,\infty)$. 

By Lemma~\ref{lem-sing-hyp}, a function $\ph$ representing an
element of $(\dsv{v,2-r}\fs)^\gm$ should extend holomorphically to
neighbourhoods of $(0,\infty)$ and $(-\infty,0)$ in $\CC$. The
$\gm$-invariance implies that $\ph$ should be holomorphic on a
$\Gm$-invariant domain. So $\ph$ should be holomorphic on at least 
a region
$-\pi-\e<\arg t < \e$ for some $\e>0$. Hence the series
in~\eqref{it-ser} represents an element of $(\dsv{v,2-r}\fs)^\gm$
precisely if it converges
on a region $-\frac\pi 2-\e < \arg(it) < \frac\pi2+\e$ with
some $\e>0$. To get a holomorphic function on an excised neighbourhood
with excised set $\{0,\infty\}$, we need to pick coefficients such
that we have convergence for $-\pi<\arg(it)<\pi$. In this way we
obtain a complete description of the, infinite-dimensional, spaces
$(\dsv{v,2-r}\fs)^\gm$ and $(\dsv{v,2-r}{\fs,\exc})^\gm$ in
the first two lines in the table in Part~a).

For the smaller modules there should be asymptotic expansions at $0$
and $\infty$. Let $k=-1$for $\dsv{v,2-r}{\fs,\smp}$ and $k=0$ for
$\dsv{v,2-r}{\fs,\infty}$. In the expansion at zero,
$d_n=0$ for $n\neq 0$, and $\al$ should be in $E_{v,2-r}(\gm) \cap
\ZZ_{\geq k}$.  The function $t\mapsto t^{r-2}\,
\ph(-1/t)$ should also have an
expansion with terms $t^{m}$ with $m\geq k$. Hence we have the
further restriction $r-2-\al \in \ZZ_{\geq k}$. So the exponents $\al
\in E_{v,2-r}(\gm)$ should satisfy
\[ \al \in E_{v,2-r}(\gm) \cap \ZZ \cap \bigl( r-2+\ZZ)\quad\text{ and
}k\leq \al \leq r-2-k\,. \]
So we should have $r\in \ZZ$. The condition on $\im \k_{v,2-r}(\gm)$
in \eqref{kpdef} implies that $\al =\k_{v,2-r}(\gm)\in \ZZ$ and
$n=0$. The remaining condition gives $k\leq \k_{v,2-r}(\gm)\leq
r-2-k$. This gives the third and fourth line in
the table. This completes the proof of Part~a).

Moreover, if $r\in \ZZ_{\geq 2}$, any invariant $t\mapsto t^\k$ that
 is in $\dsv{v,2-r}{\fs,\infty}$ is in $\dsv{v,2-r}\pol$. This gives
Part~b).
\end{proof}

\rmrke The characterization depends on the primitive hyperbolic
element~$\gm$. The element $\gm^{-1}$ is primitive hyperbolic as
well, and
\be \k_{v,2-r}(\gm^{-1})\;\equiv\; r-2-\k_{v,2-r}(\gm)\bmod 2\pi
i/\ell(\gm)\,.\ee
The transition $x\mapsto r-2-x$ maps the set $E_{v,2-r}(\gm)$ in
\eqref{hyp-Ft} into $E_{v,2-r}(\gm^{-1})$.

\begin{lem}Suppose that both fixed points $\xi$ and $\xi'$ of the
hyperbolic element $\gm\in \Gm$ are in $\RR$ and satisfy $\xi<\xi'$.
If $r$ and $\k $ in the previous lemma are integral, and $-1\leq
\k\leq r-1$, then the $\gm$-invariants in
$\dsv{v,2-r}{\fs,\smp}$ are spanned by the rational function
\[ t\mapsto (t-\xi')^{r-2-\k} (t-\xi)^\k \,.\qedhere\]
\end{lem}

\begin{proof}We use
\[ g \=|\xi-\xi'|^{-1/2}\, \matc {\xi'}\xi11 \in \SL_2(\RR)\]
to transform the geodesic $i(0,\infty)$ into the geodesic from $\xi$
to $\xi'$. We work with $t\in \lhp$, and denote by
\il{pt=}{$\stackrel\pnt=$}$\stackrel\pnt=$ that we ignore non-zero
factors that do not depend on~$t$:
\[
(it)^\k|_{2-r} g^{-1}\; (t) \;\stackrel\pnt=\; (-t+\xi')^{r-2} \,
\Bigl( \frac{t-\xi}{-t+\xi'}\Bigr)^\k \;\stackrel\pnt=\;
(t-\xi')^{r-2-\k}\,(t-\xi)^\k\,.\qedhere\]
\end{proof}


\subsection{Modules of singularities} In the sequel we will deal with
two $\Gm$-mod\-ules $V\subset W$, where $V=\dsv{v,2-r}\om$, and $W$
is one of the following larger modules:
\be\label{VW} (a):\; \dsv{v,2-r}\fs\,,\quad (b):
\;\dsv{v,2-r}{\fs,\wdg}\,,\quad
(c): \;\dsv{v,2-r}{\fs,\smp}\,,\quad (d):\; \dsv{v,2-r}{\fs,\infty}\,.
\ee

\begin{defn}
In each of the cases in~\eqref{VW} we consider the quotient
module\ir{Sg}{\Sg{}{}, \; \Sg{v,2-r}{\fs,\ast}}
\be \label{Sg}\Sg{}{} \;:=\; W/V\,,\ee
which we call the \il{msing}{module of
singularities}\il{singm}{singularities, module of}\emph{module of
singularities}. We write $\Sg{v,2-r}\fs$, $\Sg{v,2-r}{\fs,\wdg}$,
$\ldots$, if we want to indicate the case under consideration
explicitly.
\end{defn}

\begin{defn}For $\xi\in \proj\RR$ we put
\il{Sg-xi}{$\Sg\xi{}$}$\Sg\xi{} := W[\xi]/V \subset\Sg{}{}$, where
\il{W[]}{$W[\xi]$}$W[\xi]$ consists of the elements $f\in W$ with
$\bsing f \subset \{\xi\}$.
\end{defn}

\rmrks
\itmi The space $\Sg\xi{}$ is a subspace of $\Sg{}{}$, not the stalk
of a sheaf.

\itm The direct sum $\bigoplus_{\xi\in \proj\RR} \Sg\xi{}$ is a
submodule of $\Sg{}{}$.

\begin{defn}We say that the module $\Sg{}{}=W/V$ has
\il{sepsing}{separation of singularities}\emph{separation of
singularities} if
\be\label{sesidcp} 
\Sg{}{}\= \bigoplus_{\xi\in \proj\RR} \Sg\xi{}\,.\ee
\end{defn}

\begin{prop}\label{prop-sepsing}For all cases in~\eqref{VW} the module
$\Sg{}{}$ has separation of singularities.
\end{prop}

\begin{proof}In \cite[Proposition 13.1]{BLZm} this is shown for the
sheaves used in that paper: $V= \V s \om$, the sheaf of analytic
functions with action of $\PSL_2(\RR)$ specified by the spectral
parameter $s$ and $W$ a subsheaf of $\V s{\fs,\wdg}$. It is based on
the result of complex function theory that if $\Om_1$ and $\Om_2$ are
open sets in $\CC$ any holomorphic function $f$ on $\Om_1\cap \Om_2$
can be written as $f=f_1-f_2$ with $f_1\in \hol(\Om_1)$, $f_2\in
\hol(\Om_2)$. See, e.g., \cite[Proposition 1.4.5]{Ho}.

This shows that if, for some open set containing $\lhp$
and $\proj\RR \setminus\{\xi_1,\ldots,\xi_n\}$, $f\in \hol(U)$ 
represents an element of $\dsv{v,2-r}\fs$, then we can take a 
neighbourhood $U_1\supset U$
of $\lhp \cup\bigl( \proj\RR \setminus \nobreak\{\xi_1\}\bigr)$, and
$U_2\supset U$ a neighbourhood of $\lhp \cup\bigl(
\proj\RR\setminus\nobreak\{\xi_2,\ldots,\xi_n\}\bigr)$. Then $f_2$
 represents an element of $\dsv{v,2-r} \om[\xi_1]$ and $f_1$ an
element of $\dsv{v,2-r} \om[\xi_2,\ldots,\xi_n]$. Successively we can
write each element of $\dsv{v,2-r}\om[\xi_1,\ldots,\xi_n]$
non-uniquely as the sum of elements in the spaces
$\dsv{v,2-r}\om[\xi_j]$. This shows that $\Sg{v,2-r}\fs$ has
separation of singularities.

For the other spaces $\Sg{}{}=W/V$ we have $W\subset \dsv{v,2-r}\fs$.
All these subspaces are defined by conditions on the singularities of
a local nature, based on the properties of a representative at each
$\xi$ in the set of singularities separately. Addition of an element
for which $\xi$ is not in the set of boundary singularities does not
influence the condition at~$\xi$. So separation of singularities is
inherited from $\Sg{v,2-r}\fs$.
\end{proof}

\begin{lem}\label{lem-hyp-is}Let $r\in \CC$.
\begin{enumerate}
\item[i)] The space of invariants $\Sg{}\Gm$ is zero.
\item[ii)] Let $\gm\in \Gm$ be hyperbolic, with fixed points $\xi$
and~$\xi'$. By $\ell(\gm)$ we denote the length of the associated
geodesic. Then the dimensions of the spaces of invariants are as
follows.
\[ \renewcommand\arraystretch{1.4}
\begin{array}{|c|c|c|}\hline
& v(\gm) \neq e^{-r \ell(\gm)/2} &v(\gm) = e^{-r \ell(\gm)/2} \\
\hline
\dim\bigl((\Sg{v,2-r}\fs)_\xi\bigr)^\gm& \infty & \infty \\
\dim\bigl((\Sg{v,2-r}{\fs,\wdg})_\xi\bigr)^\gm& \infty & \infty \\
\dim\bigl((\Sg{v,2-r}{\fs,\smp})_\xi\bigr)^\gm& 0 & 1 \\
\dim\bigl((\Sg{v,2-r}{\fs,\infty})_\xi\bigr)^\gm& 0 & 0 \\ \hline
\end{array}
\]
\end{enumerate}
\end{lem}

\begin{proof}
For Part~i) we note that if $f\in W$ represents an element of
$\Sg{}\Gm$, then $\singr f$ is a $\Gm$-invariant subset of
$\proj\RR$. All orbits in $\proj\RR$ of the cofinite discrete group
$\Gm$ are infinite. However, elements of $\dsv{v,2-r}\fs$ have only
finitely many singularities. 

In Part~ii) we denote $V=\dsv{v,2-r}\om$, and take for $W$ one of the
modules $\dsv{v,2-r}\fs$, $\dsv{v,2-r}{\fs,\wdg}$,
$\dsv{v,2-r}{\fs,\smp}$, and $\dsv{v,2-r}{\fs,\infty}$. There is an
injective map $W^\gm/V^\gm\rightarrow \Sg{}\gm$. The image is
contained in $\Sg{\xi}{}\oplus \Sg{\xi'}{}$. With separation of
singularities, we can split each element of $\Sg{}\gm$ as a component
in $(\Sg{\xi}{})^\gm$ and a component in $(\Sg{\xi'}{})^\gm$.

We conjugate $\gm$ in $\SL_2(\RR)$ to $\matc {p^{1/2}}00{p^{-1/2}}$
with $p=e^{\ell(\gm)}$. The invariants $t\mapsto
(it)^{\k+\frac{2\pi in}{\ell(\gm)}}$ in the proof of
Lemma~\ref{lem-hyp-inv} have a singularity at~$0$, except possibly
in the case $n=0$. This leads to the first two lines in the table.

Now let $W=\dsv{v,2-r}{\fs,\smp}$ or $W=\dsv{v,2-r}{\fs,\infty}$. The
component in $(\Sg 0 {})^\gm$ of the image $f+V$ in $\Sg{}{}$ is
invariant if and only if $0\in \bsing f$ and $f|_{v,2-r}(\gm-\nobreak
1)\in V$. Let $f(t)\sim \sum_{n\geq k} c_m\, t^m$ be the asymptotic
expansion at~$0$, with $k=-1$ for $\dsv{v,2-r}{\fs,\smp}$, and $k=0$
for $\dsv{v,2-r}{\fs,\infty}$.

If $k=-1$ the term $c_{-1}t^{-1}$ can be non-zero if
$p^{-r/2}=v(\gm)$. Then $f(t)=t^{-1}$ leads to a non-zero element of
$(\Sg0{})^\gm$.

If $m_0\geq 0$ a term with $p^{-r/2+1+m_0}=v(\gm)$ leads to an
invariant in $W^\gm$ which is also in $V^\gm$, so not to a non-zero
element of $(\Sg{0}{})^\gm$. The remaining asymptotic series
$\sum_{m\geq 0,\,m\neq m_0} c_m \, t^m$ for $\ph \in W[\xi]$, leads
to an asymptotic series
\[\sum_{m\geq 0,\,m\neq m_0} c_m\, \left(v(\gm)^{-1}
p^{-r/2+1+m}-1\right)\, t^m\]
for $\ph|_{v,2-r}(\gm-\nobreak 1)$. For an invariant in $\Sg{}{}$ this
last series should be convergent on a neighbourhood of $0$ in~$\CC$.
But since $p^m$ is exponentially increasing, then also $\sum_{m\geq
0,\,m\neq m_0} c_m \, t^m$ is convergent on that
neighbourhood, and hence is in~$V$.
\end{proof}

\subsection{Mixed parabolic cohomology and parabolic cohomology}For
$\Gm$-modules $V\subset W$ as in \eqref{VW} there is a natural map
$\hpar^1(\Gm;V,W)\rightarrow \hpar^1(\Gm;W)$. We'll show that it is
injective, and investigate its surjectivity.

\begin{lem}\label{lem-esVWS}
Let $V=\dsv{v,2-r}\om\subset W$, where $W$ is one of the modules
$\dsv{v,2-r}{\fs,\cond}$ in~\eqref{VW}, or one of the corresponding
modules $\dsv{v,2-r}{\fsn,\cond}$. The following sequence is exact:
\be\label{esVWS}
0\rightarrow \hpar^1(\Gm;V,W) \rightarrow \hpar^1(\Gm;W) \rightarrow
H^1(\Gm;\Sg{}{})\,.
\ee
\end{lem}

\begin{proof}The exact sequence of $\Gm$-modules $0\rightarrow V
\rightarrow W \rightarrow \Sg{}{} \rightarrow 0$ induces a long exact
sequence in mixed parabolic cohomology. This is discussed in
\cite{BLZm} at the end of \S11. We use the following part of the long
 exact sequence:
\be \label{expcSg} H^0(\Gm;\Sg{}{}) \rightarrow \hpar^1(\Gm;V,W)
\rightarrow \hpar^1(\Gm;W)
\rightarrow H^1(\Gm;\Sg{}{})
\ee
Part~i) of Lemma~\ref{lem-hyp-is} leads to the desired sequence.
\end{proof}

The lemma shows that $\hpar^1(\Gm;V,W)
\rightarrow \hpar^1(\Gm;W)$ is injective. It is surjective if the
image of $\hpar^1(\Gm;W)
\rightarrow H^1(\Gm;\Sg{}{})$ is zero.

\begin{defn}For each $\Gm$-orbit $x\subset \proj\RR$
put\ir{Sg-orb}{\Sg{}{}\{x\}}
\be\label{Sg-orb} \Sg{}{}\{x\}:= \bigoplus_{\xi\in x} \Sg{\xi}{}\,.
\ee
\end{defn}

For each orbit $x\in \Gm\backslash \proj\RR$ the space $\Sg{}{}\{x\}$
is a $\Gm$-module. Since $\Sg{}{}$ has separation of singularities we
have
\be \Sg{}{} \= \bigoplus_{x\in \Gm\backslash\proj\RR} \Sg{}{}\{x\}\,.
\ee
To investigate the image of $\hpar^1(\Gm;W)\rightarrow
H^1(\Gm;\Sg{}{})$ we can investigate separately the images of
$H^1(\Gm;W)\rightarrow H^1(\Gm;\Sg{}{}\{x\})$. The following
statement is analogous to~\cite[Proposition 13.4]{BLZm}:

\begin{prop}\label{prop-WSg-0}Let $x$ be a $\Gm$-orbit in $\proj\RR$.
The natural map
\[ \hpar^1(\Gm;W) \rightarrow H^1(\Gm;\Sg{}{}\{x\})\]
is the zero map in each of the following cases:
\begin{enumerate}
\item[a)] the stabilizers $\Gm_\xi$ of the elements $\xi \in x$ are
equal to $\{1,-1\}$,
\item[b)] the orbit $x$ consists of cusps of~$\Gm$,
\item[c)] the stabilizers $\Gm_\xi$ of the elements $\xi \in x$
contain hyperbolic elements, with the additional condition that for
all $\gm\in \Gm_\xi$ the space of invariants $\Sg\xi\gm$ is zero.
\end{enumerate}
\end{prop}

\rmrks
\itmi This result shows an important difference between hyperbolic
elements of $\Gm$ and other elements. If the Condition~c) is not
satisfied it opens the way to construct cocycles that do not come
from automorphic forms via the injection $\coh r \om$ and the natural
map from mixed parabolic cohomology to parabolic cohomology.

\itm In Case~c) we need to check that $\Sg\xi\gm=\{0\}$ only for one
generator $\gm$ of~$\Gm_\xi$.

\itm We do not repeat the proof, since it is completely analogous to
the proof of~\cite[Proposition 13.4]{BLZm}. We explain the main
steps.
\medskip

The proof uses the description of cohomology based on a tesselation
$\tess$ of the upper half-plane, as discussed
in~\S\ref{sect-tesscoh}. We quote two lemmas from~\cite{BLZm} before
sketching the proof of Proposition~\ref{prop-WSg-0}.

\begin{lem}\label{lem-cocparbY}For each cocycle $c_1\in
Z^1(F_\pnt^\tess;W)$ there is $c\in Z^1(F_\pnt^\tess;W)$ in the same
cohomology class with the properties that $c(e)=0$ for all edges
$e\in X_{1}^\tess$ that occur in the boundary of any cuspidal
 triangle.
\end{lem}

\twocolwithpictr{We recall that each cusp $\ca$ of $\Gm$ occurs as
vertex of infinitely many faces $\pi_\ca^{-n} V_\ca\in X_2^\tess$,
$n\in \ZZ$, $\pm\pi_\ca$ generators of the stabilizer of~$\ca$. These
$\pi_\ca^{_-n}\, V_\ca$ are cuspidal triangles. The edges
$\pi_\ca^{-n}f_\ca$ form a \il{horc}{horocycle}\emph{horocycle} in
$\uhp$. If the cusp $\ca$ is in $\RR$ this horocycle is a euclidean
circle.

\quad The lemma says that we can arrange that $c$ vanishes on all
edges $\pi_\ca^{-n}e_\ca$ and $\pi_\ca^{-n}f_\ca$. }{
\setlength\unitlength{2cm}
\begin{picture}(1.4,1.8)(-.7,-.1)
\put(-.7,0){\line(1,0){1.4}}
\put(-.1,1.74){$f_\ca$}
\put(-.1,1.1){$V_\ca$}
\put(-.4,.9){$e_\ca$}
\put(-.85,1.4){$P_\ca$}
\put(.65,1.4){$\pi_\ca^{-1}P_\ca$}
\put(-.05,-.15){$\ca$}
\qbezier(0,0.)(-0.0146618,0.134618)(-0.143061,0.177639)
\qbezier(-0.143061,0.177639)(-0.27771,0.192011)(-0.347113,0.0757337)
\qbezier(0.347113,0.0757337)(0.381824,0.0916397)(0.414938,0.11065)
\qbezier(0.347113,0.0757337)(0.27771,0.192011)(0.143061,0.177639)
\qbezier(0.143061,0.177639)(0.0146618,0.134618)(0,0.)
\qbezier(0.414938,0.11065)(0.465381,0.139624)(0.511322,0.17531)
\qbezier(0.414938,0.11065)(0.322234,0.240228)(0.164964,0.214719)
\qbezier(0.164964,0.214719)(0.0158743,0.158533)(0,0.)
\qbezier(0.511322,0.17531)(0.588757,0.235581)(0.650195,0.312094)
\qbezier(0.511322,0.17531)(0.380198,0.316677)(0.19305,0.27027)
\qbezier(0.19305,0.27027)(0.0167988,0.192083)(0,0.)
\qbezier(0.650195,0.312094)(0.76972,0.462906)(0.813008,0.650407)
\qbezier(0.650195,0.312094)(0.451605,0.450243)(0.226908,0.360609)
\qbezier(0.226908,0.360609)(0.016427,0.241357)(0,0.)
\qbezier(0.813008,0.650407)(0.858212,0.858501)(0.800891,1.06359)
\qbezier(0.800891,1.06359)(0.740537,1.26781)(0.591716,1.42012)
\qbezier(0.813008,0.650407)(0.50336,0.707938)(0.250203,0.520579)
\qbezier(0.250203,0.520579)(0.011843,0.314725)(0,0.)
\qbezier(0.591716,1.42012)(0.343402,1.6575)(0,1.66667)
\qbezier(0,1.66667)(-0.343402,1.6575)(-0.591716,1.42012)
\qbezier(0.591716,1.42012)(0.022004,0.824165)(0,0.)
\qbezier(-0.591716,1.42012)(-0.740537,1.26781)(-0.800891,1.06359)
\qbezier(-0.800891,1.06359)(-0.858212,0.858501)(-0.813008,0.650407)
\qbezier(0,0.)(-0.022004,0.824165)(-0.591716,1.42012)
\qbezier(-0.813008,0.650407)(-0.76972,0.462906)(-0.650195,0.312094)
\qbezier(0,0.)(-0.011843,0.314725)(-0.250203,0.520579)
\qbezier(-0.250203,0.520579)(-0.50336,0.707938)(-0.813008,0.650407)
\qbezier(-0.650195,0.312094)(-0.588757,0.235581)(-0.511322,0.17531)
\qbezier(0,0.)(-0.016427,0.241357)(-0.226908,0.360609)
\qbezier(-0.226908,0.360609)(-0.451605,0.450243)(-0.650195,0.312094)
\qbezier(-0.511322,0.17531)(-0.465381,0.139624)(-0.414938,0.11065)
\qbezier(0,0.)(-0.0167988,0.192083)(-0.19305,0.27027)
\qbezier(-0.19305,0.27027)(-0.380198,0.316677)(-0.511322,0.17531)
\qbezier(-0.414938,0.11065)(-0.381824,0.0916397)(-0.347113,0.0757337)
\qbezier(0,0.)(-0.0158743,0.158533)(-0.164964,0.214719)
\qbezier(-0.164964,0.214719)(-0.322234,0.240228)(-0.414938,0.11065)
\end{picture}
}
\begin{proof}See the proof of~\cite[Lemma 13.2]{BLZm}.
\end{proof}

Let $x\in \Gm\backslash\proj\RR$. If a cohomology class in
$\hpar^1(\Gm;W)$ is given by a cocycle in $Z^1(F^\tess_\pnt;W)$ as in
Lemma~\ref{lem-cocparbY} its image $c\in
Z^1(F^\tess_\pnt;\Sg{}{}\{x\})$ vanishes on all edges in
$X^\tess_1\setminus X^{\tess,Y}_1$, so it is in fact a cocycle on
$F^{\tess,Y}_1$. Therefore $c$ represents a class in
$H^1(\Gm;\Sg{}{}\{x\})$. Anyhow, $c$ is a cocycle that vanishes on
all edges $\gm^{-1}f_\ca$ and $\gm^{-1}e_\ca$ with $\gm\in \Gm$ and
$\ca\in \fdcu$ (the intersection of $\proj\RR$ with the closure of
the fundamental domain $\fd$ underlying the tesselation~$\tess$).

For any edge $e\in X_1^\tess$ we denote by
\il{cexxi}{$c(e)_\xi$}$c(e)_\xi$ the component of $c(e)$ in $\Sg\xi{}$
in the decomposition $\Sg{}{}\{x\}= \bigoplus _{\xi\in x} \Sg\xi{}$.
We put, for the fixed cocycle~$c$\ir{Dxi}{D(\xi)}
\be\label{Dxi} D(\xi) \;:=\; \bigl\{ e\in X_1^\tess\;:\; c(e)_\xi\neq
0\bigr\}\,. \ee
\begin{lem}For each $\xi\in x$, $x\in \Gm\backslash\proj\RR$, the set
$D(\xi)$ consists of finitely many $\Gm_\xi$-orbits.
\end{lem}
\begin{proof}See the proof of~\cite[Lemma 13.5]{BLZm}.
\end{proof}

\begin{proof}[Sketch of the proof of Proposition~\ref{prop-WSg-0}]To
the cocycle $c\in Z^1(F^\tess_\pnt;\Sg{}{}\{x\})$ is associated a
function $X_0^{\tess, Y}\times X_0^{\tess,Y} \rightarrow
\Sg{}{}\{x\}$, also denoted $c$. The value $c(P,Q)$ is determined by
the value of $c$ on any path in $\ZZ[X_1^\tess]$ from $P$ to~$Q$.

Let $\ca$ be a cusp of~$\Gm$. If we can show that
$c(\gm^{-1}P_\ca,P_\ca)=0$ for all $\gm\in \Gm$, then the group
cocycle $\ps_\gm=c(\gm^{-1}P_\ca,P_\ca)$ vanishes, and hence the
cohomology class of $c$ is trivial.

The proof in~\cite[\S13.1]{BLZm} considers the three cases given in
Proposition~\ref{prop-WSg-0} separately. In all cases it is argued
 for a given $\xi\in x$, that there is a path in $\ZZ[X_1^{\tess,Y}]$
from $\gm^{-1}P_\ca$ to $P_\ca$ that does not contain edges in
$D(\xi)$. This gives $c(\gm^{-1}P_\ca,P_\ca)_\xi=0$, and leads to
$[c]=0$ in $H^1(\Gm;\Sg{}{}\{x\})$.\smallskip

Case~a) in Proposition~\ref{prop-WSg-0} is easiest, since in this case
$D(\xi)$ is a finite set of edges, which is easily avoided.

In Case~b) the orbit $x$ consists of cusps, and we take $\ca \in x$.
Now the set $D(\xi)$ may be infinite. Let $\gm\in\Gm$ be fixed. If
$\xi \not \in \{\ca,\gm^{-1}\ca\}$ it is shown that there is a path
from $\gm^{-1}P_\ca$ to $P_\ca$ avoiding $D(\xi)$. Then the
observation that $c(\gm^{-1}P_\ca,P_\ca) \in \Sg{\gm^{-1}\ca}{}
\oplus \Sg\ca{}$ is the basis for an argument showing that
$\gm\mapsto c(\gm^{-1}P_\ca,P_\ca)$ is a coboundary.

In Case~c) the set $D(\xi)$ may be a barrier that makes it impossible
to find a suitable path between $\gm^{-1}P_\ca$ and $P_\ca$ if they
are on opposite sides of the barrier. If this happens the cocycle
relation can be used to show that $c(\gm^{-1}P_\ca,P_\ca)_\xi $ is in
$\Sg\xi\gm$ for a generator $\gm$ of~$\Gm_\xi$. Under the additional
condition in Part~c) in Proposition~\ref{prop-WSg-0}, this invariant
is zero.
\end{proof}

\begin{thm}\label{thm-iso-mpcpc}Let $v$ be a multiplier system for the
weight $r\in \CC$ on the cofinite discrete subgroup $\Gm$ of
$\SL_2(\RR)$ with cusps. The natural map
\[ \hpar^1\bigl(\Gm;\dsv{v,2-r}\om,W\bigr) \rightarrow
\hpar^1(\Gm;W)\]
is an isomorphism for each of  the following  $\Gm$-modules~$W$:
\begin{enumerate}
\item[i)] $W$ is one of the $\Gm$-modules $\dsv{v,2-r}\fsn$,
$\dsv{v,2-r}{\fsn,\wdg}$, $\dsv{v,2-r}{\fsn,\smp}$, or
$\dsv{v,2-r}{\fsn,\infty}$.
\item[ii)] $W=\dsv{v,2-r}{\fs,\infty}$.
\item[iii)] $W=\dsv{v,2-r}{\fs,\smp}$, under the additional condition
that $v(\gm)\neq e^{-r\ell(\gm)/2}$ for all primitive hyperbolic
elements $\gm\in\Gm$. {\rm(By $\ell(\gm)$ we denote the length of the
associated closed geodesic in $\Gm\backslash\uhp$.)}
\end{enumerate}
\end{thm}

\begin{proof}We use Proposition~\ref{prop-WSg-0} to show that
$\hpar^1(\Gm;W)\rightarrow H^1(\Gm;\Sg{}{})$ is the zero map. Then
the exact sequence in~\eqref{esVWS} gives the desired bijectivity.

For the spaces $W$ in Part~i) we have $W/V=\bigoplus_{\ca \;\mathrm{
cusp}}\Sg\ca{}$, and need only Case~b) in
Proposition~\ref{prop-WSg-0}. For Parts ii) and~iii) we have to take
into account all cases in Proposition~\ref{prop-WSg-0}, and need the
vanishing of $\Sg\xi\gm$ for all hyperbolic $\gm\in \Gm$ that leave
fixed~$\xi$. Part~ii) of Lemma~\ref{lem-hyp-is} shows that this is
the case for $W=\dsv{v,2-r}{\fs,\infty}$, and also for
$\dsv{v,2-r}{\fs,\smp}$ provided $e^{-r\ell(\gm)/2}\neq v(\gm)$.
\end{proof}

\rmrk{Missing case} Missing in Theorem~\ref{thm-iso-mpcpc} is the
module $W=\dsv{v,2-r}{\fs,\exc}$. That case is discussed in
Proposition~\ref{prop-icd-hyp}.

\subsection{Related work}\label{sect-lit12}We followed closely the
approach in~\cite[\S13.1]{BLZm}.


\section{Cocycles and singularities}\label{sect-coc-sing}

There are several natural maps between cohomology groups that we did
not yet handle in the previous sections. Theorem~\ref{THMiso} in
\S\ref{sect-isocg} states explicitly some maps that are not
isomorphisms. In this section we prove those statements by
constructing cocycles with the appropriate properties.

\subsection{Cohomology with singularities in hyperbolic fixed
 points}\label{sect-cshfp}In the exceptional case in Part~iii) of
 Theorem~\ref{thm-iso-mpcpc} we want to show that the injective map
\[ \hpar^1(\Gm; \dsv{v,2-r}\om ,W) \rightarrow \hpar^1(\Gm;W)\]
is not surjective for $W=\dsv{v,2-r}{\fs,\wdg}$, or for
$W=\dsv{v,2-r}{\fs,\smp}$ under the additional condition $v(\gm)
= e^{-r \ell(\gm)/2}$ for at least one primitive hyperbolic element
of~$\Gm$.

We use the description of cohomology based on a tesselation $\tess$ of
the upper half-plane, as discussed in~\S\ref{sect-tesscoh}. 

\rmrk{Notations}We work with a hyperbolic subgroup~$\Eta$ of $\Gm$. So
$\Eta$ is generated by a primitive hyperbolic element~$\dt$, and
$-1$. All elements of $\Eta$ leave fixed the repelling fixed point
$\zt_1$ and the attracting fixed point $\zt_2$ of~$\dt$. The elements
of $\Eta$ leave invariant the geodesic between $\zt_1$ and $\zt_2$.
The image of this geodesic in $\Gm\backslash\uhp$ is a closed
 geodesic, whose length we indicate by~$\ell(\dt)$.

\begin{lem}\label{lem-path-hyp}Let $\tess$ be a $\Gm$-invariant
tesselation of~$\uhp$. Let $\dt\in \Eta$ and $\zt_1,\zt_2\in
\proj\RR$ as indicated above.

There is a path $p$ from $\zt_1$ to $\zt_2$ in $\uhp$ with the
following properties:
\begin{enumerate}
\item[a)] $p$ is an oriented $C^1$-curve in $\uhp \cup\proj\RR$, with
respect to the structure of $\proj\CC$ as a $C^1$-variety.
\item[b)] $p$ has no self-intersection, and intersects $\proj\RR$ only
in the end-points $\zt_1$ and~$\zt_2$.
\item[c)] $p$ does not go through points of $X_0^{\tess,Y}=X_0^\tess
\cap\uhp$.
\item[d)] $p$ intersects each edge $e\in X_1^\tess$ transversely, at
most a finite number of times.
\item[e)] For each edge $e\in X_1^\tess$ there are only finitely many
$\Gm$-translates $\gm^{-1}p$ that intersect~$e$.
\item[f)] $\dt^{-1}p=p$.
\end{enumerate}
\end{lem}

\rmrke All $\Gm$-translates $\gm^{-1}p$ form $C^1$-paths in $\uhp$
from $\gm^{-1}\zt_1$ to $\gm^{-1}\zt_2$ with properties b)--e), and
$(\gm^{-1}\dt\gm)^{-1}\; \gm^{-1}p=\gm^{-1}p$.

\begin{proof}Intuitively, we may start with the geodesic from $\zt_1$
to $\zt_2$ and deform it to satisfy the conditions.

More precisely, we take a point $P_0$ in the interior of a face of the
tesselation, and take a $C^1$-path $p_0$ from $P_0$ to $\dt P_0$,
taking care to arrive in $\dt P_0$ with the same derivative as $\dt
p_0$ departs from~$\dt P_0$. If $p_0$ goes through a vertex or
has a non-transversal intersection with an edge, or coincides with an
edge, we deform it locally. In this way we arrange that $p_0$
intersects finitely many edges once, transversally. Near $P_0$ and
$\dt P_0$ we have not changed $p_0$. Taking the union $\bigcup_{m\in
\ZZ} \dt^m p_0$ and closing it in $\proj\CC$, we get a $C^1$-path $p$
satisfying Properties a)-d), and~f).

The compact path $p_0$ intersects only finitely many $\Gm$-translates
of the fundamental domain $\fd$ on which the tesselation $\tess$ is
built. A given edge $e$ is contained in the closure of one
$\Gm$-translate of~$\fd$. So there are only finitely many $\gm\in
\Gm$ such that $e$ intersects $\gm^{-1}p_0$.  This
implies that the path satisfies Property~e) as well.
\end{proof}

\begin{defn}\label{epsdefn}Let $p$ be a path as in
Lemma~\ref{lem-path-hyp}, let $\gm \in \Gm$, and let $x\in \pm
X_1^\tess$ be an oriented edge.
\begin{enumerate}
\item[i)]For each point intersection point $P \in x \cap \gm^{-1}
p$ we define
\il{epsP}{$\eps_P(e,\gm^{-1}p)$}$\eps_P(x,\gm^{-1}p)\in \{\pm1\}$
depending on the orientation as indicated in Figure~\ref{fig-epsP}.
\begin{figure}
\[
\setlength\unitlength{.7cm}
\begin{array}{ccc}
\begin{picture}(4,3)
\put(4,2){\vector(-2,-1){4}}
\put(3,0){\vector(-1,3){1}}
\put(2.3,2.5){$x$}
\put(-.2,.6){$\gm^{-1} p$}
\end{picture}&&
\begin{picture}(4,3)
\put(0,0){\vector(2,1){4}}
\put(3,0){\vector(-1,3){1}}
\put(2.3,2.5){$x$}
\put(-.2,.6){$\gm^{-1}p$}
\end{picture}\\
\eps_P(x,\gm^{-1}p)=1&\quad& \eps_P(x,\gm^{-1}p)=-1\\
\text{ $\gm^{-1}p$ crosses $x$ from right to left}&& \text{
$\gm^{-1}p$ crosses $x$ from left to right}
\end{array}
\]
\caption{Choice of $\eps_P(x,\gm^{-1}p)$.}\label{fig-epsP}
\end{figure}
\item[ii)] We put\ir{epsdef}{\eps(x,\gm^{-1}p)}
\be\label{epsdef} \eps(x,\gm^{-1}p) \;:=\; \sum_{P\in x\cap \gm^{-1}p}
\eps_P(x,\gm^{-1}p)\,.\ee
\item[iii)] We extend $x\mapsto \eps(x,\gm^{-1}p)$ to a $\CC$-linear
map $\CC[X_1^\tess] \rightarrow \CC$.
\end{enumerate}
\end{defn}

\rmrks
\itmi If $x$ and $\gm^{-1}p$ have no intersection, then the sum
in~\eqref{epsdef} is empty, hence $\eps(x,\gm^{-1}p)=0$.

\itm Like in \cite{BLZm} we use the convention that $X_1^\tess$
consists of oriented edges of the tesselation, and that if $e\in
X_1^\tess$, then the edge $-e$ with the opposite orientation is not
in $X_1^\tess$.

\itm \label{espdefnd} Property~d) in Lemma~\ref{lem-path-hyp} implies
that the total number of crossing of $x$ and $\gm^{-1}p$ is finite.
So $\eps(x,\gm^{-1}p)$ in Part~ii) is well defined. It counts the
number of crossings from right to left minus the number of crossings
from left to right.

\itm The definition of $\eps$ is arranged in such a way that for each
oriented edge $x$ occurring in the boundary $\partial_2
V$ of a face $V\in X_2^\tess$ the quantity $\eps(x,\gm^{-1}p)$ counts
the number of times that $\gm^{-1}p$ enters the face $V$ through the
edge~$x$ minus the number of times it leaves $V$ through~$x$. This
gives $\eps(\partial_2
V,\gm^{-1}p)=0$ for all faces $V\in X_2^\tess$.

\itm We have $\eps_P(-x,\gm^{-1}p)=-\eps_P(x,\gm^{-1}p)$ and
$\eps(-x,\gm^{-1}p)=-\eps(x,\gm^{-1}p)$. Hence the $\CC$-linear
extension in Part~iii) is possible.

\itm \label{nettosum} For an oriented path $q\in \ZZ[X_1^\tess]$ we
can view $\eps(q,\gm^{-1}p)$ as the number of times $q$ crosses
$\gm^{-1}p$ where $\gm^{-1}p$ goes from right to left, with respect
to the orientation of $q$, minus the number of times $q$ crosses
$\gm^{-1}p$ where $\gm^{-1}p$ goes from left to right.

\itm The function $\eps$ is $\Gm$-invariant:
\be \label{eps-eqv}
\eps(\bt^{-1}x,\bt^{-1}\gm^{-1}p) \= \eps(x,\gm^{-1}p)\qquad\text{for
all }\bt\in \Gm\,. \ee
\medskip

\begin{prop}\label{prop-cocp}Let $\Eta$ be a hyperbolic subgroup of
$\Gm$, and let $p$ be a path as in Lemma~\ref{lem-path-hyp} between
the fixed points $\zt_1$ and $\zt_2$ of~$\Eta$. Let $W$ be a
$\Gm$-module.

For each $a\in W^\Eta$ we put\ir{cpa-def}{c(p,a;x)}
\be \label{cpa-def}
  c(p,a;x) \;:=\; \sum_{\gm \in \Eta\backslash \Gm}
\eps(x,\gm^{-1}p)\; a|\gm
\qquad\text{for }x\in F_1^\tess=\CC[X_1^\tess]\,. \ee
\begin{enumerate}
\item[a)] This defines a cocycle $c(p,a;\cdot)\in
Z^1\bigl( F^\tess_\pnt;W)$.
\item[b)] If $p_1$ and $p_2$ are paths as in
Lemma~\ref{lem-path-hyp} with the same initial point $\zt_1$ and the
same final point $\zt_2$, then $c(p_1,a;\cdot)$
and $c(p_2,a;\cdot)$ are in the same cohomology class in
$\hpar^1(\Gm;W)$.
\end{enumerate}
\end{prop}

\rmrke Without the counting function $\eps(\cdot,\cdot)$ the values $
c(p,a;x)$ of the cocycle $c(p,a;\cdot)$ are hyperbolic Poincar\'e
series. Hence we call the sums in~\eqref{cpa-def} \il{shPs}{signed
hyperbolic Poincar\'e series}\il{Pssh}{Poincar\'e series, signed
hyperbolic}\emph{signed hyperbolic Poincar\'e series}.

\begin{proof}The terms in the sum are invariant under $\gm\mapsto
\dt\gm$ with $\dt\in \Eta$. It is a finite sum by Property~e) in
Lemma~\ref{lem-path-hyp}. So $c(p,a;x)$ is well-defined. In
Remark~\ref{espdefnd} after Definition~\ref{epsdefn} we have noted
that $\eps(\partial_2
V,\gm^{-1}p)=0$ for each $V\in X_2^\tess$. This gives the cocycle
property. With \eqref{eps-eqv} we have for $\bt\in \Gm$
\begin{align*} c(p,a;\bt^{-1}x) &\= \sum_{\gm\in \Eta\backslash\Gm}
\eps(\bt^{-1}x, \gm^{-1}p)\; a|\gm \=\sum_{\gm\in \Eta\backslash\Gm}
\eps(\bt^{-1}x,
(\gm\bt)^{-1}p)\; a|\gm\bt \\
&\= \sum_{\gm\in \Eta\backslash\Gm} \eps(x, \gm^{-1}p)\; a|\gm \bt \=
c(p,a;x)|\bt\,.\end{align*}
This gives the $\CC[\Gm]$-linearity of $x\mapsto c(p,a;x)$, and ends
the proof of Part~i).

To prove Part~b) we consider the function $c(p,a;\cdot,\cdot)$ on
$X_0^\tess \times X_0^\tess$ given by $c(p,a;Q_1,Q_2)=c(p,a;q)$
independent of the choice of the path $q\in \ZZ[X_1^\tess]$ from
$Q_1$ to~$Q_2$. The cohomology class of $c(p,a;\cdot)$ is
determined by the group cocycle
$\gm \mapsto c(p,a;\gm^{-1} Q_0, Q_0)$ for any base point
$Q_0 \in X_0^\tess$. See the final paragraphs of~\S\ref{sect-rbot}.
We will show that for cusps $\ca$ and $\cb$ the value of
$c(p,a;\ca,\cb)$ depends only on the position of $\ca$ and $\cb$ in
relation to $\zt_1$ and $\zt_2$, and not on the actual path $p$ from
$\zt_1$ to~$\zt_2$. With a cusp as the base point $Q_0$ this gives
Part~b).

The points $\zt_j$ divide $\proj\RR$ into two \il{ci}{cyclic
interval}\emph{cyclic intervals}
\il{cyclint}{$\xi,\eta)_{\mathrm{cycl}}$}$(\zt_1,\zt_2)_{\mathrm{cycl}}$
and $(\zt_2,\zt_1)_{\mathrm{cycl}}$ for the cyclic order
on~$\proj\RR$. See Figure~\ref{fig-ci}.
\begin{figure}[ht]
\[\setlength\unitlength{.8cm}
\begin{picture}(8,2.3)(0,-1.5)
\put(0,0){\line(1,0){8}}
\put(1.5,0){\circle*{.08}}
\put(6,0){\circle*{.08}}
\put(1.4,-.4){$\zt_1$}
\put(5.8,-.4){$\zt_2$}
\put(3.2,.7){$(\zt_1,\zt_2)_{\mathrm{cycl}}$}
\put(3.2,-1.5){$(\zt_2,\zt_1)_{\mathrm{cycl}}$}
\put(3.7,.6){\vector(0,-1){.5}}
\put(3.1,-1.2){\vector(-3,1){2.8}}
\put(4.3,-1.2){\vector(3,1){2.8}}
\end{picture}
\]
\caption{}\label{fig-ci}
\end{figure}

For cusps $\ca$ and $\cb$ we choose a path $q_{\ca,\cb}\in
\ZZ[X_1^\tess]$ from $\ca$ to $\cb$. By Remark~\ref{nettosum} after
Definition~\ref{epsdefn}, the values of $\eps(q_{\ca,\cb},\gm^{-1}p)$
are zero if $\ca$ and $\cb$ are not separated in $\proj\RR$ by the
points $\gm^{-1}\zt_1$ and $\gm^{-1}\zt_2$. Table~\ref{tab-a-b-gm}
gives the values of $\eps(q_{\ca,\cb},\gm^{-1}p)$ for $\gm\in \Gm$
and the cusps $\ca$ and $\cb$ for the fixed path~$p$. See also
 Figure~\ref{fig-pq}.
\begin{table}[ht]
\[ \begin{array}{|c|cc|}\hline
\eps(q_{\ca,\cb},\gm^{-1}p)
& \cb \in (\gm^{-1}\zt_1,\gm^{-1}\zt_2)_{\mathrm{cycl}} & \cb \in
(\gm^{-1}\zt_2,\gm^{-1}\zt_1)_{\mathrm{cycl}}
\\ \hline
\ca \in (\gm^{-1}\zt_1,\gm^{-1}\zt_2)_{\mathrm{cycl}} & 0 & -1 \\
\ca \in (\gm^{-1}\zt_2,\gm^{-1}\zt_1)_{\mathrm{cycl}} & 1 & 0
\\ \hline
\end{array}
\]
\caption{}\label{tab-a-b-gm}
\end{table}
This implies that $c(p,a;\ca,\cb)$ only depends on the position of the
cusp $\ca$ and $\cb$ in relation to $\zt_1$ and $\zt_2$, not on the
actual path~$p$.\end{proof}
\begin{figure}
\[\setlength\unitlength{1cm}
\begin{picture}(10,2.4)(0,-.4)
\put(0,0){\line(1,0){10}}
\put(.7,-.4){$\gm^{-1}\zt_1$}
\put(5.7,-.4){$\gm^{-1}\zt_2$}
\put(2.9,-.4){$\ca$}
\put(7.9,-.4){$\cb$}
\qbezier(1,0)(1,1)(2,1)
\qbezier(5,1)(6,1)(6,0)
\put(2,1){\vector(1,0){.8}}
\put(2,1){\line(1,0){3}}
\put(1.4,1.2){$\gm^{-1}p$}
\qbezier(3,0)(3.5,1.7)(4,2)
\qbezier(7,2)(7.5,1.7)(8,0)
\put(4,2){\line(1,0){3}}
\put(4,2){\vector(1,0){1.5}}
\put(7.2,2){$q$}
\put(3.3,.7){$\scriptstyle \eps=-1$}
\end{picture}
\]
\caption{Illustration of case $\ca\in
(\gm^{-1}\zt_1,\gm^{-1}\zt_2)_{\mathrm{cycl}}$ and $\cb \in
(\gm^{-1}\zt_2,\gm^{-1}\zt_1)_{\mathrm{cycl}}$ in
Table~\ref{tab-a-b-gm}. The path $\gm^{-1}p$ crosses $q$ from left to
right. } \label{fig-pq}
\end{figure}

\rmrks
\itmi The cocycle is $\Gm$-equivariant in the following way:
\be c(p,a;\cdot) \= c\bigl( \gm^{-1}p,a|\gm;\cdot\bigr)\qquad\text{for
all }\gm\in \Gm\,. \ee

\itm The cocycle $c(p,a;\cdot)$ depends linearly on $a\in W^\Eta$;
{\sl ie,} for all $\ld_1,\ld_2\in \CC$
\[ c(p,\ld_1 a_1+\ld_2 a_2;\cdot) = \ld_1\, c(p,a_1;\cdot)+ \ld_2\,
c(p,a_2;\cdot)\,.\]

\itm The construction is canonical for a morphism of $\Gm$-modules
$W\rightarrow W_1$: If $a\in W^\Eta$ is mapped to $b\in W_1^\Eta$,
then
\[ c(p,a;\cdot) \mapsto c(p,b;\cdot)\quad\text{under the natural map
}Z^1(F^\tess_\pnt;W) \rightarrow Z^1(F^\tess_\pnt;W_1)\,.\]

\rmrk{Geodesics with elliptic fixed points} The geodesic from
$\frac12-\frac12\sqrt 5$ to $\frac12+\frac12\sqrt 5$ induces a closed
geodesic on $\Gmod\backslash\uhp$. The corresponding hyperbolic
subgroup $\Eta$ of $\Gmod$ can be generated by $D=\matc2111$
and~$-\Id$. This geodesic passes through the point $i\in \uhp$, which
is fixed by the elliptic element $S=\matr0{-1}10\in \Gmod$. It
induces an element $\pm S$ in $\overline\Gmod$ of order two, so $i$
is an \il{ep}{elliptic point of order $2$}elliptic point of~$\Gmod$
of order~$2$. All points $D^n \, i$, with $n\in \ZZ$, are elliptic
points of $\Gmod$ of order~$2$, fixed by $D^n S D^{-n}\in \Gmod$.

In general, a geodesic of $\Gm$ may go through elliptic fixed points
of $\Gm$ of order~$2$ in~$\bar \Gm$. Then there are elliptic elements
of order two in $\s\in \Gamma$ such that $\sigma \gm
\sigma=\gamma^{-1}$ for all $\gamma \in \Eta$. The action
of $\s$ interchanges the two fixed points of~$\Eta$. The element $\s$
normalizes $\Eta$, but is not an element of~$\Eta$. Conversely, each
$\s\in \Gm\setminus \Eta$ such that $\s\Eta\s^{-1}=\Eta$, is elliptic
with a fixed point of order $2$ on the geodesic.

\begin{lem}\label{lem-Psi-Eta}Let $V=\dsv{v,2-r}\om$, and either
$W=\dsv{v,2-r}{\fs,\smp}$ or $W=\dsv{v,2-r}{\fs,\wdg}$. Denote
$\Sg{}{}=W/V$. For each hyperbolic subgroup $\Eta\subset\Gm$ there is
a linear map\ir{PsiEta}{\Psi_\Eta,\; \tilde\Psi_\Eta}
\be \label{PsiEta} \Psi_\Eta : W^\Eta \rightarrow \hpar^1(\Gm;W)\ee
with the following properties:
\begin{enumerate}
\item[i)] The image of the composition $\tilde\Psi_\Eta:W^\Eta
\stackrel{\Psi_\Eta}\rightarrow\hpar^1(\Gm;W)
\rightarrow H^1(\Gm;\Sg{}{})$ is contained in the $\Gm$-invariant
summand $H^1\bigl(\Gm;\Sg{}{}\{\Gm\,\zt_1\} +
\Sg{}{}\{\Gm\,\zt_2\}\bigr)$ of
$H^1(\Gm;\Sg{}{})$, where $\zt_1$ and $\zt_2$ are the fixed points
of~$\Eta$.
\item[ii)] The kernel of $\tilde\Psi_\Eta$ is the space
\[ V^\Eta + \Bigl\{ a\in W^\Eta\;:\; a|_{v,2-r}\s =a\text{ for some
}\s\in \Gm\setminus \Eta \text{ normalizing }\Eta\Bigr\}\,.\]
\end{enumerate}
\end{lem}

\rmrks
\itmi The summands $\Sg{}{}\{\Gm\, \zt_1\}$ and $\Sg{}{}\{\Gm\,
\zt_2\}$ of $\Sg{}{}$ either coincide or have intersection~$\{0\}$.

\itm Any $\s\in \Gm\setminus \Eta$ normalizing $\Eta$ is
elliptic of order two. (If $\gm\in \Gm$ normalizes $\Eta$ and fixes
$\zt_1$ and $\zt_2$, then it is hyperbolic, and hence in $\Eta$. If
it interchanges the $\zt_j$ it has order two, and hence is elliptic.)

\itm The second term in the description of $\ker\tilde\Psi_\Eta$ in
Part~ii) is zero if there are no elliptic elements
normalizing~$\Eta$.  

\begin{proof}We use a path $p$ from $\zt_1$ to $\zt_2$ as in
Lemma~\ref{lem-path-hyp}. For $a\in W^\Eta$ we define $ \Psi_H(a) $
as the cohomology class of $c(p,a;\cdot)$ in
Proposition~\ref{prop-cocp}. This gives a linear map, and the
construction  and Lemma~\ref{lem-sing-hyp} show 
that the cocycle $c(p,a;\cdot)$ has values with
singularities in the $\Gm$-orbits of $\bsing a
\subset\{\zt_1,\zt_2\}$. This gives Part~i).

Suppose first that there is an elliptic $\sigma \in \Gamma$
normalizing $\Eta$.  In the sum over $\gm\in
\Eta\backslash\Gm$ in the definition of $c(p,a;\cdot)$
in~\eqref{cpa-def} we combine the summands $\gm$ and $\s\gm$. Since
$\s^{-1}p$ is $p$ with the opposite orientation, we have
$\eps\bigl(x,(\s\gm)^{-1}p)= \eps\bigl( \gm^{-1}(-p)\bigr) = -
\eps(x,\gm^{-1}p)$. The two corresponding terms in the sum
in~\eqref{cpa-def} give
\[ \eps(x,\gm^{-1}p) \,\bigl( a|_{v,2-r}\gm - a|_{v,2-r}\s\gm\bigr) \=
\eps(x,\gm^{-1}p) \; a|_{v,2-r}(1-\s)\gm\,.\]
So if $a\in W^\Eta$ satisfies $a|_{v,2-r}\s=a$, then the cocycle
$c(p,a;\cdot)$ is zero, so $a\in
\ker\Psi_\Eta\subset \ker\tilde \Psi_\Eta$.

If $a\in V^\Eta$, then $c(p,a;\cdot)$ has values in~$V$, hence the
image cocycle in $\Sg{}{}$ vanishes. This establishes the inclusion
$\supset$ in Part~ii), for the case that an elliptic
$\sigma$ normalizing $\Eta$ exists and for the other case.

To show the other inclusion, suppose that $a\in W^\Eta$ is in $\ker
\tilde \Psi_\Eta$. If there are elliptic $\s\in
 \Gm\setminus\Eta$ normalizing $\Eta$, then $a|_{v,2-r}\s=\pm a$. If
$a|_{v,2-r}\s= a$ then $a$ is in the right hand side in Part~ii) and
 and we are done. If $a|_{v,2-r}\sigma = -a$ then we will show that
$a\in V^\Eta$.

Since $a$ is $\Eta$-invariant, $\bsing a \subset\{\zt_1,\zt_2\}$ by
Lemma~\ref{lem-sing-hyp}. So the image $\tilde a $ of $a$ in
$\Sg{}{}$ is in $\Sg{\zt_1}{}\oplus \Sg{\zt_2}{}$. Since the class of
$c(p,\tilde a;\cdot)$ is zero, this cocycle is a coboundary, and
 there exists $\tilde f \in C^0(X_0^\tess;\Sg{}{}\{\zt_1,\zt_2\})$
such that $c(p,\tilde a;\cdot)=d\tilde
f$, with the notation
$\Sg{}{}\{\zt_1,\zt_2\}=\Sg{}{}\{\zt_1\}+\Sg{}{}\{\zt_2\}$.

By $\Gm$-equivariance, $f(\ca)=f(\ca)|_{v,2-r}\pi_\ca$ for all
cusps~$\ca$. So $\bsing f(\ca) \subset \{\ca\}$, by
Lemma~\ref{lem-sing-inv}. Since $\zt_1$ and $\zt_2$ are no cusps, we
have $\tilde f(\ca)=0$ for all cusps.

Let $q_{\ca,\cb}$ be a path in $\ZZ[X_1^\tess]$ from a cusp $\ca\in
(\zt_2,\zt_1)_{\mathrm{cycl}}$ to a cusp $\cb\in
(\zt_1,\zt_2)_{\mathrm{cycl}}$.  If no elliptic $\s\in \Gm$
normalizing $\Eta$ exists, then $\eps(q_{\ca,\cb},p)=1$. The
contribution to $c(p,a;q_{\ca,\cb})$ with singularities in
$\{\zt_1,\zt_2\}$ is given by $a$. So the component of $c(p,\tilde
a;q)$ in $\Sg{\zt_1}{}\oplus \Sg{\zt_2}{}$ is equal to
the component of $\tilde a$ in~$\Sg{\zt_1}{}\oplus \Sg{\zt_2}{}$ in
the decomposition \eqref{sesidcp}. On the other hand $c(p,\tilde
a;q)=\tilde f(\ca)-\tilde f(\cb)=0$. Since $\bsing a
\subset\{\zt_1,\zt_2\}$, this implies that $\tilde a=0$, hence $a\in
V$. But $a\in W^\Eta$, so $a\in V^\Eta$.

If $\s\in \Gm\setminus\Eta$ normalizes $\Eta$ some changes in
this reasoning are needed, since we have also
$\eps(q_{\ca,\cb},\s^{-1}p)=-1$. Now the contribution to
$c(p,a;q_{\ca,\cb})$ with singularities in $\{\zt_1,\zt_2\}$ is
$a-a|_{v,2-r}\s$, and the component of $c(p,\tilde a;q)$ in
$\Sg{\zt_1}{}\oplus\Sg{\zt_2}{}$ is equal to $2\tilde a$. We can
finish the proof by the same argument as in the other case.
\end{proof}

Theorem~\ref{thm-iso-mpcpc} asserts that the canonical map
$\hpar^1(\Gm;V,W)
\rightarrow \hpar^1(\Gm;W)$ is an isomorphism if $V=\dsv{v,2-r}\om$
and $W$ is one of a list of larger modules, each contained in
$\dsv{v,2-r}\fs$. Now we focus on the following two cases, for which
Theorem~\ref{thm-iso-mpcpc} does not give information:
\begin{enumerate}
\item[a)] $V=\dsv{v,2-r}\om$, $W=\dsv{v,2-r}{\fs,\wdg}$.
\item[b)] $V=\dsv{v,2-r}\om$, $W=\dsv{v,2-r}{\fs,\smp}$, and there are
primitive hyperbolic elements $\gm\in \Gm$ for which
$v(\gm)=e^{-r\ell(\gm)/2}$.
\end{enumerate}
The following result gives information concerning Case~a), and partial
information concerning Case~b).

\begin{prop}\label{prop-icd-hyp}
Let $r\in \CC$.
\begin{enumerate}
\item[i)] The natural map
\[\hpar^1(\Gm;\dsv{v,2-r}\om,\dsv{v,2-r}{\fs,\wdg})\rightarrow
\hpar^1(\Gm;\dsv{v,2-r}{\fs,\wdg})\]
\begin{enumerate}
\item[1)] is injective,
\item[2)] and its image in $\hpar^1(\Gm;\dsv{v,2-r}{\fs,\wdg})$ has
infinite codimension.
\end{enumerate}
\item[ii)] Suppose that the set
\[ P \= \Bigl\{ \gm\in \Gm\;:\; \gm \text{ is primitive hyperbolic,
and } v(\gm)=e^{-r\ell(\gm)/2} \Bigr\}\]
 is non-empty. {\rm(Recall that $\ell(\gm)$ is the length of the
 closed geodesic associated to~$\gm$.)}
\begin{enumerate}
\item[1)] The natural map
\be \label{smpmap}
\hpar^1(\Gm;\dsv{v,2-r}\om,\dsv{v,2-r}{\fs,\smp})\rightarrow
\hpar^1(\Gm;\dsv{v,2-r}{\fs,\smp})\ee
is injective.
\item[2)] It is not surjective if $r\in \ZZ_{\geq 0}$ and for some
$\gm\in P$ one of the following two conditions is satisfied:
\begin{enumerate}
\item[a)] There are no $\s\in \Gm$ such that $\s\gm\s^{-1}=\gm^{-1}$.
\item[b)] $r=0$, $v(\gm)=1$, and there exist $\s\in\Gm$ such that
$\s\gm\s^{-1}=\gm^{-1}$, and $v(\s)=1$.
\end{enumerate}
\end{enumerate}
\end{enumerate}
\end{prop}

\rmrke If in Part~ii)2) none of the conditions a) and b) holds, we do
not know whether the map in~\eqref{smpmap} is surjective.

\begin{proof}The \emph{injectivity} in Parts i)1) and~ii)1) follows
from the exact sequence~\eqref{esVWS}.

\rmrk{Part~i)2)} Let $\Eta$ be a hyperbolic subgroup of~$\Gm$, with
primitive hyperbolic generator $\gm$. We consider two cases:
\begin{itemize}
\item  Suppose that there are no elliptic $\s\in \Gm$ normalizing
$\Eta$. Lemma~\ref{lem-hyp-is} implies that $W^\Eta/V^\Eta$, with
$W=\dsv{v,2-r}{\fs,\exc})$ and $V=(\dsv{v,2-r}\om)^\Eta$, has
infinite dimension. Lemma~\ref{lem-Psi-Eta} and the exact
sequence~\eqref{esVWS} imply that its image in
$H^1(\Gm;\dsv{v,2-r}{\fs,\exc})$ has infinite dimension.

\item If there exists $\s\in \Gm\setminus \Eta$ normalizing $\Eta$, we
need an subspace of 
\[ W^\Eta \bigm/ \bigl( V^\Eta+ \ker(\s-1) \bigr)
\]
of infinite dimension. If
$f\in W^\Eta = (\dsv{v,2-r}{\fs,\exc})^\Eta$, then
$f|_{v,2-r}\s \in (\dsv{v,2-r}{\fs,\exc})^\Eta$ too
(because of $\s\gm\s^{-1}=\gm^{-1}$). Then $a:=f-f|_{v,2-r}\s$
satisfies $a|_{v,2-r}\s=-a$. By Part~ii) of Lemma~\ref{lem-hyp-is}
there are infinitely many linearly independent such $f$ not in
$V=\dsv{v,2-r}\om$.
\end{itemize}

Furthermore, Lemma~\ref{lem-Psi-Eta} shows that different hyperbolic
subgroups of $\Gm$ lead to cohomology classes with values in
different summands of~$\Sg{}{}$, which is an other source of infinite
dimensionality.

\rmrk{Part ii)2)} We show non-surjectivity of the map by producing
cocycles that have non-zero image in $H^1(\Gm;\Sg{}{})$. See
Lemma~\ref{lem-esVWS}. Lemma~\ref{lem-Psi-Eta} provides us with
cocycles. More precisely, consider $\gm\in P$. To apply
Lemma~\ref{lem-Psi-Eta} we need $a\in
(\dsv{v,2-r}{\fs,\smp})^\gm$ that is not in $(\dsv{v,2-r}{\om})^\gm$.
Lemma~\ref{lem-hyp-inv} shows that there is a one-dimensional space
with such elements, occurring for $r\in \ZZ_{\geq 0}$ and $\k\in
\{-1,r-\nobreak1\}$, with $\k$ as indicated in that lemma. That gives
$e^{\ell(\gm)(\k+1)}=1$, hence $\k=-1$. Under Condition~a)
in Part~ii)2) we conclude that there is a class in
$\hpar^1(\Gm;\dsv{v,2-r}{\fs,\smp})$ with non-trivial image in
$H^1(\Gm;\Sg{v,2-r}{\fs,\smp})$.

Under Condition~b) we have $\pm 1 = v(\gm) = e^{-r\ell(\gm)/2}$. Since
$r\in \ZZ_{\geq 0}$ this is possible only for $r=0$ and $v(\gm)=1$.
Conjugation as above brings us to the situation
$a(t)\stackrel\pnt=(it)^{-1}$. Hence $a|_{v,2-0}\s = -v(\s)\, a$. So
we need the value $v(\s)=1$ of the two possible values~$\pm 1$ to
complete the proof with Lemma~\ref{lem-Psi-Eta}.\end{proof}

\subsection{Mixed parabolic cohomology and condition at cusps}
\begin{prop}\label{prop-icod-exc-fs}Let $r\in \CC\setminus \ZZ_{\geq
2}$. The space $\hpar^1(\Gm;\dsv {v,2-r}\om,\dsv{v,2-r}{\fs,\wdg})$
has infinite codimension in the space $\hpar^1(\Gm;\dsv
{v,2-r}\om,\dsv{v,2-r}\fs)$.
\end{prop}

We prepare the proof of this proposition in two lemmas, one of
geometric nature, like Proposition~\ref{prop-cocp}, the other an
infinite codimension result.\smallskip

In the lemma with a geometric flavor, we work with a tesselation as
discussed in~\S\ref{sect-tess}, based on a fundamental domain $\fd$
of $\Gm\backslash\uhp$, which is split up in a compact set $\fd_Y$,
and cuspidal triangles $V_\cb$ where $\cb$ runs over a set of
representatives of the $\Gm$-orbits of cusps. The edge $f_\cb$ is the
intersection of the boundaries of $\fd_Y$ and~$V_\cb$.

\begin{lem}\label{lem-pcoc}Let $r\in \CC\setminus \ZZ_{\geq 2}$, let
$\ca$ be a cusp of~$\Gm$, and let $\dt\in \Gm\setminus\Gm_\ca$. Let
$a\in \dsv{v,2-r}\om$ such that
\[ a|_{v,2-r}(1-\dt^{-1})\, \in
\dsv{v,2-r}\fs\bigm|_{v,2-r}(1-\pi_\ca) \,.\]
\begin{enumerate}
\item[a)] There exists a $\CC$-linear map $a\mapsto c(a;\cdot)$ from
$\dsv{v,2-r}\om$ to $ Z^1(F^{\tess,Y}_\pnt;\dsv{v,2-r}\om)$ such that
$c(a;f_\ca)=a|_{v,2-r}(1- \dt^{-1})$, and if there are cusps $\cb$
not in the orbit $\Gm\, \ca$ then $c(a;f_\cb)=0$.
\item[b)] The $\CC[\Gm]$-equivariant linear map
$c(a;\cdot):F^{\tess,Y}_1
\rightarrow \dsv{v,2-r}\om$ has a $\CC[\Gm]$-equivariant linear
extension $\tilde c(a;\cdot)
:F^\tess_1\rightarrow \dsv{v,2-r}\fs$ such that
\[\tilde c(a;\cdot)\in
Z^1(F^\tess_\pnt;\dsv{v,2-r}\om,\dsv{v,2-r}\fs)\,.\]
\item[c)] The cohomology class $[\tilde c(a;\cdot)]\in
\hpar^1(\Gm;\dsv{v,2-r}\om,\dsv{v,2-r}\fs)$ satisfies
\[ [\tilde
c(a;\cdot)]\in\hpar^1(\Gm;\dsv{v,2-r}\om,\dsv{v,2-r}{\fs,\wdg})
\Longleftrightarrow a|_{v,2-r}(1-\dt^{-1})\in
 \dsv{v,2-r}{\fs,\wdg}\bigm|_{v,2-r}(1-\nobreak\pi_\ca)\,.\]
\end{enumerate}
\end{lem}

\newcommand\tfd{\put(-1.5,2.){\line(0,-1){1.13397}}
\qbezier(0,0.)(-0.0317542,0.559017)(-0.5,0.866025)
\qbezier(-0.5,0.866025)(-1.,1.11803)(-1.5,0.866025)
\qbezier(0.5,0.866025)(0.269594,0.730406)(0.133975,0.5)
\qbezier(0.133975,0.5)(0.00224617,0.267347)(0,0.)
\put(0.5,0.866025){\line(0,1){1.13397}}
\qbezier(1.5,0.866025)(1.26959,0.730406)(1.13397,0.5)
\qbezier(1.13397,0.5)(1.00225,0.267347)(1.,0.)
\qbezier(1.5,0.866025)(1.,1.11803)(0.5,0.866025)
\qbezier(0.5,0.866025)(0.0317542,0.559017)(0,0.)
\qbezier(0.5,0.288675)(0.333333,0.372678)(0.166667,0.288675)
\qbezier(0.166667,0.288675)(0.0105847,0.186339)(0,0.)
\qbezier(1.,0.)(0.989415,0.186339)(0.833333,0.288675)
\qbezier(0.833333,0.288675)(0.666667,0.372678)(0.5,0.288675)
}

\twocolwithpictr{\rmrke In the simplest situation, we apply the lemma
with a choice of $\dt\in \Gm$ such that the fundamental domain
$\dt^{-1}\fd$ is a neighbour of the fundamental domain $\fd$, and has
common edges with it. Since $\dt\not\in \Gm_\ca$ the edges $e_\ca$
and $\pi_\ca^{-1}e_\ca$ in $\partial_2
\fd$ that go to $\ca$ are not edges of $\dt^{-1}\fd$.

A general choice of $\dt\in \Gm\setminus \Gm_\ca$ leads to fundamental
domains $\fd$ and $\dt^{-1}\fd$ that are far apart. We can connect
them by a finite corridor of fundamental domains $\gm_j^{-1}\fd$ such
that $\gm_{j-1}^{-1}\fd$ and $\gm_j^{-1}\fd$ have a common edge.
}{\setlength\unitlength{1.8cm}
\begin{picture}(3.2,1.7)(-1.6,-.1)
\put(-1.6,0){\line(1,0){3.2}}
\put(-.05,-.2){$\cb$}
\put(.85,-.2){$\dt^{-1}\ca$}
\put(-.55,2){$\uparrow \ca$}
\tfd
\end{picture}
}
\begin{proof}To construct a cocycle $c(a;\cdot)$ with the desired
properties we adapt the geometric approach in \S\ref{sect-cshfp} to
the present needs.\vskip.3ex
\twocolwithpictl{\setlength\unitlength{1.8cm}
\begin{picture}(3.2,2.1)(-1.6,-.2)
\put(-1.6,0){\line(1,0){3.2}}
\put(-1.5,1.5){\line(1,0){2}}
\put(-.05,-.2){$\cb$}
\put(.85,-.2){$\dt^{-1}\ca$}
\tfd
\qbezier(1.2,0.6)(1.03382,0.708404)(0.850929,0.631476)
\qbezier(0.850929,0.631476)(0.679652,0.531319)(0.666667,0.333333)
\qbezier(0.333333,0.333333)(0.320348,0.531319)(0.149071,0.631476)
\qbezier(0.149071,0.631476)(-0.0338201,0.708404)(-0.2,0.6)
\put(-.58,1.8){$p$}
\put(-.6,1.5){\vector(-1,0){.6}}
\put(-1.1,1.6){$f_\ca$}
\put(-.1,.69){$\scriptstyle f_\cb$}
\put(.8,.69){$\scriptstyle \dt^{-1}f_\ca$}
\thicklines
\put(-.4,2){\vector(0,-1){.3}}
\put(-.4,2){\line(0,-1){.6}}
\qbezier(-.4,1.4)(-.4,1.2)(.2,.8)
\qbezier(.2,.8)(.9,.7)(.95,.3)
\qbezier(.95,.3)(1,.2)(1,0)
\end{picture}}{\quad We take a $C^1$ path from $\ca$ to $\dt^{-1}\ca$
not going through vertices in $X_0^\tess$, except the initial and
final points $\ca$ and $\dt^{-1}\ca$, passing through the interior of
$V_\ca$, leaving it through a point of $f_\ca$, then going on through
the interior of $\uhp_Y=\bigcup_{\gm\in \Gm}\gm^{-1}\fd_Y$, crossing
edges in $X_1^{\tess,Y}$ transversally, entering $\dt^{-1}V_\ca$ via
a point of $\dt^{-1}f_\ca$ and going through the interior of
$\dt^{-1}V_\ca$ to $\dt^{-1}\ca$.

\quad We can choose the path $p$ in such a way that it has many of the
properties in Lemma~\ref{lem-path-hyp}, namely a)--d). }\vskip.4em

In b) we replace $\zt_1$ and $\zt_2$ by $\ca$ and $\dt^{-1}\ca$. In d)
we have intersections only with the edges $f_\ca$, $\dt^{-1}f_\ca$,
and a finite number of intermediate edges in $X_1^{\tess,Y}$. Since
$p$ runs through finitely many translates of $\fd$, Property~e) is
also satisfied. Moreover, all edges $e_\cb$ to cusps $\cb$ do not
intersect~$p$. Property~f) does not apply here.

We define \il{eps1}{$\eps(x,\gm^{-1}p)$}$\eps(x,\gm^{-1}p)$ for
$x\in X_1^\tess$ and $\gm\in \Gm$ as in Definition~\ref{epsdefn}, and
next define $c_0(a;\cdot)
\in C^1(F^\tess_\pnt;\dsv{v,2-r}\om)$ by\ir{cpb-def-c}{c_0(a;x)}
\be\label{cpb-def-c}
 c_0(a;x)\;:=\; \sum_{\gm \in \{\pm
 1\}\backslash\Gm}\eps(x,\gm^{-1}p)\; a|\gm\quad\text{ for }x\in
 F^\tess_1\,. \ee
The $\Gm$-equivariance is clear from the equivariance of $\eps$,
however $c_0(a;\cdot)$ need not be a cocycle. Indeed, on the one
hand, $c_0(a;e_\ca)=c_0(a;\pi_\ca^{-1}e_\ca)=0$, since $p$ does not
intersect any $\Gamma$-translate of~$e_\ca$ in an interior point. On
the other hand, $c_0(a;f_\ca)=a|(1-\dt)$ may very well be non-zero.
However, we still have
\[ c_0\bigl(a;\partial_2
\gm^{-1}\fd_Y\bigr)=0\qquad\text{for all }\gm\in \Gm\,. \]
So the restriction $c(a;\cdot)$ of $c_0a;\cdot)$ to
$F^{\tess,Y}_1=\CC[X_1^{\tess,Y}]$ is in
$Z^1(F_\pnt^{\tess,Y};\dsv{v,2-r}\om)$.

The path $p$ intersects $f_\ca$ with $\eps(f_\ca,p)=1$ and
$\dt^{-1}f_\ca$ with $\eps(\dt^{-1}f_\ca,p)=-1$, and no other
$\Gm$-translates of edges $f_{\cb'}$ with $\cb'$ a cusp of~$\Gm$. So
no path $\gm^{-1}p$ with $\gm\in \Gm$ intersects $f_\cb$ with
$\cb\neq\ca$ in the closure of $\fd$ in $\proj\CC$. For $f_\ca$ we
find
\[ c(a;f_\ca) \= \eps(f_\ca,p)\, a + \eps(f_\ca,\dt\,p)\, a|\dt^{-1}
\= a|_{v,2-r}(1-\dt^{-1}) \,.\]
So $c(a;\cdot)$ satisfies the requirements in Part~a) of the lemma.

Part~b) asks for defining \il{tldcpa}{$\tilde c$}$\tilde c(a;e_\cb)$
for the cusps $\cb$ in the closure of~$\fd$. For $\cb\neq\ca$ this is
easy: We have $c(a;f_\cb)=0$, and define $\tilde c(a;e_\cb)=0$ to
have $\tilde c(a;\partial_2
V_\cb)=0$. The assumptions on $a$ in the lemma show that there exists
$h\in \dsv{v,2-r}\om[\ca]$ such that
$h|_{v,2-r}(1-\nobreak\pi_\ca)=a|_{v,2-r}(1-\nobreak
\dt^{-1})=c(a;f_\ca)$. By taking $\tilde c(a;e_\ca)=h$ we have
$\tilde c(a;\partial_2
V_\ca)=0$. By $\Gm$-equivariance we use this to define a cocycle
$\tilde c(a;\cdot) \in
Z^1(F^\tess_\pnt;\dsv{v,2-r}\om,\dsv{v,2-r}\fs)$ that coincides with
$c$ on~$F^{\tess,Y}_1$.

The implication $\Leftarrow$ in Part~c) is a direct consequence of the
definition of~$c$. For the implication $\Rightarrow$ we suppose that
there exists $f\in C^0(F^\tess_\pnt;\dsv{v,2-r}\om,\dsv{v,2-r}\fs)$
such that $\tilde c(a;\cdot)-(df)(\cdot) \in
Z^1(F_\pnt^\tess;\dsv{v,2-r}\om,\dsv{v,2-r}{\fs,\wdg})$. The
$\Gm$-equivariance of $f$ implies that
$f(\ca)|_{v,2-r}\pi_\ca=f(\ca)$. Denote $k= \tilde
c(a;e_\ca)-df(e_\ca)$; so $k\in \dsv{v,2-r}{\fs,\wdg}$. Then
\begin{align*} k|_{v,2-r}(1-\pi_\ca) &\= \tilde c(a;e_\ca)|
_{v,2-r}(1-\pi_\ca)
- \bigl( f(P_\ca)-f(\ca)\bigr)|_{v,2-r}(1-\pi_\ca)
\\
&\= h|_{v,2-r}(1-\pi_\ca) - f(P_\ca)|_{v,2-r}(1-\pi_\ca) + 0
\\
&\;\in\; c(a;f_\ca) + \dsv{v,2-r}\om|_{v,2-r}(1-\pi_\ca)
\qquad\text{ (since }P_\ca \in X_0^{\tess,Y})\\
&\=a|_{v,2-r}(1- \dt^{-1})
+\dsv{v,2-r}\om|_{v,2-r}(1-\pi_\ca)\,.
\end{align*}
Hence $a|_{v,2-r}(1- \dt^{-1})\in
\dsv{v,2-r}{\fs,\wdg}(1-\nobreak\pi_\ca)$.
\end{proof}

\begin{lem}\label{lem-icod-exc-fs}Let $r\in \CC\setminus\ZZ_{\geq 2}$,
$\ld,\mu \in \CC^\ast$, and $\gm=\matc abcd\in \SL_2(\RR)$ with 
$c > 0$. Then the space
\be\label{doe}
\Bigl( \dsv{2-r}\om\bigm|_{2-r}(1-\mu\,\gm^{-1})\Bigr) \cap \Bigl(
\dsv{2-r}{\fs,\exc}\bigm|_{2-r}(1-\ld^{-1} T) \Bigr)
\ee
has infinite codimension in the space
\be\label{do}
\Bigl(\dsv{2-r}\om \bigm|_{2-r}(1-\mu\,\gm^{-1})\Bigr)
\cap \Bigl( \dsv{2-r}\fs\bigm|_{2-r}|(1-\ld^{-1}T)
\Bigr)\,. \ee
\end{lem}

\begin{proof}
This may be compared with Lemma \ref{lem-*solve}, which implies, with
Lemma~\ref{lem-sing-inv}, that for $r\in \CC\setminus \ZZ_{\geq 2}$
the space
\[ \dsv{2-r}\om \cap \bigl( \dsv{2-r}\fs|_{2-r}(1-\ld^{-1}T)\bigr)\]
has finite codimension in the space $\dsv{2-r}\om$. So we have to show
that imposing the condition~``$\wdg$'' and applying
$|_{2-r}(1-\nobreak \mu\gm^{-1})$ makes an infinite-dimensional
difference. We do this by giving an infinite-dimensional space
\[ R\subset \Bigl( \dsv{2-r}\om\bigm|_{2-r}(1-\mu\,\gm^{-1})\Bigr)
\cap \Bigl( \dsv{2-r}{\fs}\bigm|_{2-r}(1-\ld^{-1} T) \Bigr)\,,\]
for which we then show that it has zero intersection with
\[ \dsv{2-r}{\fs,\wdg}|_{2-r}(1-\ld^{-1}T)\,.\]

We take $z_0\in \uhp$, on which we will impose some restrictions later
on, and put
\begin{align*} R \= \Bigl\{ \ph|_{2-r}&(1-\mu\, \gm^{-1}) \in \dsv
{2-r}\om\;:\; \ph(t)=(i-t)^{r-2}\,p(t)
\text{ where $p$ is a rational}\\
&\qquad \text{ function on $\proj\CC$, such that
$p(\infty)=p(\gm^{-1}\infty)=0$, and}\\
&\qquad \text{ $p$ has a singularity at $t=z_0$, and nowhere else in
$\proj\CC$}\Bigr\}\,.
\end{align*}
Since the order of the singularity of $p$ at $t=z_0$ is not
prescribed, this space has infinite dimension. There should be a zero
at at least two points in $\proj\CC$, so any non-zero $p$ has a
singularity at~$t=z_0$ of order at least~$2$. The factor $(i-\nobreak
t)^{r-2}$ may give $\ph$ a boundary singularity at~$t=\infty$. This
factor has no influence on the singularities of $\ph$ at $t=z_0$ and
$t=\gm\, z_0$.\vskip.4ex

\newcommand\tfe{
\put(-1.5,0){\line(1,0){3}}
\put(-.5,1.1){\circle*{.04}}
\put(-.55,1.18){$z_0$}
\put(.06,.95){$i$}
\put(.36,.32){$\gm i$}
\put(.9,-.2){$\gm\infty$}
\put(.658,.753){\circle*{.04}}
\put(.65,.85){$\gm z_0$}
\thicklines
\put(0,1){\line(0,1){.7}}
\qbezier(.5,.5)(1,.5)(1,0)
\put(1,0){\line(0,1){1.7}} }
\twocolwithpictr{\quad The singularities of $\ph$ in $\proj\CC$ occur
at $z_0$, from $p$, and on the line $i[1,\infty]$, from the factor
$(i-\nobreak t)^{r-2}$. The singularities of $\ph|_{2r}\gm^{-1}(t)
=
(a-cz)^{r-2}\, \ph(\gm^{-1}t)$ are contained in the union of
$a/c+i[0,\infty]$ and $\gm$ applied to the singularities of~$\ph$.

\quad We choose $z_0$ such that the set $z_0+\ZZ$ does not contain
points of $i[1,\infty]\cup \gm(i[1,\infty)]\cup (\gm\infty+\nobreak
i[0,\infty])\cup\{\gm \, z_0\}$. }{
\label{fig-singf}\setlength\unitlength{1.6cm}
\begin{picture}(3,1.4)(-1.5,-0)
\tfe
\end{picture}
}\vskip.4ex

Let $f=\ph|_{2-r}(1-\nobreak \mu\, \gm^{-1})\in R$. We have
$\Prj{2-r}\ph(t) = p(t)$, hence $(\Prj{2-r}\ph)(\infty)=0$; and also
$\bigl(\Prj{2-r}
(\ph|_{2-r}\gm^{-1}) \bigr)(\infty) = \bigl( p |^\prj_{2-r} \gm^{-1}
)(\infty) = 0$. (See~\eqref{prjact}.)
Using a one-sided average (Proposition~\ref{prop-osa1}) we find $h \in
\dsv{2-r}\om[\infty]$ such that
\be\label{hf-rel}
 h(t) - \ld^{-1} h(t+1) \=f(t)\= \ph(t)-\mu(\ph|_{2-r}\gm^{-1})(t)
\,,\ee
at least for $t\in \lhp$. So $f\in
\dsv{2-r}\om[\infty]|_{2-r}(1-\nobreak\ld^{-1}T)$, and
obviously also in $\dsv{2-r}\om|_{2-r}\allowbreak(1-\nobreak
\mu^{-1}\gm)$.
We have to show that if $p\neq0$, then none of the solutions
of \eqref{hf-rel} can be in $\dsv{2-r}{\om,\exc}[\infty]$.

If a solution $h$ of \eqref{hf-rel} were in
$\dsv{2-r}{\om,\exc}[\infty]$, then it extends holomorphically to an
$\{\infty\}$-excised neighbourhood. So $h$ can have singularities only
inside a strip $|\re z|\leq N$ for some $N>0$. In particular $h$ can
have singularities at $z_0+n$ only for a finite number of $n\in
\ZZ$.\vskip.4ex
\twocolwithpictl{\setlength\unitlength{1.6cm}
\begin{picture}(3,1.4)(-1.5,-0)
\tfe
\put(.5,1.1){\circle*{.04}}
\put(1.5,1.1){\circle*{.04}}
\put(-1.5,1.1){\circle*{.04}}
\end{picture}
}{\quad The right hand side in Relation~\eqref{hf-rel} has
singularities at $z_0+n$ only if $n=0$. So the maximal $n\geq 0$ such
that $z_0+\nobreak n$ is a singularity of $h$ cannot be larger than
$0$, since otherwise there would be a singularity $z_0+n+1$ as well.
Similarly, the minimum value of $n\leq 0$ such that $h$ is singular
at $z_0+n$ is also $0$. However, a singularity of $h$ only at $z_0$
is also impossible, since $f$ is holomorphic at $z_0\pm 1$.
}\vskip.4ex

So $h$ cannot have a singularity at any point of $z_0+\ZZ$. The choice
of $z_0$ shows that then $\ph$ has no singularity at $z_0$, in
contradiction with $p\neq 0$.
\end{proof}

\begin{proof}[Proof of Proposition~\ref{prop-icod-exc-fs}.] We have to
show that
\[ \dim \Bigl( \hpar^1(\Gm;\dsv{v,2-r}\om,\dsv{v,2-r}\fs) \bigm/
\hpar^1(\Gm;\dsv{v,2-r}\om,\dsv{v,2-r}{\fs,\wdg}) \Bigr)
\=\infty\,.\]
We choose a cusp $\ca$ of $\Gm$ and $\dt\in \Gm\setminus \Gm_\ca$, and
apply Parts b) and~c) of Lemma~\ref{lem-pcoc}. The map $a\mapsto
[\tilde c(a;\cdot)]$ induces a linear map
\begin{align*} \Bigl(\dsv{v,2-r}\om\bigm|_{v,2-r}&(1-\dt^{-1})
\Bigr)\cap \Bigl( \dsv{v,2-r}\fs\bigm|_{v,2-r}(1-\pi_\ca) \Bigr)\\
&\hbox{} \rightarrow \hpar^1(\Gm;\dsv {v,2-r}\om,\dsv{v,2-r}\fs)
\bigm/
\hpar^1(\Gm;\dsv{v,2-r}\om,\dsv{v,2-r}{\fs,\wdg})\,,\end{align*}
with kernel
\[ \Bigl(\dsv{v,2-r}\om \bigm|_{v,2-r} (1-\dt^{-1})\Bigr) \cap \Bigl(
\dsv{v,2-r}{\fs,\exc}\bigm|_{v,2-r}(1-\pi_\ca) \Bigr)\,.\]
So it suffices to show that this kernel has infinite codimension in
\[ \Bigl(\dsv{v,2-r}\om\bigm|_{v,2-r}(1-\dt^{-1}) \Bigr)\cap \Bigl(
\dsv{v,2-r}\fs\bigm|_{v,2-r}(1-\pi_\ca) \Bigr)\,.\]
Conjugating $\ca$ to $\infty$ and $\dt$ to $\gm$, we arrive at a
statement handled in Lemma~\ref{lem-icod-exc-fs}, with $\ld$ and
$\mu$ determined by $v(\pi_\ca)$ and~$v(\dt)$.
\end{proof}

\subsection{Recapitulation of the proof of Theorem
\ref{THMiso}}\label{sect-recap-iso}

\rmrkn{Part~i)} concerns the case $r\in \CC\setminus\ZZ_{\geq 2}$. We
have to show
\begin{enumerate}
\item[a)] $\hpar^1(\Gm;\dsv{v,2-r}\om,\dsv{v,2-r}{\fsn,\exc}) =
\hpar^1(\Gm;\dsv{v,2-r}\om,\dsv{v,2-r}{\fs,\exc})
\cong \hpar^1(\Gm;\dsv{v,2-r}{\fsn,\exc})$.
\item[b)] $\hpar^1(\Gm;\dsv{v,2-r}\om,\dsv{v,2-r}{\fsn,\exc})$ has
infinite codimension in $H^1(\Gm;\dsv{v,2-r}\om)$.
\item[c)] $\hpar^1(\Gm;\dsv{v,2-r}{\fsn,\exc}) \rightarrow
\hpar^1(\Gm;\dsv{v,2-r}{\fs,\exc})$ is injective with an image of
infinite codimension.
\end{enumerate}
In the following diagram we indicate where we have carried out the
various steps. (To save space we suppress $\Gm$ in the notation.) For
Parts i)a) and i)b) we have:
\bad\xymatrix@C=2cm@R=1cm{
\hpar^1(\dsv{v,2-r}\om,\dsv{v,2-r}{\fsn,\exc})
\ar@{=}[r]^{\text{Prop.~\ref{prop-parb*0}}}
\ar[d]_\cong^{\text{Thm.~\ref{thm-iso-mpcpc}}}
& \hpar^1(\dsv{v,2-r}\om,\dsv{v,2-r}{\fs,\exc})
\ar@{^{(}->}[d]_{\text{inf.~codim.}}
  ^{\text{Prop.~\ref{prop-icod-exc-fs}}}
\\
\hpar^1(\dsv{v,2-r}{\fsn,\exc})
& \hpar^1(\dsv{v,2-r}\om,\dsv{v,2-r}{\fs})
\ar@{^{(}->}[d]^{\text{Prop.~\ref{prop-acpc}}} _{\text{fin.~codim.}}
\\
& H^1(\dsv{v,2-r}\om)
} \ead
Part~i)c) follows from the following commuting diagram:
\bad \xymatrix@C=2cm@R=1cm{
\hpar^1(\dsv{v,2-r}\om,\dsv{v,2-r}{\fsn,\exc})
\ar@{=}[r]^{\text{Prop.~\ref{prop-parb*0}}}
\ar[d]_\cong^{\text{Thm.~\ref{thm-iso-mpcpc}}}
& \hpar^1(\dsv{v,2-r}\om,\dsv{v,2-r}{\fs,\exc})
\ar@{^{(}->}[d]_{\text{inf. cod.}}^{\text{Prop. \ref{prop-icd-hyp}}}
\\
 \hpar^1(\dsv{v,2-r}{\fsn,\exc})
 \ar[r]
& \hpar^1(\dsv{v,2-r}{\fs,\exc})
} \ead

\rmrkn{Part ii)} of the theorem states the following identities and
isomorphisms:
\begin{align*}
 \hpar^1(\Gm;\dsv{v,2-r}\om,\dsv{v,2-r}{\fsn,\infty,\exc}) &\=
 \hpar^1(\Gm;\dsv{v,2-r}\om,\dsv{v,2-r}{\fsn,\infty})
 \= \hpar^1(\Gm;\dsv{v,2-r}\om,\dsv{v,2-r}{\fs,\infty})\\
 &\;\cong\;\hpar^1(\Gm;\dsv{v,2-r}{\fsn,\infty})
 \;\cong\; \hpar^1(\Gm;\dsv{v,2-r}{\fs,\infty})\,.
\end{align*}
It follows from the diagram
\bad\xymatrix{ \hpar^1(\dsv{v,2-r}\om,\dsv{v,2-r}{\fsn,\infty,\wdg})
\ar@{=}[d]^{\text{Prop.~\ref{prop-rid-of-exc}}}\\
\hpar^1(\dsv{v,2-r}\om,\dsv{v,2-r}{\fsn,\infty})
\ar@{=}[r]^{\text{Prop.~\ref{prop-parb*0}}}
\ar[d]^{\text{Thm.~\ref{thm-iso-mpcpc}}} _\cong
& \hpar^1(\dsv{v,2-r}\om,\dsv{v,2-r}{\fs,\infty})
\ar[d]^{\text{Thm.~\ref{thm-iso-mpcpc}}} _\cong \\
\hpar^1(\dsv{v,2-r}{\fsn,\infty})
& \hpar^1(\dsv{v,2-r}{\fs,\infty})
}\ead
Only for the first equality we need $r\in \RR\setminus \ZZ_{\geq 2}$.
For all other steps $r\in \CC\setminus \ZZ_{\geq 2}$ suffices.

\rmrkn{Part~iii)} states for $r\in \RR \setminus \ZZ_{\geq 1}$:
\begin{enumerate}
\item[a)] The image $\coh r \om M_r
(\Gm,v)=\hpar^1(\Gm;\dsv{v,2-r}\om,\dsv{v,2-r}{\fsn,\smp,\wdg})$ is
 equal to
\[\hpar^1(\Gm;\dsv{v,2-r}\om,\dsv{v,2-r}{\fsn,\smp})\,,\quad
\hpar^1(\Gm;\dsv{v,2-r}\om,\dsv{v,2-r}{\fs,\smp})\,\]
and canonically isomorphic to $\hpar^1(\Gm;\dsv{v,2-r}{\fsn,\smp})$.
\item[b)] The space
$\hpar^1(\Gm;\dsv{v,2-r}\om,\dsv{v,2-r}{\fsn,\smp,\wdg})$ is
canonically isomorphic to the space
$\hpar^1(\Gm;\dsv{v,2-r}{\fs,\smp})$ if $v(\gm)\neq
e^{-r\ell(\gm)/2}$ for all primitive hyperbolic elements $\gm\in
\Gm$, where $\ell(\gm)$ is the hyperbolic length of the closed
geodesic associated to~$\gm$
\end{enumerate}
Part~iii)a) follows from the diagram
\bad\xymatrix{ \hpar^1(\dsv{v,2-r}\om,\dsv{v,2-r}{\fsn,\smp,\exc})
\ar@{=}[d]_{\text{Prop. \ref{prop-rid-of-exc}, ii)}}\\
\hpar^1(\dsv{v,2-r}\om,\dsv{v,2-r}{\fsn,\smp})
\ar@{=}[r]^{\text{Prop.~\ref{prop-parb*0}}}
\ar[d]_{\cong}^{\text{Thm.~\ref{thm-iso-mpcpc}}}
& \hpar^1(\dsv{v,2-r}\om,\dsv{v,2-r}{\fs,\smp})
\\
\hpar^1(\dsv{v,2-r}{\fsn,\smp})
}\ead
The condition that $r$ is real is needed only for the first step.
Part~iii)b) follows also from Theorem~\ref{thm-iso-mpcpc} under a
condition on hyperbolic elements.

\subsection{Related work}\label{sect-lit13}
The constructions in this section arose from a generalization of the
examples in Propositions 13.7 and~14.3 in~\cite{BLZm}. The paths~$p$
in Lemma~\ref{lem-path-hyp} and in the proof of Lemma~\ref{lem-pcoc}
represent cycles in homology. It is conceivable that they can be
related to the computations of Ash \cite{Ash89}, who computes the
parabolic cohomology with values in the rational functions by
computing first homology groups. We have not succeeded in making this
relation explicit.


\section{Quantum automorphic forms}\label{sect-qaf}
Theorem~\ref{THMiso} implies that $\coh r \om : A_r(\Gm,v) \rightarrow
H^1(\Gm;\dsv{v,2-r}\om)$ is far from 
surjective. Quantum automorphic forms may be put, for weights $r\in
\CC \smallsetminus \ZZ_{\geq 1}$, on the place of the
question mark in the diagram
\badl{qaf-diag} \xymatrix{ A_r(\Gm,v) \ar[d]_{\coh r \om}^\cong
\ar@{.}[r]
&? \ar@{.}[d]
\\
 \hpar^1(\Gm;\dsv{v,2-r}\om; \dsv{v,2-r}{\fsn,\wdg})
\ar@{^{(}->}[r]
& H^1(\Gm;\dsv{v,2-r}\om)
 }\eadl
This is similar to the role of quantum Maass forms
in~\cite[\S14.4]{BLZm}.

\subsection{Quantum modular forms}\label{sect-qmf}
Zagier \cite{Za10}
gives examples of \il{qmf}{quantum modular form}\emph{quantum modular
forms} as functions on $\QQ$ that have a modular transformation
behavior modulo a smooth function on~$\RR$.

\rmrkn{Example: } Powers of the Dedekind eta-function.
\il{poweta-q}{powers of the Dedekind eta-function} We attach a
quantum modular form to $\eta^{2r}$ with $\re r>0$.

The cusps of $\Gmod$ form one orbit, $\proj\QQ=\QQ \cup\{\infty\}$.
For each cusp $\ca\in \QQ$ the function
\be\label{hcad} h_\ca(t) \;:=\; \int_{z_0}^\ca \om_r(\eta^{2r};t,z)
\= \int_{z_0}^\ca \eta^{2r}(z)\,(z-t)^{r-2}\, dz \ee
is well defined for $t\in \lhp\cup\RR$.

Let $\dt=\matc abcd \in \Gmod$ such that $\ca,\dt^{-1}\ca \in \QQ$.
Then
\badl{tp} v_r(\dt)^{-1}&\,(c t+d)^{r-2}\, h_{\dt \ca}(\dt t) -
h_\ca(t) \= \biggl(
\int_{\dt^{-1}z_0}^\ca-\int_{z_0}^\ca\biggr)\om_r\bigl(
\eta^{2r};t,z) \\
&\= \ps_{\eta^{2r},\dt}^{z_0}(t)\,, \eadl
by Lemma~\ref{lem-omprop}. All terms in this relation are in
the space $\dsv{2-r}\infty$ of smooth vectors, hence
\be\label{peta2r} p(\ca) \;:=\; h_\ca(\ca)\qquad (\ca\in \QQ) \ee
is well defined, and satisfies
\be p|_{v_r,2-r} (\dt-1)\,(\ca) \=
\ps_{\eta^{2r},\dt}^{z_0}(\ca)\qquad
(\ca,\dt\ca\in \QQ)\,.\ee

The function $p:\QQ\rightarrow \CC$ has no reason to have a
continuous extension to~$\RR$. However, $p|_{v_r,2-r} (\dt-1)$ is the
restriction of a real-analytic function on~$\RR$. The function $p$ is
an example of a quantum modular form.

\rmrk{Strong quantum modular forms}Since $h_\ca$ as indicated above is
an element of $\dsv{2-r}\infty$, we have an asymptotic series
$h_\ca(t) \sim P(\ca,t):=\sum_{n\geq 0} c_n(\ca)\,(t-\ca)^n$,
approximating $h_\ca(t)$ as $t\rightarrow \ca$ through $\lhp\cup\RR$.
For $\dt\in \Gm$ as above we have from~\eqref{tp}:
\be v_r(\dt)^{-1}\,(c t+d)^{r-2}\, P(\dt\ca,\dt t) - P(\ca,t)
\,\sim\, \ps_{\eta^{2r},\dt}^{z_0}(t)
\ee
as $t\rightarrow \ca $ through $\lhp\cup\RR$. This means that $P$ is a
\il{sqaf}{strong quantum modular form}\il{qafs}{quantum modular form,
strong}\emph{strong quantum modular form} in the sense of
Zagier~\cite{Za10}.

\rmrk{Constant function}Now we take $r=0$ , hence $\eta^0=1\in
M_0\bigl(\Gmod,1\bigr)$. It seems sensible to take now
\be\label{hca} h_\ca(t) \= \frac1{t-\ca}\,.\ee
Now we cannot substitute $t=\ca$. However, with $\dt$ as above
\be \label{c1} v_0(\dt)^{-1}\, (ct+d)^{0-2}\, h_{\dt \ca}(\dt t) -
h_\ca(t)
\= \frac {-c}{c t+d}= \tilde\ps_\dt(t)\,, \ee
with the cocycle $\tilde\ps \in
Z^1\bigl(\Gmod;\dsv{1,2}{\fs,\wdg}\bigr)$
in~\eqref{ps-r=0}. So $P(\ca,t)=\frac 1{t-\ca}$ can be viewed as a
strong quantum automorphic form if we allow asymptotic series of the
form $\sum_{n\geq -1} c_n(\ca)\,(t-\ca)^n$.

\subsection{Quantum automorphic forms} \label{sect-qaf-def}For general
cofinite discrete groups~$\Gm$ we define quantum automorphic forms as
simply as possible for our purpose. The example of the constant
function shows that we need to use series starting at order~$-1$.

It turns out that we get satisfactory results in the context of these
notes if we use expansion starting at order $-1, $ namely,
$P(\ca,t):= c_{-1}\, (t-\ca)^{-1}+c_0 \cdots$, and leaving implicit
the terms $c_1, c_2,\ldots$.

\begin{defn}By \il{Scu}{$\cu$}$\cu$ we denote the set of cusps
of~$\Gm$. By a \il{soexp}{system of expansions}\emph{system of
expansions} $p$ on $\cu$ we mean a map assigning to all except
finitely many points $\ca\in\cu\cap\RR$ an expression
\[ p(\ca,t) \= c_{-1}(\ca) \,(t-\ca)^{-1}+ c_0 (\ca)
+ (t-\ca)\, \CC\hk{t-\ca}\,,\]
where $\CC\hk{t-\nobreak\ca}$ is the ring of formal power series in
$t-\ca$. Two such systems $p$ and~$p_1$ are equivalent if
$p(\ca,t)\equiv p_1(\ca,t) \bmod (t-\nobreak\ca)\,
\CC\hk{t-\nobreak\ca}$ for all but finitely many $\ca\in\cu\cap\RR$.
By \il{Rqaf}{$\R$}$\R$ we denote the linear space of equivalence
classes of systems of expansions.
\end{defn}

If $t\mapsto \ph(t)$ is real-analytic on a neighbourhood of~$\ca$
in~$\RR$, then multiplication by $\ph(t)$ is well defined for
elements of $(t-\nobreak \ca)^{-1}\CC\hk{t-\nobreak \ca} \bmod
(t-\nobreak \ca)\,\CC\hk{t-\nobreak\ca}$.

\begin{defn}The action \il{actGm-qaf}{$|_{v,2-r}$}$|_{v,2-r}$ of~$\Gm$
on $\R$ is induced by
\be (p|_{v,2-r}\gm)(\ca,t) \;:=\; v(\gm)^{-1}\, (ct+d)^{r-2} \,
p(\gm\ca,\gm t) \ee
for all $\ca\in \cu\cap \RR$ and $\gm =\matc\ast\ast c d\in \Gm$ for
which $p(\ca,\cdot)$ and $p(\gm \ca,\cdot)$ are defined. If
$r\not\in \ZZ$ we define $(ct+\nobreak d)^{r-2}$ by the argument
convention~\eqref{ac} for $t\in \lhp$.
\end{defn}

\rmrks
\itmi The operations in both parts of the definition preserve the
equivalence between systems of expansions. We will mostly identify an
equivalence class with a  representative  of it.

\itm The inclusion $\dsv{v,2-r}\om \rightarrow \R$ given by $\ph
\mapsto p_\ph$, where\ir{pphi-def}{p_\ph}
\be \label{pphi-def} p_\ph(\ca,t) \= \ph(\ca) +
 (t-\ca)\CC[[t-\ca]]
 \qquad\text{ for all }\ca\in \cu\cap\RR\,,\ee
is equivariant for the actions $|_{v,2-r}$ of $\Gm$ on
$\dsv{v,2-r}\om$ and $\R$.

\begin{defn}\label{defn-qaf}
Let $r\in\CC$ and let $v$ be a multiplier system for the weight~$r$.
\begin{enumerate}
\item[a)] By \il{Rv}{$\R_{v,2-r}$}$\R_{v,2-r}$ we denote $\R$ provided
with the action $|_{v,2-r}$ of~$\Gm$.
\item[b)] We define the $\Gm$-module \il{Qv}{$\Q_{v,2-r}$}$\Q_{v,2-r}
:= \R_{v,2-r}\bigm/ \dsv{v,2-r}\om$.
\item[c)] We define the space $\qA_{2-r}(\Gm,v)$ of \il{qautf}{quantum
automorphic form}\emph{quantum automorphic forms} of weight $2-r$
with multiplier system $v$ as a quotient of
$\Gm$-invariants:\ir{qA}{\qA_{2-r}(\Gm,v)}
\be\label{qA}
\qA_{2-r}(\Gm,v) \;:=\; \Q_{v,2-r}^\Gm \bigm/ \R_{v,2-r}^\Gm\,. \ee
\end{enumerate}
\end{defn}

\rmrks
\itmi So we have an exact sequence of $\Gm$-modules
\[ 0 \rightarrow \dsv{v,2-r}\om \rightarrow \R_{v,2-r} \rightarrow
\Q_{v,2-r}\rightarrow 0\,,\]
with the associated long exact sequence
\[ 0 \rightarrow(\dsv{v,2-r}\om)^\Gm \rightarrow \R_{v,2-r} ^\Gm
\rightarrow \Q_{v,2-r}^\Gm \rightarrow
H^1(\Gm;\dsv{v,2-r}\om)\rightarrow\cdots\]
We choose to define quantum automorphic forms as the quotient
$\Q_{v,2-r}^\Gm / \R_{v,2-r}^\Gm$, which can automatically mapped
into $H^1(\Gm,\dsv{v,2-r}\om)$ injectively.

In this way a quantum automorphic form is a function defined on almost
all of $\cu\cap \RR$ that has automorphic transformation behavior
modulo functions that are analytic on $\RR$ minus finitely many
points. Further we work modulo functions on $\cu \cap \RR$ that are
exactly automorphic.

\itm We leave it to the reader to explore the examples of Zagier
\cite{Za10}. The purpose of our definition is not to cover all those
examples. We are content to define quantum automorphic forms in such
a way that they fill the hole in diagram~\ref{qaf-diag}.

\subsection{Quantum automorphic forms, cohomology, and automorphic
forms}
\begin{prop}\label{prop-qA-coh}
Let $v$ be a multiplier system on~$\Gm$ for the weight $r\in \CC$.
\begin{enumerate}
\item[a)] There is an injective natural map\ir{Cdef}{\qC}
\be\label{Cdef} \qC:\qA_{2-r}(\Gm,v) \rightarrow
H^1(\Gm;\dsv{v,2-r}\om)\,. \ee
\item[b)] If $ r\in \CC \smallsetminus \ZZ_{\geq 1}$, then $\qC$ is
surjective.
\end{enumerate}
\end{prop}

\begin{proof} \emph{Injectivity, Part~a). }Definition~\ref{defn-qaf}
implies that the sequence
\be 0\rightarrow \dsv{v,2-r}\om \rightarrow \R_{v,2-r}\rightarrow
\Q_{v,2-r}\rightarrow0 \ee
is exact. The part
\[ \R_{v,2-r}^\Gm \rightarrow \Q_{v,2-r}^\Gm \rightarrow
H^1(\Gm;\dsv{v,2-r}\om)
\]
of the corresponding long exact sequence in group cohomology shows
that the connecting homomorphism induces an injective linear map
\[\qA_{2-r}(\Gm,v) \rightarrow H^1(\Gm;\dsv{v,2-r}\om)\,,\]
which we call~$\qC$. It sends a quantum automorphic form represented
by $p\in \R_{v,2-r}$ to the class of the cocycle $\gm \mapsto
p|_{v,2-r}( \gm-\nobreak1)$.

\rmrk{Surjectivity, Part~b)} Let $ r\in \CC \smallsetminus \ZZ_{\geq
1}$, and let $\ld\in \CC^\ast$. Proposition~\ref{prop-osa1} shows
that for each $f\in \dsv{2-r}\om$ at least one of the one-sided
averages $\av{T,\ld}^+f$ and $\av{T,\ld}^- f$ exists in
$\dsv{2-r}\fs$, and that $(\av{T,\ld}^\pm f)|_{2-r}(1-\ld^{-1}T)=f$.
We use $\av{T,\ld}^+ f$ if $|\ld|\geq 1$ and $\av{T,\ld}^- f$ if
$|\ld|<1$. Furthermore, by Proposition~\ref{prop-av-as}, there is an
asymptotic formula
$(\av{T,\ld}^\pm f)(t) 
= (it)^{r-2 }\,\bigl( c_{-1}t + c_0 + \oh(t^{-1})\bigr)$
as $t\rightarrow\pm \infty$ through~$\RR$, with coefficients
$c_{-1}$ and $c_0$ determined by $f$. By conjugation, we define for
parabolic $\pi=\s T \s^{-1}$, with $\s\in \SL_2(\RR)$,
$\xi=\s\infty\neq \infty$ and $f\in
\dsv{2-r}\om$:\ir{avpidef}{\av{\pi,\ld}^\pm}
\be\label{avpidef} \av{\pi,\ld}^\pm f \;:=\; \bigl(
\av{T,\ld}^\pm(f|_{2-r}\s)\bigr)|_{2-r}\s^{-1}\,.\ee

It satisfies
\be (\av{\pi,\ld}^\pm f)|_{2-r}(1-\ld^{-1}\pi) \= f\,. \ee
With the transformation $|_{2-r}\s^{-1}$ the asymptotic behavior
of $\av{T,\ld}^\pm (f|_{2-r}\s)$ at $\infty$ leads to an asymptotic
formula of the form
\be \label{ae-avpi}
(\av{\pi,\ld}^\pm f)(t) \= c_{-1}\, (t-\xi)^{-1} + c_0 + \oh(t-\xi)\,,
\ee
where $t\uparrow \xi$ for $\av{\pi,\ld}^+$ and $t\downarrow \xi$ for
$\av{\pi,\ld}^-$. The constants $c_{-1}$ and $c_0$ differ from the
constants at~$\infty$. The definition of $\av{\pi,\;d}^\pm f$ 
depends on the choice of $\s$ such that $\xi=\s\infty$. For cusps
$\ca\in \cu \cap\RR$ we use $\s_\ca$ as in~\S\ref{sect-af}.

After this preparation we consider a cohomology class in
$H^1(\Gm;\dsv{v,2-r}\om)$, represented by the cocycle $\ps\in
Z^1(\Gm;\dsv{v,2-r}\om)$. For $ \ca\in \cu \cap \RR $ we put
\be\label{psurj} p(\ca,t) \;:=\; - \bigl(\av{\pi_\ca,v(\pi_\ca)}^\pm
\ps_{\pi_\ca}\bigr)(t) + (t-\ca)\, \CC\hk{t-\ca} \,, \ee
where we choose $\pm $ so that the average exists.

Let $\dt=\matc abcd\in \Gm$ with $\ca,\dt\ca\in \RR$. Then
$$(\av{\pi_\ca,v(\pi_\ca)}^\pm \ps_{\pi_\ca} ) |_{v,2-r}(\pi_\ca-1) \=
- \ps_{\pi_\ca}.$$
Since $\pi_{\dt\ca} \= \dt \pi_\ca \dt^{-1}$, we have $\ps_{\pi_{\dt\ca}} \=
\ps_{\pi_\ca}|_{v,2-r}\dt^{-1} + \ps_\dt
|_{v,2-r}(\pi_\ca-1)\dt^{-1}$. Therefore,
\begin{align*}
&v(\dt)^{-1}\,(ct +d)^{r-2}\,
\bigl(\av{\pi_{\dt\ca},v(\pi_{\dt\ca})}^\pm
\ps_{\pi_{\dt\ca}}\bigr)(\dt t)
\= \bigl(\av{\pi_{\dt\ca},v(\pi_{\dt\ca})}^\pm
\ps_{\pi_{\dt\ca}}\bigr)|_{v,2-r}\dt\,(t)
\\
&\bigl(\av{\pi_{\dt\ca},v(\pi_{\dt\ca})}^\pm
\ps_{\pi_{\dt\ca}}\bigr)|_{v,2-r}\dt(\pi_\ca-1)
\=\bigl(\av{\pi_{\dt\ca},v(\pi_{\dt\ca})}^\pm
\ps_{\pi_{\dt\ca}}\bigr)|_{v,2-r}(\pi_{\dt\ca}-1)\dt\\
&\= -\ps_{\dt\ca}|_{v,2-r}\dt \= -\ps_{\pi_\ca} -
\ps_\dt|_{v,2-r}(\pi_\ca-1).
\end{align*}
Therefore,
$$\Bigl( \bigl(\av{\pi_{\dt\ca},v(\pi_{\dt\ca})}^\pm
\ps_{\pi_{\dt\ca}}\bigr)|_{v,2-r}\dt -
(\av{\pi_\ca,v(\pi_\ca)}^\pm \ps_{\pi_\ca} ) \Bigr)
|_{v,2-r}(\pi_\ca-1)
\=- \ps_\dt|_{v,2-r}(\pi_\ca-1).$$
From~\eqref{ae-avpi} we know that
$(\av{\pi_\ca,v(\pi_\ca)}^\pm \ps_{\pi_\ca} )\,(t)$ has a one-sided
expansion at $\ca$, and $\bigl(\av{\pi_{\dt\ca},v(\pi_{\dt\ca})}^\pm
\ps_{\pi_{\dt\ca}}\bigr)\,(t)$ at $\dt^{-1}\ca$. The transformation
formula for $|_{v,2-r}\dt$ shows that then
$\bigl(\bigl(\av{\pi_{\dt\ca},v(\pi_{\dt\ca})}^\pm
\ps_{\pi_{\dt\ca}}\bigr)|_{v,2-r}\dt\bigr)\,(t)$ has a similar
expansion at~$\ca$. The function $\ps_\dt$ is holomorphic at
$\ca$.

So the function $f:=
\bigl(\av{\pi_{\dt\ca},v(\pi_{\dt\ca})}^\pm
\ps_{\pi_{\dt\ca}}\bigr)|_{v,2-r}\dt -
(\av{\pi_\ca,v(\pi_\ca)}^\pm \ps_{\pi_\ca} )
+\ps_\dt$ has a one-sided asymptotic expansion at $\ca$ starting at a
multiple of $(t-\nobreak\ca)^{-1}$ and is invariant under
$|_{v,2-r}\pi_\ca$. Conjugating this to~$\infty$ and applying
Part~ii) of Lemma~\ref{lem-ldpr} we conclude that $f(t)\sim 0$ as $t$
approaches $\ca$ from one direction. (We use that
$r\not\in \ZZ_{\geq 1}$.)

So we have the following equality of asymptotic expansions
\[ \bigl(\av{\pi_{\dt\ca},v(\pi_{\dt\ca})}^\pm
\ps_{\pi_{\dt\ca}}\bigr)|_{v,2-r}\dt\; (t) -
(\av{\pi_\ca,v(\pi_\ca)}^\pm \ps_{\pi_\ca} ) (t) \sim
-\ps_\dt(t)
 \] 
as $t$ approaches $\ca$ through $\RR$ from
the left or the right depending on
$\pm$. We conclude that 
\be (p|_{v,2-r}\dt)(\ca,t) - p(\ca,t) \= \ps_\dt(t) 
+ (t-\ca)\, \CC[[t-\ca]]\,.
 \ee 
 So $p|_{v,2-r}(\dt-1) = \ps_\dt$ in
$\R$, and $\qC( p) =[\ps]$.
\end{proof}

\rmrke For weight $r=1$, Part~i) of Proposition~\ref{prop-osa1}
implies (after conjugating $\infty$ to $\ca\in \cu\cap\RR$) that
$\av{\pi_\ca,v(\pi_\ca)} \ps_{\pi_\ca}$ is defined if $
\ps_{\pi_\ca}(\ca)=0$. Since there are finitely many $\Gm$-orbits of
cusps, the construction in proof of surjectivity of $\qC$ goes
through for a subspace of $H^1(\Gm;\dsv{v,1}\om)$ of finite
codimension.

\begin{prop}\label{prop-Q}Let $v$ be a multiplier system on~$\Gm$ for
the weight $r\in \CC\setminus\ZZ_{\geq 1}$. There is an injective
linear map \il{Q}{$Q$}$Q:A_r(\Gm,v) \rightarrow \qA_{2-r}(\Gm,v)$
such that the following diagram commutes:
\[\xymatrix{ A_r(\Gm,v) \ar@{^{(}->}[rr]^{\coh r \om }
\ar@{^{(}->}[dr]^Q
 &
& H^1(\Gm;\dsv{v,2-r}\om)\\
& \qA_{2-r}(\Gm,v) \ar[ru]_\qC^{\cong} }\]
\end{prop}

\begin{proof}Theorem~\ref{THMac} implies that $\coh r \om$ is
injective. Since $\qC$ is bijective by Proposition~\ref{prop-qA-coh}
the map $\qC$ is invertible, so $Q=\qC^{-1}\circ \coh r \om$.
\end{proof}

\rmrke This result shows that for $r \in \CC \smallsetminus \in
\ZZ_{\geq 1}$ each class in $H^1(\Gm;\dsv{v,2-r}\om)$ is the
image of an object with automorphic flavor.

\rmrk{ Back to the examples} We now discuss the examples of
quantum modular forms in \ref{sect-qmf} after 
Definition~\ref{defn-qaf}.

For the \il{etapow-q}{powers of the Dedekind eta-function}\emph{powers
of the Dedekind eta-function} with $\re r>0$ we gave in \eqref{hcad}
\be p(\ca)=h_\ca(\ca)\qquad\text{with }h_\ca(t) = \int_{z_0}^\ca
\eta^{2r}(z)\, (z-t)^{r-2}\, dz\,.\ee
On the other hand, if $r\in \CC\setminus \ZZ_{\geq 1}$, then
$Q(\eta^{2r}) = \qC^{-1} \bigl( \coh r \om(\eta^{2r})
\bigr)$ in Proposition~\ref{prop-Q} can be given by
\be q(\ca,t) \= - \bigl(\av{\pi_\ca,v_r(\pi_\ca)}^\pm
\ps^{z_0}_{\eta^{2r},\pi_\ca}\bigr)(t)+(t-\ca)\, \CC[[t-\ca]]\,, \ee
according to the construction in~\eqref{psurj}, where $\pm$ has to be
chosen so that the one-sided average exists.

First take $r\in (0,\infty) \setminus \ZZ_{\geq 1}$. Then $h_\ca \in
\dsv{v_r,2-r}{\om,\infty,\exc}$ satisfies
$h_\ca|_{v,2-r}(\pi_\ca-\nobreak 1)= \ps_{\eta^{2r},\pi_\ca}^{z_0}$
by Lemma~\ref{lem-inf}. By Part~ii) of Lemma~\ref{lem-osa-pe}
(also conjugated to~$\ca$) we have $h_\ca=
\av{\pi_\ca,v_r(\pi_\ca)}^\pm\ps^{z_0}_{\eta^{2r},\pi_\ca}$ for both
choices of~$\pm$. So $p(\ca)
\equiv q(\ca,t)\bmod (t-\nobreak \ca)\, \CC[[t-\nobreak\ca]]$ in this
case.

For $\re r>0$, $r\in \CC\setminus \RR$, we consider only the case that
$\im r>0$; the other case goes similarly. We note that $v_r(\pi_\ca)
= v_r(T) \= e^{\pi i r/6}$ for all cusps~$\ca$. We use
$\av{\pi_\ca,v_r(\pi_\ca)}^-$, and the asymptotic of
$(\av{\pi_\ca,v_r(\pi_\ca)}^- \ps_{\eta^{2r},\pi_\ca}^{z_0})(t)$ as
$t\downarrow \ca$. Since $h_\ca$ and $-\av{\pi_\ca,v_r(\pi_\ca)}^-
\ps_{\eta^{2r},\pi_\ca}^{z_0}$ satisfy the same equation, we have
$h_\ca=-\av{\pi_\ca,v_r(\pi_\ca)}^- \ps_{\eta^{2r},\pi_\ca}^{z_0}+P$
with a$v_r(\pi_\ca)$-periodic function $P$. The asymptotic behavior
as $t\downarrow \ca$ shows that $P(t)$ has to be $\oh(t-\nobreak
\ca)$ as $t\downarrow\ca$. So $h_\ca$ and
$-\av{\pi_\ca,v_r(\pi_\ca)}^- \ps_{\eta^{2r},\pi_\ca}^{z_0}$
determine the same element of~$\R_{v_r,2-r}$. \smallskip

For the \emph{constant function} $\eta^0=1$ we used
$h_\ca(t)=(t-\nobreak \ca)^{-1}$. It leads in~\eqref{c1} to a cocycle
with values in $\dsv{1,2}{\fsn,\wdg}$, not in $\dsv{1,2}\om$. So it
does not represent $Q(1)\in \qA_2\bigl(\Gmod,1\bigr)$.

For an explicit computation, we write $\ca\in \QQ$ in the form $\s_\ca
\infty$, with $\s_\ca=\matc abcd \in \Gmod$, and hence $\ca=\frac
ac$. Then
\[ \pi_\ca^n \= \matc{1-nac}{na^2}{-nc^2}{1+nac}\,.\]
So we have, with $\pi=\pi_\ca=\s_\ca T \s_\ca^{-1}$:
\begin{align*}
\av{\pi,1}^+ &\ps^{z_0}_{1,\pi}\, (t)
\= \sum_{n\geq 0} \ps^{z_0}_{1,\pi}\bigm|_{1,2} \pi^n\, (t) \=
\sum_{n\geq 0} \int_{\pi^{-1}z_0}^{z_0}
\om_0(1;\cdot,z)|_{1,2}\pi^n\,(t)
\displaybreak[0]\\
&\= \sum_{n\geq 0}
\int_{\pi^{-n-1}z_0}^{\pi^{-n}{z_0}}\frac{dz}{(z-t)^2}\,.
\end{align*}
We have $\lim_{n\rightarrow \infty} \pi_\ca^{-n}z_0 = a/c=\ca$. Hence
\[ -\av{\pi,1}^+ \ps^{z_0}_{1,\pi}\, (t) \= -
\int_{\ca}^{z_0}\frac{dz}{(z-t)^2} \= \frac1{t-\ca}
+\frac1{z_0-t}\,. \]
This modification of the function $h_\ca$ in~\eqref{hca} leads to the
element of $\R$ given by
\[ p(\ca,t) \= \frac1{t-\ca}+(z_0-\ca)^{-1} +(t-\ca)\,
\CC\hk{t-\ca}\,,\]
which satisfies $p|_{1,2}(\dt-1)(\ca,t) \equiv \ps_{1,\dt}^{z_0}(t)
\bmod
(t-\ca)\, \CC\hk{t-\nobreak\ca}$.

\rmrk{ Dependence on the parameters} The family of modular forms
$r\mapsto \eta^{2r}$ depends holomorphically on $r$. This suggest to
look for quantum modular forms given by
\[ p(\ca,t) = \frac{\al(\ca,r)}{t-\ca} + \bt(\ca,r) + (t-\ca)\,
\CC[[t-\ca]]\,,\]
where $r\mapsto \al(\ca,r) $ and $r\mapsto \bt(\ca,r)$ are at least
continuous on $[0,\infty)$. This is impossible (proof left to the
reader). It is a phenomenon similar to the asymptotic expansion
in~\eqref{H-as}, where the coefficient in the leading term is
discontinuous in~$\ld$.

\subsection{Related work}\label{sect-lit14}The concept of quantum
automorphic is due to Zagier. His paper~\cite{Za10} gives beautiful
explicit examples of quantum modular forms. Zagier mentioned the
concept long before the appearance of~\cite{Za10}. The
paper~\cite{Br7} was written during the preparation of~\cite{BLZm},
to fill a hole in a diagram analogous to~\eqref{qaf-diag}.


\section{Remarks on the literature}\label{sect-lit}
Like we mentioned in~\S\ref{sect-lit2},  an indication  of what we now
call the Eichler integral is present in a paper of Poincar\'e in
1905, \cite{Poinc}. Eichler's definition in~\cite{Ei57} is based on
Bol's equality $\partial_\tau^{r-1} \bigl( F|_{2-r}\gm\bigr) =
F^{(r-1)}|_r \gm$, which appears in~\cite[\S8]{Bol}. In~\cite{Co56}
Cohn indicated this approach for weight~$4$. The paper \cite{Sh59} of
Shimura has a different atmosphere; it stresses cohomology with
values in a $\ZZ$-module. In the following years Gunning, Knopp,
Lehner and others studied the relation between automorphic forms and
cohomology: \cite{Gu59, Gu61, Ei65, CK65,Leh69, HK92, Leh71a, Raz77,
Raz77a, GR76}. Kra \cite{Kra69a, Kra69b} started the study of
cohomology of kleinian groups. Here the cohomology group is not
generated by Eichler integrals. We have not included in the list of
references all papers on the cohomology of kleinian groups.

Manin \cite{Ma73} discussed arithmetical questions. For a cuspidal
Hecke eigenform for $\SL_2(\ZZ)$ of even weight the ratio between the
even periods are in the field generated by the Fourier coefficients
of the cusp form; for the ratios of the odd periods the same holds.
The cocycles are present in the background, for instance in the
period relations. So apart from the Fourier coefficients there are
two, possibly transcendental, numbers involved in the coefficients of
the period polynomials. The arithmetic of the period polynomials,
associated with values of $L$-functions at integral points in the
critical strip, are an important area of study in connection with the
cocycles attached to automorphic forms. It goes further than the
central idea in these notes, which is establishing the relation
between automorphic forms and cohomology. Therefore we have not tried
to include all papers in this area in the list of references. We
mention the concept ``modular symbol''; see \cite{Ma73, Sho80}. We
mention also Haberland's paper~\cite{Ha83}, and \cite{KM87, Fu97,
Fu98, Za90}. In \cite{Za03} Zagier describes rather explicitly how to
reconstruct a cuspidal Hecke eigenform from its period polynomial.

The step from weights in~$\ZZ_{\geq 2}$ to general real weights was
done by Knopp in his paper \cite{Kn74}. For general real weights one
needs a multiplier system, which Knopp assumes to be unitary. This
definition leads to a map from cusps forms to cohomology classes with
values in the highest weight module $\dsv{v,2-r}{-\infty}$,
characterized by the condition of polynomial growth. Knopp's cocycle
integral also occurs in the paper \cite{Ni74} of Niebur. In the
proofs in \cite{Kn74} Knopp uses the construction of ``supplementary
series'', from~\cite{Kn62}. It is nice to see that with hindsight we
can view the resulting functions as mock automorphic forms.  See
for instance Pribitkin \cite{Pr99}, \cite{Pr0}. 

The isomorphism between the space of cusp forms and the cohomology
group was completed for all weights in 2010 by Knopp and Mawi,
\cite{KM10}. In \cite{CLR14} the isomorphism of Knopp is
combined with multiplication by non-zero cusp forms.

Knopp, \cite{Kn78}, started the study of rational period functions and
gave examples. He showed in~\cite{Kn81} that the singularities can
occur only in the rational points $0$, $\infty$, and in points in
real-quadratic fields
(which are hyperbolic fixed points of~$\Gmod$), and Choie\cite{Ch89}
showed the existence of rational period functions with singularities
in any real quadratic irrationals. Several authors expanded the
theory, \cite{MR84, KZ84, Ash89, Ch89, CP90, CP91, Ch92, CZ93, HK92,
Pa93, Ch94, Schm96, Fu97, DIT10}. We expect that the approach in
Sections \ref{sect-isos-pc} and~\ref{sect-coc-sing} can be applied to
cohomology with values in the module of rational functions.

In \cite{KnMa3} Knopp and Mason start the study of ``generalized
modular forms'', which are vector-valued automorphic forms with at
most exponential growth at the cusps for the modular group
$\SL_2(\ZZ)$ with real weight and matrix-valued multiplier systems
that need not be unitary. The papers \cite{KLR9, KR10, Ra11, Ra13}
deal with the cohomology classes associated to these automorphic
forms.

The $\Gm$-behavior of automorphic forms can be formulated as the
vanishing of $F|_{v,r}(\gm-\nobreak 1)$ for all $\gm\in \Gm$. This
has been generalized to the condition that
\[ F|_{v,r}(\gm_1-1)\,(\gm_2-1)\cdots (\gm_q-1) \=0\quad\text{ for all
}\gm_1,\ldots,\gm_q\in \Gm\,,\]
leading to ``higher order automorphic forms'', for which Deitmar,
\cite{Dei9, Dei12}, has studied cohomological questions. See also
Diamantis and O'Sullivan \cite{DO8}, Sim \cite{Si9}. Cohomological
techniques have also been used in the context of higher-order forms
by Taylor~\cite{Ta12}. See further~\cite{BrDi}.

In \cite{BGKO} Bringmann, Guerzhoy, Kane and Ono consider period
polynomials for $r$-harmonic modular forms with negative even
weights. Bringmann, Diamantis and Raum \cite{BDR13} extended the
construction to account for non-critical values of
$L$-functions.\medskip

The condition of holomorphy can be completely removed from the
definition of automorphic forms, and replaced by a second order
differential equation. Formulated in terms of functions on the
 universal covering group~$\tG$ this is the eigenvalue equation for
the Casimir operator. This leads to the so-called ``Maass forms'' and
their generalizations. For Maass forms of weight~$0$ the relation
between automorphic forms and cohomology has been studied by many
authors. Lewis, \cite{Le97}, gave a bijection between even Maass cusp
forms and spaces of holomorphic functions on
$\CC\setminus(-\infty,0]$ that satisfy a functional equation similar
to the equation satisfied by period function for the modular group
$\PSL_2(\ZZ)$. In the papers \cite{LZ97} and especially \cite{LZ1}
this is further discussed for the modular group. M\"uhlenbruch,
\cite{Mue3} extended this to real weights. See also~\cite{MuRa}.
Martin \cite{Mar6} uses similar methods in the context of
holomorphic modular forms of weight~$1$. A relation between
the period functions of Lewis and the hyperfunctions associated to
Maass forms was explored in~\cite{Br97}, the ideas in which were
expanded in~\cite{DH5, DH7, Dei11}.

Another, rather unexpected, relation is with eigenfunctions of the
transfer operator introduced by Mayer, \cite{May91}, in connection
with the Selberg zeta-function. Transfer operators are a concept from
mathematical physics, applied by Mayer to the geodesic flow on the
quotient $\PSL_2(\ZZ)\backslash \uhp$. The eigenfunctions of the
transfer operator with eigenvalue~$1$ are, after a simple
transformation, identical with Lewis's period functions. So the
eigenfunctions of the transfer operator are related to cohomology
classes. See \cite{LZ97}, \cite[Chap.~IV, \S3]{LZ1}, and \cite{Za2}
for a further discussion. In \cite{BrMu9} this relation with
cohomology is used to relate eigenfunctions of two transfer
operators. See also \cite{MP11, Po12, Po12a, Po13}. As far as we see,
the use of a transfer operator is less suitable in the present
context, since the space $\dsv{2-r}\om$ is not the space of global
sections of a sheaf on~$\proj\RR$.\smallskip

The aim of the paper~\cite{BLZm} is to explore the relation between
Maass forms of weight zero and cohomology more completely, for all
cofinite discrete group. For cocompact discrete groups rather
complete results were available, even in the context of automorphic
forms on more general symmetric spaces, in the work of Bunke and
Olbrich, \cite{BO95, BO98}. For groups with cusps a reasonably
complete description was obtained with use of three ideas: (1) use of
mixed parabolic cohomology groups; (2) work with boundary germs as
coefficient module; (3) description of the mixed parabolic cohomology
groups with resolutions based on a suitable tesselation of the upper
half-plane. In the present notes we tried to apply these ideas in the
context of holomorphic automorphic forms.

\appendix

\section{Universal covering group and representations} The discussion
in this appendix is not really essential for these notes, but several
 definitions and arguments become more natural if we relate them to
 the universal covering group of $\SL_2(\RR)$.

\subsection{Universal covering group} The \il{ucg}{universal covering
group}\emph{universal covering group} \il{tG}{$\tG$}$\tG$ of
$\SL_2(\RR)$ is a simply connected Lie group that is locally
isomorphic to the Lie group~$\SL_2(\RR)$.

We can describe $\tG$ with help of the \il{Iwdc}{Iwasawa
decomposition}\emph{Iwasawa decomposition} of $\SL_2(\RR)$, which
writes each $g\in \SL_2(\RR)$ uniquely as
\[ g \= \matc{\sqrt y}{\frac{x}{\sqrt y}}0{\frac1{\sqrt y}} \,
\matr{\cos\th}{\sin\th}{-\sin\th}{\cos\th}\,,\]
with $z=x+iy\in \uhp$ and $\th\in \RR\bmod2\pi\ZZ$. As an analytic
variety, $\SL_2(\RR)$ is isomorphic to $\uhp \times(\RR/2\pi\ZZ)$. A
simply connected analytic variety that covers $\uhp
\times(\RR/2\pi\ZZ)$ is $\uhp\times \RR$, with the natural map
$\RR\rightarrow \RR/2\pi\ZZ$. We denote its points as
\il{Ic}{$(z,\th)$}$(z,\th)$, with $z\in \uhp$, $\th\in \RR$. It is
possible to define a group structure on $\uhp\times\RR$ such that the
projection map $\uhp\times\RR\rightarrow\uhp\times(\RR/2\pi\ZZ)$ is
an real-analytic group homomorphism. The resulting group with
underlying space $\uhp\times \RR$ is the universal covering group
$\tG$ of $\SL_2(\RR)$, with projection homomorphism\ir{pr}{\pr}
\be\label{pr}
\pr : \tG \rightarrow\SL_2(\RR)\,.\ee

Here we do not describe the group structure of~$\tG$ explicitly. (See,
{\sf eg.,} \cite[\S2.2.1]{Br94}.)
We mention that there is a group homomorphism
\il{tldk}{$\tilde k(\th)$}$\tilde k:\RR\rightarrow \tG$, given by
$\tilde k(\th) = (i,\th)$. It covers the isomorphism
$\RR/2\pi\ZZ\rightarrow \SO(2)$ given by
$\th\mapsto\matr{\cos\th}{\sin\th}{-\sin\th}{\cos\th}$. We note that
$\bigl\{ \tilde k(2\pi n)\,:\, n\in \ZZ \bigr\}$ is the kernel of
$\pr :\tG\rightarrow \SL_2(\RR)$, and that
\il{tldZ}{$\tilde Z$}$\tZ:=\bigl\{\tilde k(\pi n)\,:\, n\in
\ZZ\bigr\}$ is the center of~$\tG$.

The most important aspect of the group structure is the lift $g\mapsto
\tilde g$ from $\SL_2(\RR)$ to~$\tG$, given by\ir{tildeg}{\tilde g}
\be\label{tildeg} \widetilde{\matc abcd} \;:=\; \Bigl(
\frac{ai+b}{ci+d}, - \arg(ci+d)
\Bigr)\qquad\text{with }-\pi<\arg(cz+d)\leq \pi\,.\ee
It takes a preimage for the covering map~$\pr$. It satisfies
\be\label{g-lift} \widetilde{\matc abcd}\, (z,\th) \= \Bigl(
\frac{az+b}{cz+d},\th-\arg(cz+d)
 \Bigr)\,. \ee
This map is continuous on the open dense subset $G_0\subset
\SL_2(\RR)$, in~\eqref{G0}. We have
\badl{tldrel} (\tilde g)^{-1} &\= \widetilde{g^{-1}}&&\text{for } g\in
G_0\,,
\\
\widetilde{gpg^{-1}}&\= \tilde g \tilde p (\tilde g)^{-1}&&\text{for
}g\in G_0,\; p=\matc{\sqrt y}{\frac x{\sqrt y}}0{\frac 1{\sqrt
y}},\;x+iy\in \uhp\,. \eadl
All elements of $\tG$ can be, non-uniquely, written as a product
$\tilde g \, k(\pi n)$, with $g\in G_0$, $n\in \ZZ$.

\subsubsection{Weight functions and actions by right and left
translation}\label{app-wf}
\begin{defn}A function $f:\tG\rightarrow \CC$ has
\il{ucwt}{weight}\emph{weight} $r\in \CC$ if $f(z,\th)=f(z,0)\,
e^{ir\th}$.
\end{defn}
A function $f$ on~$\tG$ with weight~$r$ is determined by its values on
$(z,0)$, with $z\in \uhp$. We define a corresponding function $R_r f$
on $\uhp$ by\ir{Rrf}{R_r f}
\be\label{Rrf} (R_r f)(z) \;:=\; y^{-r/2}\, f(z,0)\,,\quad\text{ hence
} f(z,\th) \= y^{r/2}\, (R_r f)(z)\, e^{ir\th}\,. \ee

\rmrk{Left translation}The group $\tG$ has a right action in the space
of functions $\tG\rightarrow \CC$ given by \il{uclt}{left
translation}\emph{left translation}\ir{Lt}{f|g}
\be\label{Lt}
\text{for }g\in \tG: f\mapsto f|g\,,\text{ given by }(f|g)(g_1)\=f
(gg_1)\,. \ee
We also use the notation \il{Lg}{$L_g f= f|g$}$L_g f = f|g$.

The action by left translation preserves the weight. Moreover, we have
\be\label{Rrcmp} \bigl( R_r(f|\tilde g ) \bigr) (z) \= (cz+d)^{-r}\,
(R_r f)(z)\qquad\text{for }g=\matc abcd\in \SL_2(\RR)\,.\ee
Thus, we see that the operators $|_r g$ in \eqref{SL2op} correspond
naturally to the representation of $\tG$ by left translation in the
functions on $\tG$ of weight~$r$. The argument convention for
$\arg(cz+\nobreak d)$ for $z\in \uhp$ in \eqref{ac} is coupled to the
choice of the argument in~\eqref{tildeg}. Then the convention for
$z\in \lhp$ is determined by the wish to have
relation~\eqref{Ci-act}. Since $g\mapsto \tilde g$ is not a group
homomorphism, the operators $|_r g$ do not form a representation of
$\SL_2(\RR)$.

\rmrk{Right translation}There is also the left action of $\tG$ on the
functions on~$\tG$ by \il{ritr}{right translation}\emph{right
translation}:\ir{rt}{R_g}
\be\label{rt} (R_g f)(g_1) \;:=\; f(g_1 g)\,.\ee
It commutes with left translations. It does not preserve the weight.

\subsubsection{Discrete subgroup}\label{app-dgu}For a cofinite
\il{dgu}{discrete subgroup}discrete subgroup $\Gm\subset\PSL_2(\RR)$
we define \ir{tGm}{\tGm}
\be\label{tGm}
\tGm \;:\;= \bigl\{ g\in \tG \;:\; \pr\, g\in \Gm\bigr\}\,.\ee
It is a discrete subgroup of $\tG$. It contains the center $\tZ$.

Any character $\ch: \tGm\rightarrow\CC^\ast$ of $\tGm$ induces a
\il{cchu}{central character}\emph{central character} of $\tZ$ which
is determined by $\ch\bigl( \tilde k(\pi)\bigr)
$, which we can write as $\ch\bigl(\tilde k(\pi)\bigr) = e^{\pi i r}$
with $r\in \CC \bmod2\pi\ZZ$. The map $v_\ch : \Gm\rightarrow
\CC^\ast$ given by\il{vch}{$v_\ch$}
\be\label{vch}
v_\ch \matc abcd \;:=\; \ch\biggl( \widetilde {\matc abcd}
\biggr)\qquad \matc abcd\in \Gm \ee
is a \il{ucmsu}{multiplier system}multiplier system on $\Gm$ for the
weight~$r$. One can check that all multiplier systems on~$\Gm$ arise
in this way.

The representation $|_{v_\ch,r}$ of $\bar\Gm$ on the functions on
$\uhp$, in \eqref{vq-act+}, corresponds to the representation
$\ch^{-1}\otimes L$ of $\tGm$ on the functions of weight~$r$ on
$\tG$. For these functions the generator $\tilde k(\pi)$ of~$\tZ$
acts as multiplication by $\ch\bigl( \tilde k(\pi)\bigr)^{-1}\,
e^{\pi i r} = 1$. So indeed, $\ch^{-1}\otimes L$ is a representation
of $\tGm/\tZ \cong \bar\Gm$.

The invariants of the representation $\ch^{-1}\otimes L$ in the
functions of weight $r$ correspond to the space of all functions on
$\uhp$ with \il{atb}{automorphic transformation
behavior}$(\Gm,v)$-automorphic transformation behavior of weight~$r$.
For automorphic forms one requires also that the functions are
eigenfunctions of a differential operator. These differential
operators can be described with the Lie algebra.
(See~\S\ref{sect-Lie}.)

\rmrk{Modular group} The \il{mguc}{modular group}modular group
$\Gmod=\SL_2(\ZZ)$ is covered by
\il{tGmod}{$\tGmod$}$\tGmod\subset\tG$. The generators $T=\matc 1101$
and $S=\matr0{-1}10$ can be lifted to give generators $t=(i+\nobreak
1,0)$ and $s=(i,-\pi/2)$ of $\tGmod$, with relations generated by
$ts^2=s^2t$ and $ts\,ts\,t=s$.

All characters of $\tGmod$ are of the form $\ch_r:\tGmod\rightarrow
\CC^\ast$ with $r\in \CC/12\pi\ZZ$ given by\ir{uc-chr}{\ch_r}
\be\label{uc-chr} \ch_r(t) \= e^{\pi i/6}\,,\qquad \ch_r(s) \= e^{-\pi
i r/2}\,, \ee
corresponding to the multiplier system \il{ucvr}{$v_r$}$v_r$
in~\eqref{vr-mod}.

\subsubsection{Lie algebra}\label{sect-Lie}The real \il{La}{Lie
algebra}Lie algebra of $\SL_2(\RR)$ is\ir{glie}{\glie_r}
\be\label{glie}
\glie_r \;:=\; \Bigl\{ g \in M_2(\RR)\;:\; \mathrm{Trace}\,g=0
\Bigr\}\,.\ee
A basis is $\WW=\matr01{-1}0$, $\HH=\matr100{-1}$, $\VV=\matc0110$.
For each $\XX\in \glie_r$ the \il{em}{exponential map}exponential
$\exp \XX=\sum_{n\geq 0} \frac1{n!}\, \XX^n$ is an element of
$\SL_2(\RR)$. For small values of $t\in \RR$ we have $\exp(t\XX)\in
G_0$; the lift $t\mapsto (\exp t\XX)^\sim$ extends to a group
homomorphism $\RR\rightarrow\tG$. This leads to differential
operators on~$\tG$:\il{LLie}{$L_\XX$}\il{RLie}{$R_\XX$}
\be (L_\XX f)(g)\;:=\;\frac d{dt} f\bigl( (\exp t\XX)^\sim
g\bigr)\Bigr|_{t=0}\,,\quad
(R_\XX f)(g) \;:=\; \frac d{dt} f\bigl( g(\exp
t\XX)^\sim\bigr)\Bigr|_{t=0}\,. \ee
This can be extended to a linear map $\XX\mapsto L_\XX$ from the
complexified Lie algebra
\il{gliec}{$\glie$}$\glie:=\CC\otimes_\RR \glie_r$ to the first order
\il{rido}{right-invariant differential operators}right-invariant
differential operators on~$\tG$. Similarly we have a linear map
$\XX\mapsto R_\XX$ from $\glie$ to the first order
\il{lido}{left-invariant differential operators}left-invariant
differential operators on~$\tG$. So the operators $R_\XX$ leave
invariant the space of invariants for the representation
$\ch^{-1}\otimes L$ in $C^\infty(\tG)$, and the operators $L_\XX$
leave invariant the space of differentiable functions with a given
weight.

The relation with the \il{Lp}{Lie product}Lie product
\il{[.,.]}{$[\cdot,\cdot]$}$[\XX,\YY]=\XX\YY-\YY\XX$
is
\be R_\XX R_\YY - R_\YY R_\XX \= R_{[\XX,\YY]}\,,\quad L_\XX L_\YY -
L_\YY L_\XX = - L_{[\XX,\YY]}\,.\ee
We also write \il{XXf}{$\XX f = R_\XX f$}$\XX f$ instead of $R_\XX f$.
For the basis $\WW$, \il{EEpm}{$\EE^\pm$}$\EE^+=\HH+i\VV$,
$\EE^-=\HH-i\VV$ of~$\glie$ we have in the coordinates $(z,\th)\in
\tG$:\badl{Wee} \WW &\=
\partial_\th\,,\\
\EE^+ &\= e^{2i\th}
\bigl(2iy\,\partial_x+2y\,\partial_y-i\partial_\th\bigr)\,,\quad
\EE^- \= e^{-2i\th}
\bigl(-2iy\,\partial_x+2y\,\partial_y+i\partial_\th\bigr)\,.
\eadl

The Lie algebra $\glie$ can be embedded in the \il{uea}{universal
enveloping algebra}\emph{universal enveloping algebra}
\il{Uenv}{$\mathcal U$}$\mathcal U$, generated by all products of
elements of $\glie$, with the relations $\XX\YY-\YY\XX=[\XX,\YY]$ for
 all $\XX,\YY\in \glie$. The maps $\XX\mapsto R_\XX$ and $\XX\mapsto
L_\XX$ can be extended to~$\mathcal U$, and describe the
ring of all left-invariant, respectively right-invariant,
differential operators on~$\tG$. The center of $\mathcal U$ is a
polynomial algebra in one variable, for which we can
take\ir{omdef}{\om}
\be \label{omdef} \om\;:=\; - \frac14\EE^-\EE^+ + \frac 14\WW^2+ \frac
i2\WW\= - \frac14\EE^+\EE^- + \frac 14\WW^2- \frac i2\WW\,.\ee
It gives rise to the following \il{bido}{bi-invariant differential
operator}bi-invariant differential operator
on~$\tG$:\ir{Cas}{L_\om=R_\om}
\be\label{Cas}L_\om\= R_\om \= e^{-2i\th}
\bigl(-2iy\,\partial_x+2y\,\partial_y+i\partial_\th\bigr)\,,
\ee
called the \il{Casop}{Casimir operator}\emph{Casimir operator}.

\subsubsection{Automorphic forms on~$\tG$}\label{app-do}One may define
an \il{ucaftG}{automorphic form}automorphic form on $\tG$ with
character~$\ch$ as a function $f:\tG\rightarrow \CC$ with
transformation behavior $f(\gm g)=\ch(\gm)\, f(g)$ for all $g\in
\tG$, $\gm\in \tGm$, that is an eigenfunction of $R_\om$
and~$R_\WW$. With this definition, an automorphic form has a
\il{uc1wt}{weight}weight~$r\in \CC$, determined by $R_\WW f = irf$,
 and an \il{ei}{eigenvalue of automorphic form}eigenvalue $\ld\in
\CC$, determined by $R_\om f=\ld f$.

There are several interesting sets of values for $(\ld,r)$. If one
wants to do spectral theory, it is convenient to take $r\in \RR$.
Then square integrability of the automorphic forms restricts $\ld$ to
a subset of~$\RR$ containing the interval $(1/4,\infty)$.

The automorphic forms considered in \cite{BLZm} correspond to $r=0$
and $\ld=s(1-\nobreak s)$ with $0<\re s<1$.

The differential operator $\Dt_r$  in~\eqref{Dtr} 
corresponds under $R_r$ in~\eqref{Lt} to $R_\om - \frac
r2(1-\nobreak\frac r2)$. If $f$ has weight~$r$, then $\EE^-
f$ has weight $r-2$. With \eqref{Rrf} we have
\be R_{r-2} (\EE^- f) \= -4i y^2 \,\partial_{\bar z} R_r f \;\Bigl( \=
\overline{2 y^{r-2}\,\shad _r F}\,\Bigr)
\,.\ee
So the condition of holomorphy corresponds to being in the kernel
of~$\EE^-$. Then \eqref{Cas} implies that $\ld=\frac
r2\bigl(1-\nobreak \frac r2\bigr)$. For the same eigenvalue there are
more eigenfunctions of the Casimir operator than there are in the
kernel of $\EE^-$. They correspond to the larger space of
$r$-harmonic automorphic forms.

\subsubsection{Polar functions}\label{app-pol-fcts} The polar
$r$-harmonic functions \il{Prmucu}{$\P{r,\mu}$}$\P{r,\mu}$,
\il{Mrmucu}{$\M{r,\mu}$}$\M{r,\mu}$, and
\il{Hrmucu}{$\H{r,\mu}$}$\H{r,\mu}$ in \S\ref{sect-pol-exp} are
specializations of functions in~\cite[\S4.2]{Br94}. Fourier terms
$F(\mu,\cdot)$
transforming according to $F(\mu,\cdot)|_r
\matr {\cos\th}{\sin\th}{-\sin\th}{\cos\th} = e^{i(r+2\mu)\th}\,
F(\mu,\cdot)$ for small values of~$\th$ are of the form $R_r
f(\mu,\cdot)$, as in~\eqref{Rrf}, where $f(\mu,\cdot):\tG\rightarrow
\CC$ satisfies
\be f\bigl(\mu,\tilde k(\eta) g \tilde k(\ps) \bigr)\=
e^{ir(\eta+\ps)+2i\mu\eta}\, f(\mu,g)\,,\quad R_\om f(\mu,g) \= \frac
r2\,\Bigl(1-\frac r2\Bigr)\,f(\mu,g)\,. \ee

Such a function can be written as
\be \tilde k(\eta)\, (it,0)\, \tilde k(\ps)
\mapsto e^{2i\mu\eta+ir(\eta+\ps)}\, \bigl(\frac
u{u+1}\Bigr)^{\mu/2}\,(u+1)^{-r/2}\, h_\mu\Bigl(\frac
1{u+1}\Bigr)\,,\ee
with $t\geq1$, $u=\bigl( (t^{1/2}+\nobreak t^{-1/2})/2\bigr)^2$, where
$h_\mu$ satisfies the differential equation in~~\cite[\S4.2.6]{Br94}.
\begin{table}[ht]
\[ \renewcommand\arraystretch{1.2}
\begin{array}{|rcl|rcl|}\hline
\text{\cite{Br94}}&&\text{here}& \text{\cite{Br94}}&&\text{here}\\
\hline
n&=&r+2\mu& l&=&r\\
u&=& \frac{(t+1)^2}{4t}& s&=& \frac{r-1}2 \\
p&=&\frac12|\mu| & \e &=&\sign \mu\\ \hline
u&=&\frac{|z-i|^2}{4y} & e^{2i \eta} &=&
\frac{z-i}{z+i}\,\frac{|z+i|}{|z-i|}\\
e^{2i(\eta+\ps)}&=& \frac{2i}{z+i}\, \frac{|z+i|}2&&&
\\ \hline
\end{array}
\]
\caption{Relations for the computation
in~\S\ref{app-pol-fcts}}\label{tab-rels}
\end{table}
In Table~\ref{tab-rels} we summarize the relation between the
variables in~\cite{Br94} and here. The solutions in \cite[4.2.6 and
4.2.9]{Br94} give:
\badl{muom} \mu\bigl(n,s;(it,0)\bigr) &\= \Bigl( \frac u{u+1}
\Bigr)^{\mu/2}\,
(u+1)^{-r/2}\,,\\
\mu\bigl(n,-s;(it,0)\bigr) &\=\Bigl( \frac u{u+1} \Bigr)^{\mu/2}\,
(u+1)^{r/2-1}\, \hypg21\Bigl( 1+\mu,1-r;2-r;\frac1{u+1}\Bigr)\,,\\
\text{if }\mu\leq 0\,:\\
\om\bigl( n,s;(it,0)\bigr)&\= \Bigl(\frac u{u+1} \Bigr)^{-\mu/2}\,
(u+1)^{-r/2}\,\hypg21\Bigl( |\mu|, r; 1+|\mu|;\frac u{u+1}\Bigr)\,.
\eadl
We write $(z,0)=\tilde k(\eta)\,
(it,0)\,\tilde k(\ps)$, and have to multiply with $y^{-r/2}\,
e^{ir(\eta+\ps)+2i\eta\mu}$ to get the corresponding function
$F(\mu,\cdot)$. Table~\ref{tab-rels} shows also that the functions
 in~\eqref{muom} correspond to $\P{r,\mu}$, $\M{r,\mu}$, and
$\H{r,\mu}$, respectively. This requires some computations and, for
$\M{r,\mu}$ with $\mu\leq 0$, use of a Kummer relation
(Relation (2),~\cite[\S2.9]{EMOT}).

\subsubsection{Resolvent kernel}\label{app-rk}Let $m_r$ denote the
function on~$\tG$ such that $R_r m_r = \M{r,0}$. So $m_r\bigl( \tilde
k(\eta) g \tilde k(\ps)\bigr)= e^{ir(\eta+\psi)}\, m_r(g)$. The
kernel function~\il{Qruc}{$Q_r(\cdot,\cdot)$}$Q_r$ in~\eqref{Qr-def}
corresponds to the function $q_r(g_1,g_2) := m_r(g_1^{-1}g_2)$, which
satisfies $q_r\bigl(g_1\tilde k(\th_1),g_2\tilde k(\th_2) \bigr)=
e^{ir(\th_2,\th_1)}\, q_r(g_1,g_2)$. So it has weight $-r$ in~$g_1$
and weight $r$ in~$g_2$, and we should have $Q_r(z_1,z_2) = y_1^{r/2}
y_2^{-r/2} \, q_r\bigl(
z_1,0),(z_2,0)\bigr)$, which is indeed the case:
\begin{align*}
y_1^{r/2} y_2^{-r/2} \,& q_r\bigl( z_1,0),(z_2,0)\bigr)
\= (y_1/y_2)^{r/2}\,m_r \biggl(\;
\matc{y_1^{-1/2}}{-x_1y_1^{-1/2}}0{y_1^{1/2}}^\sim\,(z_2,0)\biggr)
\\
&\= (y_1/y_2)^{r/2}\, m_r\bigl( (z_2-x_1)/y_1,0\bigr)
\\
&\= (y_1/y_2)^{r/2}(y_2/y_1)^{r/2} \, M_r \bigl(z_2-x_1)/y_1 \bigr) \=
M_r \bigl(z_2-x_1)/y_1 \bigr)\,.
\end{align*}
Since $q_r(gg_1,gg_2)=q_r(g_1,g_2)$ for all $g\in \tG$, this
immediately implies the invariance relation~\eqref{Qinv}.

For the differential equations we use that in weight $r$ the Casimir
operator corresponds to $\Dt_r + \frac r2\bigl(1-\nobreak\frac
r2\bigr)$. Since $\om$ is left-invariant, we have
\[R_\om q_r(g_1,\cdot)
\= \frac r2\bigl(1-\frac r2\bigr)\, q_r(g_1,\cdot)\,.\]
This corresponds to~\eqref{qh2}.

The Casimir operator commutes with $g\mapsto g^{-1}$ and with right
translations, so $\om q_r(\cdot,g_2) = \frac r2\bigl(1-\frac
r2\bigr)\, q_r(\cdot,g_2)$.
Since $Q_r$ has weight $-r$ in the first variable, we have
\[ \Bigl( \Dt_{-r} +\frac{-r}2\,\bigl(1+\frac r2\bigr)\Bigr)
Q_r(\cdot,z_2) \= \frac r2\,\bigl(1-\frac r2\bigr)\,
Q_r(\cdot,g_2)\,,\]
which is~\eqref{nhe}.

\subsection{Principal series}\label{app-ps}
\rmrk{Induced representation}\il{indrepr}{induced representation}The
set \il{tP}{$\tP$}$\tP:= \bigl\{
(z,m\pi)\in \tG\;:\; z\in \uhp,\; m\in \ZZ\}$ is a subgroup of $\tG$.
The \il{ps}{principal series}principal series representations of
$\tG$ are obtained by induction from the characters, which can be
written as
\be \ch_{s,r} : (z,m\pi) \mapsto y^s\, e^{-m\pi i r}\,, \ee
with $s\in \CC$, $r\in \CC\mod 2\ZZ$. This leads to the space
\il{Vps}{$\V{}\om[s,r]$}$\V{}\om[s,r]$ consisting of the
real-analytic functions $\tG\rightarrow\CC$ that satisfy $f(gp)
= \ch_{s,r}(p)^{-1}\, f(g)$ with $ p\in \tP$, $g\in \tG$. The action
of $\tG$ by left translation makes $\V{}\om[s,r]$ into a
representation of~$\tG$. The collection $\bigl\{ \V{}\om[s,r]\;:\;
s\in \CC,\; r\in \RR/2\ZZ \bigr\}$ is
called the \emph{principal series} of representations of $\tG$,
depending on the \il{spprm}{spectral parameter}\emph{spectral
parameter} $s\in \CC$ and the \il{cechps}{central
character}\emph{central character} $\tilde k(m\pi)
\mapsto e^{-\pi i m r}$. The superscript $\om$ indicates that, for the
moment, we consider analytic vectors.

The classes of $\tG/\tP$ can be parametrized as $\tilde k(\th)\tP$
with $\th \in \RR\bmod\pi\ZZ$. We can describe the elements of
$\V{}\om[s,r]$ as functions $f:\RR\rightarrow\CC$ that satisfy
$f(\th+\nobreak \pi) = e^{\pi i r}\, f(\th)$. With some work one can
explicitly describe $f\mapsto f\bigm|\widetilde{\matc
abcd}$ in terms of
analytic functions of $\th$ depending on $a$, $b$, $c$, and $d$.

This is not a practical way to work with principal series
representations. We choose $p\in r+2\ZZ$ and relate $f$ as above to
$\ph$ on $\proj\RR$ by $\ph(-\cot\th) = e^{-ip\th}\, f(\th)$. This
leads to a realization of the principal series $\V{}\om[s,r]$ in the
real-analytic functions on $\proj\RR$. We denote this realization by
\il{Vomsp}{$\V{}\om(s,p)$}$\V{}\om(s,p)$, and call it a
\il{prjmps}{projective model}\emph{projective model} of
$\V{}\om[s,r]$. The model depends on the choice of $p\equiv r\bmod
2$. If one carries out the computations one arrives at the following
description of the action\il{proj-model}{$|^\prj_{s,p}$}
\badl{proj-model} \ph|^\prj_{s,p} \tilde k(\pi) \,(t) &\= e^{\pi i r
}\, \ph(t)\qquad(\text{independent of $q\equiv r\bmod 2$})\,,\\
\ph|^\prj_{s,p} \tilde g\, (t) &\= (a+ic)^{-s-p/2}\,(a-ic)^{-s+p/2}\\
&\qquad\hbox{} \Bigl(\frac{t-i}{t-g^{-1}\,i}\Bigr)^{s-p/2}\,
\Bigl(\frac{t+i}{t-g^{-1}(-i)}\Bigr)^{s+p/2}\,
\ph\Bigl(\tfrac{at+b}{ct+d}\Bigr)\,,\\
\eadl
for $g=\matc abcd\in G_0\subset\SL_2(\RR)$, as defined in~\eqref{G0}.

\rmrks \itmi The description in \eqref{proj-model} is complicated. The
factor $ \bigl(\frac{t-i}{t-g^{-1}\,i}\bigr)^{s-p/2}$ is holomorphic
on $\proj\CC$ minus a path in $\uhp$ from $i$ to $g^{-1}i$, and
similarly the factor $\bigl(\frac{t+i}{t-g^{-1}(-i)}\bigr)^{s+p/2}$
is holomorphic on $\proj\CC$ minus a path in $\lhp$ from $-i$ to
$g^{-1}(-i)$. So if $\ph$ is a real-analytic function on $\proj\RR$,
then $\ph|^\prj_{s,p}\tilde g$ is also real-analytic on $\proj\RR$.

\itm Any real-analytic function on $\proj\RR$ is the restriction of a
holomorphic function on some neighbourhood of $\proj\RR$ in
$\proj\CC$. We can view $\V{}\om(s,p)$ as a space on holomorphic
functions on some neighbourhood $U_\ph$ of $\proj\RR$ in $\proj\CC$.
The action $|^\prj_{s,p}$ preserves this space.

\itm We do not have one projective model of $\V{}\om[s,r]$, but
infinitely many. Multiplication by the function $t\mapsto
\bigl(\frac{t-i}{t+i}\bigr)^\ell$, with $\ell\in \ZZ$, gives an
isomorphism
\be\label{Vqchange} \V{}\om(s,r+2\ell) \longrightarrow \V{}\om(s,r)\,.
\ee

\itm The action $|^\prj_{s,p}$ leaves invariant other spaces of
functions on~$\proj\RR$, for instance the $C^\infty$-functions. This
leads to the space \il{V*(s,p)}{$\V{}\ast(s,p)$}$\V{}\infty(s,p)$ of
smooth vectors in the principal series representation. The discussion
in~\cite[\S2]{BLZ13} of the space distribution vectors and
hyperfunction vectors can be applied here, leading to
$\V{}{-\infty}(s,p)$ and $\V{}{-\om}(s,p)$.

\itm All elements of $\V{}\om(s,p)$ can be represented as a sum
\be\label{ephisum} \sum_{\mu \in \ZZ} c_\mu \,\Bigl(
\frac{t-i}{t+i}\Bigr)^\mu\,, \ee
with $c_\mu = \oh(e^{-a|\mu|})$ for some $a>0$. For the larger spaces
$\V{}x(s,p)$ with $x=\infty,-\infty,-\om$, there are similar
descriptions, like in~\cite[(2.18)]{BLZm}, each with a condition on
the growth of the coefficients~$c_\mu$.

\subsubsection{Highest weight subspaces}\label{sect-hwssp}For general
combinations of $s, p\in \CC$ the $\tG$-module $\V{}\om(s,p)$ is
irreducible.
(Reducibility has to be understood as the existence of a \emph{closed}
non-trivial invariant subspace, for the topology on $\V{}\om(s,p)$
that in the projective model is induced by the collection of supremum
norms on the neighbourhoods $U$ of $\proj\RR$ in~$\proj\CC$.)
Reducibility occurs if $2s\equiv p$ or $2s\equiv
-p$ modulo~$2$. For our purpose we consider $2s\equiv -p\bmod 2$. In
view of the isomorphism in~\eqref{Vqchange} we can look at the case
$(s,p)
 = \bigl(1-\nobreak\frac r2,r-\nobreak 2\bigr)$. In that case the
action in~\eqref{proj-model} is given
by\ir{proj-special}{|^\prj_{s,p}}
\badl{proj-special} \ph|^\prj_{1-r/2,r-2} \tilde k(\pi) \,(t) &\=
e^{\pi i r }\, \ph(t)\,,\\
 \ph|^\prj_{1-r/2,r-2} \tilde g\, (t) &\= \,(a-ic)^{r-2}\,
 \Bigl(\frac{t-i}{t-g^{-1}\,i}\Bigr)^{2-r}\,
\ph\Bigl(\tfrac{at+b}{ct+d}\Bigr)\,.\\
\eadl

The factor $\left( (t-i)/(t-g^{-1}i) \right)^{2-r}$ has  
singularities only on a path in $\uhp$ from $i$ to $g^{-1}i$.
Hence $\V{}\om\left(1-\frac r2,r-2\right)$ contains as an invariant
subspace the vectors represented by a holomorphic function on a
neighbourhood of $\lhp\cap \proj\RR$ in~$\proj\CC$. That is just the
\il{projmod}{projective model}projective model
$\Prj{2-r}\dsv{2-r}\om$. Moreover, a comparison
of~\eqref{proj-special} with \eqref{prjact} shows that
$|^\prj_{1-r/2,r-2}\tilde g$ is the same as the operator
$|^\prj_{2-r}g$. In this way, the space $\dsv{2-r}\om$ can be viewed
as an invariant subspace of $\V{}\om\left(1-\frac r2,r-2\right)$.

In the representation \eqref{ephisum} the subspace
\il{ucDsvom}{$\dsv{2-r}\om$}$\Prj{2-r}\dsv{2-r}\om
\subset\V{}\om\bigl(1-\nobreak\frac r2,r-\nobreak 2\bigr)$ is
characterized by $c_\mu=0$ for $\mu>0$. Then the sum represents a a
holomorphic function on a neighbourhood of~$\lhp\cup\proj\RR$
in~$\proj\CC$.

The function $t\mapsto \Bigl( \frac{t-i}{t+i}\Bigr)^\mu$ is an
eigenfunction of $\tilde k(\th)$ with eigenvalue $e^{\pi i
(r+2\mu)}$. One calls $r+2\mu$ the \il{wt-repr}{weight}weight. In
$\Prj{2-r}\dsv{2-r}\om$ only weights $r+2\mu$ with $\mu\leq 0$ occur,
hence the name \il{uchw}{highest weight subspace}\emph{highest
weight} subspace.

We may proceed similarly with the larger representations spaces
$\V{}\infty\bigl(1-\nobreak\frac r2,r-\nobreak 2\bigr) $,
$\V{}{-\infty}\bigl(1-\nobreak\frac r2,r-\nobreak 2\bigr)$, and $
\V{}{-\om}\bigl(1-\nobreak\frac r2,r-\nobreak 2\bigr)$, to obtain
descriptions of the projective models of $\dsv{v,2-r}x$ with
$x=\infty,-\infty,-\om$.

\subsection{Related work}\label{app-lit}The idea to view automorphic
forms as functions on a Lie group is well-known, and has led to wide
generalizations. We have not tried to find the first place where this
idea appears in the literature. To handle automorphic forms of
non-integral weight one has to use a central extension of the Lie
group $\SL_2(\RR)$. For half-integral weights one needs a double
cover, the \il{metapl}{metaplectic group}\emph{metaplectic group}.
See, {\sl e.g.}, Gelbart`s treatment~\cite{Ge76}. For general complex
weights we need the universal covering group~$\tG$. See Selberg
\cite{Se56}, and Roelcke \cite[\S4]{Roe66}.

Covering groups are often described with a $2$-cocycle on $\SL_2(\RR)$
with values in the center, $\ZZ/2\ZZ$ for the metaplectic group,
$\tZ\cong\ZZ$ for~$\tG$. This cocycle turns up naturally in the
description of multiplier systems, even if one does not use the
language of Lie groups. Petersson gives it in~\cite[(11)]{Pe38}, and
Roelcke in~\cite[(1.7)]{Roe66}. We feel more comfortable with the
description of $\tG$ as the space $\uhp\times\RR$ provided with an
analytic group structure. This keeps the $2$-cocycle hidden in the
properties of the lift $g\mapsto\tilde g$.

For all semisimple Lie groups the principal series of representations
is important. In \cite[Chap~II]{Kna86} one finds examples. For the
universal covering group~$\tG$ of $\SL_2(\RR)$ it was developed by
Puk\'anski \cite{Pu64}, since he needed it for function theory
on~$\tG$. Chapter~VII of~\cite{Kna86} discusses the construction of
principal series representations as an induced representation.


\ifcitenumber
\newcommand\bibit[4]{
\bibitem {#1}#2: {\em #3;\/ } #4}
\newcommand\bibitq[4]{
\bibitem {#1}#2: {\em #3\/ } #4}
\else
\newcommand\bibit[4]{
\bibitem[#1] {#1}#2: {\em #3;\/ } #4.}
\newcommand\bibitq[4]{
\bibitem[#1] {#1}#2: {\em #3\/ } #4.}
\fi
\newcommand\bibitn[4]{} 
\raggedright
\iffilename
\renewcommand\subsnb{}
\fi


\newcommand\ind[2]{\item #1\quad\ #2}
\renewcommand\il[1]{\pageref{inl-#1}}
\renewcommand\ir[1]{\pageref{#1}}

\section*{Index}
\begin{multicols}{2}
 \raggedright 
\begin{trivlist}\footnotesize
\ind{analytic boundary germ}{\il{abg}}
\ind{analytic cohomology}{\il{analcoh}}
\ind{argument convention}{\ir{ac}}
\ind{asymptotic series}{\il{asmpts}}
\ind{automorphic form}{\il{af}, \il{ucaftG}}
\ind{---, holomorphic}{\il{afh}}
\ind{---, unrestricted holomorphic}{\il{afuh}}
\ind{automorphic integral}{\il{ai}}
\ind{automorphic transformation behavior}{\il{atb}}
\ind{$\Gm$-average}{\il{avr}}
\indexspace
\ind{bi-invariant differential operator}{\il{bido}}
\ind{boundary germs}{\il{bg}}
\ind{boundary operators}{\il{bo}}
\ind{boundary singularity}{\il{singul}, \il{bs1}}
\indexspace
\ind{Casimir operator}{\il{Casop}}
\ind{central character}{\il{cchu}, \il{cechps}}
\ind{closed geodesic}{\il{clgeo}}
\ind{coboundary}{\il{cob}}
\ind{cocycle}{\il{coc}}
\ind{cofinite}{\il{cft}}
\ind{cohomology group}{\il{cg}}
\ind{complex projective line}{\il{cpl}}
\ind{cusp}{\il{cu}}
\ind{cusp form}{\il{cf}}
\ind{cuspidal triangle}{\il{cutr}}
\ind{cyclic interval}{\il{ci}}
\indexspace
\ind{Dedekind eta-function}{\il{Def}, \il{Deta-pe}}
\ind{discrete subgroup}{\il{ds}, \il{dgu}}
\ind{dual weight}{\il{dw0}}
\indexspace
\ind{edge of a tesselation}{\il{edte}}
\ind{eigenvalue of automorphic form}{\il{ei}}
\ind{elliptic point of order~$2$}{\il{ep}}
\ind{entire automorphic form}{\il{eaf}}
\ind{excised neighbourhood}{\il{en}}
\ind{$E$-excised neighbourhood}{\il{Een}}
\ind{excised semi-analytic vectors}{\il{esav}}
\ind{excised set}{\il{es}}
\ind{exponential decay}{\il{expdecc}}
\ind{exponential map}{\il{em}}
\indexspace
\ind{face of a tesselation}{\il{fate}}
\ind{free space resolvent kernel}{\il{fsresk}}
\ind{Fourier coefficient}{\il{Fc}}
\ind{Fourier expansion}{\il{Fe}}
\ind{fundamental domain}{\il{funddom}}
\indexspace
\ind{Green's form}{\il{Gf}}
\indexspace
\ind{$r$-harmonic}{\il{r-harm}}
\ind{harmonic lift}{\il{hl}, \il{hlc}}
\ind{harmonic automorphic form}{\il{harmaf}}
\ind{harmonic boundary germ}{\il{hbg}}
\ind{highest weight subspace}{\il{uchw}}
\ind{holomorphic automorphic form}{\il{haf}}
\ind{---, unrestricted}{\il{hafu}}
\ind{holomorphic Eisenstein series in weight $2$}{\il{holEis}}
\ind{horocycle}{\il{horc}}
\ind{Hurwitz-Lerch zeta-function}{\il{HuLe}}
\ind{---, asymptotic behavior}{\il{HLzt-as}}
\ind{hyperbolic matrix}{\il{hypelt}}
\ind{hyperbolic measure}{\il{hpm}}
\indexspace
\ind{incomplete gamma-function}{\il{icgf}}
\ind{induced representation}{\il{indrepr}}
\ind{invariants}{\il{invar}, \il{inv}}
\ind{Iwasawa decomposition}{\il{Iwdc}}
\indexspace
\ind{kernel function}{\il{Krkf}}
\indexspace
\ind{$L$-series}{\il{Lser}}
\ind{left-invariant differential operator}{\il{lido}}
\ind{left translation}{\il{uclt}}
\ind{Lie algebra}{\il{La}}
\ind{Lie product}{\il{Lp}}
\ind{lower half-plane}{\il{lhpl}}
\indexspace
\ind{metaplectic group}{\il{metapl}}
\ind{mixed parabolic coboundary}{\il{mpcob1}}
\ind{mixed parabolic cochain}{\il{mpcc}}
\ind{mixed parabolic cocycle}{\il{mpcoc}, \il{mpcoc1}}
\ind{mixed parabolic cohomology group}{\il{mpcg}, \il{mpc1}}
\ind{mock automorphic form}{\il{maf}}
\ind{modular group}{\il{mgr}, \il{mgp}, \il{mguc}}
\ind{module of singularities}{\il{msing}}
\ind{multiplier system}{\il{ms}, \il{ucmsu}}
\ind{--- for a given weight}{\il{mssuit}}
\ind{---, unitary|}{\il{msu}}
\indexspace
\ind{one-sided average}{\il{osa}}
\ind{---, asymptotic behavior}{\il{1sa-as}}
\indexspace
\ind{parabolic cocycle}{\il{pbcoc}}
\ind{parabolic cohomology group}{\il{pchgp}}
\ind{parabolic difference equation}{\il{parbeq}}
\ind{parabolic matrix}{\il{parbm}}
\ind{period function for $\eta^{2r}$}{\il{pfeta2r}}
\ind{$\ld$-periodic}{\il{ldper}, \il{ldperbg}}
\ind{Pochhammer symbol}{\il{Poch}}
\ind{Poincar\'e series, signed hyperbolic}{\il{Pssh}}
\ind{polar expansion}{\il{polexp}}
\ind{polynomial growth}{\il{polgr}}
\ind{---, at the cusps}{\il{polgrc}}
\ind{powers of the Dedekind eta-function}{\il{etapow}, \il{etapow1},
\ir{hmodint}, \il{poweta-q}, \il{etapow-q}}
\ind{primitive hyperbolic}{\il{phe}}
\ind{primitive parabolic}{\il{prpar}}
\ind{principal series}{\il{ps}}
\ind{projective model}{ \il{prjm}, \il{prjmps}, \il{projmod}}
\indexspace
\ind{quantum automorphic form}{\il{qautf}}
\ind{quantum modular form}{\il{qmf}}
\ind{---, strong}{\il{qafs}}
\indexspace
\ind{rational period function}{\il{rpf}}
\ind{real projective line}{\il{rpl}}
\ind{resolvent kernel}{\il{rk}}
\ind{restriction morphism}{\il{rm}}
\ind{right-invariant differential operator}{\il{rido}}
\ind{right translation}{\il{ritr}}
\indexspace
\ind{semi-analytic vector}{\il{sav}}
\ind{---, excised}{\il{save}}
\ind{---, smooth}{\il{savs}}
\ind{--- with simple singularities}{\il{sav-ss}}
\ind{separation of singularities}{\il{sepsing}}
\ind{shadow operator}{\il{so}, \il{shad1}}
\ind{sheaf of $r$-harmonic functions}{\il{shf}}
\ind{sheaf of real-analytic functions}{\il{srafP}}
\ind{signed hyperbolic Poincar\'e series}{\il{shPs}}
\ind{singularity}{\il{singu}}
\ind{singularities, module of}{\il{singm}}
\ind{smooth semi-analytic vector}{\il{ssav}}
\ind{spectral parameter}{\il{spprm}}
\ind{splitting of harmonic boundary germs}{\il{splbg}}
\ind{stabilizer}{\il{stab}}
\ind{strong quantum automorphic form}{\il{sqaf}}
\ind{system of expansions}{\il{soexp}}
\indexspace
\ind{tesselation}{\il{tess}}
\indexspace
\ind{unitary multiplier system}{\il{ums}}
\ind{universal covering group}{\il{ucg}}
\ind{universal enveloping algebra}{\il{uea}}
\ind{unrestricted holomorphic automorphic form}{\il{uhaf}}
\ind{upper half-plane}{\il{uhpl}}
\indexspace
\ind{weight}{\il{wt}, \il{ucwt}, \il{uc1wt}, \il{wt-repr}}
\end{trivlist}
\end{multicols}

\section*{List of notations}
\renewcommand\ind[2]{\item $#1$\quad\ #2}
\begin{multicols}{3}
 \raggedright
\begin{trivlist}\footnotesize
\ind{A_r(\Gm,v)}{\il{Ar}}
\ind{A_r^0(\Gm,v)}{\ir{A0}}
\ind{A^\E_r(\Gm,v)}{\ir{AE}}
\ind{\qA_{2-r}(\Gm,v)}{\ir{qA}}
\ind{\avGm{v,r}}{\ir{avG}}
\ind{\av{T,\ld}}{\ir{avTdef}}
\ind{\av {T,\ld}^+,\; \av{T,\ld}^-}{\ir{avpm}, \il{osa1}}
\ind{\av{\pi,\ld}^\pm}{\ir{avpidef}}
\ind{a_n(\ca,F)}{\il{anxi}}
\indexspace
\ind{\bg_r}{\ir{bdg}, \ir{bg-decomp}}
\ind{B^1(\Gm;V)|}{\il{B1}}
\ind{ B^i(F_\pnt^\tess;V,W)}{\il{BiFtess}}
\ind{\bsing}{\il{bsing}, \il{bsing1}}
\indexspace
\ind{\cu $ set of cusps$ }{\il{Scu}}
\ind{C^i(F^\tess_\pnt;V,W)}{\ir{CiFtess}}
\ind{\qC}{\ir{Cdef}}
\ind{c_F}{\ir{cF}}
\ind{c^{z_0}_F}{\ir{cz0-def}}
\ind{\tilde c $ lift of cocycle $ c}{\il{tldc}}
\ind{c(p,a;x)}{\ir{cpa-def}}
\ind{c_0(a;x)}{\ir{cpb-def-c}}
\ind{\tilde c(p,a;\cdot)}{\il{tldcpa}}
\ind{c(e)_\xi}{\il{cexxi}}
\ind{c_r}{\il{crt}, \il{cr}}
\indexspace
\ind{\dsv{2-r}\om,\; \dsv{2-r}\infty,\; \dsv{2-r}{-\infty},\;
 \dsv{2-r}{-\om}}{\il{dsv*}, \il{ucDsvom}}
\ind{\dsv{v,2-r}\om,\; \dsv{v,2-r}\infty,\; \dsv{v,2-r}{-\infty},\;
 \dsv{v,2-r}{-\om}}{\il{dsv-v2-r}}
\ind{\dsv{2-r}\pol,\; \dsv{v,2-r}\pol}{\il{dsvpol}}
\ind{\dsv{2-r}\om[\xi_1,\ldots,\xi_n],\; \dsv{2-r}\fsn,\;
\dsv{2-r}\fs}{\ir{dsvomxi}}
\ind{\dsv{2-r}{\om,\infty}[\xi_1,\ldots,\xi_n]}{\il{dsvominf}}
\ind{ \dsv{2-r}{\om,\smp}[\xi_1,\ldots,\xi_n]}{\ir{smpdef}}
\ind{ \dsv{2-r}{\om,\wdg}[\xi_1,\ldots,\xi_n]}{\ir{dsvwdg}}
\ind{\dsv{2-r}{\fsn,\cond},\; \dsv{2-r}{\fs,\cond}}{\ir{fs-fsn-cond}}
\ind{\dsv{2-r}{\fsn,\cond_1,\cond_2},\; 
 \dsv{2-r}{\fs,\cond_1,\cond_2}}{\il{combcond}}
\ind{\dsv{v,2-r}{\om,\cond},\; \dsv{v,2-r}{\fsn,\cond},\;
\dsv{v,2-r}{\fs,\cond}}{\il{gmactfsfsn}}
\ind{D_p}{\ir{Dp}}
\ind{D_\e^\pm}{\ir{Depspm}}
\ind{D(\xi)}{\ir{Dxi}}
\ind{\partial_i}{\il{pai}}
\ind{d,\; d^i}{\il{deri}}
\indexspace
\ind{\esv r \om,\; \esv{v,r}\om}{\ir{esv-gen}, \ir{esv-iw}}
\ind{\esv{r}{\om,\wdg}[\xi_1,\ldots,\xi_n],\;
\esv{r}{\om,\wdg}[]}{\il{esv-gen[]}, \il{esv-iw[]}, \il{esv[]}}
\ind{\esv r {\fsn,\wdg},\; \esv{r}{\fs,\wdg}}{\ir{esv*}}
\ind{\esv {v,r}{\fsn,\wdg},\; \esv {v,r}{\fs,\wdg}}{\il{esvfsn}}
\ind{\EE^+,\; \EE^-}{\il{EEpm}}
\ind{E_2,\; E_s^\ast}{\ir{E2}}
\ind{e_\ca}{\il{eca}}
\indexspace
\ind{\fd}{\il{fd}}
\ind{\fd_Y}{\il{fdY}}
\ind{\F{r,n}}{\ir{Frndef}}
\ind{F_i^\tess}{\il{Fitess}}
\ind{F_i^{\tess,Y}}{\il{FitessY}}
\ind{F(\mu,\cdot)}{\ir{pol-term}}
\ind{f_r}{\ir{frdef}}
\ind{f_\ca}{\il{fca}}
\indexspace
\ind{\Gr r \ast}{\ir{Gromfs}}
\ind{\Gr{v,r}\ast}{\il{Grvr}}
\ind{\Gr r {\om,\exc}[\xi_1,\ldots,\xi_n]}{\il{Gr[}}
\ind{G_0}{\ir{G0}}
\ind{\tG $ universal covering group$ }{\il{tG}}
\ind{\glie_r}{\ir{glie}}
\ind{\glie}{\il{gliec}}
\indexspace
\ind{\sharm_r}{\il{sharm}}
\ind{\sharmb r}{\ir{sharmb}}
\ind{\sharmh r}{\ir{sharh}}
\ind{\harm_r(\Gm,v)}{\il{hn}}
\ind{\uhp,\;\lhp}{\il{ulhp}}
\ind{\H{r,\mu}}{\ir{Hrmu}, \il{Hrmucu}}
\ind{H(s,a,z)}{\ir{HLzt}}
\ind{H^1(\Gm;V)}{\ir{1coh}}
\ind{\hpar^1(\Gm;V,W)}{\il{parbc}, \il{hpar1mpc}}
\ind{\hpar^1(\Gm;V)}{\il{hpar}}
\indexspace
\ind{\I r \ld}{\il{Irld}}
\ind{I(r,s)}{\ir{Irs}}
\indexspace
\ind{K_r(\cdot;\cdot)}{\ir{Kq}, \il{Kr1}}
\ind{K_{v,r}}{\il{Kvr}}
\ind{k(\th)=\matr{\cos\th}{\sin\th}{-\sin\th}{\cos\th}}{\il{k(.)}}
\ind{\tilde k(\th) \in \tG}{\il{tldk}}
\indexspace
\ind{ L:g\mapsto  L_g = f|g }{\il{Lg}}
\ind{L_\XX}{\il{LLie}}
\ind{L\bigl(\eta^{2r},s\bigr)}{\ir{L-etap}}
\ind{\ell(\gm)}{\il{elleta}}
\indexspace
\ind{\M{r,\mu}}{\ir{mex}, \ir{Mrmu}, \il{Mrmucu}}
\ind{M_r(\Gm,v)}{\il{Mdef}}
\indexspace
\ind{\N r \om,\; \N r {\fs,\wdg}}{\il{Nr}}
\ind{N_R(\cdot)}{}
\indexspace
\ind{\hol}{\il{holdef}}
\ind{\hol_\uhp}{\il{holuhp}}
\indexspace
\ind{\Pcal=\Ci\dsv{2-r}{-\infty}}{\ir{Pkr}}
\ind{\proj\CC,\; \proj\RR}{\il{proj}}
\ind{\P{r,\mu}}{\ir{Prmu}, \il{Prmucu}}
\ind{\tP\subset\tG}{\il{tP}}
\ind{\pr:\tG \rightarrow\PSL_2(\RR)}{\ir{pr}}
\ind{\Prj{2-r}}{\ir{Prj}, \il{Prj1}}
\ind{P_\ca}{\il{Pinf}}
\ind{p_k(r)}{\il{pk-eta}}
\ind{p_r(z;\tau)}{\ir{prkernel}}
\ind{p_\ph}{\ir{pphi-def}}
\indexspace
\ind{\Q_{v,2-r}}{\il{Qv}}
\ind{Q}{\il{Q}}
\ind{Q_F}{\ir{QF}}
\ind{Q_r(\cdot;\cdot)}{\ir{Qr-def}, \il{Qruc}}
\ind{\bcoh r \om}{\il{bcoh}, \il{bcoh1}}
\ind{q(\cdot)}{\il{q}}
\indexspace
\ind{\R}{\il{Rqaf}}
\ind{\R_{v,2-r}}{\il{Rv}}
\ind{R_r f}{\ir{Rrf}}
\ind{R_g}{\ir{rt}}
\ind{R_\XX}{\il{RLie}}
\ind{\coh r \om}{\il{cohrom-thm}, \il{cohom}, \il{coh1}}
\ind{\coh r \infty}{\ir{rinf}}
\ind{r $ weight $ }{\il{rwt}}
\indexspace
\ind{\Sg{v,2-r}{},\; \Sg{v,2-r,\xi}{}\; \Sg{v,2-r}{\ast,\ast}}{\ir{Sg}}
\ind{\Sg\xi{}}{\il{Sg-xi}}
\ind{\Sg{}{}\{x\}}{\ir{Sg-orb}}
\ind{\cusp r(\Gm,v)}{\il{Scf}}
\ind{\singr}{\il{sing}}
\ind{S=\matr0{-1}10}{\il{Smat}}
\indexspace
\ind{\tess}{\il{te}}
\ind{T=\matc 1101}{\il{Tmat}}
\indexspace
\ind{\mathcal U}{\il{Uenv}}
\ind{u(C;z), \; u(\tilde c;z),\; u([c];z)}{\ir{uc}}
\indexspace
\ind{\V{2-r}\om, \V{v,2-r}\om}{\ir{Vtmrom}}
\ind{\V{}\om[s,r]}{\il{Vps}}
\ind{\V{}\ast(s,p)}{\il{Vomsp},  \il{V*(s,p)}}
\ind{V_\ca}{\il{Vca}, \il{Va}}
\ind{v $ multiplier system$ }{\il{vnc}}
\ind{v_\ch}{\il{vch}}
\ind{v_r}{\ir{vr-mod}, \il{ucvr}}
\indexspace
\ind{\W r \om,\; \W{v,r}\om}{\ir{Wdef}}
\ind{\Whol r \om}{\ir{Whol}, \il{Whol1}}
\ind{w=\frac{z-i}{z+i}}{\il{pc}}
\indexspace
\ind{X^\tess_i}{\il{Xtess}}
\ind{X^{\tess,Y}_i}{\il{XiY}}
\ind{x=\re z}{\il{xulhp}}
\indexspace
\ind{y=\im z}{\il{yulhp}}
\indexspace
\ind{\tZ $ center of $\tG}{\il{tldZ}}
\ind{Z^1(\Gm;V)}{\il{Z1}}
\ind{\zpar^1(\Gm;V,W)}{\ir{zparb}}
\ind{ Z^i(F_\pnt^\tess;V,W)}{\il{ZiFtess}}
\ind{Z^i(F_\pnt^{\tess,Y};V)}{}
\ind{(z,\th)\,\in \tG}{\il{Ic}}
\indexspace
\indexspace
\ind{\al_r}{\ir{alr}}
\ind{\al_\ca}{\il{alca}}
\indexspace
\ind{\Gm $ cofinite discrete group$ }{\il{Gm}}
\ind{\bar \Gm = \Gm\bigm/\{1,-1\}}{\il{bGm}}
\ind{\tGm \subset\tG}{\ir{tGm}}
\ind{\Gmod $ modular group$ }{\il{Gmod}, \il{Gmodgen}}
\ind{\tGmod}{\il{tGmod}}
\ind{\Gm_\ca $ stabilizer of $\ca }{\il{stabxi}, \il{Gmca}}
\indexspace
\ind{\Dt_r}{\ir{Dtr}}
\indexspace
\ind{\eps(x,\gm^{-1}p)}{\ir{epsdef}, \il{eps1}}
\ind{\eps_P(e,p)}{\il{epsP}}
\indexspace
\ind{\eta^{2r}}{\ir{eta2r}}
\indexspace
\ind{\Ci $ involution$ }{\ir{Ci}}
\indexspace
\ind{\k=\k_{v,2-r,\gm}}{\il{kp-hyp}}
\indexspace
\ind{\shad_r}{\ir{shad}, \il{sh1}}
\indexspace
\ind{\rs_r}{\ir{rsl}}
\ind{\rsp_r}{\il{restr}, \il{res1}}
\indexspace
\ind{\pi_\ca}{\il{pica}}
\indexspace
\ind{\s_\ca}{\il{gxi}}
\indexspace
\ind{\ch_r}{\ir{uc-chr}}
\indexspace
\ind{\Psi_\Eta,\; \tilde\Psi_\Eta}{\ir{PsiEta}}
\ind{\ps^{z_0}_F}{\ir{psiz0def}, \il{psiz0prop}}
\ind{\ps^\ca_F}{\ir{psparb}}
\ind{\ps_F}{\ir{psF}}
\ind{\ps:\gm\mapsto \ps_\gm \text{ cocycle}}{}
\indexspace
\ind{\om}{\ir{omdef}}
\ind{L_\om=R_\om}{\ir{Cas}}
\ind{\om_r(F;t,z)}{\ir{omr}}
\ind{\om_r^\prj(F;t,z)}{\ir{omr-prj}}
\indexspace
\indexspace
\ind{V^\Gm,\; V^\gm $
(invariants)$ }{\il{VGm}, \il{Vinvgm}}
\ind{W[\xi]}{\il{W[]}}
\ind{f\mapsto f|_r}{\ir{SL2op}}
\ind{f\mapsto f|_{v,p}}{\il{vq-act}, \il{actGm-qaf}}
\ind{f\mapsto f|g}{\ir{Lt}}
\ind{f\mapsto f|^\prj_{s,p}g}{\il{proj-model}, \ir{proj-special}}
\ind{(x)|\gm}{\il{x|g}}
\ind{\XX f \;=R_\XX f}{\il{XXf}}
\ind{\ph^\prj\mapsto \ph^\prj|^\prj_{2-r}g}{\ir{prjact}}
\ind{[\cdot,\cdot]_r}{\ir{Gr-f}}
\ind{[\cdot,\cdot]}{\il{[.,.]}}
\ind{(\xi,\eta)_{\mathrm{cycl}} $ cyclic interval in
$ \proj\RR}{\il{cyclint}}
\ind{(a)_m}{\il{Ps}}
\ind{\sim}{\ir{asser}, \il{sim1}}
\ind{\stackrel\pnt=}{\il{pt=}}
\ind{g\mapsto \tilde g\;:\; \SL_2(\RR) \rightarrow\tG}{\ir{tildeg}}
\end{trivlist}
\end{multicols}

\end{document}